\documentclass[a4paper, reqno, titlepage, twoside, 11pt]{article}
\usepackage[utf8]{inputenc}
\usepackage{fancyhdr}
\usepackage{amsthm}
\usepackage{amssymb}
\usepackage[american]{babel}
\usepackage{exscale}
\usepackage[mathscr]{euscript}    
\usepackage{url}
\usepackage[all,ps,tips,tpic]{xy}
\usepackage{color}
\usepackage{bbm}
\usepackage{bbold}   
\usepackage{verbatim}
\usepackage{a4}
\usepackage{mathdots}
\usepackage{fancyhdr}   
\usepackage{mathrsfs}
\usepackage{MnSymbol}
\fancypagestyle{plain}{
\fancyhf{}
\cfoot{\thepage}
 
  }

\newtheorem{lem}{Lemma}%[section]

\newtheorem{definition}[lem]{Definition}

\newtheorem{cor}[lem]{Corollary}

\newtheorem{prop}[lem]{Proposition}

\newtheorem*{THM}{Theorem}
\newtheorem*{Lemma}{Lemma}
\newtheorem*{Definitionwn}{Definition}
\newtheorem*{Proposition}{Proposition}
\newtheorem*{Corollary}{Corollary}

\theoremstyle{remark}

\DeclareMathOperator{\id}{id}
\DeclareMathOperator{\Name}{Name}
\DeclareMathOperator{\dom}{dom}

\DeclareMathOperator{\Ord}{Ord}
\DeclareMathOperator{\OR}{OR}

\DeclareMathOperator{\cf}{cf}

\DeclareMathOperator{\Card}{Card}

\DeclareMathOperator{\Succ}{Succ}
\DeclareMathOperator{\Lim}{Lim}
\DeclareMathOperator{\supp}{supp}
\DeclareMathOperator{\rg}{rg}

\DeclareMathOperator{\rk}{rk}
\DeclareMathOperator{\Aut}{Aut}

\newcommand{\uhr}{\hspace*{-0,5mm} \upharpoonright \hspace*{-0,5mm} }

\newcommand{\m}{\mathbb}
\newcommand{\tbl}{\textquotedblleft}
\newcommand{\tbr}{\textquotedblright}

\newcommand{\ol}{\overline}
\newcommand{\wt}{\widetilde}

\begin{document}

\title{An Easton-like Theorem for $ZF\, +  DC$}
%Oder: Easton's theorem in $ZF\, + \, DC$ for set-many cardinals
\begin{center} \Large  \bfseries An Easton-like Theorem for Zermelo-Fraenkel Set Theory with the Axiom of Dependent Choice
%$\mathbf{ZF\, \boldsymbol{+}\, DC}$ 
\end{center}

\begin{center} { \scshape Anne Fernengel and Peter Koepke } \end{center}
\vspace*{5mm}

\begin{center} \bfseries Abstract \end{center}
\begin{quote} \small We show that in the theory $ZF\, \plus\, DC\, \plus$ \tbl for every cardinal $\lambda$, the set $[\lambda]^{\aleph_0}$ is well-ordered\tbr ($AX_4$), the $\theta$-function measuring the surjective size of the powersets $\powerset (\kappa)$ can take almost arbitrary values on any set of uncountable cardinals. This complements our results from \cite{arxiv}, where we prove that in $ZF$ (without $DC$), any possible behavior of the $\theta$-function can be realized; and answers a question of Shelah in \cite{AX4}, where he emphasizes that $ZF\, \plus\, DC\, \plus \, AX_4$ is a \tbl reasonable\tbr\,theory, where much of set theory  and combinatorics is possible.

%Our constructed model $N$ is a symmetric extension by a countably closed forcing $\m{P}$ which generalizes the forcing notion introduced in [GK].

%\colorbox{yellow}{FRAGE: Den letzten Satz weglassen?}
\end{quote}

\vspace*{3mm}

%Darauf eingehen, dass man hier ohne große Kardinalzahlen arbeiten möchte?

%Den \tbl alten Textb\tbr nochmal ansehen!

%auf [GK] eingehen

%FRAGE: irgendwo das Forcing im Detail durchgehen, und auf den Fehler eingehen? Erst in der Arbeit, oder auch hier? Vll. im Kapitel  \tbl The forcing\tbr ?

%$DC$ formulieren, und erwähnen, dass aus $DC$ das abzählbare Auswahlaxiom $AC_\omega$ folgt? Erwähnen, dass man $DC$ erreichen könnte, durch abzählbar-abgeschlossenes Forcing mit abzählbar-abgeschlossenem Filter?

%Man müsste erwähnen, dass man nur für überabzählbare Kardinalzahlen hier zurecht käme... 
%das Argument durchgehen, wieso $DC$ im \tbl klassischen Forcing \tbr\, mit endlich erzeugtem Filter nicht gilt? 
%UNKLAR ist immer noch, wieso nicht Isomorphismen und Filter abzählbaren Support haben könnten! ???

%Auf die Konstruktion eingehen? Forcing, Isomorphsimen, normaler Filter, symmetrisches Modell?

%\section{Basic Notations and Facts about Symmetric Forcing Extensions
\paragraph{Introduction.}
The \textbf{Continuum Function} $\kappa \mapsto 2^\kappa$, which maps any cardinal $\kappa$ to the cardinality of its power set $\powerset (\kappa)$, has been investigated since the early beginnings of set theory. In 1878, Georg Cantor advanced the \textbf{Continuum Hypothesis (CH)} (\cite{cantor}), which states that $2^{\aleph_0} = \aleph_1$, i.e. there is no set the cardinality of which is strictly between $\aleph_0$ and the cardinality of $\powerset(\aleph_0)$. The Continuum Hypothesis was among the first statements that were shown to be independent of $ZF$: Firstly, Kurt Gödel proved in \cite{goedel} that $CH$ holds in the constructible universe $L$. On the other hand, when Paul Cohen invented the method of forcing in \cite{cohen1} and \cite{cohen2}, he proved that $2^{\aleph_0}$ could be any cardinal $\kappa$ of uncountably cofinality. \\[-3mm]

The \textbf{Generalized Continuum Hypothesis (GCH)} was formulated by Felix Hausdorff in \cite{Hausdorff1914} (with earlier versions in \cite{Hausdorff1907} and \cite{Hausdorff1908}). Asserting that $2^\kappa = \kappa^\plus$ for all cardinals $\kappa$, it is a \textit{global} statement about possible behaviors of the Continuum Function. The $GCH$ is consistent with $ZFC$, since it holds true in $L$ (see \cite{goedel}). 
In 1970, William Easton proved the following global result: For \textit{regular} cardinals $\kappa$, any reasonable behavior of the $2^\kappa$-function is consistent with $ZFC$ (\cite{Easton}). In his forcing construction, he takes \tbl many\tbr\, Cohen forcings and combines them in a way that was henceforth known as the \textit{Easton product}. \\[-3mm]

For singular cardinals $\kappa$, however, the situation is a lot more involved, since the value of $2^\kappa$ for singular $\kappa$ is strongly influenced by the behavior of the Continuum Function below. The \textbf{Singular Cardinals Hypothesis (SCH)} implies that for any singular cardinal $\kappa$ with the property that $2^\lambda < \kappa$ holds for all $\lambda < \kappa$, it already follows that $2^\kappa = \kappa^\plus$. \\[-3mm]

It turned out that the negation of the $SCH$ is tightly linked with the existence of large cardinals. Among the first results in this direction was a theorem by Menachem Magidor (\cite{scp1} and \cite{scp2}), who proved that, assuming a huge cardinal, it is possible that $GCH$ first fails at a singular strong limit cardinal. On the other hand, Ronald Jensen and Keith Devlin proved in \cite{marginaliatts} that the negation of $0^\sharp$ implies $SCH$. Motik Gitik determined in \cite{G1} and \cite{G2} the consistency strength of $\neg SCH$ being the existence of a measurable cardinal $\lambda$ of Mitchell order $\sigma (\lambda) = \lambda^{\plus \plus}$. \\[-3mm]

There are many more results about possible behaviors of the Continuum Function starting from large cardinals. For instance, by a theorem of Carmi Merimovich, the theory $ZFC\, \plus\, \forall \, \kappa\ (2^\kappa = \kappa^{\plus n})$ is consistent for each $n < \omega$ (\cite{cm}). \\[-2mm]

On the other hand, \textit{Silver's Theorem} (\cite{Silver}) states that for any singular cardinal $\kappa$ of uncountable cofinality such that $2^\lambda = \lambda^\plus$ holds for all $\lambda < \kappa$, it already follows that $2^\kappa = \kappa^\plus$. Hence, the $SCH$ holds if it holds for all singular cardinals of countable cofinality.
This result was extended by Fred Galvin and Andr\'{a}s Hajnal shortly after (\cite{galvinhajnal}).
%extended this result and proved that for singular cardinals $\kappa$ of uncountably cofinality, the value of $2^\kappa$ strongly depends on the values $2^\nu$ for $\nu < \kappa$.
 
Another prominent example concerning upper bounds on the Continuum Function of singular cardinals is the following theorem by Saharon Shelah (\cite{Shelah}): \\[-4mm]

\textit{If $2^{\aleph_n} < \aleph_\omega$ for all $n < \omega$, then $2^{\aleph_\omega} < \aleph_{\omega_4}$. }\\[-2mm]

This makes clear that there are significant constraints on possible behaviors of the Continuum Function in $ZFC$. In particular, a result like Easton's Theorem can not exist for singular cardinals. \\[-2mm]

All of the aforementioned results essentially involve the Axiom of Choice. Without $AC$, however, there is a lot more possible. In \cite{AK}, Arthur Apter and Peter Koepke examine the consistency strength of the negation of $SCH$ in $ZF\, \plus\, \neg AC$. In this context, one has to distinguish between \textit{injective} and \textit{surjective} failures. An \textit{injective failure of $SCH$ at $\kappa$} is a model of $ZF\, \plus\, \neg AC$ with a singular cardinal $\kappa$ such that $GCH$ holds below $\kappa$, but there is an injective function $\iota: \lambda \hookrightarrow \powerset (\kappa)$ for some $\lambda \geq \kappa^{\plus \plus}$. A \textit{surjective failure of $SCH$ at $\kappa$} is a model of $ZF\, \plus\, \neg AC$ with a singular cardinal $\kappa$ such that $GCH$ holds below $\kappa$, but there is a surjective function $f: \powerset (\kappa) \rightarrow \lambda$ for some cardinal $\lambda \geq \kappa^{\plus \plus}$.

On the one hand, Arthur Apter and Peter Koepke construct injective failures of the $SCH$ at $\aleph_\omega$, $\aleph_{\omega_1}$ and $\aleph_{\omega_2}$ that would contradict the theorems by Shelah and Silver in the $ZFC$-context, but have fairly mild consistency strengths in $ZF\, \plus\, \neg AC$. On the other hand, regarding a surjective failure of the $SCH$, they prove that for every $\alpha \geq 2$, $ZFC$ together with the existence of a measurable cardinal is equiconsistent with the theory \[ZF\; \plus\; \neg AC\; \plus\; \mbox{\tbl GCH holds below } \aleph_\omega\tbr\; \plus\, \]\[ \plus\, \mbox{ \tbl there exists a surjective function } f: [\aleph_\omega]^\omega \rightarrow \aleph_{\omega + 2}\tbr.\] 
%For their construction, they introduce a new notion of parallel Prikry forcing. \\[-3mm]

It follows that also without the Axiom of Choice, injective failures of the $SCH$ are inevitably linked to large cardinals. Regarding surjective failures however, it is not possible to replace in their argument the surjective function $f: [\aleph_\omega]^\omega \rightarrow \aleph_{\omega + 2}$ by a surjection $f: \powerset (\aleph_\omega) \rightarrow \aleph_{\omega + 2}$, so the following question remained: Is it possible, for $\lambda \geq \aleph_{\omega + 2}$, to construct a model of $ZF\, \plus\, \neg AC$ where $GCH$ holds below $\aleph_\omega$ and there is a surjection $f: \powerset(\aleph_\omega) \rightarrow \lambda$ without any large cardinal assumptions?\\[-3mm]

%\colorbox{yellow}{FRAGE: $s$ oder $f$ für Surjektionen?}

This question was positively answered by Motik Gitik and Peter Koepke in \cite{GK}, where a ground model $V \vDash ZFC\, \plus \, GCH$ with a cardinal $\lambda \geq \aleph_{\omega + 2}$ is extended via symmetric forcing, in a way such that the extension $N = V(G)$ preserves all $V$-cardinals, the $GCH$ holds in $N$ below $\aleph_\omega$, and there is a surjective function $f: \powerset (\aleph_\omega) \rightarrow \lambda$. \\[-3mm]

More generally, in the absence of the Axiom of Choice, where the power set $\powerset (\kappa)$ of a cardinal $\kappa$ is generally not well-ordered, the \tbl size\tbr\, of $\powerset(\kappa)$ can be measured surjectively by the $\theta$-function
\[\theta(\kappa) := \sup \{\alpha \in \Ord\ | \ \exists\, f: \powerset (\kappa) \rightarrow \alpha \mbox{ surjective function} \},\] which generalizes the value $\theta := \theta (\omega)$ prominent in descriptive set theory. In the $\neg AC$-context, the $\theta$-function provides a surjective substitute for the Continuum Function $\kappa \mapsto 2^\kappa$. \\[-3mm]

One can show that in the model constructed in \cite{GK}, it follows that indeed, $\theta (\aleph_\omega) = \lambda$. This gives rise to the question whether the behavior of the $\theta$-function might be essentially undetermined in $ZF$. \\[-2mm]

In \cite{arxiv}, we could prove that indeed, the only constraints on the $\theta$-function in $ZF$ are the obvious ones: weak monotonicity and $\theta(\kappa) \geq \kappa^{\plus \plus}$ for all $\kappa$. In other words: In $ZF$, there is an analogue of Easton's Theorem for regular \textit{and} singular cardinals.

\begin{THM}(\cite{arxiv}) Let $V$ be a ground model of $ZFC \, + \, GCH$ with a function $F$ on the class of infinite cardinals such that the following properties hold: \begin{itemize} \item $\forall \kappa\ F(\kappa) \geq \kappa^{\plus \plus}$ \item $\forall\, \kappa, \lambda \ \big(\kappa \leq \lambda \rightarrow F(\kappa) \leq F(\lambda)\big)$. \end{itemize} Then there is a cardinal-preserving extension $N \supseteq V$ with $N \vDash ZF$ such that $\theta^N (\kappa) = F(\kappa)$ holds for all $\kappa$. 
\end{THM}

In our construction, we introduce a forcing notion $\m{P}$ whose elements $p$ are functions on trees $(t, \leq_t)$ with finitely many maximal points. The trees' levels are indexed by cardinals, and on any level $\kappa$, there are finitely many vertices $(\kappa, i)$ with $i < F(\kappa)$. For successor cardinals $\kappa^\plus$, the value $p(\kappa^\plus, i)$ is a partial $0$-$1$-function on the interval $[\kappa, \kappa^\plus)$. Thus, for any condition $p$ and $(\kappa, i) \in \dom p$, it follows that $\bigcup \{ p(\nu^\plus, j)\ | \ (\nu^\plus, j) \leq_t (\kappa, i) \}$ is a partial function on $\kappa$ with values in $\{0, 1\}$. Since we do not allow splitting at limits for the trees, it follows that this forcing indeed adds $F(\kappa)$-many new $\kappa$-subsets for every cardinal $\kappa$.

We discussed in \cite[Chapter 6]{arxiv} whether it might be possible to modify our construction and use trees with countably many maximal points, which could result in a countably closed forcing, giving rise to a symmetric extension $N$ with $N \vDash ZF \, \plus\, DC$. However, this modification would have a drastic impact on the forcing notion, destroying a crucial homogeneity property.
%, which most arguments in our proof were based on.
Seemingly, our construction relies on certain finiteness properties, and hence does not opt for a symmetric extension $N$ with $N \vDash DC$. \\[-2mm]

In this paper, we treat the question whether the $\theta$-function is still essentially undetermined, if we consider a model $N \vDash ZF\, \plus \, DC$. Starting from a ground model $V \vDash ZFC\, \plus\, GCH$, we construct a cardinal-preserving symmetric extension $N \supseteq V$ with $N \vDash ZF\, \plus \, DC$, and this time, we generalize the forcing in [GK] to obtain a countably closed forcing notion $\m{P}$. \\[-2mm]

%, if we want $N \vDash ZF\, \plus \, DC$.
The \textbf{Axiom of Dependent Choice (DC)}, introduced by Paul Bernays in 1942 (\cite{Bernays}), states that for every nonempty set $X$ with a binary relation $R$ such that for all $x \in X$ there is $y \in X$ with $y R x$, it follows that there is a sequence $(x_n\ | \ n < \omega)$ in $X$ such that $x_{n + 1} R x_n$ for all $n < \omega$. \\[-3mm]

When dealing with real numbers, surprisingly often $DC$ is sufficient (instead of the full Axiom of Choice), and the theory $ZF\, \plus\, DC$ provides an interesting framework for real analysis.

Concerning combinatorial set theory however, investigations under $ZF\, \plus\, DC$ seemed rather hopeless in the first place. A crucial step in the other direction was a paper by Saharon Shelah (\cite{settheorywithoutchoice}) with the main result in $ZF\, \plus\, DC$ that whenever $\mu$ is a singular cardinal of uncountable cofinality such that $|H (\mu)| = \mu$, then $\mu^\plus$ is regular and non-measurable. In the case that the power sets $\powerset(\alpha)$ are well-orderable for all $\alpha < \aleph_{\omega_1}$ with $| \bigcup_{\alpha < \aleph_{\omega_1}} \powerset (\alpha)| = \aleph_{\omega_1}$, it essentially follows that also $\powerset(\aleph_{\omega_1})$ is well-orderable. 

%Regarding our question, this implies that whenever we want to set $\theta (\aleph_{\omega_1})$ in dependently of the values $\theta (\alpha)$ for $\alpha < \aleph_{\omega_1}$, we have to make sure that in our construction, also the well ordering of $\bigcup_{\alpha < \aleph_{\omega_1}} \powerset (\alpha)$ is destroyed -- otherwise, a surjective failure of $SCH$ at $\aleph_{\omega_1}$ would come along with an \textit{injective} failure of $SCH$ at $\aleph_{\omega_1}$, which would necessarily involve large cardinals by [AK]. \\[-3mm]

Subsequently (see \cite{pcfwc} and \cite{AX4}), Shelah showed that much of $pcf$-theory is possible in $ZF\, \plus\, DC$, if an additional axiom is adopted: \[ (AX_4)\ \ \mbox{ \textit{For every cardinal}  } \lambda, \mbox{ \textit{the set }} [\lambda]^{\aleph_0}\mbox{ \textit{can be well-ordered}.}\]

%In [pcfwc, 0.4] he points out that, considering models where the Axiom of Choice fails, adopting $AX_4$ is rather working in the opposite direction as investigating $L[\mathbb{R}]$, where we lack essentially a well-ordering of $\powerset(\omega)$. 
Starting from a ground model $V \vDash ZFC$, any symmetric extension by countably closed forcing yields a model of $ZF\, \plus\, DC\, \plus\, AX_4$ (see \cite[p.3 and p.15]{pcfwc}). In \cite[0.1]{AX4}, Shelah concludes that $ZF\, \plus\, ZF\, \plus\, AX_4$ is \tbl \textit{a reasonable theory, for which much of combinatorial set theory can be generalized}\tbr. For example, he proves a rather strong version of the $pcf$-theorem, gives a representation of $\lambda^\kappa$ for $\lambda > > \kappa$, and proves that certain covering numbers exist. Concerning applications to cardinal arithmetic, Shelah emphasizes that we \tbl cannot say much\tbr\,on possible cardinalities of $\powerset (\kappa)$, and suggests to investigate possible cardinalities of $\big(\kappa^{\aleph_0}\ | \ \kappa \in \Card\big)$ rather than $\big(\powerset (\kappa) \ | \ \kappa \in \Card\big)$ (\cite[p. 2]{pcfwc}). In \cite[0.2]{AX4} he asks, referring to \cite{GK}, if there are any bounds on $\theta (\kappa)$ for singular cardinals $\kappa$ in $ZF\, \plus\, DC\, \plus\, AX_4$.  \\[-3mm]

In this paper, we give a negative answer to this question. We prove that in $ZF\, \plus\, DC\, \plus\, AX_4$ again, the only restrictions on the $\theta$-function on a set of uncountable cardinals are the obvious ones: \\[-3mm]

Given a ground model $V \vDash ZFC\, \plus\, GCH$ with \tbl reasonable\tbr\, sequences of uncountable cardinals $(\kappa_\eta\ | \ \eta < \gamma)$ and $(\alpha_\eta\ | \ \eta < \gamma)$ for some ordinal $\gamma$, we construct a cardinal-preserving symmetric extension $N \supseteq V$ with $N \vDash ZF\, \plus\, DC\, \plus\, AX_4$, such that $\theta^N (\kappa_\eta) = \alpha_\eta$ holds for all $\eta < \gamma$. \\[-2mm]

More precisely, we prove:

\begin{THM} Let $V$ be a ground model of $ZFC \, \plus \, GCH$ with $\gamma \in \Ord$ and sequences of uncountable cardinals $(\kappa_\eta\ | \ \eta < \gamma)$ and $(\alpha_\eta\ | \ \eta < \gamma)$, such that $(\kappa_\eta\ | \ \eta < \gamma)$ is strictly increasing and closed, and the following properties hold:

\begin{itemize} \item $\forall\, \eta < \eta^\prime < \gamma\ \ \alpha_\eta \leq \alpha_{\eta^\prime}$, i.e. the sequence $(\alpha_\eta\ | \ \eta < \gamma)$ is increasing, \item $\forall\, \eta < \gamma \ \ \alpha_\eta \geq \kappa_\eta^{\plus\, \plus}$, \item $\forall\, \eta < \gamma\ \ \cf\, \alpha_\eta > \omega$, \item $\forall\, \eta < \gamma\ \ (\alpha_\eta = \alpha^\plus \rightarrow \cf\, \alpha > \omega)$. \end{itemize}

Then there is a cardinal- and cofinality-preserving extension $N \supseteq V$ with $N \vDash ZF\, \plus\, DC\, \plus\, AX_4$ such that that $\theta^N (\kappa_\eta) = \alpha_\eta$ holds for all $\eta < \gamma$. 
\end{THM}

%In other words: In $ZF\, \plus\, DC\, \plus AX_4$, the $\theta$-function can take almost arbitrary values. 

Our paper is structured as follows: In Chapter 1, we briefly review some basic definitions and fact about forcing and symmetric extensions. In Chapter 2, we state our theorem and then argue why it is not possible to drop any of our requirements on the sequences $(\kappa_\eta\ | \ \eta < \gamma)$ and $(\alpha_\eta\ | \ \eta < \gamma)$. 

%\colorbox{yellow}{FRAGE: hier erwähnen, warum man o.B.d.A. annehmen kann, dass $(\alpha_\eta\ | \ \eta < \gamma)$ streng monoton wächst?}

In Chapter 3, we introduce our forcing notion which blows up the power sets $\big(\powerset(\kappa_\eta)\ | \ \eta < \gamma\big)$ according to $(\alpha_\eta\ | \ \eta < \gamma)$; and in Chapter 4 construct a group $A$ of $\m{P}$-automorphisms and a normal filter $\mathcal{F}$ on $A$, giving rise to our symmetric extension $N := V(G)$. 
%\colorbox{red}{FRAGE: Hier erwähnen, dass es um \tbl partial automorphisms\tbl geht?}
In Chapter 5, we prove that sets of ordinals located in $N$ can be captured in fairly \tbl mild\tbr\,$V$-generic extensions, which implies that all cardinals and cofinalities are $N$-$V$-absolute. It remains to show that indeed, $\theta^N (\kappa_\eta) = \alpha_\eta$ holds for all $\eta < \gamma$. The first part, $\theta^N (\kappa_\eta) \geq \alpha_\eta$, follows by our construction (see Chapter 6.1), while for the second part, $\theta^N (\kappa_\eta) \leq \alpha_\eta$, we assume that there was a surjective function $f: \powerset( \kappa_\eta) \rightarrow \alpha_\eta$ in $N$, and obtain a contradiction by capturing a restricted version $f^\beta$ in an intermediate generic extension $V[G^\beta]$ which is sufficiently cardinal-preserving (see Chapter 6.2 and 6.3).
%and applying several isomorphisms arguments  
We treat the values $\theta(\lambda)$ for cardinals $\lambda$ in the \tbl gaps\tbr\, $\lambda \in (\kappa_\eta, \kappa_{\eta + 1})$ and $\lambda \geq \kappa_\gamma := \sup \{\kappa_\eta\ | \ \eta < \gamma\}$ in Chapter 6.4 and 6.5, respectively, and show that they are the smallest possible. We conclude with several remarks in Chapter 7.

%\colorbox{yellow}{FRAGE: Wäre \tbl function on the class of infinite cardinals\tbr\,eindeutig?}

\section{Basics.}
\label{sym forcing} 

%\colorbox{red}{ACHTUNG - erwähnen, dass $\theta(\kappa)$ anstatt $hrtg(\powerset(\kappa))$ verwendet wird!}

In this chapter, we briefly establish some basic notations and terminology about forcing and symmetric extensions. 

We write $Ord$ and $Card$ for the class of all ordinals and cardinals, respectively.

For our construction, we work with a countable transitive model $V$ of $ZFC$, our \textit{ground model}, with a notion of forcing $(\m{P}, \leq, \m{1}) \in V$. The \textit{class of $\m{P}$-names}, $Name (\m{P})$, is defined recursively as follows: $\Name_0 (\m{P}) := \emptyset$, $\Name_{\alpha + 1} (\m{P}) := $ $\powerset\big(\Name_\alpha (\m{P})\, $ $\times\, $ $\m{P}\big)$ for $\alpha \in \Ord$, and $\Name_\lambda (\m{P}) := \bigcup_{\alpha < \lambda} \Name_\alpha (\m{P})$ whenever $\lambda$ is a limit ordinal.
%The \textit{class of $\m{P}$-names} is
Then \[\Name (\m{P}) := \bigcup_{\alpha \in \Ord} \Name_\alpha (\m{P}).\]
As usual, we denote $\m{P}$-names as $\dot{x}$. For $\dot{x} \in \Name (\m{P})$ with $\dot{x} \in \Name_{\alpha + 1} (\m{P}) \setminus \Name_\alpha (\m{P})$, we write $rk \,(\dot{x}) := \alpha$ for the \textit{rank} of $\dot{x}$.

Let $G$ be a $V$-generic filter on $\m{P}$. The \textit{$V$-generic extension by $G$} is $V[G] := \{\dot{x}^G\ | \ \dot{x} \in \Name (\m{P})\}$, where the interpretation function ${(\, \cdot\, )}^G$ is defined recursively on the $\Name_\alpha (\m{P})$-hierarchy as usual. Then $V[G]$ is a transitive model of $ZFC$ with $V \subseteq V[G]$. 

For an element $a$ of the ground model, its canonical name is denoted by $\check{a}$. Whenever $x$, $y \in V[G]$ with $x = \dot{x}^G$, $y = \dot{y}^G$, there is a canonical $\m{P}$-name for the pair $(x, y)$, which will be abbreviated by $\OR_{\m{P}} (\dot{x}, \dot{y})$. \\[-2mm]

%\colorbox{yellow}{FRAGE: wird überhaupt ein Name für das geordnete Paar gebraucht?}

Regarding the construction of symmetric extensions, we follow the presentation in \cite{ID}, where the standard method for forcing with Boolean values as described in \cite{Jech} and \cite{Jech2} is translated to partial orders.  \\[-3mm]

For the rest of this chapter, fix a partial order $\m{P}$. Let $Aut (\m{P})$ denote the automorphism group of $\m{P}$. Any $\pi \in \Aut (\m{P})$ can be extended to an automorphism $\wt{\pi}$ of the name space $Name (\m{P})$ by the following recursive definition: \[\wt{\pi} (\dot{x}) := \{ \, (\wt{\pi}(\dot{y}), \pi p)\ | \ (\dot{y}, p) \in \dot{x}\, \}.\] We confuse any $\pi \in Aut (\m{P})$ with its extension $\wt{\pi}$ (which does not lead to ambiguity). For any canonical name $\check{a}$ and $\pi \in Aut (\m{P})$, it follows recursively that $\pi (\check{a}) = \check{a}$. \\[-2mm]

The \textit{forcing relation} $\Vdash$ can be defined in an outer model as follows: \\[-3mm]

If $\varphi (v_0, \ldots, v_{n-1})$ is a formula of set theory and $\dot{x}_0, \ldots, \dot{x}_{n-1} \in \Name (\m{P})$, then $p \Vdash \varphi (\dot{x}_0, \ldots, \dot{x}_{n-1})$ if for every $V$-generic filter $G$ on $\m{P}$, it follows that $V[G] \vDash \varphi( \dot{x}_0^G, \ldots, \dot{x}_{n-1}^G)$. \\[-3mm]

The forcing relation $\Vdash$ can also be defined in the ground model $V$, and the \textit{forcing theorem} holds:

%\colorbox{yellow}{TO DO: Kunen nachsehen!}

\begin{THM}[Forcing Theorem] If $\varphi(v_0, \ldots, v_{n-1})$ is a formula of set theory and $G$ a $V$-generic filter on $\m{P}$, then for every $\dot{x}_0, \ldots, \dot{x}_{n-1} \in \Name (\m{P})$, it follows that $V[G] \vDash \varphi (\dot{x}_0^G, \ldots, \dot{x}_{n-1}^G)$ if and only if there is a condition $p \in \m{P}$ with $p \Vdash \varphi (\dot{x}_0, \ldots, \dot{x}_{n-1})$. \end{THM}

%\colorbox{yellow}{ACHTUNG - man sollte die Forcing-Relation definieren! Bei Ioanna nachsehen?}

Let now $A \subseteq \Aut (\m{P})$ denote a group of $\m{P}$-automorphisms. A \textit{normal filter on $A$} is a collection $\mathcal{F}$ of subgroups $B \subseteq A$ such that $\mathcal{F} \neq \emptyset$, $\mathcal{F}$ is closed under supersets and finite intersections, and for any $B \in \mathcal{F}$ and $\pi \in A$, it follows that the conjugate $\pi^{-1} B \pi$ is contained in $\mathcal{F}$, as well. \\[-3mm]

An important property of $\m{P}$-automorphisms is the \textit{symmetry lemma}:

%\colorbox{yellow}{TO DO: Begriff \tbl closed under supersets\tbr nachsehen!} 

\begin{Lemma}[Symmetry Lemma] For a formula of set theory $\varphi (v_0, \ldots, v_{n-1})$, an automorphism $\pi \in Aut (\m{P})$ and $\m{P}$-names $\dot{x}_0, \ldots, x_{n-1}$, it follows that $p \Vdash \varphi (\dot{x}_0, \ldots, \dot{x}_{n-1})$ if and only if $\pi p \Vdash \varphi (\pi \dot{x}_0, \ldots, \pi \dot{x}_{n-1})$. \end{Lemma}

The proof is by induction over the complexity of $\varphi$. \\[-3mm]

Fix a normal filter $\mathcal{F}$ on $A$. A $\m{P}$-name $\dot{x}$ is \textit{symmetric} if the stabilizer group $\{\pi \in A\ | \ \pi \dot{x} = \dot{x}\}$ is contained in $\mathcal{F}$. Recursively, a name $\dot{x}$ is \textit{hereditarily symmetric}, $\dot{x} \in HS$, if $\dot{x}$ is symmetric and $\dot{y} \in HS$ for all $\dot{y} \in \dom \dot{x}$. \\[-3mm]

For a $V$-generic filter $G$ on $\m{P}$, the \textit{symmetric extension by $\mathcal{F}$ and $G$} is defined as follows: \[N := V(G) := \{ \dot{x}^G\ | \ \dot{x} \in HS\}. \] Then $N$ is a transitive class with $V \subseteq N \subseteq V[G]$ and $N \vDash ZF$. \\[-3mm]

The \textit{symmetric forcing relation} $\Vdash_s$ can be defined informally as follows: 

If $\varphi (v_0, \ldots, v_{n-1})$ is a formula of set theory, and $\dot{x}_0, \ldots, \dot{x}_{n-1} \in HS$, then $p \Vdash_s \varphi (\dot{x}_0, \ldots, \dot{x}_{n-1})$ if for every $V$-generic filter $G$ on $\m{P}$, it follows that $V(G) \vDash \varphi (\dot{x}_0^G, \ldots, $ $\dot{x}_{n-1}^G)$.\\[-3mm]

Note that the symmetric forcing relation $\Vdash_s$ can be defined in the ground model similar to the ordinary forcing relation $\Vdash$, but with the quantifiers and variables ranging over $HS$. It has most of the basic properties as $\Vdash$. In particular, the forcing theorem holds for $\Vdash_s$, and the \textit{symmetry lemma} is true, as well. \\

We will consider the following weak versions of the Axiom of Choice: \\[-3mm]

The \textbf{Axiom of Dependent Choice (DC)} states that whenever $X$ is a nonempty set with a binary relation $R$, such that for all $x \in R$ there exists $y \in R$ with $y\,R \,x$, there is a sequence $(x_n \ | \ n < \omega)$ in $X$ with $x_{n+1} \,R \,x_n$ for all $n < \omega$.

The \textbf{Axiom of Countable Choice ($\boldsymbol{AC_\omega}$)} states that every countable family of nonempty sets has a choice function. \\[-3mm]

%\colorbox{yellow}{Begriffe eher fett statt kursiv schreiben?} \\[-3mm]

The Axiom of Choice implies $DC$, and $DC$ implies $AC_\omega$. \\[-3mm]

The following lemma follows from \cite[Lemma 1]{Karagila}:

\begin{Lemma}[{\cite[Lemma 1]{Karagila}}] If $\m{P}$ is countably closed with a group of automorphisms $A \subseteq Aut (\m{P})$ and a normal filter $\mathcal{F}$ on $A$, such that $\mathcal{F}$ is countably closed as well, then $DC$ holds in the corresponding symmetric extension $N = V(G)$. \end{Lemma}

The theory $ZF\, \plus \, DC$ is sufficient to develop most of real analysis; while combinatorial set theory in $ZF\, \plus\, DC$ seemed rather hopeless in the first place.

In \cite{pcfwc} and \cite{AX4}, Shelah suggests that when working with $ZF\, \plus \, DC$, another axiom should be adopted to set the framework for a \tbl reasonable\tbr\, set theory, where now, surprisingly much of combinatorial set theory can be realized: 

\begin{center} $\mathbf{(AX_4)}$ \ For every cardinal $\lambda$, the set $[\lambda]^{\aleph_0}$ can be well-ordered. \end{center}

Note that $(AX_4)$ holds true in any symmetric extension by countably closed forcing (see \cite[p.3 and p.15]{pcfwc}).

%\colorbox{red}{TO DO: nach $AC_\omega$ schauen!}

%This chapter contains some basic definitions and results about class forcing. \\ 
%We start with 

%\"Ahnlich wie vorher, aber ohne Klassenforcing? Sollte man \tbl group of partial automorphisms \tbr vll. in diesem Kapitel einf\"uhren?

%Den \tbl alten Text \tbr nochmal ansehen!

\section{The Theorem}
\label{the theorem}

%\colorbox{red}{?????????? FRAGE: AK und \tbl open names\tbr\, zitieren? ??????}

We start from a ground model $V \vDash ZFC \, \plus \, GCH$ and a \textit{reasonable behavior of the $\theta$-function}: There will be sequences of uncountable cardinals $(\kappa_\eta\ | \ 0 < \eta < \gamma)$ and $(\alpha_\eta\ | \ 0 < \eta < \gamma)$ in $V$, where $\gamma \in \Ord$, for which we aim to construct a symmetric extension $N \supseteq V$ with $N \vDash ZF\, \plus\, DC\, \plus\, AX_4$, such that $V$ and $N$ have the same cardinals and cofinalities and $\theta^N (\kappa_\eta) = \alpha_\eta$ holds for all $\eta$.

(Later on, we will set $\kappa_0 := \aleph_0$, $\alpha_0 := \aleph_2$ for technical reasons -- therefore, we talk about sequences $(\kappa_\eta\ | \ 0 < \eta < \gamma)$, $(\alpha_\eta\ | \ 0 < \eta < \gamma)$ here, excluding $\kappa_0$ and $\alpha_0$. )\\[-2mm]

%\colorbox{yellow} {TO DO: NACHSEHEN, ob wirklich überall $\kappa_0$ ausgeschlossen wird!}

First, we want to discuss what properties the sequences $(\kappa_\eta\ | \ 0 < \eta < \gamma)$ and $(\alpha_\eta\ | \ 0 < \eta < \gamma)$ must have to allow for such construction. \\[-3mm]

W.l.o.g. we can assume that ($\kappa_\eta\ | \ 0 < \eta < \gamma)$ is strictly increasing and closed. \\[-3mm]

The following conditions must be satisfied:

\begin{itemize} \item For $\eta < \eta^\prime$, it follows from $\kappa_\eta < \kappa_{\eta^\prime}$ that $\alpha_\eta \leq \alpha_{\eta^\prime}$ must hold, i.e. the sequence $(\alpha_\eta\ | \ 0 < \eta < \gamma)$ must be increasing.
\item For any cardinal $\kappa$, it is possible to construct a surjection $s: \powerset(\kappa) \rightarrow \kappa^\plus$ without making use of the Axiom of Choice. 
Hence, $\alpha_\eta \geq \kappa_\eta^{\plus \plus}$ must hold for all $\eta$.

%\colorbox{yellow}{FRAGEN: Masterarbeit Ionna zitieren? \tbl Strong limits and inaccessibility}

%\colorbox{yellow}{without the axiom of choice?}

\item Since $N \vDash AC_\omega$, it follows that $cf\, \alpha_\eta > \omega$ for all $\eta$: Assume towards a contradiction, there were cardinals $\kappa$, $\alpha$ with $\theta^N (\kappa) = \alpha$, but $cf^N (\alpha) = \omega$. Let $\alpha = \bigcup_{i < \omega} \alpha_i$. By definition of $\theta^N(\kappa)$, it follows that for every $i < \omega$, there exists in $N$ a surjection from $\powerset(\kappa)$ onto $\alpha_i$. Now, $AC_\omega$ allows us to pick in $N$ a sequence $(s_i\ | \ i < \omega)$ such that each $s_i: \powerset(\kappa) \rightarrow \alpha_i$ is a surjection. This yields a surjective function $\ol{s}: \powerset(\kappa) \, \times\, \omega \rightarrow \alpha$, where $\ol{s}(X, i) := s_i (X)$ for each $(X, i) \in \powerset (\kappa)\, \times\, \omega$; which can be easily turned into a surjection $s: \powerset(\kappa) \rightarrow \alpha$. Contradiction, since $\theta^N (\kappa) = \alpha$. Hence, it follows that $cf \,\alpha_\eta > \omega$ for all $\eta$. 

\item Finally, for every $\alpha_\eta$ a successor cardinal with $\alpha_\eta = \alpha^\plus$, we will need that $cf \, \alpha > \omega$. \\[-4mm]

In this setting, it is not possible to drop this requirement: We start from a ground model $V \vDash ZFC\, \plus\, GCH$ with sequences $(\kappa_\eta\ | \ 0 < \eta < \gamma)$, $(\alpha_\eta\ | \ 0 < \eta < \gamma)$, and aim to construct $N \supseteq V$ with $N \vDash ZF\, \plus\, DC$ such that $V$ and $N$ have the same cardinals and cofinalities, and $\theta^N (\kappa_\eta) = \alpha_\eta$ holds for all $\eta$. If there was some $\eta$ with $\theta^N (\kappa_\eta) = \alpha^\plus$, where $cf\, \alpha = \omega$, one could construct in $N$ a surjective function $\ol{s}: \powerset(\kappa_\eta) \rightarrow \alpha^\plus$ as follows: \\Take a surjection $s: \powerset(\kappa_\eta) \rightarrow \alpha$ in $N$. Firstly, the canonical bijection $\kappa \leftrightarrow \kappa\, \times\, \omega$ gives a surjection $s_0: 2^\kappa \rightarrow (2^\kappa)^\omega$. Secondly, the surjection $s: \powerset(\kappa_\eta) \rightarrow \alpha$ yields in $N$ a surjection $s_1: (2^\kappa)^\omega \rightarrow \alpha^\omega$, by setting $s_1 (X_i\ | \ i < \omega) := (s(X_i)\ | \ i < \omega)$. Then $s_1$ is surjective, since for $(\alpha_i \ | \ i < \omega) \in \alpha^\omega$ given, one can use $AC_\omega$ to obtain a sequence $(Y_i\ | \ i < \omega)$ with $Y_i \in s^{-1} [\{\alpha_i\} ]$ for all $i < \omega$. Then $s_1 (Y_i\ | \ i < \omega) = (\alpha_i\ | \ i < \omega)$. %\colorbox{yellow}{FRAGE: Wäre die Schreibweise $(X_i\ | \ i < \omega)$ bzw. $s(X_i\ | \ i < \omega)$ überhaupt in Ordnung?} 
Thirdly, it follows from $cf \, \alpha = \omega$ that there is a surjection $\ol{s_2}: \alpha^\omega \rightarrow \alpha^\plus$ in $V$. Then $\ol{s_2} \in N$, and since $(\alpha^\omega)^N \supseteq (\alpha^\omega)^V$ and $(\alpha^\plus)^N = (\alpha^\plus)^V$, we obtain a surjection $s_2: \alpha^\omega \rightarrow \alpha^\plus$ in $N$. \\ Thus, it follows that $s_2\, \circ\, s_1\, \circ\, s_0 : 2^\kappa \rightarrow \alpha^\plus$ is a surjective function in $N$; contradicting that $\theta^N (\kappa_\eta) = \alpha^\plus$. \\[-3mm]

Hence, in our setting, where we want to extend a ground model $V \vDash ZFC\, \plus\, GCH$ cardinal-preservingly, it is not possible to have $\alpha_\eta = \alpha^\plus$ with $cf \,\alpha = \omega$. \\[-3mm]

The following question arises: More generally, without referring to a ground model $V$, could there be $N \vDash ZF\, \plus \, DC\, \plus\, AX_4$ with cardinals $\kappa$, $\alpha$, such that $cf^N (\alpha) = \omega$ and $\theta^N (\kappa) = \alpha^\plus$? The answer is no: Let $s: 2^\kappa \rightarrow \alpha$ denote a surjective function in $N$. Then with $DC$, it follows as before that there is also a surjective function $s_1: (2^\kappa)^\omega \rightarrow \alpha^\omega$ in $N$; and we also have a surjective function $s_0: 2^\kappa \rightarrow (2^\kappa)^\omega$. Since $\alpha^\omega$ is well-ordered by $(AX_4)$, a diagonalization argument as in \textit{König's Lemma} shows that there is also a surjection $s_2: \alpha^\omega \rightarrow \alpha^\plus$. Hence, $s_2\, \circ\, s_1\, \circ\, s_0: 2^\alpha \rightarrow \alpha^\plus$ is a surjective function in $N$ as desired.

\end{itemize}

We conclude that all the requirements on the sequences $(\kappa_\eta\ | \ 0 < \eta < \gamma)$ and $(\alpha_\eta\ | \ 0 < \eta < \gamma)$ listed above, are necessary for a model $N \vDash ZF\, \plus\, DC\, \plus\, AX_4$. \\[-3mm]

In addition, one could ask if there exists a model $N \vDash ZF\, \plus \, DC$ (without $AX_4$) with cardinals $\kappa$, $\alpha$, such that $\theta^N (\kappa) = \alpha^\plus$ and $cf^N (\alpha) = \omega$. It is not difficult to see that this is not possible under $\neg 0^\sharp$ (cf. Chapter \ref{discussion}). Hence, if one wishes to avoid large cardinal assumptions, then for every $\alpha_\eta = \alpha^\plus$ a successor cardinal, one has to require $cf\,\alpha > \omega$. \\[-2mm]
%in this context where 

%\colorbox{yellow}{FRAGE: Schon hier mehr zu $\neg 0^\sharp$ schreiben? Oder nur in Kapitel \ref{discussion}?} \\[-3mm]

%Finally, $N \vDash AX_4$, 

Our main theorem states that these are the only restrictions on the $\theta$-function for set-many uncountable cardinals in $ZF\, \plus \, DC\, \plus\, AX_4$:

%\colorbox{red}{Diesen Satz so schreiben?}

%In this paper, we give a proof of the following theorem:

\begin{THM} \label{theorem} Let $V$ be a ground model of $ZFC \, \plus \, GCH$ with $\gamma \in \Ord$ and sequences of uncountable cardinals $(\kappa_\eta\ | \ 0 < \eta < \gamma)$ and $(\alpha_\eta\ | \ 0 < \eta < \gamma)$ such that $(\kappa_\eta\ | \ 0 < \eta < \gamma)$ is strictly increasing and closed, and the following properties hold:

\begin{itemize} \item $\forall\, 0 < \eta < \eta^\prime < \gamma\ \ \alpha_\eta \leq \alpha_{\eta^\prime}$, i.e. the sequence $(\alpha_\eta\ | \ 0 < \eta < \gamma)$ is increasing, \item $\forall\, 0 < \eta < \gamma \ \ \alpha_\eta \geq \kappa_\eta^{\plus\, \plus}$, \item $\forall\, 0 < \eta < \gamma\ \ \cf\, \alpha_\eta > \omega$, \item $\forall\, 0 < \eta < \gamma\ \ (\alpha_\eta = \alpha^\plus \rightarrow \cf\, \alpha > \omega)$. \end{itemize}

Then there is a cardinal- and cofinality-preserving extension $N \supseteq V$ with $N \vDash ZF\, \plus\, DC\, \plus \, AX_4$ such that that $\theta^N (\kappa_\eta) = \alpha_\eta$ holds for all $0 < \eta < \gamma$.

\end{THM}

%\colorbox{yellow}{Ist es sinnvoll, schon hier kein $\kappa_0$ zu haben?}

%\begin{itemize} \item $\alpha_\eta \geq \kappa_\eta^{\plus \plus}$ \\ Weil auch ohne $AC$ immer eine Surjektion $s: \powerset(\kappa_\eta) \rightarrow \kappa_\eta^\plus$ existiert? \item $\cf \alpha_\eta > \omega$ für alle $\eta$ \\ Weil es sonst auch eine Surjektion $s: \powerset(\kappa_\eta) \rightarrow \alpha_\eta$ geben würde?
%\item Falls $\alpha_\eta = \beta\plus$ Nachfolgerkardinalzahl wäre, folgt $\cf \beta > \omega$. \\ Warum ist diese Bedingung notwendig? Den allgemeineren Sachverhalt durchgehen: Falls $V \vDash ZFC$, und $N \vDash ZF$ mit $V \subseteq N$, sodass Kardinalzahlen $V$-$N$-absolut sind, dann ist diese Bedingung notwendig. Dann darauf eingehen, warum diese Bedingungen sinnvoll sind, wenn man eine $(\kappa_\eta \ | \eta < \gamma)$- und eine $(\beta_\eta\ | \ \eta < \gamma)$-Folge gegeben hat? Oder umgekehrt? \tbl In this setting \tbr ... wollen wir zeigen, dass man für ein Grundmodell $V$ mit vorgegebenem Verhalten der $\theta$-Funktion (für \tbl mengen-viele \tbr Ordinalzahlen) das Grundmodell $V$ erweitern könnte zu einem Modell $N \vDash ZF \, \plus\, DC$, in dem dieses Verhalten realisiert wird. 
%[Braucht man an dieser Stelle für das Argument, dass $V$ und $N$ die selben Kardinalzahlen haben? Eher nicht, oder?] Damit würde folgen: Wenn man in $N$ eine Surjektion $s: \powerset(\kappa) \rightarrow \beta$ hätte mit $\cf\beta = \omega$, dann gibt es in $N$ auch eine Surjektion $s: \powerset (\kappa) \rightarrow \beta^\plus$. Deswegen wäre diese gestellte Bedingung notwendig.

In our construction, we will make sure that for any cardinal $\lambda$ in a \tbl gap\tbr\, $(\kappa_\eta, \kappa_{\eta+1})$, the value $\theta^N (\lambda)$ is the smallest possible, i.e. $\theta^N (\lambda) = \max \{\alpha_\eta, \lambda^{\plus \plus} \}$. Also, if we set $\kappa_\gamma := \bigcup \{ \kappa_\eta\ | \ 0 < \eta < \gamma\}$, $\alpha_\gamma  := \bigcup \{\alpha_\eta\ | \ 0 < \eta < \gamma\}$, then for every $\lambda \geq \kappa_\gamma$, we will again make sure that $\theta^N (\lambda)$ takes the smallest possible value: We will have $\theta^N (\lambda) = \max \{ \alpha_\gamma^{\plus \plus}, \lambda^{\plus \plus} \}$ in the case that $cf \,\alpha_\gamma = \omega$, $\theta^N (\lambda) = \max\{ \alpha_\gamma^\plus, \lambda^{\plus \plus}\}$ in the case that $\alpha_\gamma = \alpha^\plus$ for some cardinal $\alpha$ with $cf \,\alpha = \omega$, and $\theta^N (\lambda) = \max \{ \alpha_\gamma, \lambda^{\plus \plus}\}$, else. \\[-3mm]

%\colorbox{yellow}{TO DO: \tbl kleinstmöglich\tbr\,übersetzen!} 

This allows us to assume w.l.o.g. that the sequence $(\alpha_\eta\ | \ 0 < \eta < \gamma)$ is \textit{strictly} increasing: If not, one can start with the original sequences $(\kappa_\eta\ | \ 0 < \eta < \gamma)$ and $(\alpha_\eta\ | \ 0 < \eta < \gamma)$, and successively strike out all $\kappa_\eta$ for which the value $\alpha_\eta$ is not larger than the values $\alpha_{\ol{\eta}}$ before. This procedure results in sequences $(\wt{\kappa}_{\eta}\ | \ 0 < \eta < \wt{\gamma}) := (\kappa_{s(\eta)}\ | \ 0 < \eta < \wt{\gamma})$ and $(\wt{\alpha}_{\eta}\ | \ 0 < \eta < \wt{\gamma}) := (\alpha_{s(\eta)}\ | \ 0 < \eta < \wt{\gamma})$ for some $\wt{\gamma} \leq \gamma$ and a strictly increasing function $s: \wt{\gamma} \rightarrow \gamma$, such that $\wt{\alpha}_{\wt{\gamma}} := \bigcup \{ \wt{\alpha}_{\eta}\ | \ 0 < \eta < \wt{\gamma}\} = \bigcup \{ \alpha_{s(\eta)}\ | \ 0 < \eta < \wt{\gamma} \} = \bigcup \{ \alpha_\eta\ | \ 0 < \eta < \gamma\} = \alpha_\gamma$, and $(\wt{\alpha}_{\eta}\ | \ 0 < \eta < \wt{\gamma}) = (\alpha_{s(\eta)}\ | \ 0 < \eta < \wt{\gamma})$ is strictly increasing. If we then use the sequences $(\wt{\kappa}_{\eta}\ | \ 0 < \eta < \wt{\gamma})$, $(\wt{\alpha}_{\eta}\ | \ 0 < \eta < \wt{\gamma})$ for our construction and make sure that not only $\theta^N (\wt{\kappa}_{\eta}) = \wt{\alpha}_{\eta}$ holds for all $0 < \eta < \wt{\gamma}$, but additionally,  $\theta^N (\lambda)$ takes the smallest possible value for all cardinals $\lambda$ within the \tbl gaps\tbr\,$(\wt{\kappa}_{\eta}, \wt{\kappa}_{\eta + 1})$, 
%(i.e. $\theta^N (\lambda) = \max\{\wt{\alpha}_{\eta}, \lambda^{\plus\plus}\})$, 
and also make sure that $\theta^N (\lambda)$ takes the smallest possible value for all cardinals $\lambda \geq \wt{\kappa}_{\wt{\gamma}} := \bigcup \{ \wt{\kappa}_{\eta}\ | \ 0 < \eta < \wt{\gamma}\}$,
%(i.e. $\theta^N (\lambda) = \max\{\wt{\alpha}_{\wt{\gamma}}^{\plus \plus}, \lambda^{\plus \plus}\}$ if $cf \,\wt{\alpha}_{\wt{\gamma}} = \omega$; $\theta^N (\lambda) = \max\{\wt{\alpha}_{\wt{\gamma}}^\plus, \lambda^{\plus \plus}\}$ in the case that $\wt{\alpha}_\gamma = \alpha^\plus$ with $\cf \alpha = \omega$; and $\theta^N (\lambda) = \max\{\wt{\alpha}_{\wt{\gamma}}, \lambda^{\plus \plus}\}$, else);
then it follows, that for all $\kappa_\eta$ in the original sequence $(\kappa_\eta\ | \ 0 < \eta < \gamma)$, the values $\theta^N (\kappa_\eta) = \alpha_\eta$ are as desired. \\[-2mm]

%\colorbox{yellow}{Übersetzung von \tbl strike out\tbr?}

Hence, from now on, we assume w.l.o.g. that the sequence $(\alpha_\eta\ | \ 0 < \eta < \gamma)$ is \textbf{strictly increasing}.

%colorbox{yellow}{Wären hier für $\gamma$ auch Limiten erlaubt?} 

%Sp\"ater erw\"ahnen, was in den Lücken $(\kappa_\eta, \kappa_{\eta + 1})$ passieren sollte? Wieso man \tbl obdA\tbr annehmen kann, dass die $(\alpha_\eta\ | \ \eta < \gamma)$-Folge streng monoton w\"achst?

\section{The Forcing} \label{chapterforcing}

In this chapter, we will define our forcing notion $\m{P}$. \\[-3mm]

We start from a ground model $V \vDash ZFC\, \plus\, GCH$ with sequences $(\kappa_\eta\ | \ 0 < \eta < \gamma)$, $(\alpha_\eta\ | \ 0 < \eta < \gamma)$ that have all the properties mentioned in Chapter \ref{theorem}. \\[-2mm]

We will have to treat limit cardinals and successor cardinals separately. Let $Lim := \{ 0 < \eta < \gamma\ | \ \kappa_\eta \mbox{ is a limit cardinal}\}$, and $Succ := \{0 < \eta < \gamma\ | \ \kappa_{\eta} \mbox{ is a successor cardinal}\}$. For $\eta \in \Succ$, we denote by $\ol{\kappa_\eta}$ the cardinal predecessor of $\kappa_\eta$; i.e. $\kappa_\eta = \ol{\kappa_\eta}^{\;\plus}$. 
Our forcing will be a product $\m{P} = \m{P}_0\, \times \, \m{P}_1$, where $\m{P}_0$ deals with the limit cardinals $\kappa_\eta$, and $\m{P}_1$ is in charge of the successor cardinals. \\[-3mm]

The forcing $\m{P}_0$ is a generalized version of the forcing notion in \cite{GK}.

%FRAGE: Sollte man auf den Fehler dort eingehen? Evtl berichtigen? Erwähnen, dass die korrigierte Version von [GK] keine gute Kettenbedingung erfüllt; und dass damit ein Klassenforcing schwierig wird? \\[-2mm]

Roughly speaking, for every $\eta \in \Lim$ we add $\alpha_\eta$-many $\kappa_\eta$-subsets, which will be linked in a certain fashion, in order to make sure that not too many $\kappa$-subsets for cardinals $\kappa < \kappa_\eta$ make their way into the eventual model $N$. \\[-3mm]

For technical reasons, let $\kappa_0 := \aleph_0$, $\alpha_0 := \aleph_2$. For all $\eta$ with $\eta + 1 \in \Lim$, we take a sequence of cardinals $(\kappa_{\eta, j}\ | \ j < \cf \kappa_{\eta + 1})$ cofinal in $\kappa_{\eta + 1}$, such that $\kappa_{\eta, 0} = \kappa_\eta$, the sequence $(\kappa_{\eta, j}\ | \ j < \cf \kappa_{\eta + 1})$ is strictly increasing and closed, and any $\kappa_{\eta, j + 1}$ is a successor cardinal 
%(wo überall wird das benötigt? für die Abschlüsse braucht man doch Regularität!) 
with $\kappa_{\eta, j + 1} \geq \kappa_{\eta, j}^{\plus \plus}$ for all $j < \cf \kappa_{\eta + 1}$. \\These \tbl gaps\tbr\,between the cardinals $\kappa_{\eta, j}$ and $\kappa_{\eta, j + 1}$ will be necessary for further factoring arguments. \\[-2mm]

%\colorbox{red}{ACHTUNG - Könnte das nicht problematisch werden, wenn $\kappa_0 = \aleph_0$ gesetzt wird?} \\[-2mm]

%\colorbox{red}{DAS IST NICHT KLAR!}

For all $0 < \eta < \gamma$ for which $\eta + 1 \in \Succ$, i.e. $\kappa_{\eta + 1}$ is a successor cardinal, we set $\kappa_{\eta, 0} := \kappa_\eta$, and $\cf \kappa_{\eta + 1} := 1$ for reasons of homogeneity. \\[-3mm]

Now, in the case that $\eta \in \Lim$, the forcing $P^\eta$ will be defined like an Easton-support product of Cohen forcings for the intervals $[\kappa_{\nu, j}, \kappa_{\nu, j+1}) \subseteq \kappa_\eta$:

\begin{definition} For $\eta \in \Lim$, we let the forcing notion $(P^\eta, \supseteq, \emptyset)$ consist of all functions $p: \dom p \rightarrow 2$ such that $\dom p$ is of the following form: \\[-3mm]

 There is a sequence $(\delta_{\nu, j} \ | \ \nu < \eta\, , \, j < \cf \kappa_{\nu + 1})$ with $\delta_{\nu, j} \in [\kappa_{\nu, j}, \kappa_{\nu, j+1}) $ for all $\nu < \eta$, $j < \cf \kappa_{\nu + 1}$ with \[\dom p = \bigcup_{\substack{\nu < \eta \\ j < \cf \kappa_{\nu + 1}}} \, [\kappa_{\nu, j}, \delta_{\nu, j}),\] and for any regular $\kappa_{\nu, j}$, the domain $\dom p \, \cap \, \kappa_{\nu, j}$ is  bounded below $\kappa_{\nu,j}$. 
\end{definition}

For a set $S \subseteq \kappa_\eta$, we let $P^\eta \uhr S := \{p \in P^\eta\ | \ \dom p \subseteq S \} = \{p \uhr S \ | \ p \in P^\eta\}$. Then for any $\kappa_{\nu, j} < \kappa_\eta$, the forcing $P^\eta$ is isomorphic to the product $P^\eta \uhr \kappa_{\nu, j}\, \times \, P^\eta \uhr [\kappa_{\nu, j}, \kappa_\eta)$, where the first factor has cardinality $\leq \kappa_{\nu, j}^\plus$, and the second factor is $\leq \kappa_{\nu, j}$-closed.  \\[-3mm] 

This helps to establish:

\begin{lem} \label{prescard1} For all $\eta \in \Lim$, the forcing $P^\eta$ preserves cardinals and the $GCH$. \end{lem}

\begin{proof} Let $G^\eta$ denote a $V$-generic filter on $P^\eta$. It suffices to show that for all cardinals $\alpha$ in $V$, \[ (2^\alpha)^{V[G^\eta]} \leq (\alpha^\plus)^V,\] which implies that cardinals are $V$-$V[G^\eta]$-absolute: If not, there would be a $V$-cardinal $\alpha$ with a surjection $s: \beta \rightarrow \alpha$ in $V[G^\eta]$ for some $V[G^\eta]$-cardinal $\beta < \alpha$. Then there is also a surjection $\ol{s}: \beta \rightarrow (\beta^\plus)^V$ in $V[G^\eta]$, which gives a surjection $\ol{\ol{s}}: \beta \rightarrow \big(2^\beta\big)^{V[G^\eta]}$. Contradiction.
%\colorbox{yellow} {FRAGE: Sollte man das ausführen?} 
\begin{itemize} \item In the case that $\alpha \geq \kappa_\eta^\plus$, it follows that $(2^\alpha)^{V[G^\eta]} \leq | \powerset (\alpha \cdot |P^\eta|)|^V \leq (2^\alpha)^V = (\alpha^\plus)^V$ by the $GCH$ in $V$. \item Now, assume $\alpha \in (\kappa_{\nu, j}, \kappa_{\nu, j +1})$ for some $\kappa_{\nu, j} < \kappa_\eta$. Then the forcing $P^\eta$ can be factored as $P^\eta \uhr \kappa_{\nu, j} \, \times \, P^\eta \uhr [\kappa_{\nu,j}, \kappa_\eta)$, where $P^\eta\, \uhr \, \kappa_{\nu, j}$ has cardinality $\leq \kappa_{\nu, j}^\plus \leq \alpha$, and $P^\eta \uhr [\kappa_{\nu,j}, \kappa_\eta)$ is $\leq \alpha$-closed. Hence, \[\big (2^\alpha\big)^{V[G^\eta]} \leq \big(2^\alpha\big)^{V[G^\eta \, \uhr \, \kappa_{\nu, j}]} \leq \big|\, \powerset \big(\alpha\, \cdot \, |P^\eta \, \uhr \, \kappa_{\nu, j}|\big)\, \big|^V \leq \big(2^\alpha\big)^V = (\alpha^\plus)^V.\] \item If $\alpha = \kappa_{\nu, j}$ for some regular $\kappa_{\nu,j} < \kappa_\eta$, then $|P^\eta \uhr \kappa_{\nu, j}| = \kappa_{\nu, j}$ and $P^\eta \uhr [\kappa_{\nu, j}, \kappa_\eta)$ is $\leq \kappa_{\nu, j}$-closed; so the same argument applies. \\ If $\alpha = \kappa_\eta$ is regular, then $(2^\alpha)^{V[G^\eta]} \leq (\alpha^\plus)^V$ follows from $|P^\eta| \leq \kappa_\eta$. \end{itemize}

It remains to show that $(2^{\kappa_{\nu, j}})^{V[G^\eta]} = (\kappa_{\nu, j}^\plus)^V$ for all singular $\kappa_{\nu, j} < \kappa_\eta$, and $(2^{\kappa_\eta})^{V[G^\eta]} \leq (\kappa_\eta^\plus)^V$ in the case that $\kappa_\eta$ itself is singular. \\[-3mm]

We only prove the first part (the argument for $\kappa_\eta$ is similar).

\begin{itemize}

\item Assume the contrary and let $\kappa_{\nu, j}$ least with $\lambda := \cf \kappa_{\nu, j} < \kappa_{\nu, j}$ and $(2^{\kappa_{\nu, j}})^{V[G^\eta]} > (\kappa_{\nu, j}^\plus)^V$. Take $(\alpha_i \ | \ i < \lambda)$ cofinal in $\kappa_{\nu, j}$. By assumption and by what we have shown before, it follows that $(2^\alpha)^{V[G^\eta]} = (\alpha^\plus)^V$ for all $\alpha < \kappa_{\nu, j}$. Hence, $\Card^V \, \cap \, (\kappa_{\nu, j} + 1) = \Card^{V[G]}\, \cap \, (\kappa_{\nu, j} + 1)$, and $(2^{\alpha_i})^{V[G^\eta]} = (\alpha_i^\plus)^V$ for all $i < \lambda$. Thus, \[2^{\kappa_{\nu, j}} \leq \prod_{i < \lambda} 2^{\alpha_i} \leq \kappa_{\nu, j}^\lambda \leq \kappa_{\nu, j}^{\kappa_{\nu, j}} = 2^{\kappa_{\nu, j}}\] holds true in $V$ and $V[G^\eta]$. Let $\lambda \in [\kappa_{\mu, m}, \kappa_{\mu, m+1})$ for some $\kappa_{\mu, m} < \kappa_{\nu, j}$. If $\lambda > \kappa_{\mu, m}$, then $|P^\eta \uhr \kappa_{\mu, m}| \leq (\kappa_{\mu, m})^\plus \leq \lambda$, and $P^\eta \uhr [\kappa_{\mu, m}, \kappa_\eta)$ is $\leq \lambda$-closed. In the case that $\lambda = \kappa_{\mu, m}$, it follows by regularity of $\lambda$ that $|P^\eta \uhr \kappa_{\mu, m}| \leq \kappa_{\mu, m} = \lambda$, as well. In either case, \[\big(2^{\kappa_{\nu, j}}\big)^{V[G^\eta]} = \big(\kappa_{\nu, j}^\lambda\big)^{V[G^\eta]} \leq \big(\kappa_{\nu, j}^\lambda\big)^{V[G^\eta \, \uhr \, \kappa_{\mu, m}]} \leq \big(2^{\kappa_{\nu, j}}\big)^{V[G^\eta \, \uhr \,\kappa_{\mu, m}]} \leq \] \[ \leq \big|\,\powerset \big(\kappa_{\nu, j} \, \cdot \, |P^\eta \uhr \kappa_{\mu, m}|\big)\, \big|^V \leq  \big|\, \powerset (\kappa_{\nu, j} \, \cdot \, \kappa_{\mu, m}^\plus)\, \big|^V = (\kappa_{\nu, j}^\plus)^V,\] which gives the desired contradiction.\end{itemize}
\end{proof}

\begin{cor} \label{prescof} For every $\eta \in Lim$, the forcing $P^\eta$ preserves cofinalites. \end{cor}

\begin{proof} We show that every regular $V$-cardinal $\lambda$ is still regular in $V[G^\eta]$. If not, there would be in $V[G^\eta]$ a regular cardinal $\ol{\lambda} < \lambda$ with
%Assume towards a contradiction, there were regular $V$-cardinals $\lambda$, $\ol{\lambda}$ with $\ol{\lambda} < \lambda$, and 
a cofinal function $f: \ol{\lambda} \rightarrow \lambda$. Let $\ol{\lambda} \in [\kappa_{\nu, j}, \kappa_{\nu, j + 1})$. The forcing $P^\eta$ is isomorphic to the product $P^\eta\, \uhr\, \kappa_{\nu, j}\, \times\, P^\eta\, \uhr\, [\kappa_{\nu, j}, \kappa_\eta)$, where the second factor is $\leq \ol{\lambda}$-closed. If $\ol{\lambda} > \kappa_{\nu, j}$, then the first factor has cardinality $\leq \kappa_{\nu, j}^\plus \leq \ol{\lambda}$. In the case that $\ol{\lambda} = \kappa_{\nu, j}$, the first factor has cardinality $\leq \kappa_{\nu, j} = \ol{\lambda}$ by regularity of $\ol{\lambda}$. Hence, $f \in V[G\, \uhr\, \kappa_{\nu, j}]$. However, since $|P^\eta\, \uhr\, \kappa_{\nu, j}| < \lambda$, it follows that $\lambda$ is still a regular cardinal in the generic extension $V[G^\eta\, \uhr\, \kappa_{\nu, j}]$. Contradiction.\\ Thus, it follows that $P^\eta$ preserves cofinalities as desired.  \end{proof} 

Our eventual forcing notion $\m{P}_0$ will contain 
%For every $\sigma < \gamma$, our forcing $\m{P}$ contains 
$\alpha_\sigma$-many copies of $P^\sigma$ for every $\sigma \in \Lim$. They will be labelled $P^\sigma_i$, where $i < \alpha_\sigma$. All the $P^\sigma_i$ for $\sigma \in \Lim$, $i < \alpha_\sigma$, will be linked with a forcing notion $P_\ast$, which is a two-dimensional version of $P^\gamma$, adding $\kappa_{\nu, j + 1}$-many Cohen subsets to every interval $[\kappa_{\nu, j}, \kappa_{\nu, j + 1})$: 

%\colorbox{yellow} {TO DO: Im \tbl ersten Skript\tbr\,nach Formlierungen suchen!}

\begin{definition} We denote by $(P_\ast, \supseteq \emptyset)$ the forcing notion consisting of all functions $p_\ast: \dom p_\ast \rightarrow 2$ such that $\dom p_\ast$ is of the following form: \\[-3mm]

There is a sequence $(\delta_{\nu, j} \ | \ \nu < \gamma\, , \, j < \cf \kappa_{\nu + 1})$ with $\delta_{\nu, j} \in [\kappa_{\nu, j}, \kappa_{\nu, j+1}) $ for all $\nu  < \gamma$, $j < \cf \kappa_{\nu + 1}$ with \[\dom p_\ast = \bigcup_{\substack{\nu < \gamma \\ j < \cf \kappa_{\nu + 1}}} \, [\kappa_{\nu, j}, \delta_{\nu, j})^2,\] and for any $\kappa_{\nu, j}$ a regular cardinal, it follows that $|\dom p_\ast \, \cap \, \kappa_{\nu, j}^2| < \kappa_{\nu, j}$, and in the case that $\kappa_\gamma$ itself is regular, we require that $|\dom p_\ast| < \kappa_\gamma$.
%$|p_\ast| < \kappa_\gamma$. 
\end{definition}

For $p_\ast \in P_\ast$ and $\xi < \kappa_\gamma$, let $p_\ast (\xi) := \{ \, (\zeta, p_\ast (\xi, \zeta))\ | \ (\xi, \zeta) \in \dom p_\ast \,\}$ denote the $\xi$-th section of $p_\ast$. If $a \subseteq \kappa_\gamma$ is a set that hits every interval $[\kappa_{\nu, j}, \kappa_{\nu, j + 1})$ in at most one point, we let \[p_\ast (a) := \{ \, (\zeta, p_\ast (\xi, \zeta))\ | \ \xi \in a, (\xi, \zeta) \in \dom p_\ast \,\}.\] 

As in Lemma \ref{prescard1}, it follows that $P_\ast$ preserves cardinals and the $GCH$. \\[-2mm]

Now, we are ready to define our forcing notion $\m{P}_0$. Every $p_0 \in \m{P}_0$ is of the form \[p_0 = (p_\ast, (p^\sigma_i, a^\sigma_i)_{\sigma \in \Lim\, , \, i < \alpha_\sigma})\] with $p_\ast \in P_\ast$ and $p^\sigma_i \in P^\sigma$ for all $(\sigma, i)$. 

The \textit{linking ordinals} $a^\sigma_i$ will determine how the $i$-th generic $\kappa_\sigma$-subset $G^\sigma_i$, given by the projection of the generic filter $G$ onto $P^\sigma_i$, will be eventually linked with the $P_\ast$-generic filter $G_\ast$. 

\begin{definition} Let $\m{P}_0$ be the collection of all $p_0 = (p_\ast, (p^\sigma_i, a^\sigma_i)_{\sigma \in \Lim\, ,\, i < \alpha_\sigma})$ such that: \begin{itemize} \item The \textit{support} of $p_0$, $\supp p_0$, is countable with $p^\sigma_i = a^\sigma_i = \emptyset$ whenever $(\sigma, i) \notin \supp p_0$. \item We have $p_\ast \in P_\ast$, and $p^\sigma_i \in P^\sigma$ for all $(\sigma, i) \in \supp p_0$. \item The domains of the $p^\sigma_i$ are coherent in the following sense: \\ If $\dom p_\ast = \bigcup_{\nu < \gamma, j < \cf \kappa_{\nu + 1}} [\kappa_{\nu, j}, \delta_{\nu, j})^2$, then for every $(\sigma, i) \in \supp p_0$, it follows that $\dom p^\sigma_i = \bigcup_{\nu < \sigma , j < \cf \kappa_{\nu + 1}} [\kappa_{\nu, j}, \delta_{\nu, j})$.

We set $\dom p_0 := \bigcup_{\nu, j} [\kappa_{\nu, j}, \delta_{\nu, j})$.
%$\dom p^\sigma_i = \bigcup_{\nu < \sigma, i < \cf \kappa_{\nu + 1}} [\kappa_{\nu, j}, \delta_{\nu, j})$. 
\item For all $(\sigma, i) \in \supp p_0$, we have $a^\sigma_i \subseteq \kappa_\sigma$ with $ |a^\sigma_i \, \cap \, [\kappa_{\nu, j}, \kappa_{\nu, j+1})| = 1$ for all intervals $[\kappa_{\nu, j}, \kappa_{\nu, j + 1}) \subseteq \kappa_\sigma$. \\ If $(\sigma_0, i_0) \neq (\sigma_1, i_1)$, then $a^{\sigma_0}_{i_0} \, \cap \, a^{\sigma_1}_{i_1} = \emptyset$. (We call this the {\upshape independence property}).\end{itemize}

{ \upshape Concerning the partial ordering $\leq_0$, any linking ordinal $\{\xi\} = a^\sigma_i \, \cap \, [\kappa_{\nu, j}, \kappa_{\nu, j + 1})$ settles that whenever $q_0 \leq p_0$, the extension $q^\sigma_i \supseteq p^\sigma_i$ within in the interval $[\kappa_{\nu, j}, \kappa_{\nu, j +1})$ is determined by $q_\ast (\xi)$: } \\[-3mm]

For $p_0 = (p_\ast, (p^\sigma_i, a^\sigma_i)_{\sigma, i})$, $q_0 = (q_\ast, (q^\sigma_i, b^\sigma_i)_{\sigma, i}) \in \m{P}_0$, let $q_0 \leq_0 p_0$ if the following holds: $q_\ast \supseteq p_\ast$; $q^\sigma_i \supseteq p^\sigma_i\, , \, b^\sigma_i \supseteq a^\sigma_i$ for all $(\sigma, i) \in \supp p_0$, and whenever $\zeta \in (\dom q^\sigma_i \setminus \dom p^\sigma_i) \, \cap \, [\kappa_{\nu, j}, \kappa_{\nu, j + 1})$ with $a^\sigma_i \, \cap \, [\kappa_{\nu, j}, \kappa_{\nu, j + 1}) = \{\xi \}$, then $\xi \in \dom q_0$ with $q^\sigma_i (\zeta) = q_\ast (\xi, \zeta)$ (we call this the {\upshape linking property }). \\[-4mm]

%\colorbox{yellow} {STIMMT DAS SO??}

The {\upshape maximal element} of $\m{P}_0$ is $\m{1}_0 := (\emptyset, (\emptyset, \emptyset)_{\sigma < \gamma, i < \alpha_\sigma})$.

\end{definition}

%For a condition $p_0 \in \m{P}_0$ with $\dom p_\ast = \bigcup_{\nu, j} [\kappa_{\nu, j}, \delta_{\nu, j})^2$, we set $\dom p_0 := \bigcup_{\nu, j} [\kappa_{\nu, j}, \delta_{\nu, j})$. 

%Erwähnen, dass alle Intervalle $[\kappa_{\nu, j}, \kappa_{\nu, j + 1})$ nun \tbl riesig\tbr\, werden?
%One major difference between $\m{P}_0$ and the forcing in $[GK]$ is that we now have a much stronger independence property; so if $\Lim$ is cofinal in $\gamma$, one can easily write down $\m{P}_0$-antichains of size $\alpha_\gamma$, which stands in contrast to [GK, Lemma 3].

%\colorbox{yellow} {FRAGE: Sollte auf den Fehler in \ref{GK} hier eingegangen werden? Oder  nur in der Arbeit?}

%Hence, many cardinals will be collapsed; and
Let $G_0$ denote a $V$-generic filter on $\m{P}_0$, and $g^\sigma_i := \bigcup \{a^\sigma_i\ | \ p = (p_\ast, (p^\sigma_i, a^\sigma_i)_{\sigma, i}) \in G_0\}$. 

Note that by our strong \textit{independence property}, every interval $[\kappa_{\nu, j}, \kappa_{\nu, j+1})$ will be blown up to size $\sup\{\alpha_\sigma \ | \ \sigma \in \Lim\}$ in a $\m{P}_0$-generic extension.

Hence, since we want our eventual symmetric submodel $N$ preserve all $V$-cardinals, we will have so make sure that $N$ \tbl does not know\tbr\,the sequence of linking ordinals $(g^\sigma_i\ | \ \sigma \in \Lim\, , \, i < \alpha_\sigma)$. \\[-3mm]

A major difference between our forcing and the basic construction in \cite{GK} is the following: 
%If we stuck to [GK], then all $a^\sigma_i$ had to be finite, and 
%if we denote by $G_0$ a $V$-generic filter on $\m{P}_0$, then $g^\sigma_i := \bigcup \{a^\sigma_i\ | \ p = (p_\ast, (p^\sigma_i, a^\sigma_i)_{\sigma, i}) \in G_0\}$ 
The forcing conditions in \cite[Definition 2]{GK} have \textit{finite} linking ordinals $a^\sigma_i$; so the according generics $g^\sigma_i$ are \textbf{not} contained in the ground model $V$.
%for any $g^\sigma_i$ it would follows that 
%would not be contained in the ground model $V$. 
With our definition however, it follows for any $p \in G_0$ with $(\sigma, i) \in \supp p$ that $g^\sigma_i = a^\sigma_i \in V$.
%; hence $g^\sigma_i \in V$. 
By countable support, also countable sequences of linking ordinals $(g^{\sigma_j}_{i_j} \ | \ j < \omega)$ are contained in $V$; but for $\sigma \in \Lim$ not the whole sequence $(g^\sigma_i\ | \ i < \alpha_\sigma)$. 

This modification helps to establish that any generic $G^\sigma_i$ can be described using only $G_\ast$ and sets from the ground model $V$ (see below). \\[-2mm]

%FRAGE: Sollte man hier nur in ein paar Sätzen erwähnen, was der Unterschied ist zum ursprünglichen Forcing in [GK]? Sollte dieses korrigiert werden? Oder \tbl erst \tbr\, in der Arbeit? \\[-2mm]

Next, we define our forcing notion $\m{P}_1$, which will be in charge of the successor cardinals. For every $\sigma \in \Succ$ with $\kappa_\sigma = \ol{\kappa_\sigma\,}^{\, +}$, it follows that $\sigma =: \ol{\sigma} + 1$ must be a successor ordinal, since we have assumed in the beginning that the sequence $(\kappa_\sigma\ | \ 0 < \sigma < \gamma)$ is closed. \\[-3mm]
%say $\sigma = \ol{\sigma} + 1$, 

We denote by $P^\sigma$ the Cohen forcing \[P^\sigma := \{p : \dom p \rightarrow 2 \ \big| \ \dom p \subseteq [\ol{\kappa_\sigma}, \kappa_\sigma)\, , \, |\dom p| < \kappa_\sigma \},\] and let \[C^\sigma := \{p : \dom p \rightarrow 2 \ | \ \dom p = \dom_x p \, \times \m \dom_y p \; \subseteq \; \alpha_\sigma\, \times \, [\ol{\kappa_\sigma}, \kappa_\sigma)\; , \; |\dom p| < \kappa_\sigma\}.\]

%\colorbox{red}{ACHTUNG - Intervalle $[\ol{\kappa_\sigma}, \kappa_\sigma)$ betrachten, oder $[\kappa_{\ol{\sigma}}, \kappa_\sigma)$? DAS IST NICHT KLAR!}

Then both $P^\sigma$ and $C^\sigma$ are $< \kappa_\sigma$-closed, and if $2^{< \kappa_\sigma} = \kappa_\sigma$, i.e. $2^{\ol{\kappa_\sigma}} = \kappa_\sigma$, then they satisfy the $\kappa_\sigma^\plus$-chain condition and hence, preserve cardinals.

In particular, any forcing $P^\sigma$ or $C^\sigma$ preserves cardinals if we are working in our ground model $V$ with $V \vDash GCH$, or any $V$-generic extension by $ \leq \kappa_\sigma$-closed forcing.

%$P^\sigma$ and $C^\sigma$ preserve cardinals.

%\colorbox{red}{ACHTUNG - braucht man hier GCH im Grundmodell?} \\[-3mm]

%\colorbox{red} {ACHTUNG - da gibt es ein Problem!} 
%since they are $< \kappa_\sigma$-closed and satisfy the $\kappa_\sigma^\plus$-chain condition.

\begin{definition} The forcing notion $(\m{P}_1, \leq_1, \emptyset_1)$ consists of all $p_1 = (p^\sigma)_{\sigma \in \Succ}$ with countable support $\supp p_1 := \{ \sigma \in \Succ\ | \ p^\sigma \neq \emptyset\}$, and $p^\sigma \in C^\sigma$ for all $\sigma \in \Succ$. For $p_1 = (p^\sigma)_{\sigma \in \Succ}$, $q_1 = (q^\sigma)_{\sigma \in \Succ} \in \m{P}_1$, we let $q_1 \leq_1 p_1$ if $q^\sigma \supseteq p^\sigma$ for all $\sigma \in \Succ$; and $\m{1}_1 := (\emptyset)_{\sigma \in \Succ}$ is the maximal element. \end{definition}

For $\sigma \in \Succ$ and $i < \alpha_\sigma$, we set $p^\sigma_i = \{(\zeta, p^\sigma (i, \zeta))\ | \ (i, \zeta) \in \dom p^\sigma\}$. \\[-2mm]

Our main forcing will be the product $\m{P} :=\m{P}_0\, \times \,\m{P}_1$ with maximal element $\m{1} := (\m{1}_0, \m{1}_1)$ and order relation $\leq$. In order to simplify notation, we write conditions $p \in \m{P}$ in the form $p = (p_\ast, (p^\sigma_i, a^\sigma_i)_{\sigma \in \Lim, i < \alpha_\sigma}, (p^\sigma)_{\sigma \in \Succ})$. \\[-2mm]

%Since we want $DC$ in our eventual symmetric submodel $N$, it will be importsant that our forcing is countably closed; which follows by countable support

It is not difficult to verify:

\begin{prop} \label{countablyclosed} $\m{P}$ is countably closed.\end{prop}

This is important to make sure that $DC$ holds in our eventual symmetric extension $N$. \\

For $0 < \eta \leq \gamma$ (with $\eta \in Lim$ or $\eta \in Succ$ or $\eta = \gamma$), we define a forcing $\ol{P}^\eta$ like $P^\eta$ is defined in the case that $\eta \in Lim$: \\[-3mm]

\textit{Let $\ol{P}^\eta$ consist of all functions $p: \dom p \rightarrow 2$ such that there is a sequence $(\delta_{\nu, j}\ | \ \nu < \eta, j < \cf \kappa_{\nu + 1})$ with $\delta_{\nu, j} \in [\kappa_{\nu, j}, \kappa_{\nu, j + 1})$ for all $\kappa_{\nu, j} < \kappa_\eta$, and \[\dom p = \bigcup_{\nu, j}\, [\kappa_{\nu, j}, \delta_{\nu, j}),\] such that $|p \uhr \kappa_{\ol{\nu}, \ol{\j}}| < \kappa_{\ol{\nu}, \ol{\j}}$ whenever $\kappa_{\ol{\nu}, \ol{\j}}$ is a regular cardinal, and $|p| < \kappa_\eta$ in the case that $\kappa_\eta$ itself is regular.} \\[-3mm]

For any $0 < \eta < \lambda$ with $\kappa_\lambda$ a limit cardinal, it follows that $\ol{P}^\eta = P^\lambda \uhr \kappa_\eta$. \\

%it follows that $\m{P}$ is $\omega$-closed. \\[-2mm]

Let now $G$ be a $V$-generic filter on $\m{P}$. It induces \[G_\ast := \{ q_\ast \in P_\ast\ | \ \exists\, p = (p_\ast, (p^\sigma_i, a^\sigma_i)_{\sigma, i}, (p^\sigma)_\sigma) \in G \: : \: q_\ast \subseteq p_\ast\},\] and for $\lambda \in Lim$, $k < \alpha_\lambda$: \[ G^\lambda_k := \{q^\lambda_k \in P^\lambda \ | \ \exists\, p = (p_\ast, (p^\sigma_i, a^\sigma_i)_{\sigma, i}, (p^\sigma)_\sigma) \in G\: : \: q^\lambda_k \subseteq p^\lambda_k\}.\] As usually, these filters $G_\ast$, $G^\lambda_k$ are identified with their union $\bigcup G_\ast$, $\bigcup G^\lambda_k$. Then any $G^\lambda_k$ can be regarded a subset of $\kappa_\lambda$.\\[-3mm]

%: \kappa_\sigma \rightarrow 2$.. \\ 
Moreover, let \[g^\lambda_k := \bigcup \{a^\lambda_k\ | \ p = (p_\ast, (p^\sigma_i, a^\sigma_i)_{\sigma, i}, (p^\sigma)_\sigma) \in G \}. \] Then $g^\lambda_k = a^\lambda_k$ for any $p \in G$ with $(\lambda, k) \in \supp p_0$; %\{a^\sigma_i, \emptyset\}$, and we confuse $g^\sigma_i$ with $\bigcup g^\sigma_i \subseteq \kappa_\sigma$, 
and $g^\lambda_k$ hits any interval $[\kappa_{\nu, j}, \kappa_{\nu, j + 1}) \subseteq \kappa_\lambda$ in exactly one point. By the \textit{independence property}, it follows that $g^{\lambda_0}_{k_0} \, \cap \, g^{\lambda_1}_{k_1} = \emptyset$ whenever $(\lambda_0, k_0) \neq (\lambda_1, k_1)$.\\[-3mm]

For $\lambda \in Succ$, set \[G^\lambda := \{p^\lambda\ | \ p = (p_\ast, (p^\sigma_i, a^\sigma_i)_{\eta, i}, (p^\sigma)_\sigma) \in G\},\] and \[G^\lambda_k := \{ p^\lambda_k\ | \ p = (p_\ast, (p^\sigma_i, a^\sigma_i)_{\sigma, i}, (p^\sigma)_\sigma) \in G\}\] for any $k < \alpha_\lambda$. \\[-3mm]

Again, we confuse these filters $G^\lambda$, $G^\lambda_k$ with their union $\bigcup G^\lambda$, $\bigcup G^\lambda_k$. \\[-3mm]
%Then any

Let now $\xi \in [\kappa_{\nu, j}, \kappa_{\nu, j + 1})$. We denote by \[G_\ast (\xi) := \big\{\, q: [\kappa_{\nu, j}, \delta_{\nu, j}) \rightarrow 2\ \ \Big| \ \ \delta_{\nu, j} \in [\kappa_{\nu, j}, \kappa_{\nu, j + 1})\ \ , \ \  \exists\, p = (p_\ast, (p^\sigma_i, a^\sigma_i)_{\sigma, i}, (p^\sigma)_\sigma) \in G\ : \]\[ \ \forall\, \zeta \in \dom q\ \:q (\zeta) = p_\ast (\xi, \zeta)\, \big\} \]

%\[G_\ast (\xi) := \{ \, (\zeta, G_\ast (\xi, \zeta)) \ | \ \zeta \in [\kappa_{\nu, j}, \kappa_{\nu, j+1}) \, \}\] 

the $\xi$-th section of $G_\ast$. \\[-3mm]

If $a \subseteq \kappa_\gamma$ is a set that hits any interval $[\kappa_{\nu, j}, \kappa_{\nu, j+1}) \subseteq \kappa_\gamma$ in at most one point, we denote by $G_\ast (a)$ the set of all $q \in \ol{P}^\gamma$ such that there is $p \in G$ with $q \subseteq p_\ast (a)$.

%\begin{itemize} \item there is $p = (p_\ast, (p^\sigma_i, a^\sigma_i)_{\sigma, i}, (p^\sigma)_\sigma) \in G$ such that for all intervals $[\kappa_{\nu, j}, \kappa_{\nu, j + 1}) \subseteq \kappa_\gamma$ with $a\, \cap\, [\kappa_{\nu, j}, \kappa_{\nu, j + 1}) = \{ \xi\}$, it follows that $q (\zeta) = p_\ast (\xi, \zeta)$, \item for all intervals $[\kappa_{\nu, j}, \kappa_{\nu, j + 1})$ with $a\, \cap\, [\kappa_{\nu, j}, \kappa_{\nu, j + 1}) = \emptyset$, it follows that $\dom q\, \cap\, [\kappa_{\nu, j}, \kappa_{\nu, j + 1}) = \emptyset$. \end{itemize}

As before, we identify any $G_\ast (\xi)$ and $G_\ast (a)$ with their union $\bigcup G_\ast (\xi)$ and $\bigcup G_\ast (a)$, respectively. Then any $G_\ast (\xi)$ with $\xi \in [\kappa_{\nu j}, \kappa_{\nu, j + 1})$ can be regarded as a function $G_\ast (\xi): [\kappa_{\nu, j}, \kappa_{\nu, j + 1}) \rightarrow 2$, and any $G_\ast (a)$ becomes a function $G_\ast (a): dom\,G_\ast (a) \rightarrow 2$, where $dom \,G_\ast (a) \subseteq \kappa_\gamma$ is the union of those intervals $[\kappa_{\nu, j}, \kappa_{\nu, j + 1})$ with $a \, \cap\, [\kappa_{\nu, j}, \kappa_{\nu, j + 1}) \neq \emptyset$. \\

%$dom\, q \, \cap\, [\kappa_{\nu, j}, \kappa_{\nu, j + 1}) = \emptyset 

%\[G_\ast (a) := \big\{\, q \in \ol{P}^\gamma\ \ \Big|\ \ \exists\, p = (p_\ast, (p^\sigma_i, a^\sigma_i)_{\sigma, i}, (p^\sigma)_\sigma) \in G \ : \ \forall\, \kappa_{\nu, j} < \kappa_\gamma  \]

%\[G_\ast (a) := \big \{ \, (\zeta, G_\ast (\xi, \zeta))\ | \ \exists\, \kappa_{\nu, j} \ \big(\, \zeta \in [\kappa_{\nu, j}, \kappa_{\nu, j+1})\, \wedge\, a \,\cap \, [\kappa_{\nu, j}, \kappa_{\nu, j + 1}) = \{\xi\}\,\big)\, \big\}.\]

Now, the \textit{linking property} implies that any $G^\lambda_k \uhr [\kappa_{\nu, j}, \kappa_{\nu, j + 1})$ with $\lambda \in \Lim$, $k < \alpha_\lambda$, is eventually equal to $G_\ast (\xi)$, where $\{\xi \} := a^\lambda_k\, \cap \, [\kappa_{\nu, j}, \kappa_{\nu, j + 1})$. 

Indeed, the symmetric difference $G^\lambda_k \oplus G_\ast (g^\lambda_k)$ is always an element of the ground model $V$:
Take a condition $p \in G$ with $(\lambda, k) \in \supp p_0$, such that for any interval $[\kappa_{\nu, j}, \kappa_{\nu, j + 1}) \subseteq \kappa_\lambda$ with $\dom p_0\, \cap\, [\kappa_{\nu, j}, \kappa_{\nu, j + 1}) \neq \emptyset$ and $\{\xi\} := a^\lambda_k\, \cap \, [\kappa_{\nu, j}, \kappa_{\nu, j + 1})$, it follows that $\xi \in \dom p_0$. \big(This does not interfere with the condition that $\dom p_0$ has to be bounded below all regular $\kappa_{\ol{\nu}, \ol{\j}}$, since we do not bother the intervals $[\kappa_{\nu, j}, \kappa_{\nu, j + 1})$ with $\dom p_0\, \cap\, [\kappa_{\nu, j}, \kappa_{\nu, j + 1}) = \emptyset$.\big)\\[-3mm]

%\colorbox{red}{ACHTUNG - ist es in Ordnung, dass hier die Filter mit ihrer Vereinigung identifiziert werden?}

\begin{itemize} \item Firstly, $G^\lambda_k (\zeta) \, \oplus \, G_\ast (g^\lambda_k) (\zeta) = 0$ whenever $\zeta \notin \dom p_0$: Let $\zeta \in [\kappa_{\nu, j}, \kappa_{\nu, j + 1})$, $\zeta \notin \dom p$ with $\{ \xi\} := g^\lambda_k\, \cap\, [\kappa_{\nu, j}, \kappa_{\nu, j + 1}) = a^\lambda_k\, \cap\, [\kappa_{\nu, j}, \kappa_{\nu, j + 1})$. Take $q \in G$, $q \leq p$ with $\zeta \in \dom q_0$. Then by the \textit{linking property}, it follows that $\xi \in \dom q_0$ with $q^\lambda_k (\zeta) = q_\ast (\xi, \zeta)$.
%since $\zeta \in (\dom q \setminus \dom p)\, \cap\, [\kappa_{\nu, j}, \kappa_{\nu, j + 1})$ with $a^\lambda_k\, \cap\, [\kappa_{\nu, j}, \kappa_{\nu, j + 1}) = \{ \xi\}$. 
Hence, $G^\lambda_k (\zeta) = q^\lambda_k (\zeta) = q_\ast (\xi, \zeta) = G_\ast (g^\lambda_k) (\zeta)$, and $G^\lambda_k (\zeta) \oplus G_\ast (g^\lambda_k) (\zeta) = 0$. \item If $\zeta \in \dom p$ then $G^\lambda_k (\zeta) \oplus G_\ast (g^\lambda_k) (\zeta) = p^\lambda_k (\zeta) \oplus p_\ast (\xi, \zeta)$, where again, $\zeta \in [\kappa_{\nu, j}, \kappa_{\nu, j + 1})$ and $\{\xi\} := g^\lambda_k\, \cap\, [\kappa_{\nu, j}, \kappa_{\nu, j + 1}) = a^\lambda_k\, \cap\, [\kappa_{\nu, j}, \kappa_{\nu, j + 1})$. Here we use that for any interval $[\kappa_{\nu, j}, \kappa_{\nu, j + 1})$ with $\dom p\, \cap\, [\kappa_{\nu, j}, \kappa_{\nu, j + 1}) \neq \emptyset$, it follows that $a^\lambda_k\, \cap\, [\kappa_{\nu, j}, \kappa_{\nu, j + 1}) \subseteq \dom p_0$.
\end{itemize}

Hence, $G^\lambda_k \oplus G_\ast (g^\lambda_k)$ can be calculated in $V$. \\[-3mm]

%it follows that $G^\lambda_k (\zeta) = G_\ast (g^\lambda_k) (\zeta)$ whenever

This will be employed to keep control over the surjective size of $\powerset (\kappa_\lambda)$ in the eventual symmetric extension $N$. \\[-2mm]
% $\kappa_\lambda$-subsets make their way into the eventual model $N$. \\[-2mm]

%Now, for $0 < \eta \leq \gamma$ (with $\eta \in \Lim$ or $\eta \in \Succ$ or $\eta = \gamma$), we define a forcing $\ol{P}^\eta$ like $P^\eta$ is defined in the case that $\eta \in \Lim$: \\[-3mm]

%Let $\ol{P}^\eta$ consist of all functions $p: \dom p \rightarrow 2$ such that there is a sequence $(\delta_{\nu, j}\ | \ \nu < \eta, j < \cf \kappa_{\nu + 1})$ with $\delta_{\nu, j} \in [\kappa_{\nu, j}, \kappa_{\nu, j + 1})$ for all $\kappa_{\nu, j} < \kappa_\eta$, and \[\dom p = \bigcup_{\nu, j}\, [\kappa_{\nu, j}, \delta_{\nu, j})\] such that $|p \uhr \kappa_{\ol{\nu}, \ol{\j}}| < \kappa_{\ol{\nu}, \ol{\j}}$ whenever $\kappa_{\ol{\nu}, \ol{\j}}$ is a regular cardinal, and $|p| < \kappa_\eta$ in the case that $\kappa_\eta$ itself is regular. \\ For any $0 < \eta < \lambda$ with $\kappa_\lambda$ a limit cardinal, it follows that $\ol{P}^\eta = P^\lambda \uhr \kappa_\eta$. \\[-2mm]

%\colorbox{yellow}{TO DO: Durchgehen, wo der Zusatz $0 < ...$ fehlt!} \\[-2mm]

Now, we consider countable products $\prod_{m < \omega} P^{\sigma_m}$ and $\prod_{m < \omega} \ol{P}^{\sigma_m}$: 

\begin{definition} Let $( (\sigma_m, i_m)\ | \ m < \omega)$ be a sequence of pairwise distinct pairs with $0 < \sigma_m < \gamma$, $i < \alpha_{\sigma_m}$ for all $m < \omega$. We denote by $\prod_{m < \omega} P^{\sigma_m}$ the set of all $(p(m)\ | \ m < \omega)$ with $p(m) \in P^{\sigma_m}$ for all $m < \omega$ (with full support), and similarly, $\prod_{m < \omega} \ol{P}^{\sigma_m} := \big\{\, (p(m)\ | \ m < \omega)\ \big| \ \forall\, m < \omega\ \, p(m) \in \ol{P}^{\sigma_m}\, \big\}$. \end{definition}

For any interval $[\kappa_{\nu, j}, \kappa_{\nu, j + 1}) \subseteq \kappa_\gamma$, it follows that $\prod_{m < \omega} P^{\sigma_m}\, \uhr \, \kappa_{\nu, j}$ has cardinality $\leq \kappa_{\nu, j}$ in the case that $\kappa_{\nu, j}$ is regular, and cardinality $\leq \kappa_{\nu, j}^\plus$, else. Moreover, $\prod_{m < \omega} P^{\sigma_m}\, \uhr\, [\kappa_{\nu, j}, \kappa_{\sigma_m})$ is $\leq \kappa_{\nu, j}$-closed. Hence, as in Lemma \ref{prescard1} and Corollary \ref{prescof}, one can show that the product $\prod_{m < \omega} P^{\sigma_m}$ preserves cardinals, cofinalities and the $GCH$. 

Similarly, $\prod_{m < \omega} \ol{P}^{\sigma_m}$ preserves cardinals, cofinalities and the $GCH$. \\[-2mm]

The next lemma implies that countable products $\prod_{m < \omega} G_\ast (g^{\sigma_m}_{i_m})$ 
%for $((\sigma_m, i_m)\ | \ m < \omega)$ a sequence of pairwise distinct indices, 
are $V$-generic over $\prod_{m < \omega} \ol{P}^{\sigma_m}$:

%a rather \tbl mild\tbr\, forcing notion (cf. Lemma \ref{prescard1}):

%\colorbox{red}{ACHTUNG - \tbl product\tbr\,bedeutet doch mit endlichen Support!!}

\begin{lem} \label{generic} Consider a sequence $(a_m\ | \ m < \omega)$ of pairwise disjoint sets such that for all $m < \omega$, the following holds: $a_m$ is a subset of $\kappa_{\sigma_m}$ for some $0 < \sigma_m < \gamma$, and for all $\kappa_{\nu, j} < \kappa_{\sigma_m}$, it follows that $|a_m \, \cap \, [\kappa_{\nu, j}, \kappa_{\nu, j+1})| = 1$, i.e. $a_m$ hits every interval $[\kappa_{\nu, j}, \kappa_{\nu, j+1}) \subseteq \kappa_{\sigma_m}$ in exactly one point. 
%Furthermore, assume $a_m \, \cap \, a_{k^\prime} = \emptyset$ whenever $k \neq k^\prime$. 
Then $\prod_{m < \omega} G_\ast (a_m) := \big\{ \, (p(m)\ | \ m < \omega)\ \big| \ \forall\, m < \omega\ \,p(m) \in G_\ast (a_m)\, \big\}$ is a $V$-generic filter on $\prod_{m < \omega} \ol{P}^{\sigma_m}$. \end{lem}
%The following lemma states that countable products $G_\ast (a_0)\, \times\, \cdots\, \times \, G_\ast (a_{n-1})

\begin{proof} Let $D \subseteq \prod_{m < \omega} P^{\sigma_m}$ be an open dense set in $V$. We show that $D \, \cap \, \prod_{m < \omega} G_\ast (a_m) \neq \emptyset$. Let \[\overline{D} := \big\{p = (p_\ast, (p^\sigma_i, a^\sigma_i)_{\sigma, i}, (p^\sigma)_\sigma) \in \m{P}\ \big| \ (p_\ast (a_m)\ | \ m < \omega) \in D \big\}.\] It suffices to prove that $\ol{D}$ is dense in $\m{P}$. Assume $p = (p_\ast, (p^\sigma_i, a^\sigma_i)_{\sigma, i}, (p^\sigma)_\sigma) \in \m{P}$ given, and denote by $(q_m\ | \ m < \omega)$ an extension of $(p_\ast (a_m)\ | \ m < \omega)$ in $D$. We have to construct $\ol{p} \leq p$ such that $\ol{p}_\ast (a_m) \supseteq q_m$ for all $m < \omega$. 

Consider an interval $[\kappa_{\nu, j}, \kappa_{\nu, j + 1}) \subseteq \kappa_\gamma$. In the case that $\big(\dom p \, \cup\, \bigcup_{m < \omega} \dom q_m\big)\, \cap \, [\kappa_{\nu, j}, \kappa_{\nu, j+1}) = \emptyset$, let $\delta_{\nu, j} := \kappa_{\nu, j}$. Otherwise, we pick $\delta_{\nu, j} \in [\kappa_{\nu, j}, \kappa_{\nu, j + 1})$ such that firstly, \[\big( \dom p \, \cup \, \bigcup_{m < \omega} \dom q_m \big)\, \cap \, [\kappa_{\nu, j}, \kappa_{\nu, j + 1}) \subseteq [\kappa_{\nu, j}, \delta_{\nu, j});\] secondly, for all $m < \omega$, it follows that $a_m \, \cap \, [\kappa_{\nu, j}, \kappa_{\nu, j + 1}) \subseteq [\kappa_{\nu, j}, \delta_{\nu, j})$; and thirdly, $a^\sigma_i\, \cap \, [\kappa_{\nu, j}, \kappa_{\nu, j + 1}) \subseteq [\kappa_{\nu, j}, \delta_{\nu, j})$ for all $(\sigma, i) \in \supp p$. This is possible, since the sets $a_m\, \cap\, [\kappa_{\nu, j}, \kappa_{\nu, j + 1})$ and $a^\sigma_i\, \cap \, [\kappa_{\nu, j}, \kappa_{\nu, j + 1})$ are singletons or empty, all the domains $\dom p \, \cap\, [\kappa_{\nu, j}, \kappa_{\nu, j + 1})$ and $\dom q_m \, \cap \, [\kappa_{\nu, j}, \kappa_{\nu, j + 1})$ for $m < \omega$ are bounded below $\kappa_{\nu, j + 1}$, and $\kappa_{\nu, j + 1}$ is always a successor cardinal. \\Let \[ \dom \ol{p}_0 := \bigcup_{\nu, j} \,[\kappa_{\nu, j}, \delta_{\nu, j}).\] Then $\dom \ol{p}_0$ is bounded below all regular $\kappa_{\ol{\nu}, \ol{\j}}$, since this holds true for $\dom p$ and $\bigcup_{m < \omega} \dom q_m$. We define $\ol{p}_\ast$ on $\bigcup_{\nu, j} [\kappa_{\nu, j}, \delta_{\nu, j})^2$ as follows: Consider an interval $[\kappa_{\nu, j}, \delta_{\nu, j}) \neq \emptyset$ and $\xi$, $\zeta \in [\kappa_{\nu, j}, \delta_{\nu, j})$.  For $(\xi, \zeta) \in \dom p\, \times\, \dom p$, let $\ol{p}_\ast (\xi, \zeta) := p_\ast (\xi, \zeta)$. If $\{\xi \} = a_m\, \cap \, [\kappa_{\nu, j}, \kappa_{\nu, j + 1})$ for some $ m < \omega$ and $\zeta \in \dom q_m$, we set $\ol{p}_\ast (\xi, \zeta) := q_m (\zeta)$. This is not a contradiction towards $\ol{p}_\ast \, \uhr\, [\kappa_{\nu, j}, \delta_{\nu, j})^2 \supseteq p_\ast\, \uhr\, [\kappa_{\nu, j}, \delta_{\nu, j})^2$, since $q_m \supseteq p_\ast (a_m)$ for all $m < \omega$. Also, the $a_m$ are pairwise disjoint, so for any $\xi \in [\kappa_{\nu, j}, \delta_{\nu, j})$, there is at most one $m$ with $\xi \in a_m$. For all the remaining $(\xi, \zeta) \in \dom \ol{p}_\ast$, we can set $\ol{p}_\ast (\xi, \zeta) \in \{0, 1\}$ arbitrarily. This defines $\ol{p}_\ast$ on $\bigcup_{\nu, j} [\kappa_{\nu, j}, \delta_{\nu, j})^2$. \\[-3mm]

For all $(\sigma, i) \in \supp \ol{p}_0 := \supp p_0$, we set $\ol{a}^\sigma_i := a^\sigma_i$, and define $\ol{p}^\sigma_i \supseteq p^\sigma_i$ on the corresponding domain $\bigcup_{\kappa_{\nu, j} < \kappa_\sigma}  [\kappa_{\nu, j}, \delta_{\nu, j})$ according to the \textit{linking property}: Whenever $\zeta \in (\dom \ol{p}_0 \setminus \dom p)\, \cap \, [\kappa_{\nu, j}, \kappa_{\nu, j + 1})$ and $a^\sigma_i\, \cap \, [\kappa_{\nu, j}, \kappa_{\nu, j + 1}) = \{ \xi \}$, then $\xi \in \dom \ol{p}_0$ follows, so we can set $\ol{p}^\sigma_i (\zeta) := \ol{p}_\ast (\xi, \zeta)$. For the $\zeta \in \dom \ol{p}^\sigma_i \setminus \dom p^\sigma_i$ remaining, we can define $\ol{p}^\sigma_i (\zeta)$ arbitrarily. This completes the construction of $\ol{p}_0$. \\[-3mm]

%ACHTUNG - man müsste UMBENBENNEN!
Let $\ol{p}_1 := p_1$. It is not difficult to check that $\ol{p} \leq p$ indeed is a condition in $\m{P}$ with $\ol{p}_\ast (a_m) \supseteq q_m$ for all $m < \omega$. Hence, $(\ol{p}_\ast (a_m)\ | \ m < \omega) \in D$, and $\ol{p} \in \ol{D}$ as desired. \end{proof}

In particular, for $((\sigma_m, i_m)\ | \ m < \omega)$ a sequence of pairwise distinct pairs as before with $\sigma_m \in \Lim$, $i_m < \alpha_{\sigma_m}$ for all $m < \omega$,
%with $0 < \sigma_m < \gamma$, $i_m < \alpha_{\sigma_m}$ for all $m < \omega$, and $(\sigma_m, i_m) \neq (\sigma_{m^\prime}, i_{m^\prime})$ whenever $m \neq m^\prime$; 
it follows that $\prod_{m < \omega} G_\ast (g^{\sigma_m}_{i_m})$ 
%for $((\sigma_m, i_m)\ | \ m < \omega)$ a sequence of pairwise distinct indices, 
is a  $V$-generic filter over $\prod_{m < \omega} \ol{P}^{\sigma_m}$. \\[-3mm]

Similarly, one can show:

\begin{lem} \label{generic2} Let $((\sigma_m, i_m)\ | \ m < \omega)$ denote a sequence of pairwise distinct pairs with $ 0 < \sigma_m < \gamma$, $i_m < \alpha_{\sigma_m}$ for all $m < \omega$. Then $\prod_{m < \omega} G^{\sigma_m}_{i_m} := \big\{ \, (p(m)\ | \ m < \omega)\ \big| \ \forall\, m < \omega \ \, p(m) \in G^{\sigma_m}_{i_m}\,\big\}$ is a $V$-generic filter on $\prod_{m < \omega} P^{\sigma_m}$. \end{lem}

%The proof is similar as for Lemma \ref{generic}. 
%\colorbox{yellow} {FRAGE: Sollte man das genauer erklären? Erwähnen, dass man für die Erweiterungen} \\ \colorbox{yellow}{Freiheit in der $\ast$-Komponente hätte?}

\section{Symmetric names.} \label{symmetric names}

\subsection{Outline.} \label{outline}

For constructing our symmetric submodel $N$, we will define a group $A$ of $\m{P}$-automorphisms and a normal filter $\mathcal{F}$ on $A$. More accurately: In our setting, an automorphism $\pi \in A$ will not be defined on $\m{P}$ itself, but only on a dense subset $D_\pi \subseteq \m{P}$. We call such $\pi: D_\pi \rightarrow D_\pi$ a \textit{partial automorphism}. Hence, the set $A$ is not quite a group, but has a very similar structure:

For any $\pi$, $\sigma \in A$ with $\pi: D_\pi \rightarrow D_\pi$, $\sigma: D_\sigma \rightarrow D_\sigma$ and $p \in D_\pi\, \cap\, D_\sigma$, the image $\sigma (p)$ will be an element of $D_\pi\, \cap\, D_\sigma$ as well; and $A$ will contain a map $\nu: D_\nu \rightarrow D_\nu$ such that $D_\nu = D_\pi\, \cap\, D_\sigma$ and $\nu = \pi\, \circ\, \sigma$ on $D_\nu$. (We will call $\nu$ the \textit{concatenation} $\pi\, \circ\, \sigma$.)

Moreover, for any $\pi \in A$, there will be a map $\nu$ in $A$ with $D_\nu = D_\pi$ such that $\pi\, \circ\, \nu = \nu \, \circ\, \pi = \id_{D_\nu} = \id_{D_\pi}$ is the identity on $D_\nu = D_\pi$. (We call $\nu$ the \textit{inverse} $\pi^{-1}$.)

There will also be an \textit{identity element} $id \in A$, which is the identity map on its domain $D_{id}$, where $D_{id} \supseteq D_{\pi}$ for all $\pi \in A$. \\[-3mm]
% = \m{P}$. \\[-3mm]
%$id \,p = p$ for all $p \in \m{P}$. \\[-2mm]
%$id: \m{P} \rightarrow \m{P}$ 
%will the identity map on $\m{P}$. \\[-2mm]

%there will be a map $\nu \in D$
%The inverse maps will 
%be defined on the same dense subset 
%have the same domain $D_{\pi^{-1}} = D_\pi$; and for $\pi$, $\sigma \in A$, the concatenation $\pi \, \circ \, \sigma$ will be defined on $D_{\pi\, \circ\, \sigma} = D_\pi\, \cap\, D_\sigma$. The identity map $id_{\m{P}}: \m{P} \rightarrow \m{P}$ is defined on $D_{\id} = \m{P}$. \\

This does not quite give a group structure: For instance, for any $\pi \in A$, the concatenation $\pi\, \circ\, \pi^{-1} = \pi^{-1}\, \circ\, \pi = \id_{D_\pi}$ is not really the identity element $id$, which usually has a larger domain $D_{id}$. \\[-2mm]
%, since it would n ot be defined on the entire forcing $\m{P}$. \\[-2mm]

%give the identity map $\id_{D_{\pi}}$ on $D_\pi = D_{\pi^{-1}}$ instead of the the identity map $\id_{\m{P}}$ on the entire forcing $\m{P}$.

%We do not really and obtain a 
%Hence, this does not really give group of $\m{P}$-automorphisms, but a similar structure that we call a \textit{partial automorphism group for $\m{P}$}. \\ 
% denoted by $A$. \
%\colorbox{red}{ACHTUNG - das Wort \tbl Gruppe\tbr\,sollte wahrscheinlich nicht im Begriff vorkommen!} \\
%\colorbox{red} {FRAGE: Begriff?} \\[-3mm]

In this setting, the standard approach would be using Boolean-valued models for the construction of the symmetric submodel $N$: Any automorphism $\pi: D_\pi \rightarrow D_\pi$ can be uniquely extended to an automorphism of the complete Boolean algebra $B(\m{P})$, and thereby induces an automorphism of the Boolean valued model $V^{B(\m{P})}$. Then one can consider the group consisting of these extended automorphisms, define a normal filter and construct the corresponding symmetric submodel as described in \cite[Chapter 5]{Jech2}.

%\colorbox{yellow}{TO DO: NACHSEHEN bei Jech zur dieser Konstruktion! Kapitel nachsehen!} \\[-3mm]

%In our presentation, 
We try to avoid Boolean valued models here, and work with partial orders and automorphisms $\pi: D_{\pi} \rightarrow D_{\pi}$ on dense subsets $D_\pi \subseteq \m{P}$ instead. \\[-3mm]
%We call such $\pi: D_\pi \rightarrow D_\pi$ a \textit{partial automorphism} of $\m{P}$. 

%\colorbox{yellow}{Doppelt definiert.}\\[-2mm]

We will have a collection $\mathcal{D}$ of dense subsets $D \subseteq \m{P}$ with certain properties, and a collection $A$ of \textit{partial automorphisms} $\pi: D_\pi \rightarrow D_\pi$ with $D_{\pi} \in \mathcal{D}$ for any $\pi \in A$. Whenever $D \in \mathcal{D}$ is fixed, the automorphisms $\{\pi \in A\ | \ D_\pi = D\}$ will form a group that we denote by $A_D$. Moreover, for any $D$, $D^\prime \in \mathcal{D}$ with $D \subseteq D^\prime$ and $\pi \in A_{D^\prime}$, we will have $\pi [D] = D$, and the restriction $\pi \, \uhr\, D$ is an element of $A_D$. Hence, there are canonical homomorphisms $\phi_{D^\prime D}: A_{D^\prime} \rightarrow A_D$, $\pi \mapsto \pi\, \uhr\, D$ for any $D$, $D^\prime \in \mathcal{D}$ with $D \subseteq D^\prime$. This gives a directed system, and we can take the direct limit \[\ol{A} := \lim\limits_{\longrightarrow}{A_D} = \bigsqcup A_D \big/ \sim\] with the following equivalence relation \tbl $\sim$\tbr: Whenever $\pi \in A_D$ and $\pi^\prime \in A_{D^\prime}$, then $\pi \sim \pi^\prime$ iff there exists $D^{\prime \prime} \in A$, $D^{\prime \prime} \subseteq D\, \cap \, D^\prime$, such that $\pi$ and $\pi^\prime$ agree on $D^{\prime \prime}$. Since for any $D$, $D^\prime \in A$, the intersection $D\, \cap\, D^\prime$ will be contained in $\mathcal{D}$ as well, and $\m{P}$ is separative, this will be the case if and only if $\pi$ and $\pi^\prime$ agree on the intersection $D\, \cap\, D^\prime$. \\[-3mm]

For $\pi \in A$, we denote by $[\pi]$ its equivalence class: \[[\pi] := \{ \sigma \in A\ | \ \sigma \sim \pi\} = \{ \sigma \in A\ | \ \pi \, \uhr\, (D_\pi\, \cap\, D_\sigma) = \sigma\, \uhr\, (D_\pi\, \cap\, D_\sigma)\}.\] Then $\ol{A} = \big\{ \, [\pi]\ | \ \pi \in A\, \big\}$ becomes a group as follows: For $\pi$, $\sigma \in A$, let $[\pi]\, \circ\, [\sigma] := [\nu]$, where $\nu \in A$ with $D_\nu = D_\pi\, \cap\, D_\sigma$ and $\nu (p) = \pi (\sigma (p))$ for all $p \in D_\pi\, \cap\, D_\sigma$. Such $\nu$ will always exists by our construction of $A$, and $[\nu]$ is well-defined: If $[\pi] = [\pi^\prime]$, $[\sigma] = [\sigma^\prime]$ and $\nu$, $\nu^\prime$ as above, then for all $p \in (D_\pi\, \cap\, D_\sigma)\, \cap (D_{\pi^\prime}\, \cap\, D_{\sigma^\prime})$, it follows that $\nu (p) = \pi(\sigma(p)) = \pi^\prime(\sigma^\prime(p)) = \nu^\prime (p)$. Hence, $[\nu] = [\nu^\prime]$. \\[-3mm]

The identity element $id$ is the identity map on its domain $D_{id} \in \mathcal{D}$, with $D_{id} \supseteq D_\pi$ for all $\pi \in A$ (then $[\pi]\, \circ\, [\id] = [\id]\, \circ\, [\pi] = [\pi]$ for all $\pi \in A$ follows). \\[-3mm]

Finally, for $\pi \in A$, let $[\pi]^{-1} := [\nu]$, where $\nu \in A$ with $D_{\nu} = D_{\pi}$ and $\nu = \pi^{-1}$ on $D_\pi$, i.e. $\nu (\pi (p)) = \pi (\nu (p)) = p$ for all $p \in D_\pi$. Again, such $\nu$ will always exists by our construction of $A$, and $[\nu]$ is well-defined: Whenever $[\pi] = [\pi^\prime]$ and $\nu$, $\nu^\prime$ as above, it follows that $\nu (p) = \nu^\prime (p)$ must hold for all $p \in D_\pi\, \cap\, D_{\pi^\prime} = D_\nu\, \cap\, D_{\nu^\prime}$; hence, $[\nu] = [\nu^\prime]$. 
Moreover, $[\pi]\, \circ\, [\nu] = [\nu]\, \circ\, [\pi] = [\id_{D_\pi}] = [\id]$. \\[-2mm]

Hence, $\ol{A}$ is a group. Later on, we will define a collection of $\ol{A}$-subgroups generating a normal filter $\mathcal{F}$ on $\ol{A}$, giving rise to our notion of symmetry. \\[-2mm]

However, we first have to extend our partial automorphisms $\pi \in A$ to the name space $Name (\m{P})$. \\[-2mm]

%The partial automorphisms $\pi \in A$ can be extended to the name space $Name(\m{P})$ as follows: \\[-2mm]

For any $D \in \mathcal{D}$, we define a hierarchy $\ol{\Name_\alpha(\m{P})}^D$ recursively:\begin{itemize} \item $\ol{\Name_0(\m{P})}^D := \emptyset$ \item $\ol{\Name_{\alpha + 1} (\m{P})}^D := \{\dot{x} \in \Name (\m{P})\ | \ \dot{x} \subseteq \ol{\Name_\alpha(\m{P})}^D\, \times \, D \}$, and \item $\ol{\Name_\lambda(\m{P})}^D := \bigcup_{\alpha < \lambda} \ol{\Name_\alpha(\m{P})}^D$ for $\lambda$ a limit ordinal. \end{itemize} 

Let \[\ol{\Name(\m{P})}^D := \bigcup_{\alpha \in \Ord} \ol{\Name_\alpha(\m{P})}^D.\]

\vspace*{2mm}

In other words: $\ol{\Name(\m{P})}^D$ is the collection of all $\m{P}$-names $\dot{x}$ where only conditions $p \in D$ occur. \\[-2mm]

Whenever $\pi \in A$, ${\pi}: D_{\pi} \rightarrow D_{\pi}$, then the image $\pi \dot{x}$ can be defined es as usual as long as $\dot{x} \in \ol{\Name(\m{P})}^{D_\pi}$.
%, we can take the image ${\pi} \dot{x}$ as usual. 
In the case that $\dot{x}$ is a $\m{P}$-name with $\dot{x} \notin \ol{\Name(\m{P})}^{D_\pi}$, however, it is not clear how to apply $\pi$, so the name $\dot{x}$ has to be modified. \\[-2mm]

%to obtain $\ol{\dot{x}}^{D_{\pi}} \in \ol{\Name (\m{P})}^{D_\pi}$
%$: \\[-2mm]

Given $D \in \mathcal{D}$, we define recursively for $\dot{x} \in \Name (\m{P})$: 
\[ \ol{x}^D := \{ (\ol{y}^D, p)\ | \ \dot{y} \in \dom \dot{x}\, , \, p \in D\, , \, p \Vdash \dot{y} \in \dot{x} \}. \]

Then $\ol{x}^D \in \ol{\Name (\m{P})}^D$ with $\dot{x}^G = (\ol{x}^{D})^G$ for any $G$ a $V$-generic filter on $\m{P}$. \\[-2mm]

We will call a $\m{P}$-name $\dot{x}$ \textit{symmetric} if the collection of all $[\pi]$ with $\pi \ol{x}^{D_\pi} = \ol{x}^{D_\pi}$ is contained in our normal filter $\mathcal{F}$. 

Hereby, we have to make sure that this definition does not depend on which representative of $[\pi]$ we choose: In Lemma \ref{pisimpiprime} and \ref{pisimpiprime2}, we prove that whenever $\pi$ and $\pi^\prime$ belong to the same equivalence class $[\pi] = [\pi^\prime]$, then $\pi \ol{x}^{D_\pi} = \ol{x}^{D_\pi}$ holds if and only if $\pi^\prime \, \ol{x}^{D_{\pi^\prime}} = \ol{x}^{D_{\pi^\prime}}$. \\[-3mm]

%\colorbox{yellow}{FRAGE: $\ol{x}^{D_\pi}$ oder $\ol{\dot{x}}^{D_\pi}$ schreiben?} 

%\colorbox{red}{TO DO: Nachprüfen, dass in der Def. von Symmetrie der Repräsentant keine Rolle spielt!}

%\colorbox{yellow}{Zitat oben?}

%\colorbox{red}{TO DO: Eigenschaften} \begin{itemize} \item $\ol{\ol{x}^{D_\pi}}^{D_\sigma} = \ol{x}^{D_\sigma}$ \item $\dot{x}$ symmetrisch $\Rightarrow$ $\pi \ol{x}^{D_\pi}$ symmetrisch
%\end{itemize}

%\colorbox{red}{\vspace*{7cm}}

%\[\ol{\dot{x}}^{D} := \{ (\ol{\dot{y}}, \ol{p})\ | \ \exists\, (\dot{y}, p) \in \dot{x}\ | \ \ol{p} \leq p, \ol{p} \in D\}.\] Then for any $V$-generic filter $G$ on $\m{P}$ it follows that $(\ol{\dot{x}})^{G} = \dot{x}^{G}$. \\[-2mm]

%Let $\pi, \sigma \in A$ with $\pi: D_{\pi} \rightarrow D_{\pi}$, $\sigma: D_{\sigma} \rightarrow D_{\sigma}$ as above. It is not difficult to verify the following properties: \begin{itemize} \item For any $\dot{x} \in \Name (\m{P})$, it follows that $\ol{\ol{\dot{x}}^{D_{\pi}}}^{D_{\sigma}} = \ol{\dot{x}}^{D_{\pi}\, \cap \, D_{\sigma}}$. \item Whenever $\dot{x} \in \ol{\Name(\m{P})}^{D_{\pi}}$, then also $\sigma \ol{\dot{x}}^{D_{\sigma}} \in \ol{\Name(\m{P})}^{D_{\pi}}$ with $\pi \ol{\dot{x}}^{D_{\sigma}} = \ol{{\pi} \dot{x}}^{D_{\sigma}}$.\end{itemize}

Then we set \[ N := V(G) := \{ \dot{x}^G\ | \ \dot{x} \in HS\},  \] where $HS$ denotes the class of all \textit{hereditarily symmetric} $\m{P}$-names, defined recursively as usual: For every $\dot{x} \in \Name (\m{P})$, we have $\dot{x} \in HS$ if $\dot{x}$ is symmetric and $dom \, \dot{x} \subseteq HS$. \\[-2mm]

We have to verify that also with this modified notion of symmetry, $N = V(G)$ is a model of $ZF$. This will be done in Chapter 5.

%\colorbox{cyan}{TO DO: Peter zeigen!}

%\colorbox{yellow} {Verweis?}

%FEHLT NOCH: $\dot{x}$ symmetrisch $\Rightarrow$ ${\sigma_0} \ol{\dot{x}}^{D_{\sigma_0}}$ symmetrisch für alle $ {\sigma_0}$. Dafür müsste man erst Isomorphismen und den Filter einführen!

%?? Hier irgendwo auf [archiv] verweisen? Nur zitieren? ?? \\[-2mm]

%FRAGE: Man bräuchte doch bestimmt die Bemerkung aus [arXiv] zu $\supp \pi_1 (\eta)$? \\[-2mm]
%?? jede Abbildung aus $A$ hätte einen "`Doppelgänger"' mit anderem Domain?

%Now, we define our collection of $A$-subgroups that generate our symmetric submodel $N$, where by \textit{$A$-subgroup}, we mean a set $B \subseteq A$ such that $B$ is a group of partial $\m{P}$-automorphisms. \\[-3mm]

%We will have $A$-subgroups of the form $B_0\, \times \, A_1$, for an $A_0$-subgroup $B_0$, and $A$-subgroups of the form $A_0\, \times \, B_1$ for an $A_1$-subgroup $B_1$. \\[-3mm]

%We start with defining $A_0$-subgroups. \\[-3mm]

%, as well; so with a normal filter $\mathcal{F}$ on $A$, the notions of \textit{symmetric names}

%We define a normal filter $\mathcal{F}$ on $A$, and see how our partial $\m{P}$-automorphisms can be extended to the name space $Name(\m{P})$.  

%Then and we define a symmetric submodel $V \subseteq N \subseteq V[G]$ as usual, avoiding Boolean Algebras.\\ 
%\colorbox{yellow} {TO DO: NACHSEHEN, wie das beim Baumforcing gemacht wurde} \\ 
%\colorbox{yellow} {FRAGE: Sollte irgendwo die allgemeine Technik erklärt werden? Oder erst EVTL später?} \\[-2mm]

\subsection{Constructing A and $\boldsymbol{\ol{A}}$.}

We start with constructing $A$, our collection of partial $\m{P}$-automorphisms with the properties mentioned in Chapter 4.1. \\[-2mm]

We will have $A = A_0\, \times \, A_1$, where $A_0$ is a collection of partial $\m{P}_0$-automorphisms, 
%(we give a precise definition of this structure later on), 
and $A_1$ is a collection of partial $\m{P}_1$-automorphisms.  \\[-2mm]

Every $\pi_0 \in A_0$ will be an order-preserving bijection $\pi_0: D_{\pi_0} \rightarrow D_{\pi_0}$, where $D_{\pi_0}$ 
%$D_{\pi_0} \subseteq \m{P}_0$ such that $D_{\pi_0}$ 
is contained in our collection $\mathcal{D}_0$, defined as follows: \\[-2mm]

Let $\mathcal{D}_0$ denote the collection of all sets $D \subseteq \m{P}_0$ given by

\begin{itemize} \item a countable \textit{support} $\supp D \subseteq \{ (\sigma, i)\ | \ \sigma \in \Lim, i < \alpha_\sigma \}$, and \item a \textit{domain} $\dom D := \bigcup_{\nu < \gamma\, , \, j < \cf \kappa_{\nu + 1}} [\kappa_{\nu, j}, \delta_{\nu, j})$ such that $\delta_{\nu, j} \in [\kappa_{\nu, j}, \kappa_{\nu, j + 1})$ for all $\nu < \gamma$, $j < \cf \kappa_{\nu + 1}$; and for all regular $\kappa_{\ol{\nu}, \ol{\j}}$, it follows that $dom\, D\, \cap\, \kappa_{\ol{\nu}, \ol{\j}}$ is bounded below $\kappa_{\ol{\nu}, \ol{\j}}$, \end{itemize} such that $D$ is the set of all $p = (p_\ast, (p^\sigma_i, a^\sigma_i)_{\sigma, i}) \in \m{P}_0$ with
% \[D_{\pi_0} = \{p = (p_\ast, (p^\sigma_i, a^\sigma_i)_{\sigma, i}, (p^\sigma)_\sigma) \in \m{P}_0\ | \ 
\begin{itemize} \item $\supp p \supseteq \supp D$\ , \ $\dom p \supseteq \dom D$\,, \, and\item for all intervals $[\kappa_{\nu, j}, \kappa_{\nu, j + 1})$ with $\dom p\, \cap\, [\kappa_{\nu, j}, \kappa_{\nu, j + 1}) \neq \emptyset$, it follows that \[ \bigcup_{ (\sigma, i) \in \supp p}a^\sigma_i\, \cap\, [\kappa_{\nu, j}, \kappa_{\nu, j + 1}) \ \subseteq \ \dom p.\]

%$\forall\, \kappa_{\nu, j} < \kappa_\gamma\, : \ \big (\dom p \, \cap \, [\kappa_{\nu, j}, \kappa_{\nu, j + 1}) \neq \emptyset \Rightarrow \bigcup_{ (\sigma, i) \in \supp p} a^\sigma_i\, \cap\, [\kappa_{\nu, j}, \kappa_{\nu, j + 1})\subseteq \dom p\big)$.
\end{itemize}

In other words, $D$ is the collection of all $p \in \m{P}_0$, the domain and support of which covers a certain domain and support given by $D$; with the additional property that all the linking ordinals $\{\xi\} = a^\sigma_i\, \cap\, [\kappa_{\nu, j}, \kappa_{\nu, j + 1})$ contained in any interval $[\kappa_{\nu, j}, \kappa_{\nu, j + 1})$ hit by $\dom p$, are already contained in $\dom p$. \\[-3mm]

It is not difficult to see that any $D \in \mathcal{D}_0$ is dense in $\m{P}_0$. The sets $D \in \mathcal{D}_0$ are not open dense; but whenever $p$, $q \in \m{P}_0$ with $p \in D$ and $q \leq p$ such that $\supp q = \supp p$, then by the \textit{linking property}, it follows that also $q \in D$. \\[-3mm]

Whenever $D$, $D^\prime \in \mathcal{D}_0$, then the intersection $D\, \cap\, D^\prime$ is contained in $\mathcal{D}_0$, as well, with $supp \, (D\, \cap\, D^\prime) = supp\, D\, \cup\, supp \,D^\prime$, $dom\, (D\, \cap D^\prime) = dom\, D\, \cup dom \,D^\prime$.
\\[-2mm]

%For any $\pi_0$, $\sigma_0 \in A_0$, $\pi_0: D_{\pi_0} \rightarrow D_{\pi_0}$, $\sigma_0: D_{\sigma_0} \rightarrow D_{\sigma_0}$, the \textit{concatenation} $\nu_0 = \pi_0\, \circ\, \sigma_0$ will be contained in $A_0$ as well, with $D_{\nu_0} = D_{\pi_0}\, \cap\, D_{\sigma_0}$, and $\nu_0 (p) = \pi_0 (\sigma_0 (p))$ for all $p \in D_{\nu_0}$. Also, for any $\pi_0 \in A_0$, the \textit{inverse map} $\pi_0^{-1}$ will be contained in $A_0$, with $D_{\pi_0^{-1}} = D_{\pi_0}$, and $\pi_0^{-1} (\pi_0 (p)) = \pi_0 (\pi_0^{-1} (p)) = p$ for all $p \in D_{\pi_0}$. 
%there will be a map $\nu_0: D_{\nu_0} \rightarrow D_{\nu_0}$ in $A_0$ such that $\nu_0$ is the identity on $D_{\nu_0} = D_{\pi_0}\, \cap\, D_{\sigma_0}$. 
%Also, for any $\pi_0 \in A_0$, there will be a map $\pi_0^{-1} \in A_0$ with $D_{\pi_0^{-1}} = D_{\pi_0}$ and $\pi_0\, \circ\, \pi_0^{-1} = \pi_0^{-1}\, \circ \, \pi_0 = \id_{D_{\pi_0}}$. 
%Finally, we will have an \textit{identity element} $\id_0 \in A_0$ with $\supp \id_0 = \dom \id_0 = \emptyset$, 
%and $\id_0$ is the identity on 
%$D_{\id_0} = \{p \in \m{P}_0\ | \ \forall\, \kappa_{\nu, j} < \kappa_\gamma\, : \ (\dom p \, \cap \, [\kappa_{\nu, j}, \kappa_{\nu, j + 1}) \neq \emptyset \Rightarrow \bigcup \{ a^\sigma_i\, \cap\, [\kappa_{\nu, j}, \kappa_{\nu, j + 1})\ | \ (\sigma, i) \in \supp p \} \supseteq \dom p \}$, where $D_{\id_0} \subseteq \m{P}_0$ is dense and $D_{\pi_0} \subseteq D_{\id_0}$ for all $\pi_0 \in A_0$. \\We call this structure a \textit{group of partial $\m{P}_0$-isomorphisms}. \\[-2mm]

We now describe the two types of partial $\m{P}_0$-automorphisms that will generate $A_0$: \\[-3mm]

Our first goal is that for any two conditions $p$, $q \in \m{P}$ with the same \tbl shape\tbr, i.e. $\dom p = \dom q$, $\supp p = \supp q$ and $\bigcup a^\sigma_i = \bigcup b^\sigma_i$, there is an automorphism $\pi_0 \in A_0$ with $\pi_0 p = q$. This homogeneity property will be achieved by giving the maps $\pi_0 \in A_0$ a lot of freedom regarding what can happen on $supp \, \pi_0$ and $dom \, \pi_0$. \\[-2mm]

For $\kappa_{\nu, j} < \kappa_\gamma$, let \[\supp \pi_0 (\nu, j) := \{ (\sigma, i) \in \supp \pi_0\ | \ \kappa_{\nu, j} < \kappa_\sigma\}.\]
Regarding the linking ordinals, we want that for any $p \in D_{\pi_0}$, $p = (p_\ast, (p^\sigma_i, a^\sigma_i)_{\sigma, i})$ with $\pi p = p^\prime = ((p^\prime)_\ast, ((p^\prime)^\sigma_i, (a^\prime)^\sigma_i)_{\sigma, i})$, the sets of linking ordinals for $p$ and $p^\prime$ are the same, i.e. $\bigcup a^\sigma_i = \bigcup (a^\prime)^\sigma_i$.
In other words, for any interval $[\kappa_{\nu, j}, \kappa_{\nu, j + 1})$, the linking ordinals $\xi \in [\kappa_{\nu, j}, \kappa_{\nu, j + 1})$ will be exchanged between the coordinates $(\sigma, i) \in \supp \pi_0 (\nu, j)$,
%$a^\sigma_i \, \cap \, [\kappa_{\nu, j}, \kappa_{\nu, j + 1})$ will be \tbl swapped\tbr WORT NACHSCHLAGEN, 
which is described by an isomorphism $F_{\pi_0} (\nu, j): \supp \pi_0 (\nu, j) \rightarrow \supp \pi_0 (\nu, j)$.

Regarding the $(p^\prime)^\sigma_i$ for $(\sigma, i) \in \supp \pi_0$, there will be for every $\zeta \in [\kappa_{\nu, j}, \kappa_{\nu, j + 1})\, \cap\, \dom \pi_0$ a bijection $\pi_0 (\zeta): 2^{\supp \pi_0 (\nu, j)} \rightarrow 2^{\supp \pi_0 (\nu, j)}$ with \[  \big(\, (p^\prime)^\sigma_i (\zeta)\ | \ (\sigma, i) \in \supp \pi_0 (\nu, j) \, \big) := \pi_0 (\zeta) \big( \, p^\sigma_i (\zeta)\ | \ (\sigma, i) \in \supp \pi_0 (\nu, j) \, \big).\] Concerning $p^\prime_\ast$, we will have a similar construction for the $p^\prime_\ast (\xi, \zeta)$ in the case that $\zeta \in \dom \pi_0$ and $\xi$ is a linking ordinal contained in $\bigcup a^\sigma_i$. For all $(\xi, \zeta) \in \dom \pi_0 \, \cap \, [\kappa_{\nu, j}, \kappa_{\nu, j + 1})^2$, we will have a bijection $\pi_\ast (\xi, \zeta): 2 \rightarrow 2$, with $p^\prime_\ast (\xi, \zeta) = \pi_\ast (\xi, \zeta) (p_\ast (\xi, \zeta))$ whenever $\xi$, $\zeta \in \dom \pi_0$ and $\xi \notin \bigcup a^\sigma_i$. \\[-3mm]

Our second goal is that for any interval $[\kappa_{\nu, j}, \kappa_{\nu, j + 1})$ and $(\sigma, i)$, $(\lambda, k) \in \supp \pi_0 (\nu, j)$, there is an isomorphism $\pi_0 \in A_0$ such that $(\pi_0 G)^\lambda_k \, \cap \, [\kappa_{\nu, j}, \kappa_{\nu, j + 1}) = G^{\sigma}_{i}\, \cap \, [\kappa_{\nu, j}, \kappa_{\nu, j + 1})$. Thus, every $\pi_0 \in A_0$ will be equipped with bijections $G_{\pi_0} (\nu, j): \supp \pi_0 (\nu, j) \rightarrow \supp \pi_0 (\nu, j)$ for every $\kappa_{\nu, j}$, such that the following holds: For every $p \in D_{\pi_0}$, $p^\prime := \pi p$ and $\zeta \in \dom p \setminus \dom \pi_0$, $(\sigma, i) \in \supp \pi_0 (\nu, j)$, we have $(p^\prime)^\sigma_i (\zeta) = p^\lambda_k (\zeta)$ with $(\lambda, k) := G_{\pi_0} (\nu, j) (\sigma, i)$. 

Whenever $\zeta \in \dom \pi_0$ and $(\sigma, i) \in \supp \pi (\nu, j)$, the values $(p^\prime)^\sigma_i (\zeta)$ are described by the maps $\pi_0 (\zeta)$ mentioned in context with \tbl our first goal\tbr\,above, which allows for setting $(p^\prime)^\sigma_i (\zeta) := p^\lambda_k (\zeta)$ for any pair $(\sigma, i)$, $(\lambda, k) \in \supp \pi_0 (\nu, j)$ with $(\lambda, k) = G_{\pi_0} (\nu, j) (\sigma, i)$. \\[-3mm]

Roughly speaking, $A_0$ will be generated by these two types of isomorphism. Regarding the construction of $p_\ast^\prime$, some extra care is needed concerning the values $p^\prime_\ast (\xi, \zeta)$ for $\zeta \notin \dom \pi_0$ and $\xi \in \bigcup a^\sigma_i$ a linking ordinal, since we have to make sure that the maps $\pi_0 \in A_0$ are order-preserving: Whenever $p, q \in D_{\pi_0}$ with $q \leq p$, then also $q^\prime \leq p^\prime$ must hold; in particular, whenever $\{\xi^\sigma_i\}: = a^\sigma_i \, \cap \, [\kappa_{\nu, j}, \kappa_{\nu, j + 1})$ is a linking ordinal and $\zeta \in \dom q \ \setminus \, \dom p$ (hence, $\zeta \notin \dom \pi_0$), then $\{\xi^\sigma_i\} = (a^\prime)^\lambda_k\, \cap \, [\kappa_{\nu, j}, \kappa_{\nu, j + 1})$ with $(\sigma, i) = F_{\pi_0} (\nu, j) (\lambda, k)$, and $q^\prime_\ast (\xi^\sigma_i, \zeta) = (q^\prime)^\lambda_k (\zeta)$ by the linking property for $q^\prime \leq p^\prime$. Moreover, $(q^\prime)^\lambda_k (\zeta) = q^\mu_l (\zeta)$ with $(\mu, l) = G_{\pi_0} (\nu, j) (\lambda, k)$, and $q^\mu_l (\zeta) = q_\ast (\xi^\mu_l, \zeta)$ with $\xi^\mu_l = a^\mu_l\, \cap \, [\kappa_{\nu, j}, \kappa_{\nu, j + 1})$ by the linking property for $q \leq p$. Hence, $q^\prime_\ast (\xi^\sigma_i, \zeta) = q_\ast (\xi^\mu_l, \zeta)$ must hold, where $(\mu, l)= G_{\pi_0} (\nu, j) \, \circ \, (F_{\pi_0} (\nu, j))^{-1} (\sigma, i)$. \\[-2mm]

This gives rise to the following definition:

\begin{definition} Let $A_0$ consist of all automorphisms $\pi_0: D_{\pi_0} \rightarrow D_{\pi_0}$ such that there are \begin{itemize} \item a countable set $\supp \pi_0 \subseteq \{(\sigma, i)\ | \ \sigma \in \Lim, i < \alpha_\sigma\}$ \\(for $\kappa_{\nu, j} < \kappa_\gamma$, we set $\supp \pi_0 (\nu, j) := \{ (\sigma, i) \in \supp \pi_0 \ | \ \kappa_{\nu, j} < \kappa_\gamma\}$), \item a domain $\dom \pi_0 = \bigcup_{\nu < \gamma\, , \, j < \cf \kappa_{\nu + 1}} [\kappa_{\nu, j}, \delta_{\nu, j})$ such that $\delta_{\nu, j} \in [\kappa_{\nu, j}, \kappa_{\nu, j + 1})$ for all $\nu < \gamma$, $j < \cf \kappa_{\nu + 1}$; and for all regular $\kappa_{\ol{\nu}, \ol{\j}}$, it follows that $\dom \pi_0\, \cap\, \kappa_{\ol{\nu}, \ol{\j}}$ is bounded below $\kappa_{\ol{\nu}, \ol{\j}}$ \\(for $\kappa_{\nu, j} < \kappa_\gamma$, we set $\dom \pi_0 (\nu, j) := \dom \pi_0\, \cap\, [\kappa_{\nu, j}, \kappa_{\nu, j + 1})$), \end{itemize} such that \[D_{\pi_0} = \big\{\, p = (p_\ast, (p^\sigma_i, a^\sigma_i)_{\sigma, i}) \in \m{P}_0\ \big| \ \supp p \supseteq \supp \pi_0\, , \, \dom p \supseteq \dom \pi_0, \mbox{ and } \]\[\forall\, \kappa_{\nu, j} < \kappa_\gamma\, : \ \big(\dom p \, \cap \, [\kappa_{\nu, j}, \kappa_{\nu, j + 1}) \neq \emptyset \Rightarrow \bigcup_{(\sigma, i) \in \supp p} a^\sigma_i\, \cap\, [\kappa_{\nu, j}, \kappa_{\nu, j + 1}) \subseteq \dom p\big)\, \big \} ;\]

moreover, there are \begin{itemize} \item for all $\nu < \gamma$, $j < \cf \kappa_{\nu + 1}$, a bijection \[F_{\pi_0} (\nu, j): \supp \pi_0 (\nu, j) \rightarrow \supp \pi_0 (\nu, j)\] (which will be in charge of permuting the linking ordinals as mentioned above), \\[-3mm]

and a bijection \[G_{\pi_0} (\nu, j): \supp \pi_0 (\nu, j) \rightarrow \supp \pi_0 (\nu, j)\] (which will be in charge of permuting the verticals $p^\sigma_i$ outside $\dom \pi_0$ on the interval $[\kappa_{\nu, j}, \kappa_{\nu, j + 1})$), \item for all $\nu < \gamma$, $j < \cf \kappa_{\nu + 1}$ and $\zeta \in [\kappa_{\nu, j}, \kappa_{\nu, j + 1})\, \cap \, \dom \pi_0$, a bijection \[\pi_0 (\zeta): 2^{\supp \pi_0 (\nu, j)} \rightarrow 2^{\supp \pi_0 (\nu, j)}\] (this map will be in charge of setting the values $(\pi p)^\sigma_i (\zeta)$ for $(\sigma, i) \in \supp \pi_0 (\nu, j)$, $\zeta \in \dom \pi_0$), \item for all $\nu < \gamma$, $j < \cf \kappa_{\nu + 1}$, $\zeta \in [\kappa_{\nu, j}, \kappa_{\nu, j + 1})\, \cap \, \dom \pi_0$, and \[ \big (\xi^\sigma_i\ | \ (\sigma, i) \in \supp \pi_0 (\nu, j) \big) \in (\dom \pi_0 (\nu, j))^{\supp \pi_0 (\nu, j)}\] a sequence of pairwise distinct ordinals, a bijection \[(\pi_0)_\ast (\zeta) \big(\xi^\sigma_i\ | \ (\sigma, i) \in \supp \pi_0 (\nu, j)\big): 2^{\supp \pi_0 (\nu, j)} \rightarrow 2^{\supp \pi_0 (\nu, j)}\] (which will be in charge of setting the values $(\pi p)_\ast (\xi^\sigma_i, \zeta)$ for $\{\xi^\sigma_i\} = a^\sigma_i\, \cap \, [\kappa_{\nu, j}, \kappa_{\nu, j + 1})$ a linking ordinal and $\zeta \in [\kappa_{\nu, j}, \kappa_{\nu, j + 1})\, \cap \, \dom \pi_0$ ),
% FRAGE: Sollte man nicht eher $\xi^\sigma_i (\nu, j)$ schreiben?
\item for all $\nu < \gamma$, $j < \cf \kappa_{\nu + 1}$ and $(\xi, \zeta) \in [\kappa_{\nu, j}, \kappa_{\nu, j + 1})^2$, a bijection \[(\pi_0)_\ast (\xi, \zeta): 2 \rightarrow 2\] such that $\pi_\ast (\xi, \zeta)$ is the identity whenever $(\xi, \zeta) \notin (\dom \pi_0)^2$\\ (which will be in charge of the values $(\pi p)_\ast (\xi, \zeta)$ in the case that $\xi \notin \bigcup_{\sigma, i} a^\sigma_i$ is not a linking ordinal);\end{itemize}
which defines for $p \in D_{\pi_0}$, $p = (p_\ast, (p^\sigma_i, a^\sigma_i)_{\sigma, i})$, the image $\pi p =: p^\prime = (p^\prime_\ast, ((p^\prime)^\sigma_i, (a^\prime)^\sigma_i)_{\sigma, i})$ as follows: \\[-2mm]

We will have $\supp p^\prime = \supp p$, $\dom p^\prime = \dom p$. Moreover: \begin{itemize} \item Concerning the linking ordinals, for all $(\sigma, i) \in \supp p^\prime = \supp p$ and $\kappa_{\nu, j} < \kappa_\sigma$: \begin{itemize} \item $(a^\prime)^\sigma_i\, \cap \, [\kappa_{\nu, j}, \kappa_{\nu, j + 1}) = a^\sigma_i\, \cap \, [\kappa_{\nu, j}, \kappa_{\nu, j + 1})$ for $(\sigma, i) \notin \supp \pi_0 (\nu, j)$, \item $(a^\prime)^\sigma_i \, \cap \, [\kappa_{\nu, j}, \kappa_{\nu, j + 1}) = a^\lambda_k\, \cap \, [\kappa_{\nu, j}, \kappa_{\nu, j + 1})$ with $(\lambda, k) = F_{\pi_0} (\nu, j) (\sigma, i)$ in the case that $(\sigma, i) \in \supp \pi_0 (\nu, j)$.\end{itemize}

\item Concerning the $(p^\prime)^\sigma_i$ with $(\sigma, i) \in \supp \pi_0$, \begin{itemize} \item for $\zeta \in \dom \pi_0$, \[\big( (p^\prime)^\sigma_i (\zeta)\ | \ (\sigma, i) \in \supp \pi_0 (\nu, j) \big) = \pi_0 (\zeta) \big( p^\sigma_i (\zeta) \ | \ (\sigma, i) \in \supp \pi_0 (\nu, j) \big),\] \item and in the case that $\zeta \notin \dom \pi_0$, we will have \[(p^\prime)^\sigma_i (\zeta) = p^\lambda_k (\zeta) \mbox { with } (\lambda, k) = G_{\pi_0} (\nu, j) (\sigma, i).\] \end{itemize} 

\item Whenever $(\sigma, i) \notin \supp \pi_0$, then $(p^\prime)^\sigma_i = p^\sigma_i$. 

\item We now turn to $p^\prime_\ast$. Consider an interval $[\kappa_{\nu, j}, \kappa_{\nu, j + 1})$. For any $(\sigma, i) \in \supp \pi_0 (\nu, j)$, let $\{\xi^\sigma_i\} := a^\sigma_i\, \cap\, [\kappa_{\nu, j}, \kappa_{\nu, j + 1})$. For $\zeta \in [\kappa_{\nu, j}, \kappa_{\nu, j + 1})\, \cap \, \dom \pi_0$, we will have $\big( p^\prime_\ast (\xi^\sigma_i, \zeta)\ | \ (\sigma, i) \in \supp \pi_0 (\nu, j) \big) = $\[ (\pi_0)_\ast (\zeta) \big(\xi^\sigma_i\ | \ (\sigma, i) \in \supp \pi_0 (\nu, j)\big) \big(p_\ast (\xi^\sigma_i, \zeta)\ | \ (\sigma, i) \in \supp \pi_0 (\nu, j) \big).\] In the case that $\zeta \in [\kappa_{\nu, j}, \kappa_{\nu, j + 1})\, \cap \, (\dom p \, \setminus \, \dom \pi_0)$, we will have for $(\sigma, i) \in \supp \pi_0 (\nu, j)$: \[p^\prime_\ast (\xi^\sigma_i, \zeta) := p_\ast (\xi^\lambda_k, \zeta),\] where $(\lambda, k) = G_{\pi_0} (\nu, j) \, \circ \, (F_{\pi_0} (\nu, j))^{-1} (\sigma, i)$. \\ Finally, if $(\xi, \zeta) \in (\dom p)^2$ with $\xi, \zeta \in [\kappa_{\nu, j}, \kappa_{\nu, j + 1})$ such that $\xi \notin \bigcup_{\sigma, i} a^\sigma_i$, then \[p^\prime_\ast (\xi, \zeta) = (\pi_0)_\ast (\xi, \zeta) \big(p_\ast (\xi, \zeta)\big).\]

\end{itemize}

\end{definition}

For any $\pi \in A_0$, we have $D_{\pi_0} \in \mathcal{D}_0$ with $\supp D_{\pi_0} := \supp \pi_0$ and $\dom D_{\pi_0} := \dom \pi_0$.

Moreover, whenever $p$ is a condition in $D_{\pi_0}$, then $p^\prime := \pi_0 p \in \m{P}_0$ is well-defined with $p^\prime \in D_{\pi_0}$, since $\supp p^\prime = \supp p$, $\dom p^\prime = \dom p$, and $\bigcup_{\sigma, i} a^\sigma_i = \bigcup_{\sigma, i} (a^\prime)^\sigma_i$  by construction. 

Here we use that $\pi_0$ is only defined on $D_{\pi_0}$ and not on the entire forcing $\m{P}_0$, since we have to make sure that in our construction of the $p^\prime_\ast (\xi^\sigma_i, \zeta)$ for $\zeta \notin \dom \pi_0$, we do not run out of $dom \, p$. \\[-3mm]

%when constructing the $\p^\prime_\ast (\xi^\sigma_, \zeta)$ for $\zeta \notin \dom \pi_0$: 
%we do not run out of $\dom p$
%We need the dense sets $D_{\pi_0}$ in order to make sure that in our construction of the $p^\prime_\ast (\xi^\sigma_i, \zeta)$ for $\zeta \notin \dom {\pi_0}_0$, we do not run out of $\dom p$ when constructing the $\p^\prime_\ast (\xi^\sigma_, \zeta)$ for $\zeta \notin \dom \pi_0$: 

%For any $p \in D_{\pi_0}$, it follows that for every interval $[\kappa_{\nu, j}, \kappa_{\nu, j + 1})$ hit by $\dom p$, we have $\xi^\sigma_i \in \dom p$ for all the linking ordinals $\{\xi^\sigma_i\} = a^\sigma_i\, \cap \, [\kappa_{\nu, j}, \kappa_{\nu, j + 1})$. This implies that w

It is not difficult to see that for any $p$, $q \in D_{\pi_0}$ with $q \leq p$, also $q^\prime \leq p^\prime$ holds. The linking property follows readily from our definition of the $p^\prime_\ast (\xi^\sigma_i, \zeta)$ for $\zeta \notin \dom \pi_0$. \\[-3mm]
%$\zeta \in (\dom p \, \setminus \, \dom \pi_0) \, \cap \, [\kappa_{\nu, j}, \kappa_{\nu, j + 1})$ and $\xi^\sigma_i = a^\sigma_i\, \cap\, [\kappa_{\nu, j}, \kappa_{\nu, j + 1})$ the corresponding linking ordinal. \\[-3mm]

Whenever $\pi_0 \in A_0$ and $D \in \mathcal{D}_0$ with $D \subseteq D_{\pi_0}$, it follows that the map $\ol{\pi}_0 := \pi_0\, \uhr\, D$ is contained in $A_0$, as well. Here we have to use that the maps $\pi_0$ do not disturb the conditions' domain or support, and merely \textit{permute} the linking ordinals. In particular, whenever $p \in D$, it follows that the image $\pi_0 p$ is contained in $D$, as well. \\[-2mm]

It remains to verify that $A_0$ can be endowed with a group structure. More precisely: We will show that for any $D \in \mathcal{D}_0$, the collection $(A_0)_D := \{\pi \in A_0\ | \ D_{\pi_0} = D\}$ gives a group; and then take the direct limit $\ol{A}_0 := \lim\limits_{\longrightarrow}{(A_0)_D}$. \\[-3mm]

Firstly, for any $\pi_0 \in A_0$, it is not difficult to write down a map $\nu_0 \in A_0$ with $D_{\nu_0} = D_{\pi_0}$ such that $\nu_0$ is the inverse of $\pi_0$: \\[-3mm]

%$\nu_0 (\pi_0 (p)) = \pi_0 (\nu_0 (p)) = p$ holds for all $p \in D_{\pi_0}$: \\ 

Let $\supp \nu_0 := \supp \pi_0$ and $\dom \nu_0 := \dom \pi_0$. For any $\kappa_{\nu, j} < \kappa_\gamma$, we set $F_{\nu_0} (\nu, j) := (F_{\pi_0} (\nu, j))^{-1}$, $G_{\nu_0} (\nu, j) := (G_{\pi_0} (\nu, j))^{-1}$; and for $\zeta \in [\kappa_{\nu, j}, \kappa_{\nu, j + 1})$, we let $\nu_0 (\zeta) := (\pi_0 (\zeta))^{-1}$. Regarding $(\nu_0)_\ast$ we use the following notation: \\[-3mm]

\textit{For sets $\mathcal{I}$, $\mathcal{J}$ with a bijection $b: \mathcal{I} \rightarrow \mathcal{J}$ and a sequence $(x_j\ | \ j \in \mathcal{J})$, we denote by $\ol{b} (x_j\ | \ j \in \mathcal{J})$ the induced sequence parametrized by $\mathcal{I}$:} \[\ol{b} (x_j\ | \ j \in \mathcal{J}) := (y_i\ | \ i \in \mathcal{I}) \mbox{ with  }  y_i := x_{b(i)} \mbox{ for all } i \in \mathcal{I}.\]

Whenever $\zeta \in [\kappa_{\nu, j}, \kappa_{\nu, j + 1})\, \cap \, \dom \pi_0$, and \[\big(\xi^\sigma_i\ | \ (\sigma, i) \in \supp \pi_0 (\nu, j)\big) \in (\dom \pi_0 (\nu, j))^{\supp \pi_0 (\nu, j)}\] is a sequence of pairwise distinct ordinals, we set $(\nu_0)_\ast \big(\xi^\sigma_i \ | \ (\sigma, i) \in \supp \pi_0 (\nu, j)\big) := $ \[\ol{F_{\pi_0} (\nu, j)} \, \circ \,  \left[(\pi_0)_\ast \Big(\,  \ol{F_{\pi_0} (\nu, j)}^{\ -1} \big( \xi^\sigma_i\ | \ (\sigma, i) \in \supp \pi_0 (\nu, j) \big )\, \Big)\right] ^{-1} \, \circ \, \ol{F_{\pi_0} (\nu, j)}^{ \ -1},\] which is a bijection on $2^{\supp \pi_0 (\nu, j)}$. \\[-3mm]

For $(\xi, \zeta) \in [\kappa_{\nu, j}, \kappa_{\nu, j + 1})^2$, let $(\nu_0)_\ast (\xi, \zeta) := \big((\pi_0)_\ast (\xi, \zeta)\big)^{-1}$. \\[-2mm]

It is not difficult to verify that indeed, $\pi_0 (\nu_0 (p)) = \nu_0 (\pi_0 (p)) = p$ holds for all $p \in D_{\pi_0} = D_{\nu_0}$. \\[-2mm]

%$\pi_0\, \circ \, \nu_0 = \pi_0^{-1} \, \circ \, \pi_0 = \id_{D_{\pi_0}}$. \\[-3mm]

Secondly, one has to make sure that for any $\pi_0: D_{\pi_0} \rightarrow D_{\pi_0}$, $\sigma_0: D_{\sigma_0} \rightarrow D_{\sigma_0}$ in $A_0$, there is a map $\tau_0 \in A_0$ with $D_{\tau_0} = D_{\pi_0}\, \cap \, D_{\sigma_0}$ such that $\tau_0 (p) = \pi_0 (\sigma_0 (p))$ holds for all $p \in D_{\tau_0}$.

Setting $\dom \tau_0 := \dom \pi_0 \, \cup \, \dom \sigma_0$, $\supp \tau_0 := \supp \pi_0 \, \cup \, \supp \sigma_0$, one can write down $\tau_0$ explicitly (similarly as for the inverse map), and verify that $\tau_0 (p) = \pi_0 (\sigma_0 (p))$ holds for all $p \in D_{\tau_0}$. \\[-3mm]

%\colorbox{yellow}{FRAGE: Müsste man das nachprüfen?} \\[-3mm]

Finally, the identity element $id_0$, which is the identity map on its domain $D_{\id_0} := $ \[ \big\{\, p \in \m{P}_0\ \, \big| \, \ \forall\, \kappa_{\nu, j} < \kappa_\gamma\, : \, \big(\dom p \, \cap \, [\kappa_{\nu, j}, \kappa_{\nu, j + 1}) \neq \emptyset \Rightarrow \bigcup_{(\sigma, i) \in \supp p} a^\sigma_i\, \cap\, [\kappa_{\nu, j}, \kappa_{\nu, j + 1}) \subseteq \dom p \big)\, \big\}, \] is contained in $A_0$ as well with $D_{\id_0} \supseteq D_{\pi_0}$ for all $\pi_0 \in A_0$. \\[-3mm]

It follows that for any $D \in \mathcal{D}_0$, the collection $(A_0)_D := \{\pi_0 \in A_0\ | \ D_{\pi_0} = D\}$ gives a group with the identity element $id_D := id_0\, \uhr\, D$; and whenever $D$, $D^\prime \in \mathcal{D}$ with $D \subseteq D^\prime$, then there is a canonical homomorphism $\phi_{D^\prime D}: (A_0)_{D^\prime} \rightarrow (A_0) _D$ that maps any $\pi_0 \in (A_0)_{D^\prime}$ to its restriction $\pi_0\, \uhr \, D$. We take the direct limit \[\ol{A}_0 := \lim\limits_{\longrightarrow}{(A_0)_D} = \bigsqcup (A_0)_D \big/ \sim \;, \]  with $\pi_0 \sim \pi_0^\prime$ iff $\pi_0 (p) = \pi_0^\prime (p)$ holds for all $p \in D_{\pi_0}\, \cap\, D_{\pi_0^\prime}$. Then $\ol{A}_0$ becomes a group with concatenation and the inverse elements as described in Chapter \ref{outline}, and the identity element $\ol{\id}_0 := [id_0]$ as defined above. \\[-2mm]

%\colorbox{yellow} {FRAGE: Genügt das so?}

%(cf. Section \ref{outline}) as follows:
%\begin{itemize} \item For $[\pi_0]$, $[\sigma_0] \in A_0$, let $[\pi_0]\, \circ\, [\sigma_0] := [\nu_0]$, where $\nu_0$ denotes the map in $A_0$ with $D_{\nu_0} = D_{\pi_0}\, \cap\, D_{\sigma_0}$ and $\nu_0 (p) = \pi_0 (\sigma_0 (p))$ for all $p \in D_{\pi_0}\, \cap\, D_{\sigma_0}$. This is well-defined 
%\begin{lem} $A_0$ is a group of partial $\m{P}_0$-isomorphisms. \end{lem} 
%NACHFRAGEN!

Now, we turn to $\m{P}_1$ and define $A_1$, our collection of partial $\m{P}_1$-isomorphisms. Every $\pi_1 \in A_1$ will be a bijection $\pi_1: D_{\pi_1} \rightarrow D_{\pi_1}$ with a dense set $D_{\pi_1} \in \mathcal{D}_1$, where $\mathcal{D}_1$ is defined as follows:\\[-2mm]

%\colorbox{yellow} {FRAGE: hier und für $\mathcal{D}_0$ als Definition formulieren?} \\[-2mm] 

Let $\mathcal{D}_1$ denote the collection of all $D \subseteq \m{P}_1$ given by:
\begin{itemize} \item a countable \textit{support} $\supp D \subseteq Succ$, and \item for every $\sigma \in \supp D$, $ \kappa_\sigma = \ol{\kappa_\sigma}^{\,\plus}$, a \textit{domain} $\dom D (\sigma) = \dom_x D (\sigma) \, \times \, \dom_y D (\sigma) \subseteq \alpha_\sigma \, \times \, [\ol{\kappa_\sigma}, \kappa_\sigma)$ with $|\dom \pi_1 (\sigma)| < \kappa_\sigma$,
\end{itemize} such that \[D = \{p \in \m{P}_1\ | \ \supp p \supseteq \supp D\; \wedge \; \forall\, \sigma \in \supp D\ \dom p^\sigma \supseteq \dom D  (\sigma) \}.\]

Then every set $D \in \mathcal{D}_1$ is open dense; and whenever $D$, $D^\prime \in \mathcal{D}_1$, then the intersection $D\, \cap\, D^\prime$ is contained in $\mathcal{D}_1$ as well, with $supp\,(D\, \cap\, D^\prime) = supp\,D\, \cup\, supp\,D^\prime$, and $dom_x (D\, \cap\, D^\prime) (\sigma) = dom_x D (\sigma)\, \cup\, dom_x D^\prime (\sigma)$, $dom_y (D\, \cap\, D^\prime) (\sigma) = dom_y D (\sigma) \, \cup \, dom_y D^\prime (\sigma)$ for all $\sigma \in supp\,(D\, \cap\, D^\prime)$.\\[-3mm]

We now describe the two types of partial $\m{P}_1$-isomorphisms that will generate $A_1$:

%can now define $A_1$, our group of partial $\m{P}_1$-isomorphisms. 

Like for $A_0$, our first goal is that for any $p, q \in \m{P}_1$ which have \tbl the same shape\tbr, i.e. $supp\,p = supp\,q$ and $dom\,p^\sigma = dom\,q^\sigma$ for every $\sigma \in supp\,p$, there is an isomorphism $\pi_1 \in A_1$ with $\pi_1 p = q$. These isomorphisms will be of the following form: For every $\sigma \in \supp \pi_1$, we will have a collection of $\pi (\sigma) (i, \zeta): 2 \rightarrow 2$ for $(i, \zeta) \in \dom \pi_1 (\sigma)$, such that 
%every $\pi_1 (\sigma) (i, \zeta): 2 \rightarrow 2$ is a bijection; and 
for any $p \in D_{\pi_1}$, the map $\pi_1$ changes the value of $p^\sigma (i, \zeta)$ if and only if $\pi_1 (\sigma) (i, \zeta) \neq id$. In other words, $(\pi_1 p)^\sigma (i, \zeta) = \pi_1 (\sigma) (i, \zeta) \big(p^\sigma (i, \zeta)\big)$. 
%Hence, there is great freedom about what can happen on the domains $\dom \pi_1 (\sigma)$ for $\sigma \in \supp \pi_1$; which

This allows for constructing an isomorphism $\pi_1$ with $\pi_1 p = q$ for any pair of conditions $p$, $q$ that have the same supports and domains: One can simply set $\pi_1 (\sigma) (i, \zeta) = id$ if $p^\sigma(i, \zeta) = q^\sigma(i, \zeta)$, and $\pi_1 (\sigma) (i, \zeta) \neq id$ in the case that $p^\sigma (i, \zeta) \neq q^\sigma(i, \zeta)$. \\[-2mm]

Secondly, for every pair of generic $\kappa_\sigma$-subsets $G^\sigma_i$ and $G^\sigma_{i^\prime}$ for $\sigma \in \Succ$ and $i, i^\prime < \alpha_\sigma$, we want an isomorphism $\pi \in A_1$ such that $\pi G^\sigma_i = G^\sigma_{i^\prime}$. Therefore, we include into $A_1$ all isomorphisms $\pi_1 = (\pi_1 (\sigma)\ | \ \sigma \in \supp \pi_1)$ such that for every $\sigma \in \supp \pi_1$, there is a bijection $f_{\pi_1} (\sigma)$ on a countable set $\supp \pi_1 (\sigma) \subseteq \alpha_\sigma$; and $\pi_1$ is defined as follows: Whenever $p \in D_{\pi_1}$,
% with $\supp p \supseteq \supp \pi_1$ and $\dom_y p^\sigma \supseteq \supp \pi_1 (\sigma)$ for all $\sigma \in \supp \pi_1$, 
then $(\pi_1 p)^\sigma (i, \zeta) = p^\sigma \big(f_{\pi_1} (\sigma) (i)\,, \,\zeta\big)$ for all $(i, \zeta) \in \dom p^\sigma$. Then $\pi G^\sigma_i = G^\sigma_{f_{\pi_1} (\sigma)(i)}$. \\[-2mm]

Roughly speaking, $A_1$ will be generated by these two types of isomorphisms. In order to retain a group structure, the values $(\pi_1 p)^\sigma (i, \zeta)$ for $(i, \zeta) \in \dom \pi_1 (\sigma)$ and $i \in \supp \pi_1 (\sigma)$ have to be treated separately: For every $\zeta \in \dom_y \pi_1 (\sigma)$, there will be a bijection $\pi_1 (\zeta): 2^{\supp \pi_1 (\sigma)} \rightarrow 2^{\supp \pi_1 (\sigma)}$ such that $ \big((\pi_1 p)^\sigma (i, \zeta)\ | \ i \in \supp \pi_1 (\sigma)\big) = \pi_1 (\zeta) \big(p^\sigma(i, \zeta)\ | \ i \in \supp \pi_1 (\sigma)\big)$. \\[-2mm] 

This gives the following definition:

\begin{definition} \label{defpi1} $A_1$ consists of all isomorphisms $\pi_1: D_{\pi_1} \rightarrow D_{\pi_1}$, $\pi_1 = (\pi_1 (\sigma)\ | \ \sigma \in \supp \pi_1)$ with countable support $\supp \pi_1 \subseteq \Succ$, such that for all $\sigma \in \supp \pi_1$, $\kappa_\sigma = \ol{\kappa_\sigma}^{\,\plus}$, there are 

\begin{itemize} \item a countable set $\supp \pi_1 (\sigma) \subseteq \alpha_\sigma$ with a bijection $f_{\pi_1} (\sigma): \supp \pi_1 (\sigma) \rightarrow \supp \pi_1 (\sigma)$,\item a domain $\dom \pi_1 (\sigma) = \dom_x \pi_1 (\sigma)\, \times \, \dom_y \pi_1 (\sigma) \subseteq \alpha_\sigma \, \times \, [\ol{\kappa_{\sigma}}, \kappa_\sigma)$ with $|\dom \pi_1 (\sigma)| < \kappa_\sigma$, such that $\supp \pi_1 (\sigma) \subseteq \dom_x \pi_1 (\sigma)$,

\item for every $(i, \zeta) \in \alpha_\sigma \, \times \, [\ol{\kappa_\sigma}, \kappa_\sigma)$ a bijection $\pi_1 (\sigma) (i, \zeta): 2 \rightarrow 2$, with $\pi_1 (\sigma) (i, \zeta) = \id$ whenever $(i, \zeta) \notin \dom \pi_1 (\sigma)$, \mbox{\upshape and }

\item for every $\zeta \in \dom_y \pi_1 (\sigma)$ a bijection $\pi_1 (\zeta): 2^{\supp \pi_1 (\sigma)} \rightarrow 2^{\supp \pi_1 (\sigma)}$

\end{itemize}

with \[ D_{\pi_1} = \{p \in \m{P}_1\ | \ \supp p \supseteq \supp \pi_1\, \wedge\, \forall\, \sigma \in \supp \pi_1\ \dom p^\sigma \supseteq \dom \pi_1 (\sigma)\},\]

and for every $p \in D_{\pi_1}$, the image $\pi_1 p$ is defined as follows: \\[-2mm]

We will have $\supp (\pi_1 p) = \supp p$ with $(\pi_1 p)^\sigma = p^\sigma$ whenever $\sigma \notin \supp \pi_1$. Moreover, for $\sigma \in \supp \pi_1$, \begin{itemize} 
\item for every $(i, \zeta) \in \dom p^\sigma$ with $i \notin \supp \pi_1 (\sigma)$, we have $(\pi_1 p)^\sigma (i, \zeta) = \pi_1 (\sigma) (i, \zeta) \big(p^\sigma (i, \zeta) \big)$, \item for every $i \in \supp \pi_1 (\sigma)$ and $\zeta \in \dom_y p^\sigma \, \setminus \dom_y \pi_1 (\sigma)$, \[ (\pi_1 p)^\sigma (i, \zeta) = p^\sigma \big(f_{\pi_1} (\sigma) (i), \zeta\big), \mbox{ \upshape and}\] \item for all $\zeta \in \dom_y \pi_1 (\sigma)$, \[\ \big( (\pi_1 p)^\sigma (i, \zeta)\ | \ i \in \supp \pi_1 (\sigma) \big) = \pi_1 (\zeta) \big( p^\sigma (i, \zeta)\ | \ i \in \supp \pi_1 (\sigma) \big). \]

\end{itemize}

\end{definition} 

In other words: Outside the domain $dom\,\pi_1 (\sigma)$, we have a swap of the horizontal lines $p^\sigma (i, \cdot)$ for $i \in \supp \pi_1 (\sigma)$, according to $f_{\pi_1} (\sigma)$. If $\zeta \in \dom \pi_1 (\sigma)$, then the values $(\pi_1 p)^\sigma (i, \zeta)$ for $i \in \supp \pi_1 (\sigma)$ are given by the map $\pi_1 (\zeta)$. Any remaining value $(\pi_1 p)^\sigma (i, \zeta)$ with $i \notin \supp \pi_1 (\sigma)$ is equal to $p^\sigma (\zeta, i)$ or not, depending on whether $\pi_1 (\sigma) (i, \zeta): 2 \rightarrow 2$ is the identity or not. \\[-2mm]

We need the dense sets $D_{\pi_1}$ in order to make sure that $dom\,(\pi_1 p)^\sigma = dom\,p^\sigma$. In particular, we do not want to run out of $\dom_x p^\sigma$ when permuting the $p^\sigma (i, \cdot)$ for $i \in \supp \pi_1 (\sigma)$.
%according to $f_{\pi_1} (\sigma)$. 
\\[-2mm]

It is not difficult to see that any map $\pi_1: D_{\pi_1} \rightarrow D_{\pi_1}$ as in Definition \ref{defpi1} is order-preserving. \\[-3mm]

Whenever $\pi_1 \in A_1$ and $D \in \mathcal{D}_1$ with $D \subseteq D_{\pi_1}$, then the map $\ol{\pi}_1 := \pi_1\, \uhr\, D$ is contained in $A_1$, as well. Here we have to use that the maps $\pi_1$ do not disturb the conditions' support or domain. In particular, whenever $p \in D$, it follows that $\pi_1 p$ is contained in $D$, as well. \\[-2mm]

It remains to verify that $A_1$ can be endowed with a group structure, which happens similarly as for $A_0$: \\[-3mm]

Firstly, for any $\pi_1 \in A_1$, $\pi_1: D_{\pi_1} \rightarrow D_{\pi_1}$, one can write down a map $\nu_1 \in A_1$ with $D_{\nu_1} = D_{\pi_1}$ such that $\nu_1$ is the inverse of $\pi_1$. 

%the inverse map $\pi_1^{-1}: D_{\pi_1} \rightarrow D_{\pi_1}$ with $\supp \pi_1^{-1} = \supp \pi_1$ and $\dom \pi_1^{-1} (\sigma) = \dom \pi_1 (\sigma)$ for every $\sigma \in \supp \pi_1$ can be written down explicitly such that $\pi_1^{-1} \, \circ \, \pi_1 = \pi_1 \, \circ \, \pi_1^{-1} = \id_{D_{\pi_1}}$. 

Secondly, whenever $\pi_1, \sigma_1 \in A_1$, $\pi_1: D_{\pi_1} \rightarrow D_{\pi_1}$, $\sigma_1: D_{\sigma_1} \rightarrow D_{\sigma_1}$, one can explicitly write down a map $\tau_1 \in A_1$ with $D_{\tau_1} = D_{\pi_1}\, \cap\, D_{\sigma_1}$ such that $\tau_1 (p) = \pi_1 (\sigma_1 (p))$ holds for all $p \in D_{\tau_1}$.

%$\nu_1 = \pi_1 \, \circ \, \sigma_1$ on $D_{\nu_1} = D_{\pi_1} \, \cap \, D_{\sigma_1}$, where $\supp \nu_1 = \supp \pi_1\, \cap \, \supp \sigma_1$, and $\dom \nu_1 (\sigma) = \dom \pi_1 (\sigma) \, \cap \, \dom \sigma_1 (\sigma)$ for all $\sigma \in \supp \nu_1$. 

Thirdly, $A_1$ contains the identity element $\id_1$, which is the identity on its domain $D_{\id_1} = \m{P}_1$. \\[-2mm]

%Finally, the identity element $\id_1$ with $\supp \id_1 = \emptyset$ such that $\id_1$ is the identity on $D_{\id_1} = \m{P}_1$ is contained in $A_1$, as well.
%Hence, 

As before, we  set $(A_1)_D := \{\pi_1 \in A_1\ | \ D_{\pi_1} = D\}$ for $D \in \mathcal{D}_1$, and take the direct limit \[\ol{A}_1 := \lim\limits_{\longrightarrow}{(A_1)_D} = \bigsqcup (A_1)_D \big/ \sim \;,\]  with $\pi_1 \sim \pi_1^\prime$ iff $\pi_1 (p) = \pi_1 ^\prime (p)$ holds for all $p \in D_{\pi_1}\, \cap\, D_{\pi_1^\prime}$. Then $\ol{A}_1$ becomes a group with the identity element $\ol{\id}_1 := [id_1]$. \\[-3mm]

%Then $\ol{A}_0$ becomes a group with concatenation $\circ$ and the inverse elements as in Section \ref{outline}, and the identity element $id_0$ with $D_{\id_0}$ as defined above.
%As for $A_0$, we can tale

%\begin{lem} $A_1$ is a group of partial $\m{P}_1$-automorphisms. \end{lem}

Now, we can define our group $\ol{A}$:

\begin{definition} 
Let $\mathcal{D} := \mathcal{D}_0\, \times\, \mathcal{D}_1$ and $A := A_0 \, \times \, A_1$; i.e. $A$ is the collection of all $\pi = (\pi_0, \pi_1)$ with $\pi_0 \in A_0$ and $\pi_1 \in A_1$ as defined above, and domain $D_\pi = D_{\pi_0}\, \times\, D_{\pi_1} \in \mathcal{D}$. Let $\ol{A} := \ol{A}_0\, \times\, \ol{A}_1$. This is a group with the identity element $id = (\ol{\id}_0, \ol{\id}_1)$.

%Then $A$ is a group of partial $\m{P}$-automorphisms.
\end{definition}

%ACHTUNG! die Betrachtungen zu $\ol{\Name}^D$ etc. bräuchte man doch für $\m{P}_0$ und $\m{P}_1$ zusammen!

For $D \in \mathcal{D}$, we define $\ol{\Name (\m{P})}^D$ as in Chapter \ref{outline}. For $D \in  \mathcal{D}$ and $\dot{x} \in \Name (\m{P})$, we define recursively:
%For any $D = D_\pi = D_{\pi_0}\, \times \, D_{\pi_1}$ with $\pi = (\pi_0, \pi_1) \in A$ as above, we define a hierarchy $\ol{\Name(\m{P})}^D$ recursively as follows: \begin{itemize} \item $\ol{\Name (\m{P})}_0^D := \emptyset$ \item $\ol{\Name (\m{P})}_{\alpha+1}^D := \{\dot{x} \in \Name (\m{P})\ | \ \dot{x} \subseteq \ol{\Name(\m{P})}_\alpha^D\, \times \, D \}$, and \item $\ol{\Name (\m{P})}_\lambda^D := \bigcup_{\alpha < \lambda} \ol{\Name(\m{P})}^D_\alpha$ for $\lambda$ a limit ordinal. \end{itemize} Let \[\ol{\Name(\m{P})}^D := \bigcup_{\alpha \in \Ord} \ol{\Name(\m{P})}^D_\alpha.\]
%\vspace*{2mm}
%Whenever ${\pi}: D_{\pi} \rightarrow D_{\pi}$ and $\dot{x} \in \ol{\Name(\m{P})}$, we can define ${\pi} \dot{x}$ as usual. However, for a $\m{P}$-name $\dot{x}$ with $\dot{x} \notin \ol{\Name(\m{P})}^D$, it is not clear how to apply $\pi$, so we need an extension $\ol{\dot{x}}^{D_{\pi}}$: \\[-2mm]
%For $D \subseteq \m{P}$, $D = D_\pi = D_{\pi_0} \, \times\, D_{\pi_1}$ as above, we define recursively for $\m{P}$-names $\dot{x}$: 
\[\ol{x}^{D} := \{ (\ol{y}^D, p)\ | \ \dot{y} \in \dom{\dot{x}}\, , \, p \in D\, , \, p \Vdash \dot{y} \in \dot{x}\}.\]

Then $\ol{x}^D \in \ol{\Name (\m{P})}^D$; and for any $G$ a $V$-generic filter on $\m{P}$, it follows that $(\ol{x}^D)^G = \dot{x}^G$. \\[-3mm]

It is not difficult to check that whenever $D$, $D^\prime \in \mathcal{D}$ and $\dot{x} \in \Name (\m{P})$, then $\ol{\ol{x}^{D}}^{D^\prime} = \ol{x}^{D^\prime}$. \\[-3mm]

The next lemma is important to establish a notion of symmetry that is coherent with the equivalence relation $\sim$:

\begin{lem} \label{pisimpiprime} Let $\pi$, $\pi^\prime \in A$ with $\pi \sim \pi^\prime$, i.e. $\pi\, \uhr\, (D_\pi\, \cap \, D_{\pi^\prime}) = \pi^\prime\, \uhr\, (D_\pi\, \cap\, D_{\pi^\prime})$. Then for $\dot{x} \in \Name (\m{P})$, it follows that $\pi \ol{x}^{D_\pi} = \ol{x}^{D_\pi}$ if and only if $\pi^\prime \ol{x}^{D_{\pi^\prime}} = \ol{x}^{D_{\pi^\prime}}$. \end{lem}

We prove the following more general statement by induction over $\alpha$:

\begin{lem} \label{pisimpiprime2} Let $\pi$, $\pi^\prime \in A$ with $\pi \sim \pi^\prime$, i.e. $\pi\, \uhr\, (D_\pi\, \cap \, D_{\pi^\prime}) = \pi^\prime\, \uhr\, (D_\pi\, \cap\, D_{\pi^\prime})$, and $\alpha \in \Ord$. Then for any $\dot{y}$, $\dot{z} \in \Name (\m{P})$ with $\rk \dot{y} = \rk \dot{z} = \alpha$, it follows that $\pi \ol{y}^{D_\pi} = \ol{z}^{D_\pi}$ if and only if $\pi^\prime \ol{y}^{D_{\pi^\prime}} = \ol{z}^{D_{\pi^\prime}}$. \end{lem}

%\colorbox{yellow} {TO DO: Rang für Namen einführen!}

\begin{proof} W.l.o.g. we can assume that $D_{\pi^\prime} \subseteq D_\pi$, since the map $\wt{\pi} := \pi\, \uhr\, (D_\pi\, \cap\, D_{\pi^\prime}) = \pi^\prime\, \uhr\, (D_\pi\, \cap\, D_{\pi^\prime})$ is contained in $A$ as well, with $D_{\wt{\pi}} = D_\pi\, \cap\, D_{\pi^\prime}$. \\[-3mm]

Consider $\alpha \in \Ord$, and assume that the statement is true for all $\beta < \alpha$. Let $\dot{y}$, $\dot{z} \in \Name (\m{P})$ with $\rk \dot{y} = \rk \dot{z} = \alpha$. \begin{itemize} \item[\tbl $\Rightarrow$\tbr:] First, assume that $\pi \ol{y}^{D_\pi} = \ol{z}^{D_\pi}$. We only prove $\ol{z}^{D_{\pi^\prime}} \subseteq \pi^\prime \ol{y}^{D_{\pi^\prime}}$; the other inclusion is similar. 

Let $(\ol{x}^{D_{\pi^\prime}}, \ol{p}) \in \ol{z}^{D_{\pi^\prime}}$, i.e. $\dot{x} \in \dom \dot{z}$, $\ol{p} \in D_{\pi^\prime}$, and $\ol{p} \Vdash \dot{x} \in \dot{z}$. Then also $\ol{p} \in D_\pi$ holds. Hence, $(\ol{x}^{D_\pi}, \ol{p}) \in \ol{z}^{D_\pi}$, and $\ol{z}^{D_\pi} = \pi \ol{y}^{D_\pi}$ by assumption; so there must be $\dot{u} \in \dom \dot{y}$ with $\ol{x}^{D_\pi} = \pi \ol{u}^{D_\pi}$. Setting $\ol{q} := \pi^{-1} \ol{p}$, it follows that $\ol{q} \Vdash \ol{u}^{D_\pi} \in \ol{y}^{D_\pi}$ and $\ol{q} \Vdash \dot{u} \in \dot{y}$. 

Since $\rk \dot{u} = \rk \dot{x} < \alpha$, our inductive assumption implies that $\ol{x}^{D_{\pi^\prime}} = \pi^\prime \ol{u}^{D_{\pi^\prime}}$ holds. Hence, $(\ol{x}^{D_{\pi^\prime}}, \ol{p}) = (\pi \ol{u}^{D_{\pi^\prime}}, \pi \ol{q})$, which is contained in $\pi^\prime \ol{y}^{D_{\pi^\prime}}$, since $\dot{u} \in \dom \dot{y}$, $\ol{q} \in D_{\pi^\prime}$ (since $\ol{p} \in D_{\pi^\prime}$ and $\ol{q} = \pi^{-1} \ol{p}$), and $\ol{q} \Vdash \dot{u} \in \dot{y}$.

\item[\tbl $\Leftarrow$\tbr:] Now, we assume that $\pi^\prime \ol{y}^{D_{\pi^\prime}} = \ol{z}^{D_{\pi^\prime}}$. As before, we only prove the inclusion $\ol{z}^{D_\pi} \subseteq \pi \ol{y}^{D_\pi}$. 

Consider $(\ol{x}^{D_\pi}, \ol{p}) \in \ol{z}^{D_\pi}$, i.e. $\dot{x} \in \dom \dot{z}$, $\ol{p} \in D_\pi$ and $\ol{p} \Vdash \dot{x} \in \dot{z}$. Let $\wt{p} \leq \ol{p}$ with $\wt{p} \in D_{\pi^\prime}$. Then $(\ol{x}^{D_{\pi^\prime}}, \wt{p}) \in \ol{z}^{D_{\pi^\prime}} = \pi^\prime \ol{y}^{D_{\pi^\prime}}$, so there must be $\dot{u} \in \dom \dot{y}$ with $\ol{x}^{D_{\pi^\prime}} = \pi^\prime \ol{u}^{D_{\pi^\prime}}$. By the inductive assumption, it follows that $\ol{x}^{D_\pi} = \pi \ol{u}^{D_\pi}$, since $\rk \dot{u} = \rk \dot{x} < \alpha$. Let $\ol{q} := \pi^{-1} \ol{p}$. We have to show that $(\pi \ol{u}^{D_\pi}, \pi \ol{q}) \in \pi \ol{y}^{D_\pi}$. Since $\dot{u} \in \dom \dot{y}$ and $\ol{q} \in D_\pi$, it suffices to verify that $\ol{q} \Vdash \dot{u} \in \dot{y}$. We prove that whenever $r \leq \ol{q}$, $r \in D_{\pi^\prime}$, then $r \Vdash \dot{u} \in \dot{y}$. Consider $r \leq \ol{q}$ with $r \in D_{\pi^\prime}$. Then $\pi r \in D_{\pi^\prime}$, and $\pi r \leq \ol{p}$ implies that $\pi r \Vdash \dot{x} \in \dot{z}$. Hence, $(\ol{x}^{D_{\pi^\prime}}, \pi r) \in \ol{z}^{D_{\pi^\prime}}$, and $\ol{z}^{D_{\pi^\prime}} = \pi^\prime \ol{y}^{D_{\pi^\prime}}$ by assumption. Now, $(\pi^\prime \ol{u}^{D_{\pi^\prime}}, \pi^\prime r) = (\ol{x}^{D_{\pi^\prime}}, \pi r) \in \pi^\prime \ol{y}^{D_{\pi^\prime}}$ implies that $r \Vdash \ol{u}^{D_{\pi^\prime}} \in \ol{y}^{D_{\pi^\prime}}$; hence, $r \Vdash \dot{u} \in \dot{y}$ as desired. 

%Let $\ol{\ol{q}} := (\pi^\prime)^{-1} \ol{\ol{p}}$.
%Let 
\end{itemize}
\end{proof}

%Let $\pi, \sigma \in A$ with $\pi: D_{\pi} \rightarrow D_{\pi}$, $\sigma: D_{\sigma} \rightarrow D_{\sigma}$ as above. It is not difficult to verify the following properties: \begin{itemize} \item For any $\dot{x} \in \Name (\m{P})$, it follows that $\ol{\ol{\dot{x}}^{D_{\pi}}}^{D_{\sigma}} = \ol{\dot{x}}^{D_{\pi}\, \cap \, D_{\sigma}}$. \item Whenever $\dot{x} \in \ol{\Name(\m{P})}^{D_{\pi}}$, then also $\sigma \ol{\dot{x}}^{D_{\sigma}} \in \ol{\Name(\m{P})}^{D_{\pi}}$ with $\pi \ol{\dot{x}}^{D_{\sigma}} = \ol{{\pi} \dot{x}}^{D_{\sigma}}$.\end{itemize}

\subsection{Constructing $\boldsymbol{\mathcal{F}}$.} \label{constructingf}

Now, we define our collection of $\ol{A}$-subgroups that will generate a normal filter $\mathcal{F}$ on $\ol{A}$, establishing our notion of symmetry. \\[-3mm]

%symmetric submodel $N$. \\[-3mm]

We will introduce two different types of subgroups. \\[-3mm]

%, such that countable intersections of those generate a normal filter $\mathcal{F}$ on $\ol{A}$. Then we can use $\mathcal{F}$ to establish our notion of symmetry. \\[-3mm]

%where by \textit{$A$-subgroup}, we mean a set $B \subseteq A$ such that $B$ is a group of partial $\m{P}$-automorphisms. \\[-3mm]

%We will have $A$-subgroups of the form $B_0\, \times \, A_1$, for an $A_0$-subgroup $B_0$, and $A$-subgroups of the form $A_0\, \times \, B_1$ for an $A_1$-subgroup $B_1$. \\[-3mm]

%We start with defining $A_0$-subgroups. \\[-3mm]

Firstly, for any $0 < \eta < \gamma$,  $i < \alpha_\eta$ (with $\eta \in \Lim$ or $\eta \in \Succ$), let \[Fix (\eta, i) := \big\{\, [\pi] \in \ol{A} \ | \ \forall p \in D_\pi\ (\pi p)^\eta_i = p^\eta_i\, \big\}.\] 

Whenever $\pi \sim \pi^\prime$, it follows that $(\pi p)^\eta_i = p^\eta_i$ for all $p \in D_\pi$ if and only if $(\pi^\prime p)^\eta_i = p^\eta_i$ for all $p \in D_{\pi^\prime}$. Hence, $Fix (\eta, i)$ is well-defined, and clearly, any $Fix (\eta, i)$ is a subgroup of $\ol{A}$. \\[-3mm]

By including $Fix (\eta, i)$ into our filter $\mathcal{F}$, we make sure that any canonical name $\dot{G}^\eta_i$ for the $i$-th generic $\kappa_\eta$-subset $G^\eta_i$ is hereditarily symmetric, since $\pi \dot{G}^\eta_i = \dot{G}^\eta_i$ for all $\pi \in Fix (\eta, i)$. Hence, our eventual model $N$ will contain any generic $\kappa_\eta$-subset $G^\eta_i$. \\[-2mm]

%ACHTUNG - wurde schon erwähnt, dass man $G^\eta_i$, $\bigcup G^\eta_i$, und die zugehörige $\kappa_\eta$-Teilmenge nicht sauber trennt?
%\[\dot{G}^\eta_i := \{\]

Now, we turn to the second type of $\ol{A}$-subgroup. For any $0 < \lambda < \gamma$ and $k < \alpha_\lambda$ (with $\lambda \in \Lim$ or $\lambda \in \Succ$), we need in $N$ a surjection $s: \powerset(\kappa_\lambda) \rightarrow k$ in order to make sure that $\theta^N (\kappa_\lambda) \geq \alpha_\lambda$. However, the sequence $(G^\lambda_i\ | \ i < \alpha_\lambda)$ must \textit{not} be included into $N$, since $\theta^N (\kappa_\lambda) \leq \alpha_\lambda$, so $N$ must not contain a surjection $s: \powerset (\kappa_\lambda) \rightarrow \alpha_\lambda$. \\[-2mm]

The idea (which appears in \cite{GK} in a slightly different setting, and in \cite{arxiv} similar as here) is that for any $ 0 < \lambda < \gamma$ and $k < \alpha_\lambda$, we define a \tbl cloud\tbr\,around each $G^\lambda_i$ for $i\leq k$, denoted by $(\wt{G^\lambda_i})^{(k)}$, and make sure that the \tbl sequence of clouds\tbr\, $(\, (\wt{G^\lambda_i})^{(k)}\ | \ i < k)$ makes its way into $N$. \\[-2mm]

When defining these subgroups, we have to treat limit cardinals and successor cardinals separately. \\[-2mm]

For $\lambda \in \Lim$, $k < \alpha_\lambda$, let \[H^\lambda_k := \Big\{\, [\pi] \in \ol{A}\ \, \Big| \ \, \exists\, \kappa_{\ol{\nu}, \ol{\j}} < \kappa_\lambda\ \: \forall\, \kappa_{\nu, j} \in [\kappa_{\ol{\nu}, \ol{\j}}, \kappa_\lambda)\ \: \forall\, i \leq k\: : \]\[\Big( (\lambda, i) \notin \supp \pi_0 (\nu, j)\, \vee\, G_{\pi_0} (\nu, j) (\lambda, i) = (\lambda, i) \Big) \, \Big\}.\]

It is not difficult to verify that any $H^\lambda_k$ is well-defined and indeed a subgroup of $\ol{A}$. \\[-3mm]

%This is well-defined: Whenever $\pi \sim \pi^\prime$, and there is $\kappa_{\ol{\nu}, \ol{\j}}$ as above, then for all $\kappa_{\nu, j}, \geq \kappa_{\ol{\nu}, \ol{\j}}$ and $i \leq k$ it also follows that $(\lambda, i) \notin \supp \pi^\prime_0 (\nu, j)$ or $G_{\pi^\prime_0} (\nu, j) (\lambda, i) = (\lambda, i)$. \\ Again, it is not difficult to verify that $H^\lambda_k$ is indeed an subgroup of $\ol{A}$.\\[-3mm]

Roughly speaking, $H^\lambda_k$ contains all $[\pi] \in \ol{A}$ such that above some $\kappa_{\ol{\nu}, \ol{\j}} < \kappa_\lambda$, there is no permutation of the vertical lines $P^\lambda_i\, \uhr \, [\kappa_{\ol{\nu}, \ol{\j}}, \kappa_\lambda)$ for $i \leq k$.

%the following holds for $\pi_0$: For each interval $[\kappa_{\nu, j}, \kappa_{\nu, j + 1})$ above some $\kappa_{\ol{\nu}, \ol{j}} (\pi) < \kappa_\lambda$, there is no permutation of the vertical lines $P^\lambda_i\, \uhr \, [\kappa_{\nu, j}, \kappa_{\nu, j + 1})$ for $i \leq k$. 

This implies that for any $i, j < k$ with $i \neq j$ and $[\pi] \in H^\lambda_k$, it is not possible that $\pi G^\lambda_i = G^\lambda_j$. Hence, for any $i < k$, we can define a \tbl cloud\tbr\,around $G^\lambda_i$ as follows: \[(\dot{\wt{G^\lambda_i}})^{(k)} := \big \{ \big( \pi \ol{G^\lambda_i}^{D_\pi}, \m{1} \big)\ | \ [\pi] \in H^\lambda_k \big\}.\] With $(\wt{G^\lambda_i})^{(k)} := \big((\dot{\wt{G^\lambda_i}})^{(k)}\big)^G$, it follows that $(\wt{G^\lambda_i})^{(k)}$ is the orbit of $G^\lambda_i$ under $H^\lambda_k$; so two distinct orbits $(\wt{G^\lambda_i})^{(k)}$ and $(\wt{G^\lambda_j})^{(k)}$ for $i \neq j$ are disjoint. The sequence $\big( (\wt{G^\lambda_i})^{(k)}\ | \ i < k \big)$, which has a canonical symmetric name stabilized by all $\pi$ with $[\pi] \in H^\lambda_k$, gives a surjection $s: \powerset (\kappa_\lambda) \rightarrow k$ in $N$ (see Chapter 6.1). \\[-2mm]

%FRAGE: Den oberen Index $(k)$ evtl weglassen? Oder wäre dieser später nötig? \\[-2mm]

Now, we consider the case that $\lambda \in \Succ$. For $k < \alpha_\lambda$, let \[H^\lambda_k := \big \{\, [\pi] \in \ol{A}\ | \ \forall\, i \leq k\ \big(i \notin \supp \pi_1 (\lambda)\, \vee \, f_{\pi_1} (\lambda) (i) = i \big)\, \big\}.\] 

Again, 
%for $\pi \sim \pi^\prime$, it follows that whenever $i \notin \supp \pi_1 (\lambda)$ or $f_{\pi_1} (\lambda) (i) = i$ for all $i \leq k$, then the same holds for $\pi^\prime$. Hence, $H^\lambda_k$ is well-defined; and 
one can easily check that $H^\lambda_k$ is well-defined and indeed an $\ol{A}$-subgroup. \\[-2mm]

Whenever $[\pi]$ is contained in $H^\lambda_k$, then $\pi$ does not 
%Then $H^\lambda_k$ does not contain any of those automorphisms that 
interchange any $G^\lambda_i$ and $G^\lambda_j$ for $i, j < k$ in the case that $i \neq j$. Thus, as for $\lambda \in \Lim$, we can define \tbl clouds\tbr\,$(\wt{G^\lambda_i})^{(k)}$ for $i \leq k$ and obtain a surjection $s: \powerset(\kappa_\eta) \rightarrow k$ in $N$ (see Chapter 6.1). \\[-2mm]

%for any $i < k$, we can define a \tbl cloud\tbr\,around $G^\eta_i$ with the symmetric name \[(\dot{\wt{G^\eta_i}})^{(k)} := \big\{\big( \pi \dot{G}^\eta_i , \m{1} \big)\ | \ \pi \in H^\eta_k \big\},\] such that for $(\wt{G^\eta_i})^{(k)} := \big( (\dot{\wt{G^\eta_i}})^{(k)}\big)^G$ it follows that any two distinct \tbl clouds\tbr\, $(\wt{G^\eta_i})^{(k)}$ and $(\wt{G^\eta_j})^{(k)}$ for $i, j < k$ with $i \neq j$ are disjoint. Hence, the sequence $\big( (\wt{G^\eta_i})^{(k)}\ | \ i \leq k \big)$, which has a canonical symmetric name stabilized by all $\pi \in H^\lambda_k$, gives a surjection $s: \powerset (\kappa_\eta) \rightarrow k$ in $N$.

%ACHTUNG - aus der \tbl Bemerkung\tbr\,zur Eindeutigkeit von $\supp \pi_1 (\eta)$, die hier noch nicht zu finden ist, würde folgen, dass $F_{\pi_1} (i) = i$ impliziert, dass $i \notin \supp \pi_1 (\eta)$!. \\[-3mm]

%It is not difficult to see that any $H^\lambda_k$ for $\lambda \in \Lim$ or $\lambda \in \Succ$ is indeed an $A$-subgroup. \\[-2mm]

We are now ready to define our normal filter $\mathcal{F}$ on $\ol{A}$. Note that the $Fix (\eta, i)$ and $H^\lambda_k$ are \textit{not} normal $\ol{A}$-subgroups: For instance, if $[\pi] \in Fix (\eta, i)$ for some $\eta \in \Lim$, $i < \alpha_\eta$, and $\sigma \in A$ with $G_{\sigma_0} (\nu, j)(\eta, i) = (\eta, i^\prime)$ for all $\kappa_{\nu, j}, < \kappa_\eta$ such that $[\pi] \notin Fix (\eta, i^\prime)$, then in general, $[\sigma]^{-1}[\pi][\sigma]$ is not contained in $Fix (\eta, i)$. \\[-3mm]
%consider $p \in D_\pi\, \cap\, D_\sigma$. We have %$((\sigma^{-1}\, \circ \, \pi \, \circ \, \sigma) p)^\eta_i

However, it is not difficult to verify: \begin{lem} \label{normalfilter} \begin{itemize} \item For all $\sigma \in A$, and $\eta \in \Lim$, $i < \alpha_\eta$, \[[\sigma] Fix (\eta, i) [\sigma]^{-1} \supseteq Fix (\eta, i) \, \cap\, \bigcap \{ Fix (\eta_m, i_m)\ | \ m < \omega, (\eta_m, i_m) \in \supp \sigma_0 \}. \] In the case that $\sigma \in A$, and $\eta \in \Succ$, $i < \alpha_\eta$, \[[\sigma] Fix (\eta, i) [\sigma]^{-1} \supseteq Fix (\eta, i)\, \cap\, \bigcap \{ Fix(\eta, i_m)\ | \ m < \omega, i_m \in \supp \sigma_1 (\eta) \}.\] \item For $\sigma \in A$ and $\lambda \in \Lim$, $k < \alpha_\lambda$, \[[\sigma] H^\lambda_k [\sigma]^{-1} \supseteq  H^\lambda_k\, \cap \, \bigcap \{ Fix (\eta_m, i_m)\ | \ m < \omega, (\eta_m, i_m) \in \supp \sigma_0 \}.\] In the case that $\lambda \in \Succ$, $k < \alpha_\lambda$, \[[\sigma] H^\lambda_k [\sigma]^{-1} \supseteq H^\lambda_k\, \cap \, \bigcap \{ Fix(\lambda, i_m)\ | \ m < \omega, i_m \in \supp \sigma_1 (\lambda) \}.\]  \end{itemize} \end{lem}

%\colorbox{yellow} {FRAGE: Müsste man erwähnen, warum das nicht nachgerechnet wird?}

Hence, it follows that countable intersections of the $\ol{A}$-subgroups $Fix (\eta, i)$ and $H^\lambda_k$ generate a normal filter on $\ol{A}$.

\begin{definition} Let $\mathcal{F}$ denote the filter on $\ol{A}$ defined as follows: \\[-3mm]

A subgroup $B \subseteq \ol{A}$ is contained in $\mathcal{F}$ if there are $( (\eta_m, i_m)\ | \ m < \omega)$, $( (\lambda_m, k_m)\ | \ m < \omega)$ with \[B \supseteq \bigcap_{m < \omega} Fix (\eta_m, i_m)\, \cap \, \bigcap_{m < \omega} H^{\lambda_m}_{k_m}.\] \end{definition}

Then by Lemma \ref{normalfilter}, it follows that $\mathcal{F}$ is a normal filter on $\ol{A}$; which is countably closed. \\[-3mm]

%\colorbox{yellow}{FRAGE: bei Definitionen \tbl if and only if\tbr\, nicht so schreiben?}

Now, we can use $\mathcal{F}$ to establish our notion of symmetry.

\begin{definition} A $\m{P}$-name $\dot{x}$ is { \upshape symmetric} if \[ \big\{ \, [\pi] \in \ol{A}\ | \ \pi \ol{x}^{D_\pi} = \ol{x}^{D_\pi} \, \big\} \in \mathcal{F}. \] Recursively, a name $\dot{x}$ is {\upshape hereditarily symmetric}, $x \in HS$, if $\dot{x}$ is symmetric, and $\dot{y}$ is hereditarily symmetric for all $\dot{y} \in \dom \dot{x}$. \end{definition}

%there are $( (\eta_m, i_m)\ | \ m < \omega)$, $( (\lambda_m, k_m)\ | \ m < \omega)$ such that \[\{\pi \in A\ | \ \pi \ol{x}^{D_\pi} = \ol{x}^{D_\pi} \} \supseteq \bigcap_{m < \omega} Fix (\eta_m, i_m)\, \cap \, \bigcap_{m < \omega} H^{\lambda_m}_{k_m}.\] 
By Lemma \ref{pisimpiprime}, this is well-defined, since for $\pi \sim \pi^\prime$ and $\dot{x} \in \Name (\m{P})$, it follows that $\pi \ol{x}^{D_\pi} = \ol{x}^{D_\pi}$ if and only if $\pi^\prime \ol{x}^{D_{\pi^\prime}} = \ol{x}^{D_{\pi^\prime}}$. \\[-3mm]

%\colorbox{yellow}{FRAGE: Sollte man nochmal schreiben, wann eine Name $\dot{x}$ symmetrisch ist?} 

%\colorbox{yellow}{FRAGE: Sollte man nochmal erwähnen, dass $\pi \ol{x}^{D_\pi} = \ol{x}^{D_\pi}$ nicht vom Repräsentanten abhängt?}

We will use the following properties: If $\dot{x} \in HS$ and $\pi \in A$, then firstly, it is not difficult to verify that also $\ol{x}^{D_\pi} \in HS$ holds, and secondly, $\pi \ol{x}^{D_\pi} \in HS$. For the second claim, one can check that whenever $\sigma \in A$ with $\sigma \ol{x}^{D_\sigma} = \ol{x}^{D_\sigma}$, then \[(\pi \sigma \pi^{-1})\, \ol{\pi \ol{x}^{D_\pi}}^{\ D_{\pi \sigma \pi^{-1}}} = \ol{\pi \ol{x}^{D_\pi}}^{\ D_{\pi \sigma \pi^{-1}}},\] and then use the normality of $\mathcal{F}$.

%The following properties are not difficult to verify: \begin{itemize} \item If $\dot{x} \in HS$ and $\pi \in A$, then also $\ol{\dot{x}}^{D_\pi} \in HS$. \item IF $\dot{x} \in HS$ and $\pi \in A$ with $\dot{x} \in \ol{\Name(\m{P})}^{D_\pi}$, then also $\pi \dot{x} \in HS$. \end{itemize}

%The map $\pi_0$ will be given by the following parameters:

%we will work with symmetric names.
%partial orders and try to avoid using Boolean Algebras.

\section{The symmetric submodel.} 

Let $G$ be a $V$-generic filter on $\m{P}$. Our symmetric extension is \[N := V(G) = \{\dot{x}^G\ | \ \dot{x} \in HS\}. \] 

Since we do not use the standard method for constructing symmetric extensions, we first have to make sure that $N \vDash ZF$.\\[-2mm]

%Since our construction

%Since we do not use the standard method of constructing symmetric submodels, 
%where an isomorphisms $\pi: D_\pi \rightarrow D_\pi$ on a dense set $D_\pi$ is uniquely extended to an automorphism of the corresponding Boolean Algebra, and $V(G)$ is constructed using the Boolean valued model $V^B$, 
%we have to make sure that our model $N$ is indeed a model of $ZF$. \\[-2mm] 

%\colorbox{red}{ACHTUNG - GILT ÜBERHAUPT DAS FORCING THEOREM? DA GIBT ES EIN PROBLEM}

%\colorbox{red}{EVTL: im Skript der Vorlesung nachsehen?}

The symmetric forcing relation \tbl\,$\Vdash_s$\tbr\, can be defined in the ground model as usually, and the forcing theorem holds. \\[-3mm]

Whenever $\dot{x}$, $\dot{y} \in HS$ and $p \in \m{P}$, then $p \Vdash_s \dot{y} \in \dot{x}$ if and only if $p \Vdash \dot{y} \in \dot{x}$ (with the ordinary forcing relation $\Vdash$) and $p \Vdash_s \dot{x} = \dot{y}$ if and only if $p \Vdash \dot{x} = \dot{y}$. In particular, for any $\dot{x} \in HS$ and $D \in \mathcal{D}$, we have \[\ol{x}^D = \{(\ol{y}^D, p) \ | \ \dot{y} \in \dom \dot{x}\, , \, p \in D\, , \, p \Vdash_s \dot{y} \in \dot{x} \}. \]

\vspace*{2mm} Since $V \subseteq V(G) \subseteq V[G]$ and $V(G)$ is transitive, it follows that $V(G)$ satisfies the axioms of $Emptyset$, $Foundation$, $Extensionality$ and $Infinity$. The proofs of the axioms of $Pairing$, $Union$ and $Separation$ are similar to the proofs for the standard construction, with some extra care needed to make sure that all the involved names are indeed symmetric. \\[-3mm]

We give a proof of $Power$ $Set$ and $Replacement$.

%\colorbox{red}{ACHTUNG - das wären doch nicht alle Axiome! Was wäre mit Separation??}

\begin{lem} \label{power set} $V(G) \vDash \; Power\; Set$. \end{lem} \begin{proof} Consider $X\in N$, $X = \dot{X}^G$ with $\dot{X} \in HS$. We have to show that $\powerset^N (X) \in N$. Let \[\dot{B} := \{ (\dot{Y}, p)\ | \ \dot{Y} \in HS\, , \, \dot{Y} \subseteq \dom \dot{X}\, \times\, \m{P}\, , \, p \Vdash_s \dot{Y} \subseteq \dot{X}\}.\] Then $\dot{B}^G = \powerset^N (X)$, since for any $Y \in N$ with $Y \subseteq X$, there exists a name $\dot{Y} \in HS$, $\dot{Y}^G = Y$, such that $\dot{Y} \subseteq \dom \dot{X}\, \times\, \m{P}$.

%\colorbox{yellow}{TO DO: $\Vdash_s$ müsste definiert werden!}

It remains to make sure that the name $\dot{B}$ is symmetric. Consider $\pi \in A$ with $\pi \ol{X}^{D_\pi} = \ol{X}$. Then \[\ol{B}^{D_\pi} = \big \{ \,(\ol{Y}^{D_\pi}, \ol{p}) \ | \ \dot{Y} \in HS\, , \, \dot{Y} \subseteq \dom \dot{X} \, \times \, \m{P}\, , \, \ol{p} \Vdash_s \dot{Y} \in \dot{B}\, , \, \ol{p} \in D_\pi\, \big\}. \] It is not difficult to check that \[\ol{B}^{D_\pi} = \big \{ \,(\ol{Y}^{D_\pi}, \ol{p}) \ | \ \dot{Y} \in HS\, , \, \dot{Y} \subseteq \dom \dot{X} \, \times \, \m{P}\, , \, \ol{p} \Vdash_s \dot{Y} \subseteq \dot{X}\, , \, \ol{p} \in D_\pi\, \big\}, \] since for any $p \in D_\pi$ and $\dot{Y} \in HS$, $\dot{Y} \subseteq \dom{\dot{X}}\, \times\, \m{P}$, it follows that $p \Vdash_s \dot{Y} \in \dot{B}$ if and only if $p \Vdash_s \dot{Y} \subseteq \dot{X}$. Hence,
\[\pi \ol{B}^{D_\pi} = \big\{ \,(\pi \ol{Y}^{D_\pi}, \pi \ol{p})\ | \ \dot{Y} \in HS\, , \, \dot{Y} \subseteq \dom \dot{X}\, \times\, \m{P}\, , \, \pi \ol{p} \Vdash_s \pi \ol{Y}^{D_\pi} \subseteq  \pi \ol{X}^{D_\pi}\, , \, \pi \ol{p} \in D_\pi\, \big\}.\]

It remains to show that $\ol{B}^{D_\pi} = \pi \ol{B}^{D_\pi}$; then \[ \big\{\,  [\pi] \in \ol{A}\ | \ \pi \ol{B}^{D_\pi} = \ol{B}^{D_\pi}\, \big\} \supseteq \big\{ \, [\pi] \in \ol{A}\ | \ \pi \ol{X}^{D_\pi} = \ol{X}^{D_\pi} \, \big\} \in \mathcal{F}\] as desired. \\[-3mm]

For the inclusion $\ol{B}^{D_\pi} \subseteq \pi \ol{B}^{D_\pi}$, consider $(\ol{Y}^{D_\pi}, \ol{p}) \in \ol{B}^{D_\pi}$ as above. It suffices to construct $\dot{Y_0} \in HS$, $\dot{Y_0} \subseteq \dom \dot{X} \, \times \, \m{P}$ with $\pi \ol{Y_0}^{D_\pi} = \ol{Y}^{D_\pi}$. Then setting $\ol{p}_0 := \pi^{-1} \ol{p}$, it follows that $(\ol{Y}^{D_\pi}, \ol{p}) = (\pi \ol{Y_0}^{D_\pi}, \pi \ol{p}_0) \in \pi \ol{B}^{D_\pi}$, since $\ol{p} \Vdash_s \ol{Y}^{D_\pi} \subseteq \ol{X}^{D_\pi}$ and $\pi \ol{X}^{D_\pi} = \ol{X}^{D_\pi}$ gives $\pi \ol{p}_0 \Vdash_s \pi \ol{Y_0}^{D_\pi} \subseteq \pi \ol{X}^{D_\pi}$. 

Let \[\dot{Y_0} := \big \{ \,(\dot{z}, p)\ | \ \dot{z} \in \dom \dot{X}\, , \, p \in D_\pi\, , \, \pi \ol{z}^{D_\pi} \in \dom \ol{Y}^{D_\pi}\, , \,  \pi p \Vdash_s \pi \ol{z}^{D_\pi} \in \ol{Y}^{D_\pi} \big\}.\] 
Then \[\pi \ol{Y_0}^{D_\pi} = \big\{ (\pi \ol{z}^{D_\pi}, \pi p) | \ \dot{z} \in \dom \dot{X}\, , \, p \in D_\pi\, , \, \pi \ol{z}^{D_\pi} \in \dom \ol{Y}^{D_\pi}\, , \, p \Vdash_s \dot{z} \in \dot{Y}_0\, \big\}. \]

%\exists\,\ p \in D_\pi\: : \: \ol{p} \leq p, (\pi \ol{\dot{z}}^{D_\pi}, \pi p) \in \ol{\dot{Y}}^{D_\pi} \big\}.\] Clearly, $\pi \ol{\dot{Y_0}}^{D_\pi} \subseteq \ol{\dot{Y}}^{D_\pi}$. 

We first show that whenever $\dot{z} \in \dom \dot{X}$, $p \in D_\pi$ and $\pi \ol{z}^{D_\pi} \in \dom \ol{Y}^{D_\pi}$ as above, then $p \Vdash_s \dot{z} \in \dot{Y}_0$ if and only if $\pi p \Vdash_s \pi \ol{z}^{D_\pi} \in \ol{Y}^{D_\pi}$. \begin{itemize}  \item[\tbl$\Rightarrow$\tbr : ] If $\pi p \Vdash_s \pi \ol{z}^{D_\pi} \in \ol{Y}^{D_\pi}$, it follows that $(\dot{z}, p) \in \dot{Y}_0$; hence, $p \Vdash_s \dot{z} \in \dot{Y}_0$ as desired. \item[\tbl$\Leftarrow$\tbr :]Now, assume that $p \Vdash_s \dot{z} \in \dot{Y}_0$. Let $H$ be a $V$-generic filter on $\m{P}$ with $\pi p \in H$. We have to show that $(\pi \ol{z}^{D_\pi})^H \in (\ol{Y}^{D_\pi})^H$. Let $H^\prime := \pi^{-1} H$. Then $(\pi \ol{z}^{D_\pi})^H = \dot{z}^{H^\prime}$, and $p \in H^\prime$. Hence, $\dot{z}^{H^\prime} \in \dot{Y}_0^{H^\prime}$ implies that there must be $(\dot{u}, r) \in \dot{Y}_0$ with $\dot{u}^{H^\prime} = \dot{z}^{H^\prime}$ and $r \in H^\prime$. Then $\pi r \Vdash_s \pi \ol{u}^{D_\pi} \in \ol{Y}^{D_\pi}$ by construction of $\dot{Y}_0$. Since $\pi r \in H$, it follows that $(\pi \ol{u}^{D_\pi})^H \in (\ol{Y}^{D_\pi})^H$, with $(\pi \ol{u}^{D_\pi})^H = \dot{u}^{H^\prime} = \dot{z}^{H^\prime} = (\pi \ol{z}^{D_\pi})^H$ as desired.
\end{itemize}

Hence, \[\pi \ol{Y_0}^{D_\pi} = \big\{ (\pi \ol{z}^{D_\pi}, \pi p) | \ \dot{z} \in \dom \dot{X}\, , \, p \in D_\pi\, , \, \pi \ol{z}^{D_\pi} \in \dom \ol{Y}^{D_\pi}\, , \, \pi p \Vdash_s \pi \ol{z}^{D_\pi} \in \ol{Y}^{D_\pi}\, \big\}. \]

We have to make sure that $\pi \ol{Y_0}^{D_\pi} = \ol{Y}^{D_\pi}$. The inclusion $\pi \ol{Y_0}^{D_\pi} \subseteq \ol{Y}^{D_\pi}$ is clear. Regarding $\ol{Y}^{D_\pi} \subseteq \pi \ol{Y_0}^{D_\pi}$, consider $(\ol{u}^{D_\pi}, q) \in \ol{Y}^{D_\pi}$ with $\dot{u} \in \dom \dot{Y} \subseteq \dom \dot{X}$ and $q \in D_\pi$ such that $q \Vdash_s \dot{u} \in \dot{Y}$. From $\ol{u}^{D_\pi} \in \dom \ol{X}^{D_\pi} = \dom \pi \ol{X}^{D_\pi}$, it follows that there must be $\dot{v} \in \dom \dot{X}$ with $\ol{u}^{D_\pi} = \pi \ol{v}^{D_\pi}$. Let $r := \pi^{-1} q$. Then $(\ol{u}^{D_\pi}, q) = (\pi \ol{v}^{D_\pi}, \pi r) \in \pi \ol{Y_0}^{D_\pi}$, since $q \Vdash_s \dot{u} \in \dot{Y}$ implies $\pi r \Vdash_s \pi \ol{v}^{D_\pi} \in \ol{Y}^{D_\pi}$ as desired. \\[-3mm]

Thus, we have constructed $\dot{Y}_0 \subseteq \dom \dot{X}\, \times\, \m{P}$ with $\pi \ol{Y_0}^{D_\pi} = \ol{Y}^{D_\pi}$. It remains to make sure that $\dot{Y}_0 \in HS$. Firstly, $\dot{Y}_0 \subseteq \dom \dot{X}\, \times\, \m{P}$, so $\dom \dot{Y}_0 \subseteq HS$. Secondly, for any $\sigma \in A$ with $\sigma \ol{Y}^{D_\sigma} = \ol{Y}^{D_\sigma}$, it follows that \[(\pi^{-1} \sigma \pi)\, \ol{Y_0}^{\,D_{\pi^{-1} \sigma \pi}} = (\pi^{-1} \sigma \pi)\, \ol{\ol{Y_0}^{D_\pi}}^{\,D_{\pi^{-1} \sigma \pi}} = (\pi^{-1} \sigma \pi)\ \ol{\pi^{-1} \ol{Y}^{D_\pi}}^{\,D_{\pi^{-1} \sigma \pi}}, \] 

and since $\sigma \ol{Y}^{D_\sigma} = \ol{Y}^{D_\sigma}$, one can easily check that \[(\pi^{-1} \sigma \pi)\ \ol{\pi^{-1} \ol{Y}^{D_\pi}}^{\ D_{\pi^{-1} \sigma \pi}} = \ol{\pi^{-1} \ol{Y}^{D_\pi}}^{\ D_{\pi^{-1} \sigma \pi}},\] and

\[\ol{\pi^{-1} \ol{Y}^{D_\pi}}^{\ D_{\pi^{-1} \sigma \pi}} = \ol{\ol{Y_0}^{D_\pi}}^{\ D_{\pi^{-1} \sigma \pi}} = \ol{Y_0}^{\ D_{\pi^{-1} \sigma \pi}}.\]

Since the name $\dot{Y}$ is symmetric, it follows by normality of $\mathcal{F}$ that $\dot{Y}_0$ is symmetric, as well. Hence, $\dot{Y}_0$ has all the desired properties; and it follows that $\ol{B}^{D_\pi} \subseteq \pi \ol{B}^{D_\pi}$.\\[-3mm]

The inclusion $\pi \ol{B}^{D_\pi} \subseteq \ol{B}^{D_\pi}$ is similar.

%\colorbox{red}{ACHTUNG - Ist es egal ob man $\Vdash$ oder $\Vdash_s$ schreibt?} \\
%\colorbox{red}{ACHTUNG - EVTL anmerken, dass das für $\in$ und $\subseteq$ egal wäre?}\\

%\colorbox{red}{ACHTUNG - es fehlt die Inklusion $\pi \ol{B}^{D_\pi} \subseteq \ol{B}^{D_\pi}$!!}
%\colorbox{red}{??? DAS IST NICHT KLAR!}

%For the inclusion $\ol{\dot{Y}}^{D_\pi} \subseteq \pi \ol{\dot{Y_0}}^{D_\pi}$, consider $(\ol{\dot{z_0}}^{D_\pi}, \ol{p}) \in \ol{\dot{Y}}^{D_\pi}$ with $\dot{z_0} \in \dom \dot{Y} \subseteq \dom \dot{X}$, and $\ol{p} \in D_\pi$. Then $\ol{\dot{z_0}} \in \dom \ol{\dot{X}}^{D_\pi} = \dom \pi \ol{\dot{X}}^{D_\pi}$ implies that $\ol{\dot{z_0}}^{D_\pi} = \pi \ol{\dot{z_1}}^{D_\pi}$ for some $\dot{z_1} \in \dom \dot{x}$. Let $\ol{p_1} := \pi^{-1} (\ol{p})$. Then $(\ol{\dot{z_0}}^{D_\pi}, \ol{p}) = (\pi \ol{\dot{z_1}}^{D_\pi}, \pi \ol{p_1})$, where $\dot{z_1} \in \dom \dot{x}$, $\ol{p_1} \in D_\pi$, and $(\pi \ol{\dot{z_1}}^{D_\pi}, \pi \ol{p_1}) = (\ol{\dot{z_0}}^{D_\pi}, \ol{p}) \in \ol{\dot{Y_0}}^{D_\pi}$. Hence, $(\ol{\dot{z_0}}^{D_\pi}, \ol{p}) \in \pi \ol{\dot{Y_0}}^{D_\pi}$ as desired. \\[-3mm]

%Hence, $\ol{\dot{Y}}^{D_\pi} = \pi \ol{\dot{Y_0}}^{D_\pi}$ as desired; so $\ol{\dot{B}}^{D_\pi} \subseteq \pi \ol{\dot{B}}^{D_\pi}$ follows. The inclusion $\pi \ol{\dot{B}}^{D_\pi}$ is similar.

%DIE NOTATION KANN NICHT SO BLEIBEN!

%TO DO: Man müsste $\vDash_s$ definieren!
%FRAGE: Bei Namen $\ol{x}^{D_\pi}$ anstatt $\ol{\dot{x}}^{D_\pi}$ schreiben?

\end{proof}

\begin{lem} \label{replacement} $V(G) \vDash\, Replacement$. \end{lem}
\begin{proof} Consider $a \in N$ such that $N \vDash \forall x \in a \ \exists y \ \varphi(x, y)$. We have to show that there is $b \in N$ with \[ N \vDash \forall x \in a\ \exists\, y \in b\ \varphi (x, y). \] Let $a = \dot{a}^G$ with $\dot{a} \in HS$. We proceed like in the proof of $Replacement$ in ordinary forcing extensions. For $\dot{x} \in \dom \dot{a}$ and $p \in \m{P}$, let \[\alpha (\dot{x}, p) := \min \{\alpha\ | \ \exists\, \dot{w} \in \Name_\alpha (\m{P})\, \cap\, HS\: : \: p \Vdash_s  (\varphi (\dot{x}, \dot{w})\, \wedge\, \dot{x} \in \dot{a}) \}\] if such $\alpha$ exists, and $\alpha (\dot{x}, p) := 0$, else.

By $Replacement$ in $V$, take $\beta \in \Ord$ with $\beta \geq \sup \{\alpha (\dot{x}, p)\ | \ \dot{x} \in \dom \dot{a}\, , \, p \in \m{P}\}$. Let \[\dot{b} := \{ (\dot{y}, \m{1})\ | \ \dot{y} \in \Name_\beta(\m{P})\, \cap \, HS\},\]and $b := \dot{b}^G$. Then for all $x \in a$, it follows that there exists $y \in b$ with $N \vDash \varphi (x, y)$. It remains to show that the name $\dot{b}$ is symmetric. Let $\pi \in A$. Then \[\ol{b}^{D_\pi} = \{ (\ol{y}^{D_\pi}, q)\ | \ \dot{y} \in \Name_\beta(\m{P}) \, \cap\, HS\; , \; q \in D_\pi\},\] and \[\pi \ol{b}^{D_\pi} = \{ (\pi \ol{y}^{D_\pi}, \pi q)\ | \ \dot{y} \in \Name_\beta(\m{P})\, \cap\, HS\; , \; q \in D_\pi\}. \] 

We will show that $\pi \ol{b}^{D_\pi} = \ol{b}^{D_\pi}$.\\[-3mm]

Since it is not possible to apply $\pi$ to arbitrary $\m{P}$-names $\dot{y}$ with $\dot{y} \notin \ol{\Name (\m{P})}^{D_\pi}$, we construct an alternative $\wt{\pi}$ which is sufficient for our purposes here. Recursively, we define for $\m{P}$-names $\dot{y}$: \[ \wt{\pi} (\dot{y}) := \big\{ \, (\wt{\pi} (\dot{z}), \pi \ol{q})\ | \ \exists\, (\dot{z}, q) \in \dot{y}\; , \; \ol{q} \leq q\, , \, \ol{q} \in D_\pi\, \big\}. \] Then for all $\dot{y} \in \Name_\beta (\m{P})$, it follows that $\wt{\pi} (\dot{y}) \in \Name_\beta (\m{P})$, as well. \\[-3mm]

Whenever $H$ is a $V$-generic filter on $\m{P}$, $H^\prime := \pi^{-1} H$ and $\dot{y} \in \Name (\m{P})$, then $(\wt{\pi} (\dot{y}))^H = \dot{y}^{H^\prime}$, and one can easily check that
 \[ \pi \ol{y}^{D_\pi} = \ol{\wt{\pi} (\dot{y})}^{D_\pi}. \] Moreover, whenever $\sigma \in A$ with $\sigma \ol{y}^{D_\sigma} = \ol{y}^{D_\sigma}$, then setting $\tau := \pi \sigma \pi^{-1}$, it follows recursively that \[ \tau \; \ol{\wt{\pi} (\dot{y})}^{D_\tau} = \ol{\wt{\pi} (\dot{y})}^{D_\tau}.\]

%\[(\pi \sigma \pi^{-1} ) \ol{\wt{\pi}(\dot{y})}^{D_{\pi \sigma \pi^{-1}}} = \ol{\wt{\pi} (\dot{y})}^{D_{\pi \sigma \pi^{-1}}}\] follows. 

%\colorbox{yellow} {FRAGE: Sollte man das nachrechnen?} \\
%\colorbox{yellow}{TO DO: Man müsste die $\Name_\beta (\m{P})$-Hierarchie definieren!}

Hence, \[ \big\{\, [\tau] \in \ol{A}\ \; \big| \ \; \tau \; \ol{\wt{\pi} (\dot{y})}^{D_\tau} = \ol{\wt{\pi} (\dot{y})}^{D_\tau} \, \big\} \supseteq \big\{\, [\pi] [\sigma] [\pi]^{-1}\ \; \big| \; \ [\sigma] \in \ol{A}\, , \, \sigma \ol{y}^{D_\sigma} = \ol{y}^{D_\sigma} \, \big\}. \] In the case that $\dot{y}$ is symmetric, i.e. $\big\{\, [\sigma] \in \ol{A}\ \; \big| \ \; \sigma \ol{y}^{D_\sigma} = \ol{y}^{D_\sigma}\, \big\} \in \mathcal{F}$, it follows by normality that also $\big\{\, [\tau] \in \ol{A}\ \; \big| \; \ \tau \; \ol{\wt{\pi} (\dot{y})}^{D_\tau} = \ol{\wt{\pi} (\dot{y})}^{D_\tau} \, \big\} \in \mathcal{F}$. Hence, $\wt{\pi} (\dot{y}) \in HS$ whenever $\dot{y} \in HS$. \\[-2mm]

Now, $\pi \ol{b}^{D_\pi} = \ol{b}^{D_\pi}$ follows: For the inclusion \tbl$\subseteq$\tbr, consider $(\pi \ol{y}^{D_\pi}, \pi q) \in \pi \ol{b}^{D_\pi}$ with $\dot{y} \in \Name_\beta(\m{P})\, \cap \, HS$, $q \in D_\pi$. Then also $\pi q \in D_\pi$, and $\pi \ol{y}^{D_\pi} = \ol{\wt{\pi} (\dot{y})}^{D_\pi}$, where $\wt{\pi} (\dot{y}) \in \Name_\beta (\m{P})\, \cap \, HS$; so $(\pi \ol{y}^{D_\pi}, \pi q) = (\ol{\wt{\pi} (\dot{y})}^{D_\pi}, \pi q) \in \ol{b}^{D_\pi}$ follows. The inclusion \tbl $\supseteq$\tbr\;is similar.\\[-3mm]

Hence, $\dot{b} \in HS$ as desired. 

\end{proof}

Thus, our symmetric extension $N$ is a model of $ZF$. Since our forcing $\m{P}$ is countably closed (Proposition \ref{countablyclosed}), and our normal filter $\mathcal{F}$ generating $N$ is countably closed, it follows that $N \vDash DC$ (see for example \cite[Lemma 1]{Karagila}). Moreover, $N  \vDash AX_4$ (see \cite[p.3 and p.15]{pcfwc}): For any cardinal $\lambda$, we have $([\lambda]^{\aleph_0})^N = ([\lambda]^{\aleph_0})^V$; so the set $[\lambda]^{\aleph_0}$ can be well-ordered in $N$, using the according well-ordering of $[\lambda]^{\aleph_0}$ in $V$. \\[-2mm]

%\colorbox{red}{TO DO: Referenz suchen! Oder kurz erklären?}
%\colorbox{red}{TO DO: Erwähnen, dass $\m{P}$ countably closed ist! z.B. als Proposition}

%\colorbox{red}{ACHTUNG: Könnte die symmetrische Forcing-Relation Probleme bereiten?}

%\colorbox{yellow}{FRAGE: Was sollte ins erste Kapitel?}

%\colorbox{yellow} {EVTL das erste Kapitel weglassen, und das Nötige ungefähr hier versuchen, aufzuschreiben?}

Next, we want to show that $N$ preserves all $V$-cardinals; which will follow from the fact that any set of ordinals $X \subseteq \alpha$, $X \in N$, can be captured in a \tbl mild\tbr\,$V$-generic extension by a forcing notion as in Lemma \ref{generic} and Lemma \ref{generic2}. 

This \textit{Approximation Lemma} demonstrates how our symmetric extension $N$ can be approximated from within by fairly nice $V$-generic extensions. Later on, this will be a crucial step in keeping control over the values $\theta^N (\kappa_\eta)$. 

\begin{lem}[Approximation Lemma] \label{approx} 

Consider $X \in N$, $X \subseteq \alpha$ with $X = \dot{X}^G$ such that $\pi \ol{X}^{D_\pi} = \ol{X}^{D_\pi}$ holds for $\pi \in A$ with $[\pi]$ contained in the intersection \[\bigcap_{m < \omega} Fix( \eta_m, i_m)\, \cap\, \bigcap_{m < \omega} Fix (\ol{\eta}_m, \ol{i}_m)\; \cap \; \bigcap_{m < \omega} H^{\lambda_m}_{k_m}\, \cap \, \bigcap_{m < \omega} H^{\ol{\lambda}_m}_{\ol{k}_m},\] 
 where $( (\eta_m, i_m)\ | \ m < \omega)$, $((\ol{\eta}_m, \ol{i}_m)\ | \ m < \omega)$, $((\lambda_m, k_m)\ | \ m < \omega)$ and $( (\ol{\lambda}_m, \ol{k}_m)\ | \ m < \omega)$ denote sequences with $\eta_m \in \Lim$, $i_m < \alpha_{\eta_m}$; $\ol{\eta}_m \in \Succ$, $\ol{i}_m < \alpha_{\ol{\eta}_m}$ for all $m < \omega$; and $\lambda_m \in \Lim$, $k_m < \alpha_{\lambda_m}$; $\ol{\lambda}_m \in \Succ$, $\ol{k}_m < \alpha_{\ol{\lambda}_m}$ for all $m < \omega$. \\[-4mm]
 
%where $\eta_m \in \Lim$, $i_m < \alpha_{\eta_m}$, $\ol{\eta}_m \in \Succ$, $\ol{i}_m < \alpha_{\ol{\eta}_m}$ for all $m < \omega$; and $\lambda_m \in \Lim$, $k_m < \alpha_{\lambda_m}$, $\ol{\lambda}_m \in Succ$, $\ol{k}_m < \alpha_{\ol{\lambda}_m}$ for all $m < \omega$. 
Then \[X \in V \big[\prod_{m < \omega} G^{\eta_m}_{i_m} \, \times \, \prod_{m < \omega} G^{\ol{\eta}_m}_{\ol{i}_m}\big].\]\end{lem}

\begin{proof} Let \[X^\prime := \Big\{\, \beta < \alpha\ \ \big| \ \ \exists\, p = (p_\ast, (p^\sigma_i, a^\sigma_i)_{\sigma, i}, (p^\sigma)_\sigma) \; : \ p \Vdash_s \beta \in \dot{X}\; , \; \forall m: \ (\eta_m, i_m) \in \supp p_0\; , \] \[\forall m\, : \ a^{\eta_m}_{i_m} = g^{\eta_m}_{i_m}\; ,\; (p^{\eta_m}_{i_m})_{m < \omega} \in \prod_{m < \omega} G^{\eta_m}_{i_m}\; , \; ( p^{\ol{\eta}_i}_{\ol{i}_m})_{m < \omega} \in \prod_{m < \omega} G^{\ol{\eta}_m}_{\ol{i}_m}\, \Big\}.\] Then \[X^\prime \in V \big[\prod_{m < \omega} G^{\eta_m}_{i_m} \, \times \, \prod_{m < \omega} G^{\ol{\eta}_m}_{\ol{i}_m}\big], \] since the sequence $(g^{\eta_m}_{i_m})_{m < \omega}$ is contained in $V$. It remains to show that $X = X^\prime$. The inclusion $X \subseteq X^\prime$ follows from the forcing theorem. Concerning \tbl $\supseteq$\tbr\,, assume towards a contradiction there was $\beta \in X^\prime \setminus X$. Take $p$ as above with  $(\eta_m, i_m) \in \supp p_0$ for all $m < \omega$, and \[p \Vdash_s \beta \in \dot{X}\; , \; \forall m\; : \ a^{\eta_m}_{i_m} = g^{\eta_m}_{i_m}\; , \; (p^{\eta_m}_{i_m})_{m < \omega} \in \prod_{m < \omega} G^{\eta_m}_{i_m}\; , \; ( p^{\ol{\eta}_i}_{\ol{i}_m})_{m < \omega} \in \prod_{m < \omega} G^{\ol{\eta}_m}_{\ol{i}_m}. \] Since $\beta \notin X$, we can take $p^\prime \in G$, $p^\prime = \big(p^\prime_\ast, \big( (p^\prime)^\sigma_i, (a^\prime)^\sigma_i\big)_{\sigma, i}, \big((p^\prime)^\sigma\big)_\sigma\big)$ with $p^\prime \Vdash_s \beta \notin \dot{X}$, such that $(\eta_m, i_m) \in \supp p^\prime_0$ for all $m < \omega$. \\[-2mm]

%\colorbox{yellow}{FRAGE: Die Notation $p$ und $\ol{p}$, $\eta_m$ und $\ol{\eta}_m$ ist SCHLECHT!}
%\colorbox{yellow}{Könnte man das EVTL so lassen?}

First, we want to extend $p$ and $p^\prime$ and obtain conditions $\ol{p} \leq p$, $\ol{p}^\prime \leq p^\prime$, $\ol{p} = (\ol{p}_\ast, (\ol{p}^\sigma_i, \ol{a}^\sigma_i)_{\sigma, i}, (\ol{p}^\sigma)_\sigma)$, $\ol{p}^\prime = \big(\ol{p}^\prime_\ast, \big((\ol{p}^\prime)^\sigma_i, (\ol{a}^\prime)^\sigma_i\big)_{\sigma, i}, \big((\ol{p}^\prime)^\sigma\big)_\sigma\big)$ such that the following holds: \begin{itemize} \item $\forall\, m < \omega\ \ \ol{p}^{\eta_m}_{i_m} = (\ol{p}^\prime)^{\eta_m}_{i_m}$\; , \; $\ol{a}^{\eta_m}_{i_m} = (\ol{a}^\prime)^{\eta_m}_{i_m}$ \item $\forall\, m < \omega\ \ \ol{p}^{\ol{\eta}_m}_{\ol{i}_m} = (\ol{p}^\prime)^{\ol{\eta}_m}_{\ol{i}_m}$ 

%\colorbox{red}{ACHTUNG - $\sigma$ wurde hier als Index, aber auch als Isomorphismus vorher $\sigma \in A$ verwendet!}

\item $\dom \ol{p}_0 = \dom \ol{p}^\prime_0$ \item $\supp \ol{p}_0 = \supp \ol{p}^\prime_0$ \item $\bigcup_{(\sigma, i) \in \supp \ol{p}_0} \ol{a}^\sigma_i = \bigcup_{ (\sigma, i) \in \supp \ol{p}^\prime_0} (\ol{a}^\prime)^\sigma_i$ \item $\forall\, (\nu, j)\ \; :\ \dom \ol{p}_0\, \cap \, [\kappa_{\nu, j}, \kappa_{\nu, j + 1}) \neq \emptyset\ \rightarrow\ \big( \bigcup_{ \sigma, i} \, \ol{a}^\sigma_i\, \cap \, [\kappa_{\nu, j}, \kappa_{\nu, j + 1}) \big) \subseteq \dom \ol{p}_0$\ , \\ $\forall\, (\nu, j)\ \; :\ \dom \ol{p}^\prime_0\, \cap \, [\kappa_{\nu, j}, \kappa_{\nu, j + 1}) \neq \emptyset\ \rightarrow\ \big( \bigcup_{\sigma, i}\, (\ol{a}^\prime)^\sigma_i\, \cap \, [\kappa_{\nu, j}, \kappa_{\nu, j + 1}) \big) \subseteq \dom \ol{p}^\prime_0$ \item $\supp \ol{p}_1 = \supp \ol{p}^\prime_1$ 
%and $\dom \ol{p}_1 (\eta) = \dom \ol{p}^\prime_1 (\eta)$ for all $\eta \in \supp \ol{p}_1 = \supp \ol{p}^\prime_1$ 
\item $\forall\, \sigma \in \supp \ol{p}_1 = \supp \ol{p}^\prime_1\ : \ \dom \ol{p}_1 (\sigma) = \dom \ol{p}^\prime_1 (\sigma)$. 
\end{itemize}

We will now describe how $\ol{p}_0$ and $\ol{p}^\prime_0$ can be constructed. First, we need a set $\supp_0 := \supp \ol{p}_0 = \supp \ol{p}^\prime_0$. Consider \[s := \sup \{\kappa_\sigma\ | \  \sigma \in \Lim\; , \; \exists\, i < \alpha_\sigma\: : \: (\sigma, i) \in \supp p_0\, \cup \supp p^\prime_0 \}. \] Then by closure of the sequence $(\kappa_\sigma\ | \ 0 < \sigma < \gamma)$, it follows that $s = \kappa_{\ol{\gamma}}$ for some $\ol{\gamma} \leq \gamma$. If $\ol{\gamma} = \gamma$, then $cf \, \kappa_\gamma = \omega$ and we can take  $((\sigma_k, l_k)\ | \ k < \omega)$ with $\sigma_k \in \Lim$, $l_k < \alpha_{\sigma_k}$ for all $k < \omega$ such that $(\kappa_{\sigma_k}\ | \ k < \omega)$ is cofinal in $\kappa_{\gamma}$, and $(\sigma_k, l_k) \notin \supp p_0\, \cup \supp p^\prime_0$ for all $k < \omega$. Let \[\supp_0 := \supp \ol{p}_0 := \supp \ol{p}^\prime_0 := \supp p_0\, \cup \, \supp p^\prime_0\, \cup \, \{ (\sigma_k, l_k)\ | \ k < \omega\}.\]

If $\ol{\gamma} < \gamma$, we can set $\sigma_k := \ol{\gamma} \in \Lim$ for all $k < \omega$ and take $(l_k\ | \ k < \omega)$ such that $l_k < \alpha_{\sigma_k}$ with $(\sigma_k, l_k) = (\ol{\gamma}, l_k) \notin \supp p_0\, \cup\, \supp p^\prime_0$ for all $k < \omega$. Let \[\supp_0 := \supp \ol{p}_0 := \supp \ol{p}^\prime_0 := \supp p_0\, \cup \, \supp p^\prime_0\, \cup \, \{ (\sigma_k, l_k)\ | \ k < \omega\}\] as before.

%This defines $\supp_0 := \supp \ol{p}_0 = \supp \ol{p}^\prime_0$. \\[-2mm]

The next step is to define the linking ordinals. Take a set $X \subseteq \kappa_{\ol{\gamma}}$ such that for all intervals $[\kappa_{\nu, j}, \kappa_{\nu, j + 1}) \subseteq \kappa_{\ol{\gamma}}$, it follows that $|X\, \cap \, [\kappa_{\nu, j}, \kappa_{\nu, j + 1})| = \aleph_0$; and $X\, \cap\, \big( \bigcup_{(\sigma, i) \in \supp p_0} a^\sigma_i\, \cup \, \bigcup_{ (\sigma, i) \in \supp p^\prime_0} (a^\prime)^\sigma_i \big) = \emptyset$. Let \[\ol{X} := X\, \cup \, \bigcup_{\sigma, i} a^\sigma_i\, \cup \, \bigcup_{\sigma, i} (a^\prime)^\sigma_i.\] Our aim is to construct $\ol{p}$ and $\ol{p}^\prime$ such that $\bigcup_{\sigma, i} \, \ol{a}^\sigma_i = \bigcup_{ \sigma, i} \, (\ol{a}^\prime)^\sigma_i = \ol{X}$. \\[-3mm]

Consider an interval $[\kappa_{\nu, j}, \kappa_{\nu, j + 1}) \subseteq \kappa_{\ol{\gamma}}$. For every $(\sigma, i) \in \supp p_0$ with $\kappa_\sigma > \kappa_{\nu, j}$, we let $\ol{a}^\sigma_i\, \cap\, [\kappa_{\nu, j}, \kappa_{\nu, j + 1}) := a^\sigma_i\, \cap\, [\kappa_{\nu, j}, \kappa_{\nu, j + 1})$. \\[-3mm]

Define \[\{\xi_k (\nu, j)\ | \ k < \omega\} := \big( \ol{X} \, \cap \, [\kappa_{\nu, j}, \kappa_{\nu, j + 1}) \big) \setminus \bigcup_{ \sigma, i  } a^\sigma_i.\] This set has cardinality $\aleph_0$ by construction of $X$. \\[-3mm]

%\colorbox{yellow}{ACHTUNG - sollte man nicht eher $\aleph_0$ anstatt $\omega$ schreiben, wenn es um Kardinalitäten geht?}

Moreover, let \[\big \{\, \big(\ol{\sigma}_k (\nu, j), \ol{l}_k (\nu, j)\big)\ | \ k < \omega\, \big\} =: \{ (\sigma, i) \in \supp \ol{p}_0 \setminus \supp p_0 \ | \ \kappa_\sigma > \kappa_{\nu, j}\}. \] This set also has cardinality $\aleph_0$ by construction of $\supp \ol{p}_0 = \supp_0$. \\[-3mm]

Now, for any $k < \omega$, let \[\ol{a}\, ^{\ol{\sigma}_k (\nu, j)}_{\ol{l}_k (\nu, j)}\, \cap\, [\kappa_{\nu, j}, \kappa_{\nu, j + 1}) := \{\xi_k (\nu, j) \}.\] Together with same construction for $\ol{p}^\prime$, we obtain the linking ordinals $\ol{a}^\sigma_i$, $(\ol{a}^\prime)^\sigma_i$ for $(\sigma, i) \in \supp_0 = \supp \ol{p}_0 = \supp (\ol{p}^\prime)_0$ such that the \textit{independence property} holds, and $\bigcup_{\sigma, i} \ol{a}^\sigma_i = \bigcup_{\sigma, i} (\ol{a}^\prime)^\sigma_i = X$. \\[-2mm]

Next, we construct $\dom_0 := \dom \ol{p}_0 = \dom (\ol{p}^\prime)_0 := \bigcup_{\nu, j}  [\kappa_{\nu, j}, \delta_{\nu, j})$ as follows: Consider an interval $[\kappa_{\nu, j}, \kappa_{\nu, j + 1}) \subseteq \kappa_\gamma$. In the case that $\dom p_0\, \cap\, [\kappa_{\nu, j}, \kappa_{\nu, j + 1}) = \dom (p^\prime)_0 \, \cap \, [\kappa_{\nu, j}, \kappa_{\nu, j + 1}) = \emptyset$, let $\delta_{\nu, j} := \kappa_{\nu, j}$. Otherwise, take $\delta_{\nu, j} \in [\kappa_{\nu, j}, \kappa_{\nu, j + 1})$ such that $(\dom p_0\, \cup \, \dom p^\prime_0\, \cup \, \ol{X})\, \cap \, [\kappa_{\nu, j}, \kappa_{\nu, j + 1}) \subseteq [\kappa_{\nu, j}, \delta_{\nu, j})$. (This is possible since the set $\ol{X}\, \cap\, [\kappa_{\nu, j}, \kappa_{\nu, j + 1})$ is countable, and any $\kappa_{\nu, j + 1}$ is a successor cardinal.) \\ Let \[\dom_0 := \dom \ol{p}_0 := \dom \ol{p}^\prime_0 := \bigcup_{\nu, j} [\kappa_{\nu, j}, \delta_{\nu, j}).\] 

%\colorbox{yellow} {FRAGE: Braucht man die Bezeichnungen $\supp_0$ und $\dom_0$?} 

This set is bounded below all regular $\kappa_{\ol{\nu}, \ol{\j}}$ by construction, since $\dom p_0$ and $\dom p^\prime_0$ are bounded below all regular $\kappa_{\ol{\nu}, \ol{\j}}$. \\[-3mm]

Now, for $(\sigma, i) \in \supp \ol{p}_0$, let $\ol{p}^\sigma_i: \dom \ol{p}^\sigma_i \rightarrow 2$ on the corresponding domain with $\dom \ol{p}^\sigma_i = \dom \ol{p}_0 \cap\, \kappa_\sigma$, such that $\ol{p}^\sigma_i \supseteq p^\sigma_i $ for all $(\sigma, i) \in \supp p_0$, and in the case that $(\sigma, i) = (\eta_m, i_m)$ for some $m < \omega$, we additionally require that $\ol{p}^{\eta_m}_{i_m} \supseteq (p^\prime)^{\eta_m}_{i_m}$. This is possible, since $p^\prime \in G$ and $p^{\eta_m}_{i_m} \in G^{\eta_m}_{i_m}$, so $p^{\eta_m}_{i_m}$ and $(p^\prime)^{\eta_m}_{i_m}$ are compatible. \\[-3mm]
%let $\ol{p}^\eta_i \supseteq p^\eta_i$ with $\dom \ol{p}^\eta_i = \dom \ol{p}_0\, \cap\, \kappa_\eta$.

% \rightarrow 2$ with $\ol{p}^\eta_i (\zeta) = p^\eta_i (\zeta)$ for all $\zeta \in \dom p_0$. 
We define $\ol{p}_\ast$ on the according domain $\bigcup_{\nu, j} [\kappa_{\nu, j}, \delta_{\nu, j})^2$ such that $\ol{p}_\ast \supseteq p_\ast$, and the \textit{linking property} holds for $\ol{p}_0 \leq p_0$: Consider an interval $[\kappa_{\nu, j}, \kappa_{\nu, j + 1})$ with $\delta_{\nu, j} > \kappa_{\nu, j}$. For $\zeta \in (\dom \ol{p}_0 \setminus \dom p_0)\, \cap \, [\kappa_{\nu, j}, \kappa_{\nu, j + 1})$ and $\{\xi\} := a^\sigma_i\, \cap \, [\kappa_{\nu, j}, \kappa_{\nu, j + 1})$ for some $(\sigma , i) \in \supp p_0$, it follows by construction that $\xi \in \dom \ol{p}_0$. Let $\ol{p}_\ast (\xi, \zeta) := \ol{p}^\sigma_i (\zeta)$. For all $\xi$, $\zeta \in \dom p_0\, \cap \, [\kappa_{\nu, j}, \kappa_{\nu, j + 1})$, we set $\ol{p}_\ast (\xi, \zeta) := p_\ast (\xi, \zeta)$; and $\ol{p}_\ast (\xi, \zeta) \in \{0, 1\}$ arbitrary for the $\xi$, $\zeta \in \dom \ol{p}_0$ remaining. \\[-3mm]

Concerning $\ol{p}^\prime$, we set $(\ol{p}^\prime)^{\eta_m}_{i_m} = \ol{p}^{\eta_m}_{i_m}$ for all $m < \omega$. Then $(\ol{p}^\prime)^{\eta_m}_{i_m} \supseteq (p^\prime)^{\eta_m}_{i_m}$ by construction. For the $(\sigma, i) \in \supp (\ol{p}^\prime)_0 = \supp \ol{p}_0$ remaining, we can set $(\ol{p}^\prime)^\sigma_i$ arbitrarily on the given domain such that $(\ol{p}^\prime)^\sigma_i \supseteq (p^\prime)^\sigma_i$. 

Finally, we let $(\ol{p}^\prime)_\ast \supseteq (p^\prime)_\ast$ according to the \textit{linking property} for $\ol{p}_0^\prime \leq p_0^\prime$ (same construction as for $\ol{p}_\ast$). \\[-3mm]

It follows that $\ol{p}_0 \leq p_0$ and $\ol{p}^\prime_0 \leq p^\prime_0$, and $\ol{p}_0$ and $\ol{p}^\prime_0$ have all the required properties. \\[-3mm] 

%WIRKLICH? DAS MÜSSTE MAN NACHPRÜFEN!! \\

The construction of $\ol{p}_1 \leq p_1$ and $\ol{p}^\prime_1 \leq p^\prime_1$ is similar. \\[-2mm]

Our aim is to write down an isomorphism $\pi \in A$ with the following properties:

\begin{itemize} \item $\ol{p} \in D_\pi$ with $\pi \ol{p} = \ol{p}^\prime$, \item $\pi \in \bigcap_{m < \omega} Fix( \eta_m, i_m)\, \cap\, \bigcap_{m < \omega} Fix (\ol{\eta}_m, \ol{i}_m)\; \cap \; \bigcap_{m < \omega} H^{\lambda_m}_{k_m}\, \cap \, \bigcap_{m < \omega} H^{\ol{\lambda}_m}_{\ol{k}_m}$\\ (then $\pi \ol{X}^{D_\pi} = \ol{X}^{D_\pi}$ follows).\end{itemize}

From $\ol{p} \Vdash_s \beta \in \dot{X}$, we will then obtain $\pi \ol{p} \Vdash_s \beta \in  \pi\ol{X}^{D_\pi}$; hence, $\ol{p}^\prime \Vdash_s \beta \in \ol{X}^{D_\pi}$. This will be a contradiction towards $p^\prime \Vdash_s \beta \notin \dot{X}$. \\[-2mm]

We start with $\pi_0$. Let $\dom \pi_0 := \dom \ol{p}_0 = \dom \ol{p}^\prime_0$, and $\supp \pi_0 := \supp \ol{p}_0 = \supp \ol{p}^\prime_0$. \begin{itemize} \item Consider an interval $[\kappa_{\nu, j}, \kappa_{\nu, j + 1})$. 
%Recall that we have set $\supp \pi_0 (\nu, j) := \{ (\sigma, i) \in \supp \pi_0\ | \ \kappa_{\nu, j} < \kappa_\sigma$. 
We define $F_{\pi_0} (\nu, j): \supp \pi_0 (\nu, j) \rightarrow \supp \pi_0 (\nu, j)$ as follows: Let $F_{\pi_0} (\nu, j) (\sigma, i) := (\lambda, k)$ in the case that $(\ol{a}^\prime)^\sigma_i\, \cap \, [\kappa_{\nu, j}, \kappa_{\nu, j + 1}) = \ol{a}^\lambda_k\, \cap\, [\kappa_{\nu, j}, \kappa_{\nu, j + 1})$. This is well-defined by the \textit{independence property},
%(wurde die Verwendung der Independence Property sonst noch irgendwo erwähnt? Sollte man wohl!), 
and since we have arranged $\bigcup_{\sigma, i} \ol{a}^\sigma_i = \bigcup_{\sigma, i} (\ol{a}^\prime)^\sigma_i$. 

%\colorbox{yellow}{ACHTUNG - wurde sonst nicht eher $\bigcup_{\sigma, i}$ geschrieben?} 
\item For every interval $[\kappa_{\nu, j}, \kappa_{\nu, j + 1})$, let $G_{\pi_0} (\nu, j) (\sigma, i) = (\sigma, i)$ for all $(\sigma, i) \in \supp \pi_0 (\nu, j)$.

(These maps $G_{\pi_0} (\nu, j)$ will be the only parameters of $\pi_0$ which are \textit{not} determined by the requirement that $\pi_0 \ol{p}_0 = \ol{p}^\prime_0$. However, in order to make sure that $\pi \in \bigcap_{m < \omega} Fix(\eta_m, i_m)\, \cap\,\bigcap_{m < \omega} H^{\lambda_m}_{k_m}$, we firstly need $G_{\pi_0} (\nu, j) (\eta_m, i_m) = (\eta_m, i_m)$ for all $m < \omega$; and secondly, whenever $m < \omega$ and $i \leq k_m$, we need that $G_{\pi_0} (\nu, j) (\lambda_m, i) = (\lambda_m, i)$ for all $\kappa_{\nu, j}$ above a certain $\kappa_{\ol{\nu}, \ol{\j}}$.)

\item For $\zeta \in [\kappa_{\nu, j}, \kappa_{\nu, j + 1})\, \cap \, \dom \pi_0$, we define $\pi_0 (\zeta): 2^{\supp \pi_0 (\nu, j)} \rightarrow 2^{\supp \pi_0 (\nu, j)}$ as follows: For $\big(\epsilon_{(\sigma, i)}\ | \ (\sigma, i) \in \supp \pi_0 (\nu, j)\big) \in 2^{\supp \pi_0 (\nu, j)}$ given, let $\pi_0 (\zeta) \big(\epsilon_{(\sigma, i)}\ | $ $\ (\sigma, i) \in \supp \pi(\nu, j)\big) := \big( \wt{\epsilon}_{(\sigma, i)}\ | \ (\sigma, i) \in \supp \pi_0 (\nu, j)\big)$ such that $\wt{\epsilon}_{(\sigma, i)} = \epsilon_{(\sigma, i)}$ whenever $\ol{p}^\sigma_i (\zeta) = (\ol{p}^\prime)^\sigma_i (\zeta)$, and $\wt{\epsilon}_{(\sigma, i)} \neq \epsilon_{(\sigma, i)}$ in the case that $\ol{p}^\sigma_i (\zeta) \neq (\ol{p}^\prime)^\sigma_i (\zeta)$. 

\item Let now $\zeta \in \dom \pi_0\, \cap\, [\kappa_{\nu, j}, \kappa_{\nu, j + 1})$, and $\big(\xi^\sigma_i (\nu, j)\ | \ (\sigma, i) \in \supp \pi_0 (\nu, j)\big) \in \dom \pi_0 (\nu, j)^{\,\supp \pi_0 (\nu, j)}$. The map $\pi_\ast (\zeta) \big(\xi^\sigma_i (\nu, j)\ | \ (\sigma, i) \in \supp \pi_0 (\nu, j)\big): 2^{\supp \pi_0 (\nu, j)} \rightarrow 2^{\supp \pi_0 (\nu, j)}$ is defined as follows: A sequence $\big(\epsilon_{(\sigma, i)}\ | \ (\sigma, i) \in \supp \pi_0 (\nu, j)\big)$ is mapped to $\big(\wt{\epsilon}_{(\sigma, i)}\ | \ (\sigma, i) \in \supp \pi_0 (\nu, j)\big)$ with $\wt{\epsilon}_{(\sigma, i)} = \epsilon_{(\sigma, i)}$ if $\ol{p}_\ast (\xi^\sigma_i (\nu, j), \zeta) = \ol{p}^\prime_\ast (\xi^\sigma_i (\nu, j), \zeta)$, and $\wt{\epsilon}_{(\sigma, i)} \neq \epsilon_{(\sigma, i)}$ in the case that $\ol{p}_\ast (\xi^\sigma_i (\nu, j), \zeta)$ $\neq$ $\ol{p}^\prime_\ast (\xi^\sigma_i (\nu, j), \zeta)$. 

\item For $(\xi, \zeta) \in [\kappa_{\nu, j}, \kappa_{\nu, j + 1})^2$, the map $\pi_\ast (\xi, \zeta): 2 \rightarrow 2$ is defined as follows: We let $\pi_\ast (\xi, \zeta) = id$ in the case that $(\xi, \zeta) \notin (\dom \pi_0 (\nu, j))^2$. If $\xi, \zeta \in \dom \pi_0 (\nu, j)$, let $\pi_\ast (\xi, \zeta) = id$ if $\ol{p}_\ast (\xi, \zeta) = \ol{p}^\prime_\ast (\xi, \zeta)$, and $\pi_\ast (\xi, \zeta) \neq id$ in the case that $\ol{p}_\ast (\xi, \zeta) \neq \ol{p}^\prime_\ast (\xi, \zeta)$.
\end{itemize}

This defines $\pi_0$. Directly by construction, it follows that $\pi_0 \ol{p}_0 = \ol{p}^\prime_0$: Let $\pi_0 \ol{p}_0 =: ( (\pi \ol{p})_\ast, ( (\pi \ol{p})^\sigma_i, (\pi \ol{a})^\sigma_i)_{\sigma, i}, (\pi \ol{p}^\sigma)_\sigma)$. Then for any $(\sigma, i) \in \supp (\pi_0 \ol{p}_0) = \supp \ol{p}_0$ and $\kappa_{\nu, j} < \kappa_\sigma$, we have $(\pi \ol{a})^\sigma_i \, \cap\, [\kappa_{\nu, j}, \kappa_{\nu, j + 1}) = \ol{a}^\lambda_k\, \cap \, [\kappa_{\nu, j}, \kappa_{\nu, j + 1})$, where $(\lambda, k) = F_{\pi_0} (\nu, j) (\sigma, i)$; hence, $\ol{a}^\lambda_k\, \cap \, [\kappa_{\nu, j}, \kappa_{\nu, j + 1}) = (\ol{a}^\prime)^\sigma_i\, \cap \, [\kappa_{\nu, j}, \kappa_{\nu, j + 1})$ as desired.
%. Hence, it follows that $(\pi \ol{a})^\sigma_i \, \cap\, [\kappa_{\nu, j}, \kappa_{\nu, j + 1}) = (\ol{a}^\prime)^\sigma_i\, \cap \, [\kappa_{\nu, j}, \kappa_{\nu, j + 1})$ as desired. \\[-3mm]

For any $\zeta \in \dom \ol{p}_0$, it follows by definition of $\pi_0 (\zeta)$ that $\big( (\pi \ol{p})^\sigma_i (\zeta)\ | \ (\sigma, i) \in \supp \pi_0 (\nu, j)\big) = \big( (\ol{p}^\prime)^\sigma_i (\zeta)\ | \ (\sigma, i) \in \supp \pi_0 (\nu, j)\big)$, and similarly, $(\pi \ol{p})_\ast (\xi, \zeta) = \ol{p}^\prime_\ast (\xi, \zeta)$ for all $(\xi, \zeta) \in \dom (\pi \ol{p})_\ast = \dom \ol{p}_\ast$. \\ Hence, $\pi_0 \ol{p}_0 = \ol{p}^\prime_0$. \\[-3mm]

It remains to verify that $\pi \in \bigcap_m Fix (\eta_m, i_m) \, \cap\, \bigcap_m H^{\lambda_m}_{k_m}$. Consider a condition $r \in D_{\pi_0}$ and let $r^\prime := \pi_0 r$. Take an interval $[\kappa_{\nu, j}, \kappa_{\nu, j + 1}) \subseteq \kappa_\gamma$. Then for any $m < \omega$ with $(\eta_m, i_m) \in \supp \pi_0 (\nu, j)$ and $\zeta \in \dom \pi_0 (\nu, j)$, it follows that $(r^\prime)^{\eta_m}_{i_m} (\zeta) = r^{\eta_m}_{i_m} (\zeta)$ by construction of the map $\pi_0 (\zeta)$, since we have arranged $\ol{p}^{\eta_m}_{i_m} (\zeta) = (\ol{p}^\prime)^{\eta_m}_{i_m} (\zeta)$. In the case that $\zeta \in [\kappa_{\nu, j}, \kappa_{\nu, j + 1})$ with $\zeta \in \dom \ol{r}_0 \setminus \dom \pi_0$, it follows for $m < \omega$ that $(r^\prime)^{\eta_m}_{i_m} (\zeta) = r^\lambda_k (\zeta)$, where $(\lambda, k) = G_{\pi_0} (\nu, j) (\eta_m, i_m) = (\eta_m, i_m)$ as desired. Hence, $(r^\prime)^{\eta_m}_{i_m} = r^{\eta_m}_{i_m}$ for all $m < \omega$. Since $r \in D_{\pi_0}$ was arbitrary, it follows that $\pi_0 \in \bigcap_{m < \omega} Fix (\eta_m, i_m)$. \\[-3mm]

Similarly, $\pi_0 \in \bigcap_m H^{\lambda_m}_{k_m}$ follows from the fact that $G_{\pi_0} (\nu, j) = id$ for all intervals $[\kappa_{\nu, j}, \kappa_{\nu, j + 1}) \subseteq \kappa_\gamma$. \\[-2mm]

Now, we turn to the map $\pi_1$. \\[-3mm]

Let $\supp \pi_1 := \supp \ol{p}_1 = \supp \ol{p}^\prime_1$, and $\dom \pi_1 (\sigma) := \dom \ol{p}_1 (\sigma) = \dom \ol{p}_1^\prime (\sigma)$ for $\sigma \in \supp \pi_1$. We set $\supp \pi_1 (\sigma) := \emptyset$ for all $\sigma \in \supp \pi_1$. Then we only have to define maps $\pi_1 (\sigma) (i, \zeta): 2 \rightarrow 2$ for $\sigma \in \supp \pi$, $(i, \zeta) \in \dom \pi_1 (\sigma)$: Let $\pi_1 (\sigma) (i, \zeta) = id$ if $\ol{p} (\sigma) (i, \zeta) = \ol{p}^\prime (\sigma) (i, \zeta)$, and $\pi_1 (\sigma) \neq id$ in the case that $\ol{p} (\sigma) (i, \zeta) \neq \ol{p}^\prime (\sigma) (i, \zeta)$. \\[-3mm]

Clearly, $\pi_1 \ol{p}_1 = \ol{p}^\prime_1$. Moreover, $\pi \in \bigcap_m Fix (\ol{\eta}_m, \ol{i}_m)$: Let $m < \omega$ and $r \in D_{\pi_1}$ with $\ol{\eta}_m \in \supp r$ and $\ol{i}_m \in \dom_x r (\ol{\eta}_m)$. In the case that $\ol{\eta}_m \in \supp \pi_1$, it follows for any $\zeta \in \dom_y r (\ol{\eta}_m)$ that $(\pi r) (\ol{\eta}_m) (\ol{i}_m, \zeta) = \pi_1 (\ol{\eta}_m) (i_m, \zeta) \big(r (\ol{\eta}_m) (\ol{i}_m, \zeta)\big) = r (\ol{\eta}_m) (\ol{i}_m, \zeta)$ by construction of $\pi_1$, since we have arranged that $\ol{p}^\prime (\ol{\eta}_m) (\ol{i}_m, \zeta) = \ol{p} (\ol{\eta}_m) (\ol{i}_m, \zeta)$ whenever $(\ol{i}_m, \zeta) \in \dom  \ol{p} (\ol{\eta}_m) = \dom \ol{p}^\prime (\ol{\eta}_m) = \dom \pi_1 (\ol{\eta}_m)$. If $\ol{\eta}_m \notin \supp \pi_1$, then $(\pi r) (\ol{\eta}_m) = r (\ol{\eta}_m)$ by construction.

Finally, $\pi \in \bigcap_m H^{\ol{\lambda}_m}_{\ol{k}_m}$ follows from the fact that $\supp \pi_1 (\lambda) = 0$ for all $\lambda \in \supp\pi_1$. \\[-2mm]

Hence, the map $\pi$ has all the desired properties. \\[-2mm]

This finishes the proof of $X = X^\prime$, and \[X = X^\prime \in V \big[\prod_{m < \omega} G^{\eta_m}_{i_m}\, \times \, \prod_{m < \omega} G^{\ol{\eta}_m}_{\ol{i}_m}\big]\] follows.

\end{proof}

It is not difficult to see that with the exception of the maps $G_{\pi_0} (\nu, j)$, all the parameters describing $\pi$ are given by the requirement that $\pi \ol{p} = \ol{p}^\prime$. We call an isomorphism $\pi \in A$ of this form a \textit{standard isomorphism for $\pi \ol{p} = \ol{p}^\prime$}. \\[-2mm]

With the same proofs as for Lemma \ref{prescard1} and \ref{prescof}, one can show: 

\begin{lem} \label{prescard2anfang} Let $( (\sigma_m, i_m)\ | \ m < \omega)$, $((\ol{\sigma}_m, \ol{i}_m)\ | \ m < \omega)$ with $\sigma_m \in Lim$, $i_m < \alpha_{\sigma_m}$, and $\ol{\sigma}_m \in Succ$, $\ol{i}_m < \alpha_{\ol{\sigma}_m}$ for all $m < \omega$. Then $\prod_{m < \omega} P^{\sigma_m}\, \times\, \prod_{m < \omega} P^{\ol{\sigma}_m}$ preserves cardinals, cofinalities and the $GCH$. \end{lem}

Hence, the $Approximation$ $Lemma$ \ref{approx} implies:

\begin{cor} Cardinals and cofinalities are $V$-$N$-absolute. \end{cor}

We will now take a closer look at the intermediate generic extensions introduced in the \textit{Approximation Lemma} \ref{approx}. Firstly, we replace the generic filters $G^{\sigma_m}_{i_m}$ by $G_\ast (g^{\sigma_m}_{i_m})$, and secondly, we factor at $\kappa_\eta$ (or $\kappa_{\eta + 1}$).

%We will also use the \textit{Approximation Lemma} to find intermediate generic extensions where the $\kappa_\eta$-subsets located in $N$ can be captured.

%, which will be crucial in our proof of $\theta^N (\kappa_\eta) \leq \alpha_\eta$.

\begin{definition} For $0 < \eta < \gamma$, we say that $\big( (a_m)_{m < \omega}, (\ol{\sigma}_m, \ol{i}_m)_{m < \omega}\big)$ is an \textit{ \upshape $\eta$-good pair} if the following hold: \begin{itemize} \item $ (a_m\ | \ m < \omega)$ is a sequence of pairwise disjoint $\kappa_\eta$-subsets, such that for all $m < \omega$ and $\kappa_{\ol{\nu}, \ol{\j}} < \kappa_\eta$, it follows that $|a_m\, \cap \, [\kappa_{\ol{\nu}, \ol{\j}}, \kappa_{\ol{\nu}, \ol{\j} + 1})| = 1$, \item for all $m < \omega$, we have $\ol{\sigma}_m \in \Succ$ with $\ol{\sigma}_m \leq \eta$, $\ol{i}_m < \alpha_{\ol{\sigma}_m}$, \item if $m \neq m^\prime$, then  $(\ol{\sigma}_m, \ol{i}_m) \neq (\ol{\sigma}_{m^\prime}, \ol{i}_{m^\prime})$.\end{itemize}
\end{definition}

As in Lemma \ref{generic} and Lemma \ref{generic2}, it follows that for any $\eta$-good pair  $\big( (a_m)_{m < \omega}, (\ol{\sigma}_m, \ol{i}_m)_{m < \omega} \big)$, \[\prod_{m < \omega} G_\ast (a_m) \, \times\, \prod_{m < \omega} G^{\ol{\sigma}_m}_{\ol{i}_m}\] is a $V$-generic filter on $\prod_{m < \omega} (\ol{P}^\eta)^\omega\, \times\, \prod_{m < \omega} P^{\ol{\sigma}_m}$.

\begin{prop} \label{Xsubseteqkappaeta} Let $0 < \eta < \gamma$ and $X \in N$ with $X \subseteq \kappa_\eta$. If $\kappa_{\eta + 1} > \kappa_\eta^\plus$ (or $\kappa_\eta = \kappa_{\ol{\gamma}}$ with $\gamma = \ol{\gamma} + 1$),
% or $\kappa_\eta = \kappa_\gamma$, 
it follows that there is an $\eta$-good pair $\big( (a_m)_{m < \omega}, (\ol{\sigma}_m, \ol{i}_m)_{m < \omega} \big)$ with \[X \in V \big[\prod_{m < \omega} G_\ast (a_m)\, \times\, \prod_{m < \omega} G^{\ol{\sigma}_m}_{\ol{i}_m}\big]. \]

\end{prop}

\begin{proof} By the \textit{Approximation Lemma} \ref{approx}, there are sequences $( (\sigma_m, i_m) \ | \ m < \omega)$, $( (\ol{\sigma}_m, \ol{i}_m)\ | \ m < \omega)$ of pairwise distinct pairs with $\sigma_m \in Lim$, $i_m < \alpha_{\sigma_m}$; $\ol{\sigma}_m \in Succ$, $\ol{i}_m < \alpha_{\ol{\sigma}_m}$ for all $m < \omega$, such that 

\[X \in V \big[ \prod_{m < \omega} G^{\sigma_m}_{i_m}\, \times \, \prod_{m < \omega} G^{\ol{\sigma}_m}_{\ol{i}_m}\big].  \]

%Consider $( (\sigma_m, i_m)\ | \ m < \omega)$ with $\sigma_m \in Lim$, $i_m < \alpha_{\sigma_m}$ for all $m < \omega$. 
The sequence of linking ordinals $(g^{\sigma_m}_{i_m}\ | \ m < \omega)$ is contained in $V$, and by the \textit{linking property}, it follows that $V[\prod_{m < \omega} G^{\sigma_m}_{i_m}] = V[\prod_{m < \omega} G_\ast (g^{\sigma_m}_{i_m})]$.
% If $p$ is a condition in $G$ with $(\sigma_m, i_m) \in \supp p_0$ for all $m < \omega$, then $G^{\sigma_m}_{i_m} (\zeta) = G_\ast (\xi, \zeta)$ whenever $\zeta \notin \dom p_0$, and $\zeta$ is contained in an interval $[\kappa_{\nu, j}, \kappa_{\nu, j + 1})$ with $g^{\sigma_m}_{i_m}\, \cap\, [\kappa_{\nu, j}, \kappa_{\nu, j + 1}) = \{\xi\}$.

%\colorbox{red}{ACHTUNG - sollte man das ausführen?}

Hence,
 \[X \in V \big[ \prod_{m < \omega} G_\ast (g^{\sigma_m}_{i_m})\, \times\, \prod_{m < \omega} G^{\ol{\sigma}_m}_{\ol{i}_m} \big].\]

The forcing $\prod_{m < \omega} P^{\sigma_m}\, \times \, \prod_{m < \omega} P^{\ol{\sigma}_m}$ can be factored as \[ \big(\prod_{m < \omega} P^{\sigma_m} \uhr \kappa_\eta\; \times\; \prod_{\ol{\sigma}_m \leq \eta} P^{\ol{\sigma}_m} \big)\; \times\; \big( \prod_{m < \omega} P^{\sigma_m} \uhr [\kappa_\eta, \kappa_{\sigma_m})\, \times\, \prod_{\ol{\sigma}_m > \eta} P^{\ol{\sigma}_m} \big),\]

%\colorbox{yellow}{ACHTUNG - bräuchte man nicht $P^{\sigma_m}_{i_m}$? Ok?}

where the \tbl lower part\tbr\,has cardinality $\leq \kappa_\eta^\plus$ by the $GCH$ in $V$, and the \tbl upper part\tbr\, is $\leq \kappa_\eta^{\plus}$-closed: If $\kappa_{\eta + 1}$ is a limit cardinal, this follows from the fact that $\kappa_{\eta, j + 1} \geq \kappa_{\eta, j}^{\plus \plus}$ for all $j < \cf \kappa_{\eta + 1}$ by construction (in particular, $\kappa_{\eta, 1} \geq \kappa_\eta^{\plus \plus}$); and if $\kappa_{\eta + 1}$ is a successor cardinal, we use our assumption that $\kappa_{\eta + 1} > \kappa_\eta^\plus$. Hence, \[X \in V \big[ \prod_{m < \omega} G_\ast (g^{\sigma_m}_{i_m}\, \cap \, \kappa_\eta)\, \times\, \prod_{\ol{\sigma}_m \leq \eta} G^{\ol{\sigma}_m}_{\ol{i}_m} \big].\] Setting $a_m := g^{\sigma_m}_{i_m}\, \cap \, \kappa_\eta$ for $m < \omega$, it follows by the \textit{independence property} that $ \big( (a_m)_{m < \omega}, (\ol{\sigma}_m, \ol{i}_m)_{m < \omega\, , \, \ol{\sigma}_m \leq \eta} \big)$ is an $\eta$-good pair with \[X \in V[ \prod_{m < \omega} G_\ast (a_m) \, \times\, \prod_{\ol{\sigma}_m \leq \eta} G^{\ol{\sigma}_m}_{\ol{i}_m} \big].\]
\end{proof}

In the case that $\kappa_{\eta + 1} = \kappa_\eta^\plus$, we use our notion of an \textit{$\eta$-almost good pair}, which is defined like an $\eta$-good pair, with the exception that for an \textit{$\eta$-almost good pair} $\big( (a_m)_{m < \omega}, (\ol{\sigma}_m, \ol{i}_m)_{m < \omega} \big)$, we have $a_m \subseteq \kappa_{\eta + 1}$ for all $m < \omega$.

\begin{definition} \label{etaalmostgoodpair} For $0 < \eta < \gamma$ with $\kappa_{\eta + 1} = \kappa_\eta^\plus$, we say that $\big( (a_m)_{m < \omega}, (\ol{\sigma}_m, \ol{i}_m)_{m < \omega} \big)$ is an \textit{ \upshape $\eta$-almost good pair} if the following hold: \begin{itemize} \item $ (a_m\ | \ m < \omega)$ is a sequence of pairwise disjoint $\kappa_{\eta + 1}$-subsets, such that for all $m < \omega$ and $\kappa_{\ol{\nu}, \ol{\j}} < \kappa_{\eta + 1}$, it follows that $|a_m\, \cap \, [\kappa_{\ol{\nu}, \ol{\j}}, \kappa_{\ol{\nu}, \ol{\j} + 1})| = 1$, \item for all $m$, we have $\ol{\sigma}_m \in \Succ$ with $\ol{\sigma}_m \leq \eta$, and $\ol{i}_m < \alpha_{\ol{\sigma}_m}$, \item if $m \neq m^\prime$, then  $(\ol{\sigma}_m, \ol{i}_m) \neq (\ol{\sigma}_{m^\prime}, \ol{i}_{m^\prime})$.\end{itemize}
\end{definition}

The counterpart of Proposition \ref{Xsubseteqkappaeta} states: 

\begin{prop} \label{Xsubseteqkappaeta2} Let $0 < \eta < \gamma$ and $X \in N$ with $X \subseteq \kappa_\eta$. In the case that $\kappa_{\eta + 1} = \kappa_\eta^\plus$, there is an $\eta$-almost good pair $\big( (a_m)_{m < \omega}, (\ol{\sigma}_m, \ol{i}_m)_{m < \omega} \big)$ with \[X \in V \big[\prod_{m < \omega} G_\ast (a_m)\, \times\, \prod_{m < \omega} G^{\ol{\sigma}_m}_{\ol{i}_m}\, \times\, G^{\eta + 1} \big]. \]

\end{prop}

\begin{proof} We follow the proof of Proposition \ref{Xsubseteqkappaeta} with a slightly different factorization: Let  \[X \in V \big[ \prod_{m < \omega} G_\ast (g^{\sigma_m}_{i_m})\, \times\, \prod_{m < \omega} G^{\ol{\sigma}_m}_{\ol{i}_m} \big]\] as before with $\sigma_m \in Lim$, $i_m < \alpha_{\sigma_m}$; $\ol{\sigma}_m \in Succ$, $\ol{i}_m < \alpha_{\ol{\sigma}_m}$ for all $m < \omega$. The forcing $\prod_{m < \omega} P^{\sigma_m}_{i_m}\, \times\, \prod_{m < \omega} P^{\ol{\sigma}_m}_{\ol{i}_m}$ can be factored as \[ \big(\prod_{m < \omega} P^{\sigma_m} \uhr \kappa_{\eta + 1}\; \times\; \prod_{\ol{\sigma}_m \leq \eta + 1} P^{\ol{\sigma}_m} \big)\; \times\; \big( \prod_{m < \omega} P^{\sigma_m} \uhr [\kappa_{\eta + 1}, \kappa_{\sigma_m})\, \times\, \prod_{\ol{\sigma}_m > \eta + 1} P^{\ol{\sigma}_m} \big),\] 
%\colorbox{yellow}{ACHTUNG - bräuchte man nicht $P^{\sigma_m}_{i_m}$?}

where the \tbl lower part\tbr\,has cardinality $\leq \kappa_{\eta + 1}$ by the $GCH$ in $V$ (since $\kappa_{\eta + 1} = \kappa_\eta^\plus$), and the \tbl upper part\tbr\,is $\leq \kappa_{\eta + 1}$-closed. Hence, \[X \in V\big[ \prod_{m < \omega} G_\ast (g^{\sigma_m}_{i_m} \, \cap \, \kappa_{\eta + 1})\, \times\, \prod_{\ol{\sigma}_m \leq \eta + 1} G^{\ol{\sigma}_m}_{\ol{i}_m}\big] \subseteq V\big[ \prod_{m < \omega} G_\ast (g^{\sigma_m}_{i_m} \, \cap \, \kappa_{\eta + 1})\, \times\, \prod_{\ol{\sigma}_m \leq \eta} G^{\ol{\sigma}_m}_{\ol{i}_m}\, \times\, G^{\eta + 1}\big].\] With $a_m := g^{\sigma_m}_{i_m}\, \cap \, \kappa_{\eta + 1}$ for $m < \omega$, it follows that $ \big( (a_m)_{m < \omega}, (\ol{\sigma}_m, \ol{i}_m)_{m < \omega\, , \, \ol{\sigma}_m \leq \eta} \big)$ is an $\eta$-almost good pair with \[X \in V[ \prod_{m < \omega} G_\ast (a_m) \, \times\, \prod_{\ol{\sigma}_m \leq \eta} G^{\ol{\sigma}_m}_{\ol{i}_m} \, \times\, G^{\eta + 1} \big]\] as desired.

\end{proof}

\section{\bfseries $ \mathbf{ \boldsymbol{\forall \eta}\ \;\boldsymbol{\theta^N} \boldsymbol{(\kappa_\eta)  = \alpha_\eta}}$.} \label{chapter6} 

It remains to make sure that in our $ZF$-model $N$, the values $\theta^N (\kappa_\eta)$ are as desired. Firstly, in Chapter \ref{6.1}, \ref{6.2} and \ref{6.3}, we will show that $\theta^N (\kappa_\eta) = \alpha_\eta$ holds for all $0 < \eta < \gamma$. After that, in Chapter \ref{6.4} and \ref{6.5}, we will see that for any cardinal $ \lambda \in (\kappa_\eta, \kappa_{\eta + 1})$ in a \tbl gap\tbr, or $\lambda \geq \kappa_\gamma = \sup \{\kappa_\eta\ | \ 0 < \eta < \gamma\}$, the value $\theta^N (\lambda)$ is the smallest possible.

By our remarks from Chapter \ref{the theorem}, this justifies our assumption from the beginning that the sequence $(\alpha_\eta\ | \ 0 < \eta < \gamma)$ is strictly increasing.

%\colorbox{yellow}{TO DO EVTL: Nummern der Kapitel sind so NICHT RICHTIG! Nochmal durchgehen? }

\subsection{\bfseries $ \mathbf{ \boldsymbol{\forall \eta }\ \;\boldsymbol{\theta^N} \boldsymbol{(\kappa_\eta) \geq \alpha_\eta}}$.} \label{6.1}

Using the subgroups $H^\eta_k$, it is not difficult to see that 
%$\theta^N (\kappa_\eta) \geq \alpha_\eta$ holds for all $0 < \eta < \gamma$. In other words: F
for all $k < \alpha_\eta$, there exists in $N$ a surjection $s: \powerset (\kappa_\eta) \rightarrow k$.

\begin{prop} Let $0 < \eta < \gamma$. Then $\theta^N (\kappa_\eta) \geq \alpha_\eta$. \end{prop} 

\begin{proof}

%\colorbox{yellow}{ACHTUNG - ist überall $\eta = 0$ ausgeschlossen?}

Let $k < \alpha_\eta$. We construct in $N$ a surjection $s: \powerset(\kappa_\eta) \rightarrow k$. As already outlined in Chapter \ref{constructingf}, we define around each $G^\eta_i$ with $i < k$ a \tbl cloud\tbr\,as follows: \[\big(\wt{G^\eta_i}\big)^{(k)} \; := \; \Big( \, \big(\dot{\wt{G^\eta_i}}\big)^{(k)} \, \Big)^G,\] where

\[\big(\dot{\wt{G^\eta_i}}\big)^{(k)} := \Big \{ \, (\pi \ol{G^\eta_i}^{D_\pi}, \m{1} )\ | \ [\pi] \in H^\eta_k \, \Big\}; \]

and we take the following canonical name for the $i$-th generic $\kappa_\eta$-subset:

\[\dot{G}^\eta_i := \big\{\, (a, p)\ \big| \ p \in \m{P}\: , \: \exists\, \zeta < \kappa_\eta \ \exists\, \epsilon \in \{0, 1\}\: : \: a = \mbox{OR}_{\m{P}} (\check{\zeta}, \check{\epsilon})\, \wedge\, p^\eta_i (\zeta) = \epsilon\, \big\}. \]

%\zeta < \kappa_\eta\, , \,  \epsilon \in 2\ : \  p^\eta_i (\zeta) = \epsilon\big) \big\}\] 

%\colorbox{yellow}{TO DO: widecheck} \\[-3mm]

Roughly speaking, $\big(\wt{G^\eta_i}\big)^{(k)}$ is the orbit of $G^\eta_i$ under the $\ol{A}$-subgroup $H^\eta_k$; hence, its canonical name $\big(\dot{\wt{G^\eta_i}}\big)^{(k)}$ is fixed by all automorphisms in $H^\eta_k$. 

More precisely: \\[-3mm]

%\colorbox{yellow}{Kann man das so schreiben?} \\[-3mm]

Let $\sigma \in A$ with $[\sigma] \in H^\eta_k$. Then \[\ol{(\dot{\wt{G^\eta_i}})^{(k)}}^{D_\sigma} = \big \{\; \big( \; \ol{\pi \ol{G^\eta_i}^{D_\pi}}^{D_\sigma}, p\; \big)\ | \ [\pi] \in H^\eta_k\; , \; p \in D_\sigma  \;\big\}. \]

Moreover, for all $\pi$, \[\ol{\pi \ol{G^\eta_i}^{D_\pi}}^{D_\sigma} = \big \{\; \big(\;\ol{a}^{D_\sigma}\; , \; p \; \big)\ | \ p \in D_\sigma\; , \;  p \Vdash_s a \in \pi \ol{G^\eta_i}^{D_\pi}\; , \; \exists\, \zeta < \kappa_\eta \ \exists\, \epsilon \in \{0, 1\}\: : \: a = \mbox{OR}_{\m{P}} (\check{\zeta}, \check{\epsilon}) \;\big\},\]

since for any $a = \mbox{OR}_{\m{P}} (\check{\zeta}, \check{\epsilon})$ as above, it follows that $\ol{\pi \ol{a}^{D_\pi}}^{D_\sigma} = \ol{ \ol{a}^{D_\pi}}^{D_\sigma} = \ol{a}^{D_\sigma}$. \\[-3mm]

Now, it is not difficult to see that $p \in D_\sigma$ with $p \Vdash_s a \in \pi \ol{G^\eta_i}^{D_\pi}$ if and only if $p \in D_\sigma$ and for all $q \leq p$ with $q \in D_\pi\, \cap \, D_\sigma$ and $\zeta \in \dom q_0$, it follows that $(\pi^{-1} q )^\eta_i (\zeta) = \epsilon$. \\[-4mm]

Also, $\sigma \ol{a}^{D_\sigma} = \ol{a}^{D_\sigma}$ holds for all $\sigma$. \\[-2mm]

Hence, \[\sigma \, \ol{\pi \ol{G^\eta_i}^{D_\pi}}^{D_\sigma} = \big \{\; \big(\; \sigma \ol{a}^{D_\sigma}\; , \; \sigma p \; \big)\ | \ p \in D_\sigma\; , \; \exists\, \zeta < \kappa_\eta\; \exists\, \epsilon \in \{0, 1\} \ \;a = \mbox{OR}_{\m{P}} (\check{\zeta}, \check{\epsilon}) \; ,\] \[\forall\, q \in D_\pi\, \cap\, D_\sigma\ \big(\, (q \leq p \, \wedge\, \zeta \in \dom q_0) \Rightarrow (\pi^{-1} q)^\eta_i (\zeta) = \epsilon \, \big)\; \big\} \]

\[ = \big \{\; \big(\; \ol{a}^{D_\sigma}\; , \;  p \; \big)\ | \ p \in D_\sigma\; , \; \exists\, \zeta < \kappa_\eta\; \exists\, \epsilon \in \{0, 1\} \ \;a = \mbox{OR}_{\m{P}} (\check{\zeta}, \check{\epsilon}) \; ,\] \[\forall\, q \in D_\pi\, \cap\, D_\sigma\ \big(\, (q \leq p \, \wedge\, \zeta \in \dom q_0) \Rightarrow (\pi^{-1} \sigma^{-1} q)^\eta_i (\zeta) = \epsilon \, \big)\; \big\}.\]

Setting $\tau := \sigma \pi$, it follows that \[\sigma \ol{\pi \ol{G^\eta_i}^{D_\pi}}^{D_\sigma} = \ol{\tau \ol{G^\eta_i}^{D_\tau}}^{D_\sigma}.\]

Now, any element of \; $\sigma \, \ol{(\dot{\wt{G^\eta_i}})^{(k)}}^{D_\sigma}$ is of the form \[\big( \; \sigma \ol{\pi \ol{G^\eta_i}^{D_\pi}}^{D_\sigma}, \sigma p\; \big)\] with $[\pi] \in H^\eta_k$ and $p \in D_\sigma$. Since \[\big( \; \sigma \ol{\pi \ol{G^\eta_i}^{D_\pi}}^{D_\sigma}, \sigma p\; \big) = \big( \; \ol{\tau \ol{G^\eta_i}^{D_\tau}}^{D_\sigma}, \ol{p}\; \big),\] where $\tau := \sigma \pi$ and $\ol{p} := \sigma p$ satisfy $[\tau] \in H^\eta_k$ and $\ol{p} \in D_\sigma$, it follows that \[\big( \; \sigma \ol{\pi \ol{G^\eta_i}^{D_\pi}}^{D_\sigma}, \sigma p\; \big) \in \ol{(\dot{\wt{G^\eta_i}})^{(k)}}^{D_\sigma}.\] Hence, \[\sigma \,\ol{(\dot{\wt{G^\eta_i}})^{(k)}}^{D_\sigma} \subseteq \ol{(\dot{\wt{G^\eta_i}})^{(k)}}^{D_\sigma}.\] 

The inclusion \tbl $\supseteq$\tbr\,is similar. \\[-2mm]

%\[\ol{(\dot{\wt{G^\eta_i}})^{(k)}}^{D_\sigma} \subseteq \sigma \ol{(\dot{\wt{G^\eta_i}})^{(k)}}^{D_\sigma}\] is similar.

%= \big\{ \; \big(\; \pi \ol{\ol{ \dot{G}^\eta_i}^{D_\pi}}^{D_\sigma}, p \; \big) \ | \ \pi \in H^\eta_k\; , \; p \in D_\sigma\;\big\},\] and 

%\begin{eqnarray*} \sigma \; \ol{ \big(\dot{\wt{G^\eta_i}}\big)^{(k)}}^{D_\sigma} & = & \big \{\; \big( \; \sigma \; \ol{\pi \ol{\dot{G}^\eta_i}^{D_\pi}}^{D_\sigma}, \sigma p\; \big)\ | \ \pi \in H^\eta_i\; , \; p \in D_\sigma \big\} \\ & = & \big\{ \; \big( \; \sigma \pi \ol{\dot{G}^\eta_i}^{D_\pi \, \cap \, D_\sigma}, p \big)\ | \ \pi \in H^\eta_k \; ,\; p \in D_\sigma \big\} \\ & = &\big \{  \; \big(\sigma \pi \ol{\dot{G}^\eta_i}^{D_{\sigma \pi}}, p \; \big)\ | \ \pi \in H^\eta_k, p \in D_\sigma  \; \big\}. \end{eqnarray*} 

%Since $\pi \in H^\eta_k$ implies that also $\sigma \pi \in H^\eta_k$ with \[\ol{\dot{G}^\eta_i}^{D_{\sigma \pi}} = \ol{\ol{\dot{G}^\eta_i}^{D_\sigma \pi}}^{D_\sigma},\] it follows that \[\sigma \ol{\big(\dot{\wt{G^\eta_i}}\big)^{(k)}}^{D_\sigma} = \ol{\big(\dot{\wt{G^\eta_i}}\big)^{(k)}}^{D_\sigma}\] for all $\sigma \in H^\eta_k$.

Thus, \[ \Big( \; \big(\dot{\wt{G^\eta_i}}\big)^{(k)}\ \ \big| \ \ i < k \; \Big) := \Big \{\; \Big(\, \OR_{\m{P}} \Big( \; \check{i}, \big(\dot{\wt{G^\eta_i}}\big)^{(k)}\; \Big) \,, \, \m{1}\, \Big) \ \ \Big| \ \ i < k \; \Big\}, \]  is a name for the sequence $\Big(\big(\wt{G^\eta_i}\big)^{(k)}\ | \ i < k \Big)$ that is stabilized by all $\sigma$ with $[\sigma] \in H^\eta_k$. Hence, $\Big(\big(\wt{G^\eta_i}\big)^{(k)}\ | \ i < k \Big) \in N$. \\[-2mm]

Now, we can define in $N$ a surjection $s: \powerset(\kappa_\eta) \rightarrow k$ as follows: For $X \in N$, $X \subseteq \kappa_\eta$, let $s(X) := i$ in the case that $X \in \big(\wt{G^\eta_i}\big)^{(k)}$ if such $i$ exists, and $s(x) := 0$, else. \\[-3mm]

The surjectivity of $s$ is clear, since $G^\eta_i \in N$ for all $i < k $ with $s\big(G^\eta_i\big) = i$. It remains to show that $s$ is well-defined; i.e. for any $i, i^\prime < k$ with $i \neq i^\prime$, it follows that $\big(\wt{G^\eta_i}\big)^{(k)}\; \cap \; \big(\wt{G^\eta_{i^\prime}}\big)^{(k)} = \emptyset$. \\[-2mm]

First, let $\eta\in Lim$, and take $i$, $i^\prime < k$ with $i \neq i^\prime$. \\[-3mm]

The point is that the automorphisms in $H^\eta_k$ do not permute the vertical lines $P^\eta_i\, \uhr\, [\kappa_{\ol{\nu}, \ol{\j}}, \kappa_\eta)$ and $P^\eta_{i^\prime}\ \uhr\, [\kappa_{\ol{\nu}, \ol{\j}}, \kappa_\eta)$ above some $\kappa_{\ol{\nu}, \ol{\j}} < \kappa_\eta$. Thus, the orbits of $G^\eta_i$ and $G^\eta_{i^\prime}$ under $H^\eta_k$ must be disjoint: \\[-3mm]

Assume towards a contradiction there was $X \in \big( \wt{G^\eta_i}\big)^{(k)}\; \cap \; \big(\wt{G^\eta_{i^\prime}} \big)^{(k)}$. Then we have \[\big( \pi \ol{G^\eta_i}^{D_\pi} \big)^G = \big( \tau \ol{G^\eta_{i^\prime}}^{D_\tau} \big)^G  \] for some $\pi, \tau$ with $[\pi] \in H^\eta_k$ and $[\tau] \in H^\eta_k$. Hence, $(\pi^{-1} G)^\eta_i = (\tau^{-1} G)^\eta_{i^\prime}$. Take $\kappa_{\ol{\nu}, \ol{\j}} < \kappa_\eta$ such that for all $\kappa_{\nu, j} \in [\kappa_{\ol{\nu}, \ol{\j}}, \kappa_\eta)$ and $l < k$, it follows that $G_{\pi_0} (\nu, j) (\eta, l) = (\eta, l)$ whenever $(\eta, l) \in \supp \pi_0 (\nu, j)$, and $G_{\tau_0} (\nu, j) (\eta, l) = (\eta, l)$ whenever $(\eta, l) \in \supp \tau_0 (\nu, j)$.

By genericity, take $q \in G$ with $q \in D_\pi\, \cap\, D_\tau$ such that there is $\zeta \in \dom q \setminus (\dom \pi_0\, \cap\, \dom \tau_0)$, $\zeta \in [\kappa_{\ol{\nu}, \ol{\j}}, \kappa_\eta)$ with $q^\eta_i (\zeta) \neq q^\eta_{i^\prime} (\zeta)$.

% since the set \[\wt{D} := \{q \in \m{P}\ | \ q \in D_\pi\, \cap\, D_\tau\; ; \; (\eta, i), (\eta, j) \in \supp q_0\; ; \; \exists\, \zeta \in [\kappa_{\ol{\nu}, \ol{j}}, \kappa_\eta)\, \cap \, \]\[ \cap\, \big(\dom q \setminus (\dom \pi_0\, \cap\, \dom \tau_0)\big) \ q^\eta_i (\zeta) \neq q^\eta_j (\zeta) \}\] is dense in $\m{P}$. 
W.l.o.g., let $q^\eta_i (\zeta) = 1$, $q^\eta_{i^\prime} (\zeta) = 0$. With $r := \pi^{-1} q$, $r^\prime := \tau^{-1} q$, it follows by construction of the isomorphism that $r^\eta_i (\zeta) = q^\eta_i (\eta) = 1$ and $(r^\prime)^\eta_{i^\prime} (\zeta) = q^\eta_{i^\prime} (\zeta) = 0$, which would contradict $(\pi^{-1}G)^\eta_i = (\tau^{-1}G)^\eta_{i^\prime}$. \\[-3mm]

Hence, $s: \powerset(\kappa_\eta) \rightarrow k$ is a well-defined surjection in $N$. \\[-3mm]

The case $\eta \in Succ$ is similar.
\end{proof} 

\subsection{\bfseries $ \mathbf{ \boldsymbol{\forall} \boldsymbol{\eta } \ \boldsymbol {\big(\, \kappa_{\eta + 1} > \kappa_\eta^\plus\, \longrightarrow \,} \boldsymbol{\theta^N} \boldsymbol{(\kappa_\eta)}  \leq \boldsymbol{\alpha_\eta\, \big)} }$.} 
\label{6.2}

%It remains to show that $\theta^N (\kappa_\eta) \leq \alpha_\eta$ for all $\eta$. \\ 
Let $0 < \eta < \gamma$. Throughout this Chapter \ref{6.2}, we assume that \[ \boldsymbol{\kappa_{\eta + 1} > \kappa_\eta^\plus.}\] Then Proposition \ref{Xsubseteqkappaeta} can be applied. \\[-3mm]

In Chapter \ref{6.3}, we discuss the case that $\kappa_{\eta + 1} = \kappa_\eta^\plus$, where the proof can be structured the very same way; except that the intermediate generic extensions where the $\kappa_\eta$-subsets in $N$ are located are given by Proposition \ref{Xsubseteqkappaeta2}. Thus, we will have to take care of an extra factor $G^{\eta+1}$ in our products describing these intermediate generic extensions, which will lead to a couple of modifications. In Chapter \ref{6.3}, we take a brief look at each step in the proof presented here, and go through the major changes.\\[-1mm]

%Proposition \ref{Xsubseteqkappaeta2} describes the intermediate generic extensions where the $\kappa_\eta$-subsets in $N$ are located. \\[-2mm]

Assume towards a contradiction that there was a surjective function $f: \powerset (\kappa_\eta) \rightarrow \alpha_\eta$ in $N$. Let $f = \dot{f}^G$ with $\dot{f} \in HS$, such that $\pi \ol{f}^{D_\pi} = \ol{f}^{D_\pi}$ holds for all $\pi \in A$ with $[\pi]$ contained in the intersection \[\bigcap_{m < \omega} Fix (\eta_m, i_m)\, \cap\, \bigcap_{m < \omega} H^{\lambda_m}_{k_m} \hspace*{5,5cm} (A_{\dot{f}}).\] By Proposition \ref{Xsubseteqkappaeta}, it follows that any $X \in dom\,f$ is of the form \[X = \dot{X}^{\prod_{m < \omega} G_\ast (a_m)\, \times\, \prod_{m < \omega} G^{\ol{\sigma}_m}_{\ol{i}_m}}, \] where $ \big( (a_m)_{m < \omega}, (\ol{\sigma}_m, \ol{i}_m)_{m < \omega} \big)$ is an $\eta$-good pair. \\[-3mm]

Our proof will be structured as follows: We pick some $\boldsymbol{\beta < \alpha_\eta}$ \textit{large enough for the intersection $(A_{\dot{f}})$} (we give a definition of this term on the next page) and consider a map $f^\beta$, which will be obtained from $f$ by restricting its domain to those $X$ that are contained in a generic extension  \[V \big[\prod_{m < \omega} G_\ast (a_m)\, \times\, \prod_{m < \omega}G^{\ol{\sigma}_m}_{\ol{i}_m} \big]\] for an $\eta$-good pair $ \big( (a_m)_{m < \omega}, (\ol{\sigma}_m, \ol{i}_m)_{m < \omega} \big)$ such that $\boldsymbol{\ol{i}_m < \beta}$ \textbf{for all} $\boldsymbol{m < \omega}$. 

We wonder if this restricted function $f^\beta$ could still be surjective onto $\alpha_\eta$. \\[-2mm]

The main steps of our proof (similar as in \cite[Chapter 5]{arxiv}) can be outlined as follows: \\[-3mm]

First, we assume that also $f^\beta: \dom f^\beta \rightarrow \alpha_\eta$ was surjective onto $\alpha_\eta$.

\begin{itemize} \item[A)] We define a forcing notion $\m{P}^\beta\, \uhr \, (\eta + 1)$, which will be obtained from $\m{P}$ by essentially \tbl cutting off \hspace*{-0,3mm}\tbr\, at height $\eta + 1$ and width $\beta$. We show that there is a projection of forcing posets $\rho^\beta: \m{P} \rightarrow \m{P}^\beta\, \uhr\, (\eta + 1)$.
%\colorbox{yellow}{TO DO: \tbl complete projection\tbr\,nachschlagen und definieren!}
%colorbox{red}{FRAGE: Wo wurde eine $V$-generischer Filter $G$ auf $\m{P}$ gewählt?}
%$\pi: \m{P} \rightarrow \m{P}^\beta\, \uhr\, (\eta + 1)$. 
Then the $V$-generic filter $G$ on $\m{P}$ induces a $V$-generic filter $G^\beta\, \uhr\, (\eta + 1)$ on $\m{P}^\beta\, \uhr\, (\eta + 1)$. \item[B)] We show that $f^\beta$ is contained in an intermediate generic extension similar to $V[G^\beta\, \uhr\, (\eta + 1)]$. \item[C)] We prove that the forcing $\m{P}^\beta\, \uhr\, (\eta + 1)$ preserves cardinals $\geq \alpha_\eta$. \item[D)] We construct in $V[G^\beta\, \uhr\, (\eta + 1)]$ a set $\wt{\powerset} (\kappa_\eta) \supseteq \dom f^\beta$ with an injection $\iota: \wt{\powerset} (\kappa_\eta) \hookrightarrow \beta$.

%\colorbox{yellow}{FRAGE: Sollte man direkt $\dom \iota = \powerset^N (\kappa_\eta)$ nehmen?}

\end{itemize}

Then D) together with B) and C) gives the desired contradiction. \\[-3mm]

Hence, $f^\beta: \dom f^\beta \rightarrow \alpha_\eta$ must \textit{not} be surjective.

\begin{itemize} \item[E)] We consider $\alpha < \alpha_\eta$ with $\alpha \in rg\, f \setminus rg \,f^\beta$, and use an isomorphism argument to obtain a contradiction, again.\end{itemize}

We see that either case, whether $f^\beta$ was surjective or not, leads into a contradiction. Thus, our initial assumption must be wrong, and we can finally conclude: \\[-3mm]

\textit{There is no surjective function $f: \powerset(\kappa_\eta) \rightarrow \alpha_\eta$}. \\[-2mm]

Before we start with Chapter \ref{6.2} A), we first define our term \textit{large enough for the intersection} $(A_{\dot{f}})$:

\begin{definition} \label{largeenough} A limit ordinal $\wt{\beta} < \alpha_\eta$ is { \upshape large enough for the intersection $(A_{\dot{f}})$} if the following hold: \begin{itemize} \item $\wt{\beta} > \kappa_\eta^\plus$ \item $\wt{\beta} > \sup \{ i_m\ | \ \eta_m \leq \eta\}$ \item $\wt{\beta} > \sup \{k_m\ | \ \lambda_m \leq \eta\}$ \end{itemize}
\end{definition}

\vspace*{2mm}
(We use that $\alpha_\eta \geq \kappa_\eta^{\plus \plus}$, and $\cf \alpha_\eta > \omega$.) \\[-2mm]

%\colorbox{red}{ACHTUNG: Sollte man $\wt{\beta}$ nicht kleinstmöglich nehmen?}

%\colorbox{yellow}{TO DO: Erwähnen, dass das Vorgehen ähnlich ist wie im Baum-Forcing? Hier oder am Anfang?}

%\colorbox{yellow}{FRAGE: Wäre der Aufbau im Baumforcing nicht besser?}

Fix a limit ordinal $\wt{\beta} < \alpha_\eta$ \textit{large enough for the intersection $(A_{\dot{f}})$}, and let $\beta := \wt{\beta} + \kappa_\eta^\plus$ (addition of ordinals). \\[-3mm]

The restriction $f^\beta$ is defined as follows:

\begin{definition} \label{deffbeta}\[f^\beta := \Big \{\  (X, \alpha) \in f \ \Big| \ \exists\, \big( (a_m)_{m < \omega}, (\ol{\sigma}_m, \ol{i}_m)_{m < \omega} \big) \ \eta\mbox{-good pair}\; : \]\[(\forall  m\ \; \ol{i}_m < \beta)\, \wedge \exists\, \dot{X} \in \Name \big( (\ol{P}^\eta)^\omega\, \times\, \prod_{m < \omega} P^{\ol{\sigma}_m}\big)\ \ X = \dot{X}^{\prod_{m < \omega} G_\ast (a_m)\, \times\, \prod_{m < \omega} G^{\ol{\sigma}_m}_{\ol{i}_m}} \ \Big\}.\]\end{definition}

First, we assume towards a contradiction that $\boldsymbol{f^\beta: \dom f^\beta \rightarrow \alpha_\eta}$ \textbf{ is surjective}.

%\colorbox{yellow}{FRAGE: diese Stelle hervorheben? Ähnlich wie bei $\kappa_{\eta + 1} > \kappa_\eta^\plus$?}

\subsubsection*{A) Constructing $\boldsymbol{\m{P}^\beta \, \uhr\, (\eta + 1)}$.} 

Our aim is to construct a forcing notion $\m{P}^\beta\, \uhr\, (\eta + 1)$ that is obtained from $\m{P}$ by essentially \tbl cutting off\tbr\, at height $\eta$ and width $\beta$; i.e. only the cardinals $\kappa_\sigma$ for $\sigma \leq \eta$ should be considered, and for any such $\kappa_\sigma$, we add at most $\beta$-many new $\kappa_\sigma$-subsets $G^\sigma_i$. 

Regarding our $V$-generic filter $G$ on $\m{P}$, we need that the restriction $G^\beta\,\uhr\, (\eta + 1) := G \uhr \big(\m{P}^\beta\, \uhr (\eta + 1)\big)$ is a $V$-generic filter on $\m{P}^\beta\, \uhr\, (\eta + 1)$, which will be guaranteed by making sure that the canonical map $\rho^\beta: \m{P} \rightarrow \m{P}^\beta\, \uhr\, (\eta + 1)$, $p \mapsto p^\beta\, \uhr \, (\eta + 1)$ is a projection of forcing posets.  \\[-3mm]

%\colorbox{yellow}{Begriff?}

A first attempt to define $\m{P}^\beta\, \uhr \, (\eta + 1)$ could be the following: \\[-3mm]

For $p \in \m{P}$, let \[p^\beta\, \uhr (\eta + 1) = \big(p_\ast \uhr \kappa_\eta^2, (p^\sigma_i, a^\sigma_i)_{\sigma \leq \eta, i < \min \{\alpha_\sigma, \beta\}}, (p^\sigma\, \uhr\, (\min \{\alpha_\sigma, \beta\}\, \times\, \dom_y p^\sigma)_{\sigma \leq \eta} \big)\] denote the canonical restriction; and set \[\m{P}^\beta\, \uhr\, (\eta + 1) := \{p^\beta\, \uhr\, (\eta + 1)\ | \ p \in \m{P} \}.\] But then, $G^\beta\, \uhr \, (\eta + 1) := \{p^\beta \, \uhr \, (\eta + 1)\ | \ p \in G\}$ would not be a $V$-generic filter on $\m{P}^\beta\, \uhr\, (\eta + 1)$: Consider a linking ordinal $\xi \in g^{\ol{\sigma}}_{\ol{i}}$ for some $(\ol{\sigma}, \ol{i})$, such that $\eta < \ol{\sigma} < \gamma$, $\ol{i} < \alpha_{\ol{\sigma}}$ holds; or $\ol{\sigma} \leq \eta$, $\beta \leq \ol{i} < \alpha_{\ol{\sigma}}$. The set $D := \{p \in \m{P}^\beta\, \uhr\, (\eta + 1)\ | \ \xi \in \bigcup_{\sigma \leq \eta, i < \beta} a^{\sigma}_i \}$ is dense in $\m{P}^\beta\, \uhr\, (\eta + 1)$; but $D\, \cap\, G^\beta\, \uhr\, (\eta + 1) = \emptyset$ by the \textit{independence property}.
% it follows that there is no $\ol{p} \in G$ with $\xi \in \ol{a}^\sigma_i$ for some $\sigma \leq \eta$, $i < \beta$. 
Hence, $G^\beta\, \uhr\, (\eta + 1)$ can not be a $V$-generic filter on $\m{P}^\beta\, \uhr \, (\eta + 1)$. \\[-3mm]

This shows that the conditions in $\m{P}^\beta\, \uhr \, (\eta + 1)$ should contain some information about which linking ordinals are \tbl forbidden\tbr\, for $\bigcup_{\sigma \leq \eta, i < \beta} \, a^\sigma_i$, being already occupied by some index $(\ol{\sigma}, \ol{i})$ with $\ol{\sigma} > \eta$ or $\ol{i} \geq \beta$. 

Thus, for $p \in \m{P}$, we add to $p^\beta\, \uhr \, (\eta + 1)$ a new coordinate $X_p$, which is essentially the union of all $a^\sigma_i\, \cap\, \kappa_\eta$ for $\sigma > \eta$ or $i \geq \beta$. 
%\bigcup_{\sigma > \eta \; \vee\; i \geq \beta} (a^\sigma_i\, \cap\, \kappa_\eta)\; .\] 
Then $X_p$ is a subset of $\kappa_\eta$ that hits any interval $[\kappa_{\nu, j}, \kappa_{\nu, j + 1})$ in at most countably many points. \\[-2mm] 

%\colorbox{red}{ACHTUNG - wäre $X_p$ als Bezeichnung in Ordnung??? Man hat doch $(X, \alpha)$!} \\[-2mm]

Let $\wt{\eta} := \sup \{\sigma < \eta\ | \ \sigma \in Lim\}$. By closure of the sequence $(\kappa_\sigma\ | \ 0 < \sigma < \gamma)$, it follows that $\wt{\eta} \in Lim$ with $\wt{\eta} = \max \{\sigma \leq \eta\ | \ \eta \in Lim\}$, and $\kappa_{\wt{\eta}} = \sup \{ \kappa_\sigma\ | \ \sigma \in Lim\, , \, \sigma < \wt{\eta}\}$. \\[-3mm] 

W.l.o.g. we restrict to the case that \[\boldsymbol{\beta < \alpha_{\wt{\eta}}} \textbf{ \ \ or\ \ }\boldsymbol{Lim\, \cap\, (\eta, \gamma) \neq \emptyset}\,;\] which is the same as requiring that there exist coordinates $(\sigma, i)$ with $\sigma \in Lim$, and $\sigma > \eta$ or $i \geq \beta$. 
(Otherwise, the forcing $\m{P}^\beta\, \uhr \, (\eta + 1)$ already contains all coordinates $(\sigma, i)$ with $\sigma \in Lim$, and there are no \tbl forbidden\tbr\,linking ordinals. In that case, we can indeed set $\m{P}^\beta\, \uhr\, (\eta + 1) := \big\{\, \big (p_\ast\, \uhr\, \kappa_\eta^2, (p^\sigma_i, a^\sigma_i)_{\sigma \leq \eta, i < \beta}, (p^\sigma\, \uhr\, (\beta\, \times\, \dom_y p^\sigma))_{\sigma \leq \eta}\big)\ | \ p \in \m{P} \, \big\}$, and obtain that $G^\beta\, \uhr\, (\eta + 1)$ is a $V$-generic filter on $\m{P}^\beta\, \uhr \, (\eta + 1)$. ) \\[-2mm]

For a condition $p \in \m{P}$, let \[X_p := \bigcup \big\{\; a^\sigma_i \, \cap \, \kappa_{\wt{\eta}}\ | \ \sigma \in Lim \mbox{ with } (\sigma > \eta \mbox{ or }
 i \geq \beta)\; \big\}\,,\] and \[p^\beta\, \uhr \, (\eta + 1) := \big(\, p_\ast\ \uhr \, \kappa_\eta^2, (p^\sigma_i, a^\sigma_i)_{\sigma \leq \wt{\eta}, i < \beta}, (p^\sigma\, \uhr \, (\beta\, \times\, \dom_y p^\sigma))_{\sigma \leq \eta}, X_p\, \big).\] For reasons of homogeneity, we include into $\m{P}^\beta\, \uhr \, (\eta + 1)$ only those conditions $p^\beta\, \uhr \, (\eta + 1)$ for which the set $X_p$ hits every interval $[\kappa_{\nu, j}, \kappa_{\nu, j + 1}) \subseteq \kappa_{\wt{\eta}}$ in countably many points, which is the same as requiring $|\{ (\sigma, i) \in \supp p_0\ | \ \sigma > \wt{\eta} \mbox{ or } i \geq \beta\}| = \aleph_0$.

%\colorbox{yellow}{FRAGE: $\omega$ oder $\aleph_0$ für Kardinalitäten?} 

\begin{definition} \label{defpbetauhretapluseins}$\m{P}^\beta\, \uhr\, (\eta + 1) := $ \[\big\{\; p^\beta\, \uhr\, (\eta + 1)\ \ \big| \ \ p \in \m{P}\; , \; |\{ (\sigma, i) \in \supp p_0\ | \ \sigma > \eta \mbox{ or } i \geq \beta\}| = \aleph_0\, \big\}  \]\[ \cup \; \{\m{1}^\beta_{\eta + 1}\},\] with $\m{1}^\beta_{\eta + 1}$ as the maximal element. \\[-3mm]

%\colorbox{red}{ACHTUNG - EVTL die Bezeichnung $\emptyset$ weglassen?}

For conditions $p^\beta\, \uhr \, (\eta + 1)$, $q^\beta\, \uhr \, (\eta + 1)$ in $\m{P}^\beta\, \uhr \, (\eta + 1) \setminus \{ \m{1}^\beta_{\eta + 1}\}$, let $q^\beta\, \uhr \, (\eta + 1) \leq^\beta_{\eta +1} p^\beta\, \uhr \, (\eta + 1)$ if $X_q \supseteq X_p$, and $\big(q_\ast\, \uhr\, \kappa_\eta^2, (q^\sigma_i, b^\sigma_i)_{\sigma \leq \eta, i < \beta}, (q^\sigma\, \uhr \, (\beta\, \times\, \dom_y q^\sigma)_{\sigma \leq \eta}\big) \leq \big(p_\ast\, \uhr\, \kappa_\eta^2, (p^\sigma_i, a^\sigma_i)_{\sigma \leq \eta, i < \beta}, (p^\sigma\, \uhr \, (\beta\, \times\, \dom_y p^\sigma)_{\sigma \leq \eta}\big)$ regarded as conditions in $\m{P}$. \end{definition}

%\colorbox{yellow}{FRAGE: eine besser Schreibweise für $\uhr\, (\eta + 1)$ suchen}

%\colorbox{yellow}{Sollte man nicht $p_\ast \, \uhr\, \kappa_{\wt{\eta}}^2$ nehmen?}

% and also EINE BESSERE SCHREIBWEISE FÜR $\ \uhr\, (\eta + 1)$ SUCHEN?? $X_q \supseteq X_p$. \end{definition}

In other words: $\m{P}^\beta\,\uhr \, (\eta + 1)$ is the collection of all $(p_\ast, (p^\sigma_i, a^\sigma_i)_{\sigma \leq \eta, i < \beta}, (p^\sigma)_{\sigma \leq \eta}, X_p)$ such that \begin{itemize} \item $p := (p_\ast, (p^\sigma_i, a^\sigma_i)_{\sigma \leq \eta, i < \beta}, (p^\sigma)_{\sigma \leq \eta})$ is a condition in $\m{P}$ with $\dom p_0 \subseteq \kappa_\eta$, $\supp p_0 \subseteq \{ (\sigma, i)\ | \ \sigma \leq \eta, i < \beta\}$, and $\supp p_1 \subseteq \eta + 1$ with $\forall\, \sigma \in \supp p_1 \: : \: \dom_x p^\sigma \subseteq \beta$, \item $X_p \subseteq \kappa_{\wt{\eta}}$ with $\forall [\kappa_{\nu, j}, \kappa_{\nu, j + 1}) \subseteq \kappa_{\wt{\eta}}\ \ |X_p\, \cap \, [\kappa_{\nu, j}, \kappa_{\nu, j + 1})| = \aleph_0$, and $X_p\, \cap\, \bigcup_{\sigma \leq \eta\, , \, i < \beta}\, a^\sigma_i = \emptyset$. \end{itemize} For $p$, $q \in \m{P}$ with $q \leq p$ and $|\{ (\sigma, i) \in \supp p_0\ | \ \sigma > \eta \mbox{ or } i \geq \beta\}| = \aleph_0$, it follows that $q^\beta\, \uhr \, (\eta + 1) \leq p^\beta\, \uhr \, (\eta + 1)$.

\begin{definition} \[G^\beta\, \uhr \, (\eta + 1) := \big \{\,p \in \m{P}^\beta\, \uhr \, (\eta + 1)\ \  \big| \ \ \exists\, \ol{p} \in G\; : \ |\{ (\sigma, i) \in \supp \ol{p}_0\ | \ \sigma > \eta \mbox{ or } i \geq \beta\}| = \aleph_0\; ,\]\[ \ol{p}^\beta \, \uhr \, (\eta + 1) \leq^\beta_{\eta + 1} p\, \big\}.\]\end{definition}

\vspace*{2mm}
We will now show that $G^\beta\, \uhr \, (\eta + 1)$ is a $V$-generic filter on $\m{P}^\beta\, \uhr \, (\eta + 1)$. \\[-3mm]

Let $\ol{\m{P}} \subseteq \m{P}$ denote the collection of all $p \in \m{P}$ with the property that $|\, \{ (\sigma, i) \in \supp p_0\ | \ \sigma > \eta\, \vee\, i \geq \beta\}\,| = \aleph_0$, together with the maximal element $\m{1}$. Then $\ol{\m{P}}$ is a dense subforcing of $\m{P}$.

\vspace*{1mm}

\begin{prop} \label{projectionrho} The map $\rho^\beta: \ol{\m{P}} \rightarrow \m{P}^\beta\, \uhr \, (\eta + 1)$ with $p \mapsto p^\beta\, \uhr \, (\eta + 1)$ in the case that $|\, \{ (\sigma, i) \in \supp p_0\ | \ \sigma > \eta\, \vee\, i \geq \beta\}\,| = \aleph_0$, and $\m{1} \mapsto \m{1}^\beta_{\eta + 1}$, 
%in the case that $|\{ (\sigma, i) \in \supp p_0\ | \ \sigma > \eta\, \vee\, i \geq \beta\}| = \aleph_0$, and $p \mapsto \m{1}^\beta_{\eta + 1}$ else, 
is a \textit{projection of forcing posets}:
\begin{itemize} \item $\rho^\beta (\m{1}) = \m{1}^\beta_{\eta + 1}$\,, \item if $\ol{p}$, $\ol{q} \in \ol{\m{P}}$ with $\ol{q} \leq \ol{p}$, it follows that $\rho^\beta (\ol{q}) \leq^\beta_{\eta + 1} \rho^\beta (\ol{p})$\,, \item for any $\ol{p} \in \ol{\m{P}}$ and $q \in \m{P}^\beta\, \uhr \, (\eta + 1)$ with $q \leq^\beta_{\eta + 1} \rho^\beta (\ol{p})$, there exists $\ol{q} \in \ol{\m{P}}$ such that $\ol{q} \leq \ol{p}$ and $\rho^\beta (\ol{q}) \leq q$.\end{itemize}

Hence, $G^\beta\, \uhr (\eta + 1)$ is a $V$-generic filter on $\m{P}^\beta\, \uhr \, (\eta + 1)$. 
 \end{prop}

%\colorbox{red}{ACHTUNG - DIE ÄNDERUNGEN MÜSSTE MAN NOCH DURCHGEHEN! PROPOSITION 35??}

\begin{proof} Clearly, the map $\rho^\beta$ as defined above is order-preserving with $\rho^\beta (\m{1}) = \m{1}^\beta_{\eta + 1}$. Consider $\ol{p} = (\ol{p}_\ast, (\ol{p}^\sigma_i, \ol{a}^\sigma_i)_{\sigma, i}, (\ol{p}^\sigma)_\sigma) \in \ol{\m{P}}$ 
%\colorbox{red}{ACHTUNG - WÄRE HIER AUCH DER FALL $\ol{p} = \m{1}$ IN ORDNUNG?}
%(w.l.o.g. we can assume $|\{ (\sigma, i) \in \supp \ol{p}\ | \ \sigma > \eta\, \vee\, i \geq \beta\}| = \omega$) 
and $q = (q_\ast\, \uhr\, \kappa_\eta^2, (q^\sigma_i, b^\sigma_i)_{\sigma \leq \eta, i < \beta}, (q^\sigma)_{\sigma \leq \eta}, X_q) \in \m{P}^\beta\, \uhr \, (\eta + 1)$ with $q \leq^\beta_{\eta + 1} \rho^\beta (\ol{p}) = \ol{p}^\beta\, \uhr \, (\eta + 1)$. Then \begin{eqnarray} (q_\ast\, \uhr \, \kappa_\eta^2, (q^\sigma_i, a^\sigma_i)_{\sigma \leq \eta, i < \beta}) & \leq_0 & (\ol{p}_\ast \, \uhr \, \kappa_\eta^2, (p^\sigma_i, a^\sigma_i)_{\sigma \leq \eta, i < \beta}) \mbox{ \hspace*{0,8cm} in }\m{P}_0  \ \ , \nonumber \\
(q^\sigma)_{\sigma \leq \eta}  & \leq_1 & (p^\sigma\, \uhr\, (\beta\, \times\, \dom_y p^\sigma))_{\sigma \leq \eta} \mbox{ \hspace*{0,8cm} in }\m{P}_1\ \ , \mbox{ \hspace*{2mm} \textit{and} } \nonumber \\ X_q  &\supseteq  &\bigcup \{ \ol{a}^\sigma_i\, \cap\, \kappa_{\wt{\eta}}\ | \ \sigma > \eta\, \vee\, i \geq \beta\}.  \nonumber \end{eqnarray}
We have to construct $\ol{q} \in \m{P}$, $\ol{q} = (\ol{q}_\ast, (\ol{q}^\sigma_i, \ol{b}^\sigma_i)_{\sigma, i}, (\ol{q}^\sigma)_\sigma)$, with $\ol{q} \leq \ol{p}$ and $\rho^\beta (\ol{q}) = \ol{q}^\beta\, \uhr \, (\eta + 1) \leq^\beta_{\eta + 1} q$. \\[-3mm]

We start with $\ol{q}_0$: 
%\colorbox{yellow}{TO DO: MOTIVATION FÜR DIESE KONSTRUKTION?}
\begin{itemize} 
%\item Let $\dom \ol{q}_0 \, \cap\, \kappa_\eta := \dom q_0$, and $\dom \ol{q}_0\, \cap\, [\kappa_\eta, \kappa_\gamma) := \dom \ol{p}_0\, \cap \, [\kappa_\eta, \kappa_\gamma)$. We set $\ol{q}_\ast \, \uhr \, \kappa_\eta^2 := q_\ast \supseteq \ol{p}_\ast \, \uhr \, \kappa_\eta^2$, and $\ol{q}_\ast\, \uhr\, [\kappa_\eta, \kappa_\gamma)^2 := \ol{p}_\ast \, \uhr \, [\kappa_\eta, \kappa_\gamma)^2$. 
\item 

%\colorbox{yellow}{TO DO: sollte man den Index $\wt{\wt{\eta}}$ umbenennen?} 
In order to achieve $X_{\ol{q}} \supseteq X_q$, we will enlarge $\supp \ol{p}_0\, \cup\, \supp q_0$ by countably many $((\hat{\eta}, m_k)\ | \ k < \omega)$, where $\hat{\eta} > \eta$ or $m_k \geq \beta$ for all $k < \omega$, and arrange that any $\xi \in X_q \setminus X_p$ occurs as a linking ordinal in some $\ol{b}^{\,\hat{\eta}}_{\,m_k}$. \\More precisely: Let $\supp \ol{q}_0 := \supp \ol{p}_0\, \cup\, \supp q_0 \, \cup\, \supp_\ast$, where $\supp_\ast := \{(\hat{\eta}, m_k)\ | \ k < \omega\}$ such that $(\hat{\eta}, m_k) \notin \supp \ol{p}_0\, \cup\, \supp q_0$ for all $k <  \omega$, and since we are working in the case that $\beta < \alpha_{\wt{\eta}}$ or $(\eta, \gamma)\, \cap\, Lim \neq \emptyset$, we can take either $\hat{\eta} := \wt{\eta}$ and $m_k \in (\beta, \alpha_{\wt{\eta}})$ for all $k < \omega$; or $\hat{\eta} \in (\eta, \gamma)\, \cap\, Lim$. Then for all $(\hat{\eta}, m_k)$, it follows that $\hat{\eta} > \eta$ or $m_k \geq \beta$.

\item Next, we define the linking ordinals $\ol{b}^\sigma_i$ for $(\sigma, i) \in \supp \ol{q}_0$ such that $X_{\ol{q}} \supseteq X_q$. 

For $(\sigma, i) \in \supp q_0$, we let $\ol{b}^\sigma_i := b^\sigma_i \supseteq \ol{a}^\sigma_i$; and in the case that $(\sigma, i) \in \supp \ol{p}_0 \setminus \supp q_0$, we set $\ol{b}^\sigma_i := \ol{a}^\sigma_i$. Finally, we define $(\ol{b}^{\,\hat{\eta}}_{\,m_k}\ | \ k < \omega)$ with the following properties:

\begin{itemize} \item as usual, every $\ol{b}^{\, \hat{\eta}}_{\,m_k}$ is a subset of $\kappa_{\hat{\eta}}$ that hits every interval $[\kappa_{\nu, j}, \kappa_{\nu, j + 1}) \subseteq \kappa_{\hat{\eta}}$ in exactly one point\,, \item $\bigcup \big \{ \ol{b}^{\,\hat{\eta}}_{\,m_k}\, \cap\, \kappa_{\wt{\eta}}\ | \ k < \omega \big\} \supseteq X_q \setminus X_{\ol{p}}$\,, \item $\ol{b}^{\,\hat{\eta}}_{\,m_k}\, \cap \, b^{\ol{\sigma}}_{\ol{i}} = \emptyset$ for all $k < \omega$ and $(\ol{\sigma}, \ol{i}) \in \supp q_0$\,, \item $\ol{b}^{\,\hat{\eta}}_{\,m_k}\, \cap\, \ol{a}^{\ol{\sigma}}_{\ol{i}} = \emptyset$ for all $k < \omega$ and $(\ol{\sigma}, \ol{i}) \in \supp \ol{p}_0 \setminus \supp q_0$ \\(since $q \leq^\beta_{\eta + 1} \ol{p}^\beta\, \uhr\, (\eta + 1)$, it follows that in this case, $\ol{\sigma} > \eta$ or $\ol{i} \geq \beta$)\,,\item $\ol{b}^{\,\hat{\eta}}_{\,m_k}\, \cap \, \ol{b}^{\,\hat{\eta}}_{\,m_{k^\prime}} = \emptyset$ whenever $k \neq k^\prime$. \end{itemize}

This is possible, since 
%$| (X_q \setminus X_{\ol{p}})\, \cap \, [\kappa_{\nu, j}, \kappa_{\nu, j + 1})| \leq \omega$ for all $\kappa_{\nu, j} < \kappa_\eta$; moreover, 
$X_q\, \cap\, b^{\ol{\sigma}}_{\ol{i}} = \emptyset$ for any $(\ol{\sigma}, \ol{i}) \in \supp q$ by construction of $\m{P}^\beta\, \uhr\, (\eta + 1)$; and whenever $(\ol{\sigma}, \ol{i}) \in \supp \ol{p}_0 \setminus \supp q_0$, then $\ol{\sigma} > \eta$ or $\ol{i} \geq \beta$ implies $\ol{a}^{\ol{\sigma}}_{\ol{i}} \subseteq X_{\ol{p}}$; thus $(X_q \setminus X_{\ol{p}})\, \cap\, \ol{a}^{\ol{\sigma}}_{\ol{i}} = \emptyset$.

\item We now define $\dom \ol{q}_0$. For any interval $[\kappa_{\nu, j}, \kappa_{\nu, j + 1}) \subseteq \kappa_\eta$, take $\delta_{\nu, j} \in [\kappa_{\nu, j}, \kappa_{\nu, j + 1})$ as follows: In the case that $\dom q_0\, \cap \, [\kappa_{\nu, j}, \kappa_{\nu, j + 1}) = \emptyset$, let $\delta_{\nu, j} := \kappa_{\nu, j}$. If $\dom q_0\, \cap \, [\kappa_{\nu, j}, \kappa_{\nu, j + 1}) \neq \emptyset$, we take $\delta_{\nu, j} \in (\kappa_{\nu, j}, \kappa_{\nu, j + 1})$ such that $\bigcup \{ \ol{b}^{\,\sigma}_{\,i} \, | \, (\sigma, i) \in \supp \ol{q}_0 \}\, \cap\, [\kappa_{\nu, j}, \kappa_{\nu, j + 1}) \subseteq [\kappa_{\nu, j}, \delta_{\nu, j})$ and $\dom q_0\, \cap \, [\kappa_{\nu, j}, \kappa_{\nu, j + 1}) \subseteq [\kappa_{\nu, j}, \delta_{\nu, j})$. Since $\dom q_0$ is bounded below all regular cardinals $\kappa_{\ol{\nu}, \ol{\j}}$, this is also true for $\bigcup  \{\,[\kappa_{\nu, j}, \delta_{\nu, j})\ | \ \kappa_{\nu, j} < \kappa_\eta\,\}$. Let \[\dom \ol{q}_0\, \cap\, \kappa_\eta := \bigcup  \{\,[\kappa_{\nu, j}, \delta_{\nu, j})\ | \ \kappa_{\nu, j} < \kappa_\eta\,\},\] and $\dom \ol{q}_0 \, \cap\, [\kappa_\eta, \kappa_\gamma) := \dom \ol{p}_0 \, \cap\, [\kappa_\eta, \kappa_\gamma)$.

\item We take $\ol{q}_\ast \, \uhr \, \kappa_\eta^2 \supseteq q_\ast \,\uhr\, \kappa_\eta^2$ arbitrary on the given domain; and $\ol{q}_\ast \, \uhr \, [\kappa_\eta, \kappa_\gamma)^2 := \ol{p}_\ast\, \uhr \, [\kappa_\eta, \kappa_\gamma)^2$.

\item It remains to define $\ol{q}^\sigma_i$ for $(\sigma, i) \in \supp \ol{q}_0$.

For $(\sigma, i) \in \supp q_0$, we define $\ol{q}^\sigma_i \supseteq q^\sigma_i$ on the given domain $\bigcup_{\kappa_{\nu, j} < \kappa_\sigma} [\kappa_{\nu, j}, \delta_{\nu, j})$ according to the \textit{linking property}: Consider an interval $[\kappa_{\nu, j}, \kappa_{\nu, j + 1})$ with $\delta_{\nu, j} > \kappa_{\nu, j}$. For any $\zeta \in (\dom \ol{q}_0 \setminus \dom q_0)\, \cap\, [\kappa_{\nu, j}, \kappa_{\nu, j + 1})$, set $\ol{q}^\sigma_i (\zeta) := \ol{q}_\ast (\xi, \zeta)$, where $\{\xi\} := b^\sigma_i\, \cap\, [\kappa_{\nu, j}, \kappa_{\nu, j + 1}) = \ol{b}^\sigma_i\, \cap \, [\kappa_{\nu, j}, \kappa_{\nu, j + 1})$. (Note that $\xi \in \dom \ol{q}_0$ by construction). For $(\sigma, i) \in \supp \ol{p}_0 \setminus \supp q_0$, we set $\ol{q}^\sigma_i\, \uhr \, [\kappa_\eta, \kappa_\gamma) := \ol{p}^\sigma_i\, \uhr \, [\kappa_\eta, \kappa_\gamma)$, and define $\ol{q}^\sigma_i\, \uhr\, \kappa_\eta \supseteq \ol{p}^\sigma_i\, \uhr\, \kappa_\eta$ on the given domain according to the \textit{linking property} as before. \\ Finally, $\ol{q}^{\,\hat{\eta}}_{\,m_k}$ for $k < \omega$ can be arbitrary on the given domain.
\end{itemize} 

Then $\ol{q}_0 = (\ol{q}_\ast, (\ol{q}^\sigma_i, \ol{b}^\sigma_i)_{\sigma, i})$ is a condition in $\m{P}_0$. In particular, the \textit{independence property} holds for the linking ordinals $\ol{b}^\sigma_i$: Firstly, by construction of $(\ol{b}^{\,\hat{\eta}}_{\,m_k}\ | \ k < \omega)$, it follows that $\ol{b}^{\,\hat{\eta}}_{\,m_k}\, \cap\, \ol{b}^{\,\ol{\sigma}}_{\,\ol{i}} = \emptyset$ for any $(\ol{\sigma}, \ol{i}) \in \supp q_0\, \cup \, \supp \ol{p}_0$. Secondly, whenever $(\sigma_0, i_0) \in \supp q_0$ and $(\sigma_1, i_1) \in \supp \ol{p}_0 \setminus \supp q_0$, then $\sigma_1 > \eta$ or $i_1 \geq \beta$; hence, $\ol{b}^{\sigma_1}_{i_1} = \ol{a}^{\sigma_1}_{i_1} \subseteq X_{\ol{p}} \subseteq X_q$. Since $\ol{b}^{\sigma_0}_{i_0}\, \cap \, X_q = b^{\sigma_0}_{i_0}\, \cap\, X_q = \emptyset$, this implies $\ol{b}^{\sigma_0}_{i_0}\, \cap \, \ol{b}^{\sigma_1}_{i_1} = \emptyset$ as desired. Thus, the independence property holds for $\ol{q}_0$.\\[-2mm]

Moreover, $(\rho^\beta (\ol{q}))_0 = (\ol{q}_\ast\, \uhr\, \kappa_\eta^2, (\ol{q}^\sigma_i, \ol{b}^\sigma_i)_{\sigma \leq \eta, i < \beta}, X_{\ol{q}}) \leq q_0$ by construction; in particular, $X_q \subseteq X_{\ol{q}}$: Consider $\xi \in X_q$. In the case that $\xi \in X_{\ol{p}}$, it follows that $\xi \in \ol{a}^{\ol{\sigma}}_{\ol{i}}$ for some $(\ol{\sigma}, \ol{i}) \in \supp \ol{p}_0$ with $\ol{\sigma} > \eta$, or $\ol{i} \geq \beta$. Then $(\ol{\sigma}, \ol{i}) \in \supp \ol{p}_0 \setminus \supp q_0$; hence, $\ol{b}^{\ol{\sigma}}_{\ol{i}} = \ol{a}^{\ol{\sigma}}_{\ol{i}}$, and it follows that $\xi \in \ol{b}^{\ol{\sigma}}_{\ol{i}} \subseteq X_{\ol{q}}$ as desired. In the case that $\xi \in X_q \setminus X_{\ol{p}}$, we have $\xi \in \ol{b}^{\,\hat{\eta}}_{\,m_k}$ for some $k < \omega$; so again, $\xi \in X_{\ol{q}}$ as desired.\\[-2mm]

Finally, $\ol{q}_0 \leq \ol{p}_0$ by construction; and it follows that $\ol{q}_0$ has all the desired properties. \\[-2mm]

The construction of $\ol{q}_1$ is similar. Thus, the map $\rho^\beta : \ol{\m{P}} \rightarrow \m{P}^\beta\, \uhr\, (\eta + 1)$ as defined above,
%$p \mapsto p^\beta\, \uhr\, (\eta + 1)$ in the case that $ |\{ (\sigma, i) \in \supp p\ | \ \sigma > \eta \mbox{ or } i \geq \beta\}| = \omega$, and $p \mapsto \emptyset$, else, 
is indeed a projection of forcing posets. \\[-3mm]

%\colorbox{yellow}{Etwas weniger ausführlich?}

It follows that $G^\beta\, \uhr\, (\eta + 1)$ is a $V$-generic filter on $\m{P}^\beta\, \uhr\, (\eta + 1)$: For genericity, consider an open dense set $D \subseteq \m{P}^\beta\, \uhr \, (\eta + 1)$. It suffices to show that the set $\ol{D} := \{ \ol{p} \in \ol{\m{P}}\ | \ \ol{p}^\beta\, \uhr \, (\eta + 1) \in D\}$ is dense in $\m{P}$. Take a condition $p \in \m{P}$, and let $\ol{p} \leq p$ with $\ol{p} \in \ol{\m{P}}$. Since $D \subseteq \m{P}^\beta\, \uhr \, (\eta + 1)$ is dense, there exists $q \in \m{P}^\beta\, \uhr \, (\eta + 1)$ with $q \leq^\beta_{\eta + 1} \ol{p}^\beta\, \uhr \, (\eta + 1)$. By what we have just shown, we there exists $\ol{q} \leq \ol{p}$ with $\ol{q}^\beta\, \uhr \, (\eta + 1) \leq q$. Then $\ol{q}$ is an extension of $p$ in $\ol{D}$ as desired. \end{proof}

%\colorbox{yellow}{FRAGE: Könnte man vll die Bezeichnung für $\m{P}^\beta\, \uhr\, (\eta + 1)$ ändern?}

%\colorbox{yellow}{War es im Baumforcing genauso?}

\subsubsection*{B) Capturing $\boldsymbol{f^\beta}$.}

In this section, we will show that the map $f^\beta$ is contained in a generic extension similar to $V[G^\beta\, \uhr\, (\eta + 1)]$. \\[-3mm]

Recall that we are working in the case that $\boldsymbol{\kappa_{\eta + 1} > \kappa_\eta^\plus}$, and $\boldsymbol{\beta < \alpha_{\wt{\eta}}}$ \textbf{ or } $\boldsymbol{(\wt{\eta}, \gamma)\, \cap\, Lim \neq \emptyset}$, where $\wt{\eta} := \max \{\sigma < \eta\ | \ \sigma \in Lim\}$. \\[-2mm]

Recall that any $X \in \dom f$ is of the form \[X = \dot{X}^{\prod_{m < \omega} G_\ast (a_m)\, \times\, \prod_{m < \omega} G^{\ol{\sigma}_m}_{\ol{i}_m}}, \] where $\dot{X} \in \Name ( (\ol{P}^\eta)^\omega)\, \times\, \prod_{m < \omega} P^{\ol{\sigma}_m}_{\ol{i}_m})$ and $ \big( (a_m)_{m < \omega}, (\ol{\sigma}_m, \ol{i}_m)_{m < \omega}\big)$ is an $\eta$-good pair. Moreover, \[f^\beta := \big\{\  (X, \alpha) \in f \ \big| \ \exists\, \big( (a_m)_{m < \omega}, (\ol{\sigma}_m, \ol{i}_m)_{m < \omega} \big) \ \eta\mbox{\textit{-good pair}}\ :  \]\[(\forall  m\ \;\ol{i}_m < \beta\,) \, \wedge \, \exists\, \dot{X} \in \Name \big( (\ol{P}^\eta)^\omega\, \times\, \prod_{m < \omega} P^{\ol{\sigma}_m}\big)\ \: X = \dot{X}^{\prod_{m < \omega} G_\ast (a_m)\, \times\, \prod_{m < \omega} G^{\ol{\sigma}_m}_{\ol{i}_m}}\ \big\}.\]

Fix an $\eta$-good pair $\varrho = \big( (a_m)_{m < \omega}, (\ol{\sigma}_m, \ol{i}_m)_{m < \omega} \big)$. We use recursion over the $\Name ( (\ol{P}^\eta)^\omega\, $ $\times \,$ $\prod_{m < \omega} P^{\ol{\sigma}_m})$-hierarchy to define a map $\tau_{\varrho}: \Name ( ( \ol{P}^\eta)^\omega\, \times \, \prod_{m < \omega} P^{\ol{\sigma}_m}) \rightarrow \Name (\m{P})$ that maps any name $\dot{Y} \in \Name ( (\ol{P}^\eta)^\omega\, \times \, \prod_{m < \omega} P^{\ol{\sigma}_m})$ to a name $\tau_{\varrho} (\dot{Y}) \in \Name (\m{P})$ % which will be denoted by $\wt{\dot{Y}}$, 
such that \[\dot{Y}^{\prod_{m < \omega} G_\ast (a_m)\, \times\, \prod_{m < \omega}G^{\ol{\sigma}_m}_{\ol{i}_m}} = (\tau_{\varrho} (\dot{Y}))^G. \]

\begin{definition} \label{tauvarrho} For an $\eta$-good pair $ \varrho = \big( (a_m)_{m < \omega}, (\ol{\sigma}_m, \ol{i}_m)_{m < \omega} \big)$, we define recursively for $\dot{Y} \in \Name ( (\ol{P}^\eta)^\omega\, \times \, \prod_{m < \omega} P^{\ol{\sigma}_m})$: \[\tau_\varrho (\dot{Y}) := \Big \{ \; (\tau_{\varrho} (\dot{Z}), q)\ \, \big| \, \ q \in \m{P}\; , \; \exists\, \big(\,\dot{Z}\, , \, \big( (p_\ast (a_m))_{m < \omega}\, , \,  (p^{\ol{\sigma}_m}_{\ol{i}_m})_{m < \omega}\big) \, \big) \in \dot{Y}\; :\ \]\[\forall m\ \; \big(\, q_\ast (a_m) \supseteq p_\ast (a_m) \; , \; q^{\ol{\sigma}_m}_{\ol{i}_m} \supseteq p^{\ol{\sigma}_m}_{\ol{i}_m}\, \big) \;\Big\}.\] 

\end{definition}

%\colorbox{yellow}{FRAGE: $a_m$ oder $\ol{a}_m$?}

%\colorbox{yellow}{FRAGE: die Bezeichnung $\tau$ nicht lieber weglassen?} \\[-2mm]

It is not difficult to check that indeed, $\dot{Y}^{\prod_{m < \omega} G_\ast (a_m)\, \times\, \prod_{m < \omega}G^{\ol{\sigma}_m}_{\ol{i}_m}} = (\tau_{\varrho} (\dot{Y}))^G$ holds for all $\dot{Y} \in \Name ( (\ol{P}^\eta)^\omega\, \times \, \prod_{m < \omega} P^{\ol{\sigma}_m})$. \\[-2mm]

%When the $\eta$-good pair $\varrho = \big( (a_m\ | \ m < \omega), ((\ol{\sigma}_m, \ol{i}_m) \ | \ m < \omega) \big)$ is clear from the context, we will usually just write $\tau (\dot{Y})$ or $\wt{\dot{Y}}$. 

%\colorbox{yellow}{FRAGE: Könnte man die Bezeichnungen $\tau (\dot{Y})$ oder $\wt{\dot{Y}}$ wirklich weglassen?} \\[-2mm]

%Although this definition depends on the good pair $\big( (a_m\ | \ m < \omega), ((\ol{\sigma}_m, \ol{i}_m) \ | \ m < \omega) \big)$, we will usually just write $\wt{\dot{Y}}$ for $\tau (\dot{Y})$, when it is clear from the context which good pair we are considering.

Now, we define a map $(f^\beta)^\prime \supseteq f^\beta$, which is contained in an intermediate generic extension similar to $V[G^\beta\, \uhr \, (\eta + 1)]$. We will then use an isomorphism argument to show that actually, $(f^\beta)^\prime = f^\beta$. \\[-3mm]

Recall that $f = \dot{f}^G$, where $\pi \ol{f}^{D_\pi} = \ol{f}^{D_\pi}$ whenever $[\pi]$ contained in the intersection $\bigcap_{m < \omega} Fix (\eta_m, i_m)\, \cap \, \bigcap_{m < \omega} H^{\lambda_m}_{k_m}$ denoted by $(A_{\dot{f}})$. \\[-3mm]

The idea is that we include into $\m{P}^\beta\, \uhr\, (\eta + 1)$ the verticals $P^{\eta_m}_{i_m}$ for $\eta_m \in \Lim$, $\eta_m > \eta$. Below $\kappa_\eta$, the linking property will be important, so we also have to include the linking ordinals $a^{\eta_m}_{i_m}\, \cap\, \kappa_\eta$. \\[-2mm]

%\colorbox{yellow}{FRAGE: Ist \tbl vertical line\tbr\,ein korrekter Begriff?}

For a condition $p \in \m{P}$, we set \[\wt{X}_p := \bigcup \{ a^\sigma_i\, \cap\, \kappa_{\wt{\eta}} \ | \ \sigma \in \Lim\; , \; (\sigma, i) \neq (\eta_m, i_m) \mbox{ for all } m < \omega\; , \; (\sigma > \eta \mbox{ or } i \geq \beta) \}.\] Then $\wt{X}_p$ is similar to $X_p$, but excludes the linking ordinals $a^{\eta_m}_{i_m}$ for $\eta_m \in Lim$. \\[-3mm]

%\colorbox{red}{TO DO: die Notation $(p^\beta\, \uhr\, (\eta + 1))^{( \eta_m, i_m)_{m < \omega}}$ ist SCHLECHT! DAS MÜSSTE MAN ÄNDERN!}

%\colorbox{red}{Versuche: s. letzte Seite im Block} \\[-2mm]

For reasons of notational convenience and better clarity, we introduce the following ad-hoc notation: \\[-2mm]

Let \[ (p^\beta\, \uhr\, (\eta + 1))^{( \eta_m, i_m)_{m < \omega}} := \big(\,p_\ast\, \uhr\, \kappa_\eta^2\; , \; (p^\sigma_i, a^\sigma_i)_{\sigma \leq \eta\,, \,i < \beta}\; , \; (p^{\eta_m}_{i_m}\, \uhr\, \kappa_\eta, a^{\eta_m}_{i_m}\, \cap\, \kappa_\eta)_{m < \omega\,, \,\eta_m > \eta}, \]\[(p^\sigma\, \uhr\, ( \beta\, \times\, \dom_y p^\sigma))_{\sigma \leq \eta}, \wt{X}_p\, \big).\] Then $(p^\beta\, \uhr\, (\eta + 1))^{( \eta_m, i_m)_{m < \omega}}$ can be obtained from $p^\beta\, \uhr\, (\eta + 1)$ by using $\wt{X}_p$ instead of $X_p$, and including $(p^{\eta_m}_{i_m}\, \uhr\, \kappa_\eta, a^{\eta_m}_{i_m}\, \cap\, \kappa_\eta)$ for $\eta_m \in Lim$ with $\eta_m > \eta$. (Note that for $\eta_m \leq \eta$, it follows that $i_m < \beta$, so $(p^{\eta_m}_{i_m}, a^{\eta_m}_{i_m})$ is already part of the condition $p^\beta\, \uhr\, (\eta + 1)$.) \\[-3mm]

%\colorbox{yellow}{FRAGE: sollte man in der Definition $\wt{\eta}$ verwenden? oder stattdessen $\eta$?}\\[-3mm]

We are now ready to define our forcing notion $(\m{P}^\beta\, \uhr\, (\eta + 1))^{( \eta_m, i_m)_{m < \omega}}$. The order relation is defined similarly as for the forcing notion $\m{P}^\beta\, \uhr\, (\eta + 1)$; but additionally, we require for $(q^\beta\, \uhr \, (\eta + 1))^{( \eta_m, i_m)_{m < \omega}} \leq (p^\beta\, \uhr\, (\eta + 1))^{( \eta_m, i_m)_{m < \omega}}$ that the linking property below $\kappa_\eta$ holds for all $(\eta_m, i_m)$ with $\eta_m \in \Lim$, $\eta_m > \eta$. 

\begin{definition} \label{defpbetauhreta+1kompliziert} Let $(\m{P}^\beta\, \uhr\, (\eta + 1))^{( \eta_m, i_m)_{m < \omega}}$ denote the collection of all $(p^\beta\, \uhr\, (\eta + 1))^{( \eta_m, i_m)_{m < \omega}}$ such that $p \in \ol{\m{P}}$ (i.e. $p \in \m{P}$ with $| \{ (\sigma, i) \in \supp p_0\ | \ \sigma > \eta \mbox{ or } i \geq \beta\}| = \aleph_0$); together with $(\m{1}^\beta_{\eta + 1})^{( \eta_m, i_m)_{m < \omega}}$ as the maximal element. \\[-3mm]

%\colorbox{red}{FRAGE: $\emptyset$ beim maximalen Element weglassen?}

For conditions $p$, $q \in \ol{\m{P}}$, let $(q^\beta\, \uhr\, (\eta + 1))^{( \eta_m, i_m)_{m < \omega}} \leq (p^\beta\, \uhr\, (\eta + 1))^{( \eta_m, i_m)_{m < \omega}}$ if \begin{itemize} \item $\wt{X}_q \supseteq \wt{X}_p$, \item $\big(q_\ast\, \uhr\, \kappa_\eta^2, (q^\sigma_i, b^\sigma_i)_{\sigma \leq \eta, i < \beta}, (q^\sigma\, \uhr \, (\beta\, \times\, \dom_y q^\sigma))_{\sigma \leq \eta}\big) \leq \big(p_\ast\, \uhr\, \kappa_\eta^2, (p^\sigma_i, a^\sigma_i)_{\sigma \leq \eta, i < \beta}, (p^\sigma\, \uhr \, (\beta\, \times\, \dom_y p^\sigma))_{\sigma \leq \eta}\big)$ regarded as conditions in $\m{P}$, \item $\forall\, \eta_m > \eta \ : \ q^{\eta_m}_{i_m} \, \uhr\, \kappa_\eta \supseteq p^{\eta_m}_{i_m}\, \uhr\, \kappa_\eta$, \item $\forall\, \eta_m > \eta\; , \; (\eta_m, i_m) \in \supp p\ : \ b^{\eta_m}_{i_m} = a^{\eta_m}_{i_m}$,
% \mbox{ whenever } (\eta_m, i_m) \in \supp p$, 
\item for all intervals $[\kappa_{\nu, j}, \kappa_{\nu, j + 1}) \subseteq \kappa_\eta$ and $\eta_m > \eta$ with $a^{\eta_m}_{i_m}\, \cap\, [\kappa_{\nu, j}, \kappa_{\nu, j + 1}) = \{\xi\}$, it follows that $q^{\eta_m}_{i_m} (\zeta) = q_\ast (\xi, \zeta)$ whenever $\zeta \in (\dom q \setminus \dom p)\, \cap\, [\kappa_{\nu, j}, \kappa_{\nu, j + 1})$.\end{itemize} 

\end{definition}

%\colorbox{yellow}{FRAGE: Evtl $(p^\beta\, \uhr\, (\eta + 1))^{( \eta_m, i_m)_{m < \omega}} := \emptyset$ definieren für gewisse $p \in \m{P}$? NEIN!! DAS GEHT NICHT!}

%Recall that we are working in the case that $\beta < \alpha_{\wt{\eta}}$ or $(\wt{\eta}, \gamma)\, \cap\, Lim \neq \emptyset$, so there are always enough indices $(\sigma, i)$ with $\sigma \in Lim$ and $\sigma > \eta$ or $i \geq \beta$. \\[-3mm]

%\colorbox{yellow}{ACHTUNG - sollte man weiter oben $\m{1}^\beta\, \uhr\, (\eta + 1)$ oder $\m{1}^\beta_{\eta + 1}$ schreiben?}

%\colorbox{yellow}{oder besser keine Namen für Ordnungsrelation und das maximale Element einführen?}

Finally, for constructing our intermediate generic extension for capturing $f^\beta$, we also have to include the verticals $P^{\eta_m}\, \uhr\, [\kappa_\eta, \kappa_{\eta_m})$ for $\eta_m > \eta$. \\[-2mm]

This gives a product \[(\m{P}^\beta\, \uhr\, (\eta + 1))^{(\eta_m, i_m)_{m < \omega}}\; \times\; \prod_{m < \omega} P^{\eta_m} \, \uhr\, [\kappa_\eta, \kappa_{\eta_m}),\] which is the set of all \[ \big(\; (p^\beta\, \uhr\, (\eta + 1))^{( \eta_m, i_m)_{m < \omega}}\; , \; (p^{\eta_m}_{i_m}\, \uhr\, [\kappa_\eta, \kappa_{\eta_m}))_{m < \omega} \; \big)\] such that $p \in \ol{\m{P}}$ (i.e. $p \in \m{P}$ with $| \{ (\sigma, i) \in \supp p_0\ | \ \sigma > \eta \mbox{ or } i \geq \beta\}| = \aleph_0$); together with a maximal element $(\ol{\m{1}}^\beta_{\eta + 1})^{( \eta_m, i_m)_{m < \omega}}$. \\[-3mm]

Then \[(G^\beta\, \uhr\, (\eta + 1))^{( \eta_m, i_m)_{m < \omega}}\, \times\, \prod_{m < \omega} G^{\eta_m}_{i_m}\, \uhr\, [\kappa_\eta, \kappa_{\eta_m})\] is the set of all $\big(\, (p^\beta\, \uhr\, (\eta + 1))^{( \eta_m, i_m)_{m < \omega}}\; , \; (p^{\eta_m}_{i_m}\, \uhr\, [\kappa_\eta, \kappa_{\eta_m}))_{m < \omega}\, \big)$ such that there exists $q \in G\, \cap\, \ol{\m{P}}$ with $(q^\beta\, \uhr \, (\eta + 1))^{(\eta_m ,i_m)_{m < \omega}} \leq (p^\beta\, \uhr\, (\eta + 1))^{( \eta_m, i_m)_{m < \omega}}$ and $q^{\eta_m}_{i_m} \, \uhr\, [\kappa_\eta, \kappa_{\eta_m}) \supseteq p^{\eta_m}_{i_m}\, \uhr\, [\kappa_\eta, \kappa_{\eta_m})$ for all $m < \omega$; together with the maximal element $(\ol{\m{1}}^\beta_{\eta + 1})^{( \eta_m, i_m)_{m < \omega}}$. \\[-3mm]

In order to show that $(G^\beta\, \uhr\, (\eta + 1))^{( \eta_m, i_m)_{m < \omega}}\, \times\, \prod_{m < \omega} G^{\eta_m}_{i_m}\, \uhr\, [\kappa_\eta, \kappa_{\eta_m})$ is a $V$-generic filter on $(\m{P}^\beta\, \uhr\, (\eta + 1))^{ (\eta_m, i_m)_{m <  \omega}}\; \times\; \prod_{m < \omega} P^{\eta_m} \, \uhr\, [\kappa_\eta, \kappa_{\eta_m})$, we proceed similarly as in Proposition \ref{projectionrho}:

\begin{prop} \label{projectionrho2} The map $(\rho^\beta)^{( \eta_m, i_m)_{m < \omega}}\,: \,\ol{\m{P}} \rightarrow (\m{P}^\beta\, \uhr\, (\eta + 1))^{ (\eta_m, i_m)_{m < \omega}}\; \times\; \prod_{m < \omega} P^{\eta_m} \, \uhr\, [\kappa_\eta, \kappa_{\eta_m})$, \[p \mapsto \big(\, (p^\beta\, \uhr\, (\eta + 1))^{( \eta_m, i_m)_{m < \omega}}\; , \; (p^{\eta_m}_{i_m}\, \uhr\, [\kappa_\eta, \kappa_{\eta_m}))_{m < \omega}\, \big) \] in the case that $|\{ (\sigma, i) \in \supp p_0\ | \ \sigma > \eta\, \vee\, i \geq \beta\}| = \aleph_0$, and $\m{1} \mapsto (\ol{\m{1}}^\beta_{\eta + 1})^{( \eta_m, i_m)_{m < \omega}}$, is a projection of forcing posets. \end{prop}

\begin{proof}We closely follow the proof of Proposition \ref{projectionrho}. Consider $\ol{p} \in \ol{\m{P}}$ with $|\{ (\sigma, i) \in \supp \ol{p}_0\ | \ \sigma > \eta\, \vee\, i \geq \beta\}| = \aleph_0$, and a condition \[q = \Big(\, q_\ast\, \uhr\, \kappa_\eta^2\, , \, (q^\sigma_i, b^\sigma_i)_{\sigma \leq \eta, i < \beta}\, , \, (q^{\eta_m}_{i_m}\, \uhr\, \kappa_\eta, b^{\eta_m}_{i_m}\, \cap\, \kappa_\eta)_{m < \omega, \eta_m > \eta}\,, \]\[ (q^\sigma)_{\sigma \leq \eta}\, , \, \wt{X}_q \,, \,(q^{\eta_m}_{i_m}\, \uhr\, [\kappa_\eta, \kappa_{\eta_m}))_{m < \omega}\, \big)\] in \[(\m{P}^\beta\, \uhr\, (\eta + 1))^{(\eta_m, i_m)_{m < \omega}}\, \times\, \prod_{m < \omega} P^{\eta_m}\, \uhr\, [\kappa_\eta, \kappa_{\eta_m})\] with $q \leq (\rho^\beta)^{( \eta_m, i_m)_{m < \omega}} (\ol{p})$. We have to construct $\ol{q} \leq \ol{p}$, $\ol{q} = (\ol{q}_\ast, (\ol{q}^\sigma_i, \ol{b}^\sigma_i)_{\sigma, i}, (\ol{q}^\sigma)_\sigma)$, such that $(\rho^\beta)^{( \eta_m, i_m)_{m < \omega}} (\ol{q}) \leq q$. \\[-3mm]

%\colorbox{yellow}{TO DO EVTL: $q_\ast\, \uhr\, \kappa_\eta^2$ anstatt $q_\ast$ hier und auch in Proposition \ref{projectionrho} ändern?}

%FRAGE: Eher $\wt{\m{P}}^\beta\, \uhr\, (\eta + 1)$ schreiben anstatt $\m{P}^\beta\, \uhr\,(\eta + 1))^{( \eta_m, i_m)_{m < \omega}}$?
%TO DO!

%W.l.o.g. we can assume that $(\eta_m, i_m) \in \supp \ol{p}_0$ for all $\eta_m \in Lim$. \\[-3mm]  NEIN!!

%\colorbox{red}{ACHTUNG - so klar ist das NICHT!!} \\[-3mm]

We start with $\ol{q}_0$.
\begin{itemize} \item Similarly as in Proposition \ref{projectionrho}, we construct $\supp_\ast = \{ (\check{\eta}, m_k)\ | \ k < \omega\}$ such that $\check{\eta} > \eta$ or $m_k \geq \beta$, and $(\check{\eta}, m_k) \notin \supp \ol{p}_0\, \cup\, \supp q_0$ for all $k < \omega$; with the additional property that for all $k < \omega$, we have $(\check{\eta}, m_k) \notin \{ (\eta_m, i_m)\ | \ m < \omega\, , \, \eta_m \in \Lim\}$. We set $\supp \ol{q}_0 = \supp \ol{p}_0\, \cup\, \supp q_0\, \cup\, \supp_\ast\, \cup \, \{ (\eta_m, i_m)\ | \ m < \omega\, , \, \eta_m \in \Lim\}$. 

\item Next, we define the linking ordinals $\ol{b}^\sigma_i$ for $(\sigma, i) \in \supp \ol{q}_0$, such that $\wt{X}_{\ol{q}} \supseteq \wt{X}_q$ holds: \\[-4mm]

First, we consider the case that $(\sigma, i) \notin \{(\eta_m, i_m)\ | \ m < \omega\, , \, \eta_m \in \Lim\}$. For $(\sigma, i) \in \supp q_0$, we let $\ol{b}^\sigma_i := b^\sigma_i \supseteq \ol{a}^\sigma_i$, and $\ol{b}^\sigma_i := \ol{a}^\sigma_i$ in the case that $(\sigma, i) \in \supp \ol{p}_0 \setminus \supp q_0$. 
 
We construct $(\ol{b}^{\,\check{\eta}}_{m_k}\ | \ k < \omega)$ as in Proposition \ref{projectionrho}. \\[-4mm]

%, with the additional property that $\ol{b}^{\,\check{\eta}}_{m_k}\, \cap\, \ol{b}^{\eta_m}_{i_m} = \emptyset$ for all $k < \omega$ and $m < \omega$ with $\eta_m \in \Lim$. (This is possible, since we have arranged that $\bigcup \{ \ol{b}^{\eta_m}_{i_m}\ | \ m < \omega\, , \, \eta_m \in \Lim\}\, \cap\, \wt{X}_q = \emptyset$.
After that, we define the linking ordinals $(\ol{b}^{\eta_m}_{i_m}\ | \ m < \omega\, , \, \eta_m \in \Lim)$ with the following properties:
\begin{itemize} \item As usual, every $\ol{b}^{\eta_m}_{i_m}$ is a subset of $\kappa_{\eta_m}$ that hits any interval $[\kappa_{\nu, j}, \kappa_{\nu, j + 1}) \subseteq \kappa_{\eta_m}$ in exactly one point. \item The $\ol{b}^{\eta_m}_{i_m}$ are pairwise disjoint, 
%and for every $m < \omega$, we have $\ol{b}^{\eta_m}_{i_m}\, \cap\, \wt{X}_q = \emptyset$, 
and $\ol{b}^{\eta_m}_{i_m}\, \cap\, \ol{b}^\sigma_i = \emptyset$ for every $m < \omega$ and $(\sigma, i) \in \supp \ol{q}_0$
%\supp q_0\, \cup\, \supp \ol{p}_0\, \cup\, \{(\check{\eta}, m_k)\ | \ k < \omega\}$ 
with $(\sigma, i) \neq (\eta_m, i_m)$. \item For every $(\eta_m, i_m) \in \supp \ol{p}_0$, we set $\ol{b}^{\eta_m}_{i_m} := a^{\eta_m}_{i_m}$; for every $(\eta_m, i_m) \in \supp q_0 \setminus \supp \ol{p}_0$ with $\eta_m \leq \eta$, we set $\ol{b}^{\eta_m}_{i_m} := b^{\eta_m}_{i_m}$; and whenever $(\eta_m, i_m) \in \supp q_0 \setminus \supp \ol{p}_0$ with $\eta_m > \eta$, we let $\ol{b}^{\eta_m}_{i_m} \supseteq b^{\eta_m}_{i_m}\, \cap \, \kappa_\eta$.  \end{itemize}

This concludes our construction of the linking ordinals $\ol{b}^\sigma_i$.

\item We define $\dom \ol{q}_0 = \bigcup_{\nu, j} [\kappa_{\nu, j}, \delta_{\nu, j})$ as follows: Let $dom := \dom \ol{p}_0\, \cup\, \dom q_0\, \cup \, \bigcup_{\eta_m \in Lim} \dom q^{\eta_m}_{i_m}\, \uhr\, [\kappa_\eta, \kappa_{\eta_m})$. For every interval $[\kappa_{\nu, j}, \kappa_{\nu, j + 1})$ with $dom\, \cap\, [\kappa_{\nu, j}, \kappa_{\nu, j + 1}) = \emptyset$, we set $\delta_{\nu, j} := \kappa_{\nu, j}$; and whenever $dom\, \cap\, [\kappa_{\nu, j}, \kappa_{\nu, j + 1}) \neq \emptyset$, we pick $\delta_{\nu, j} \in (\kappa_{\nu, j}, \kappa_{\nu, j + 1})$ with the property that $dom\, \cap\, [\kappa_{\nu, j}, \kappa_{\nu, j + 1}) \subseteq [\kappa_{\nu, j}, \delta_{\nu, j})$, and $\ol{b}^\sigma_i\, \cap\, [\kappa_{\nu, j}, \kappa_{\nu, j + 1}) \subseteq [\kappa_{\nu, j}, \delta_{\nu, j})$ for all $(\sigma, i) \in \supp \ol{q}_0$.

%also the linking ordinals $\ol{b}^\sigma_i\, \cap\, [\kappa_{\nu, j}, \kappa_{\nu, j + 1})$ 
%for $(\sigma, i) \in \supp \ol{q}_0$ 
%are contained in $[\kappa_{\nu, j}, \delta_{\nu, j})$. \\ 

Since $\dom \,\ol{p}_0$, $\dom \,q_0$ and the domains $\dom\, q^{\eta_m}_{i_m}\, \uhr\, [\kappa_\eta, \kappa_{\eta_m})$ are bounded below all regular cardinals, this is also true for $dom$ and $\dom \,\ol{q}_0$.

\item We take $\ol{q}_\ast \, \uhr\, \kappa_\eta^2 \supseteq q_\ast\, \uhr\, \kappa_\eta^2$ arbitrary on the given domain. 

The verticals $\ol{q}^\sigma_i\, \uhr \, \kappa_\eta$ for $(\sigma, i) \in (\supp q_0\, \cup \, \supp \ol{p}_0) \setminus \{ (\eta_m, i_m)\ | \ m < \omega\, , \, \eta_m \in \Lim\}$ can be defined according to the \textit{linking property} as in Proposition \ref{projectionrho}. 

The verticals $\ol{q}^{\,\hat{\eta}}_{\,m_k}\, \uhr\, \kappa_\eta$ with $(\hat{\eta}, m_k) \in \supp_\ast$ can be set arbitrarily on the given domain.

%For $(\sigma, i) \in \supp q$ (then $\sigma \leq \eta$, $i < \beta$), we define $\ol{q}^\sigma_i \supseteq q^\sigma_i$ on the given domain according to the \textit{linking property}: For every interval $[\kappa_{\nu, j}, \kappa_{\nu, j + 1}) \subseteq \kappa_\eta$ and $\zeta \in \dom \ol{q}_0 \setminus \dom q_0$, we set $\ol{q}^\sigma_i (\zeta) := \ol{q}_\ast (\xi, \zeta)$, where $\{\xi\} = b^\sigma_i\, \cap\, [\kappa_{\nu, j}, \kappa_{\nu, j + 1})$. (Note that $\xi \in \dom \ol{q}_0$ follows by construction.) \\ In the case that $(\sigma, i) \in \supp \ol{p}_0 \setminus \supp q_0$ (then $\sigma > \eta$ or $i \geq \beta$), but $(\sigma, i) \notin \{(\eta_m, i_m)\ | \ \eta_m > \eta\}$, we proceed similarly for intervals $[\kappa_{\nu, j}, \kappa_{\nu, j + 1}) \subseteq \kappa_\eta$, and define $\ol{q}^\sigma_i \, \uhr\, [\kappa_{\nu, j}, \kappa_{\nu, j + 1}) \supseteq \ol{p}^\sigma_i\, \uhr\, [\kappa_{\nu, j}, \kappa_{\nu, j + 1})$ according to the \textit{linking property} as before. On intervals $[\kappa_{\nu, j}, \kappa_{\nu, j + 1}) \subseteq [\kappa_\eta, \kappa_\gamma)$, we set $\ol{q}^\sigma_i \, \uhr\, [\kappa_{\nu, j}, \kappa_{\nu, j + 1}) \supseteq \ol{p}^\sigma_i\, \uhr\, [\kappa_{\nu, j}, \kappa_{\nu, j + 1})$ arbitrary on the given domain. \\ We can set the $\ol{q}^{\hat{eta}}_{m_i}$ arbitary on the given domain. EVTL nicht so ausführlich aufschreiben? 

Now, consider $(\eta_m, i_m)$ with $\eta_m \in \Lim$. In the case that $(\eta_m, i_m) \in \supp q_0$ with $\eta_m \leq \eta$, we can proceed as before, and define $\ol{q}^{\eta_m}_{i_m} \supseteq q^{\eta_m}_{i_m}$ according to the \textit{linking property} as in Proposition \ref{projectionrho}. \\
Concerning the verticals $\ol{q}^{\eta_m}_{i_m}\, \uhr\, \kappa_\eta$ for $(\eta_m, i_m) \in \supp q_0$ with $\eta_m > \eta$, we define $\ol{q}^{\eta_m}_{i_m}\, \uhr\, [\kappa_{\nu, j}, \kappa_{\nu, j + 1}) \supseteq q^{\eta_m}_{i_m}\, \uhr \, [\kappa_{\nu, j}, \kappa_{\nu, j + 1})$ on intervals $[\kappa_{\nu, j}, \kappa_{\nu, j + 1}) \subseteq \kappa_\eta$ according to the \textit{linking property}, and use that we have incorporated the linking ordinals $b^{\eta_m}_{i_m}\, \cap\, \kappa_\eta$ into our forcing notion $(\m{P}^\beta\, \uhr\, (\eta + 1))^{( \eta_m, i_m)_{m < \omega}}$: For $\zeta \in (\dom \ol{q}_0 \setminus \dom q_0)\, \cap\, [\kappa_{\nu, j}, \kappa_{\nu, j + 1})$, we set $\ol{q}^{\eta_m}_{i_m} (\zeta) := \ol{q}_\ast (\xi, \zeta)$, where $\{\xi\} = b^{\eta_m}_{i_m}\, \cap\, [\kappa_{\nu, j}, \kappa_{\nu, j + 1}) = \ol{b}^{\eta_m}_{i_m}\, \cap \, [\kappa_{\nu, j}, \kappa_{\nu, j + 1})$. (Note that $\xi \in \dom \ol{q}_0$ by construction.) 

In the case that $(\eta_m, i_m) \notin \supp q_0$, it follows that also $(\eta_m, i_m) \notin \supp \ol{p}_0$, and we can set $\ol{q}^{\eta_m}_{i_m}\, \uhr\, \kappa_\eta$ arbitrarily on the given domain.

\item Next, consider an interval $[\kappa_{\nu, j}, \kappa_{\nu, j + 1}) \subseteq [\kappa_\eta, \kappa_\gamma)$. We first set the verticals $\ol{q}^\sigma_i \, \uhr\, [\kappa_{\nu, j}, \kappa_{\nu, j + 1})$ for $(\sigma, i) \in \supp \ol{q}_0$, $\sigma > \eta$, on the given domain, with the property that $\ol{q}^{\eta_m}_{i_m}\, \uhr \, [\kappa_{\nu, j}, \kappa_{\nu, j + 1}) \supseteq q^{\eta_m}_{i_m}\, \uhr \, [\kappa_{\nu, j}, \kappa_{\nu, j + 1})$ for all $m < \omega$ with $(\eta_m, i_m) \in \supp q_0$, and $\ol{q}^\sigma_i \, \uhr\, [\kappa_{\nu, j}, \kappa_{\nu, j +1}) \supseteq \ol{p}^\sigma_i\, \uhr\, [\kappa_{\nu, j}, \kappa_{\nu, j + 1})$ whenever $(\sigma, i) \in \supp \ol{p}_0$. After that, we define $\ol{q}_\ast\, \uhr\, [\kappa_{\nu, j}, \kappa_{\nu, j + 1})^2 \supseteq \ol{p}_\ast \, \uhr \, [\kappa_{\nu, j}, \kappa_{\nu, j + 1})^2$ according to the \textit{linking property}: Whenever $\zeta \in (\dom \ol{q}_0 \setminus \dom \ol{p}_0)\, \cap\, [\kappa_{\nu, j}, \kappa_{\nu, j + 1})$ and $\{\xi\} = \ol{a}^\sigma_i\, \cap\, [\kappa_{\nu, j}, \kappa_{\nu, j + 1}) = \ol{b}^\sigma_i\, \cap\, [\kappa_{\nu, j}, \kappa_{\nu, j + 1})$ for some $(\sigma, i) \in \supp \ol{p}_0$, then $\ol{q}_\ast (\xi, \zeta) := \ol{q}^\sigma_i (\zeta)$. 
%Similarly, in the case that $\zeta \in (\dom \ol{q}_0 \setminus \dom q^{\eta_m}_{i_m})\, \cap\, [\kappa_{\nu, j}, \kappa_{\nu, j + 1})$ and $\{\xi\} = b^{\eta_m}_{i_m}\, \cap \, [\kappa_{\nu, j}, \kappa_{\nu, j + 1})$, we set $\ol{q}_\ast (\xi, \zeta) := \ol{q}^{\eta_m}_{i_m} (\zeta)$. 
Otherwise, $\ol{q}_\ast (\xi, \zeta)$ can be set arbitrarily. 
%Then $\ol{q}_\ast\, \uhr\, [\kappa_\eta, \kappa_\gamma)^2$ is well-defined by the independence property. 
\end{itemize}

%\colorbox{red}{TO DO: Das müsste man nochmal DURCHGEHEN! AUCH Proposition \ref{projectionrho}!}

This defines $\ol{q}_0$. The construction of $\ol{q}_1$ is similar; and it is not difficult to see that $\ol{q} \leq \ol{p}$ with $(\rho^\beta)^{( \eta_m, i_m)_{m < \omega}} (\ol{q}) \leq q$. \\[-3mm]

Hence, $(\rho^\beta)^{( \eta_m, i_m)_{m < \omega}}$ is a projection of forcing posets.

\end{proof}

%\colorbox{yellow}{TO DO EVTL: BEZEICHNUNGEN ÄNDERN!! $\wt{\rho^\beta}$ anstatt $(\rho^\beta)^{( \eta_m, i_m)_{m < \omega}}$?}

%\colorbox{yellow}{Wegen den Bezeichnungen mal nachfragen?}

Thus, it follows that $(G^\beta\, \uhr\, (\eta + 1))^{( \eta_m, i_m)_{m < \omega}}\, \times\, \prod_{m < \omega} G^{\eta_m}_{i_m}\, \uhr\, [\kappa_\eta, \kappa_{\eta_m})$ is a $V$-generic filter on the forcing notion $(\m{P}^\beta\, \uhr\, (\eta + 1))^{( \eta_m, i_m)_{m < \omega}}\, \times\, \prod_{m < \omega} P^{\eta_m}\, \uhr\, [\kappa_\eta, \kappa_{\eta_m})$. \\[-2mm]

The aim of Chapter \ref{6.2} B) is to show that $f^\beta$ is contained in the intermediate $V$-generic extension $V[\, (G^\beta\, \uhr\, (\eta + 1))^{( \eta_m, i_m)_{m < \omega}}\, \times\, \prod_{m < \omega} G^{\eta_m}_{i_m}\, \uhr\, [\kappa_\eta, \kappa_{\eta_m})\, ]$.

\begin{definition} \label{deffbetaprime}Let $(f^\beta)^\prime$ denote the set of all $(X, \alpha)$ for which there exists an $\eta$-\textit{good pair} $\varrho = \big( (a_m)_{m < \omega}, (\ol{\sigma}_m, \ol{i}_m)_{m < \omega}\big)$ with $\ol{i}_m < \beta$ for all $m < \omega$ such that \[X = \dot{X}^{\prod_m G_\ast (a_m)\, \times \, \prod_m G^{\ol{\sigma}_m}_{\ol{i}_m}},\] and there is a condition $p \in \m{P}$ with \begin{itemize} \item $|\{ (\sigma, i) \in \supp p_0\ | \ \sigma > \eta \mbox{ or } i \geq \beta\}| = \aleph_0$, \item $p \Vdash_s \big(\tau_\varrho (\dot{X}), \alpha\big) \in \dot{f}$ 
%(where $\wt{\dot{X}} := \tau_\varrho (\dot{X})$ denotes the transformation of names introduced in Definition \ref{tauvarrho}), 
\item $\big(\, (p^\beta\, \uhr\, (\eta + 1))^{( \eta_m, i_m)_{m < \omega}}\; , \; ( p^{\eta_m}_{i_m}\, \uhr\, [\kappa_\eta, \kappa_{\eta_m}))_{m < \omega}\, \big) \in (G^\beta\, \uhr\, (\eta + 1))^{( \eta_m, i_m)_{m < \omega}}\, \times\, \prod_{m < \omega} G^{\eta_m}_{i_m} \, \uhr \, [\kappa_\eta, \kappa_{\eta_m})$, \item $\forall\, \eta_m \in \Lim \; : \; (\eta_m, i_m) \in \supp p_0$ with $a^{\eta_m}_{i_m} = g^{\eta_m}_{i_m}$. \end{itemize} \end{definition}

Then $(f^\beta)^\prime \in V[\, (G^\beta\, \uhr\, (\eta + 1))^{( \eta_m, i_m)_{m < \omega}}\, \times\, \prod_{m < \omega} G^{\eta_m}_{i_m}\, \uhr\, [\kappa_\eta, \kappa_{\eta_m})\, ]$, since the sequence $(g^{\eta_m}_{i_m}\ | \ m < \omega)$ is contained in the ground model $V$. \\[-3mm]

We will now use an isomorphism argument and show that $f^\beta = (f^\beta)^\prime$.

%the following hold: \[p^\beta\, \uhr \, (\eta + 1) \in G^\beta\, \uhr\, (\eta + 1)\; , \; \prod_{j < \omega} p^{\eta_j}_{i_j} \in \prod_{j < \omega} G^{\eta_j}_{i_j}\; , \; \prod_{j < \omega\; , \; \eta_j \in Lim} a^{\eta_j}_{i_j} = \prod_{j < \omega\; , \; \eta_j \in Lim} g^{\eta_j}_{i_j}.\]\end{definition}  

%EVTL irgendwo erwähnen, für was Indizes $\eta_j$ und $\ol{\eta}_j$ verwendet werden?

%\colorbox{red}{ACHTUNG - wurde überall zwischen $\Vdash$ und $\Vdash_s$ unterschieden?}

\begin{prop} \label{fbetafbetaprime} $f^\beta = (f^\beta)^\prime$.\end{prop}
\begin{proof} By the forcing theorem, it follows that $(f^\beta)^\prime \supseteq f^\beta$. Assume towards a contradiction, there was $(X, \alpha) \in (f^\beta)^\prime \setminus f^\beta$. Let \[X = \dot{X}^{\prod_{m < \omega} G_\ast (a_m)\, \times\, \prod_{m < \omega} G^{\ol{\sigma}_m}_{\ol{i}_m}}\] for an $\eta$-good pair $\varrho = \big( (a_m)_{m < \omega},  (\ol{\sigma}_m, \ol{i}_m)_{m < \omega} \big)$ with $\ol{i}_m < \beta$ for all $m < \omega$. Take $p \in \m{P}$ as in Definition \ref{deffbetaprime} with $p \Vdash_s (\tau_{\varrho (\dot{X})}, \alpha) \in \dot{f}$; and
%with $p \Vdash_s (\wt{\dot{X}}, \alpha) \in \dot{f}$, $| \{ (\sigma, i) \in \supp p_0\ | \ \sigma > \eta\mbox{ or } i \geq \beta\}| = \omega$, and  (\wt{\dot{X}}, \alpha) \in \dot{f}$ and $|\{ (\sigma, i) \in \supp p_0\ | \ \sigma > \eta \mbox{ or } \sigma \geq \eta, i \geq \beta\}| = \omega$, and $\big(\, (p^\beta\, \uhr\, (\eta + 1))^{( \eta_m, i_m)_{m < \omega}}, ( p^{\eta_m}_{i_m}\, \uhr\, [\kappa_\eta, \kappa_{\eta_m}))_{m < \omega}\, \big) \in (G^\beta\, \uhr\, (\eta + 1))^{( \eta_m, i_m)_{m < \omega}}\, \times\, \prod_{m < \omega} G^{\eta_m}_{i_m} \, \uhr \, [\kappa_\eta, \kappa_{\eta_m})$.
%$p^\beta\, \uhr \, (\eta + 1) \in G^\beta\, \uhr \, (\eta + 1)$, $\prod p^{\eta_m}_{i_m} \in \prod G^{\eta_m}_{i_m}$, and $\prod a^{\eta_m}_{i_m} = \prod g^{\eta_m}_{i_m}$. 
since $(X, \alpha) \notin f^\beta$, we can take $p^\prime \in G$ with $p^\prime \Vdash_s (\tau_\varrho (\dot{X}), \alpha) \notin \dot{f}$ and $(\eta_m, i_m) \in \supp p^\prime_0$ for all $\eta_m \in Lim$. \\[-3mm] 

Our first step will be to extend the conditions $p$ and $p^\prime$ and obtain $\ol{p} \leq p$, $\ol{p}^\prime \leq p^\prime$ such that $\ol{p}$ and $\ol{p}^\prime$ have \tbl the same shape\tbr\, similarly as in the \textit{Approximation Lemma} \ref{approx}; but additionally, $\ol{p}^\beta\, \uhr\, (\eta + 1) = (\ol{p}^\prime)^\beta\, \uhr\, (\eta + 1)$ holds, and $\ol{p}^{\eta_m}_{i_m} = (\ol{p}^\prime)^{\eta_m}_{i_m}$ for all $m < \omega$, and $\ol{a}^{\eta_m}_{i_m} = (\ol{a}^\prime)^{\eta_m}_{i_m}$ for all $m < \omega$ with $\eta_m \in Lim$.

After that, we construct an isomorphism $\pi$ such that firstly, $\pi \ol{p} = \ol{p}^\prime$; secondly, $\pi$ should not disturb the forcing $\m{P}^\beta\, \uhr \, (\eta + 1)$ (which will imply $\pi \,\ol{\tau_{\varrho} (\dot{X})}^{D_\pi} = \ol{\tau_{\varrho} (\dot{X})}^{D_\pi}$); and thirdly $[\pi]$ should be contained in the intersection $\bigcap_m Fix (\eta_m, i_m)\, \cap \, \bigcap_m H^{\lambda_m}_{k_m}$ (which implies $\pi \ol{f}^{D_\pi} = \ol{f}^{D_\pi}$).

%\colorbox{yellow}{FRAGE: Könnte man nicht $\wt{\dot{X}}$ überall ersetzen durch $\tau_{\varrho} (\dot{X})$?}

%\colorbox{yellow}{Beide Bezeichnunge sind SCHLECHT!}

Then from $\ol{p} \Vdash_s (\tau_{\varrho} (\dot{X}), \alpha) \in \dot{f}$ it follows $\pi \ol{p} \Vdash_s (\pi \, \ol{\tau_{\varrho}(\dot{X})}^{D_\pi}, \alpha) \in \pi \ol{f}^{D_\pi}$. Together with $\ol{p}^\prime \Vdash_s (\ol{\tau_{\varrho} (\dot{X})}^{D_\pi}, \alpha) \notin \ol{f}^{D_\pi}$, this gives our desired contradiction. \\[-2mm]

In order to make such an isomorphism $\pi$ possible, the extensions $\ol{p} \leq p$ and $\ol{p}^\prime \leq p^\prime$ will satisfy the following properties:

\begin{itemize} \item $\ol{\supp}_0 := \supp \ol{p}_0 = \supp \ol{p}^\prime_0$ \item $\ol{\dom}_0 := \dom \ol{p}_0 = \dom \ol{p}^\prime_0$ \item $\bigcup \ol{a} := \bigcup_{(\sigma, i) \in \ol{\supp}_0} \ol{a}^\sigma_i = \bigcup_{ (\sigma, i) \in \ol{\supp}_0} (\ol{a}^\prime)^\sigma_i$ \item $\forall\; \nu\,,\,j\ : \ \big( \ol{\dom}_0 \, \cap \, [\kappa_{\nu, j}, \kappa_{\nu, j + 1}) \neq \emptyset \Rightarrow \bigcup \ol{a} \, \cap \, [\kappa_{\nu, j}, \kappa_{\nu, j + 1}) \subseteq \ol{\dom}_0\big)$ \item $\ol{\supp}_1 := \supp \ol{p}_1 = \supp \ol{p}^\prime_1$ \item $\forall\, \sigma \in \ol{\supp}_1\ : \ \ol{\dom}_1 (\sigma) := \dom \ol{p}^\sigma = \dom (\ol{p}^\prime)^\sigma$. \end{itemize}

Additionally, we want: \begin{itemize} \item $\forall\, m < \omega\ : \ \ol{p}^{\eta_m}_{i_m} = (\ol{p}^\prime)^{\eta_m}_{i_m}$ \item  $\forall\, m < \omega\; , \; \eta_m \in \Lim\ : \ \ol{a}^{\eta_m}_{i_m} = (\ol{a}^\prime)^{\eta_m}_{i_m}$ \item $\ol{p}^\beta\, \uhr\, (\eta + 1) = (\ol{p}^\prime)^\beta\, \uhr\, (\eta + 1)$, i.e. \begin{itemize} \item $\ol{p}_\ast\, \uhr\, \kappa_\eta^2 = \ol{p}^\prime_\ast \, \uhr \, \kappa_\eta^2$ \item $\forall\, \sigma \in Lim$, $\sigma \leq \eta$, $i < \min \{\alpha_\sigma, \beta\}\ : \ \ol{p}^\sigma_i = (\ol{p}^\prime)^\sigma_i$, $\ol{a}^\sigma_i = (\ol{a}^\prime)^\sigma_i$ \item $\forall\, \sigma \in Succ$ $\sigma \leq \eta\ \; : \ \ol{p}^\sigma\, \uhr\, (\beta\, \times\, \dom_y \ol{p}^\sigma) = (\ol{p}^\prime)^\sigma\, \uhr\, (\beta\, \times\, \dom_y (\ol{p}^\prime)^\sigma)$.\end{itemize} 

Then it follows that $\wt{X}_{\ol{p}} = \wt{X}_{\ol{p}^\prime}$.
%\colorbox{red}{ACHTUNG - $X$ wird mehrfach verwendet!}
%\colorbox{red}{Für die Linking Ordinalzahlen einen anderen Buchstaben versuchen?}
\end{itemize}

Note that $\ol{a}^{\eta_m}_{i_m} = (\ol{a}^\prime)^{\eta_m}_{i_m}$ for $\eta_m \in \Lim$ follows automatically, since $a^{\eta_m}_{i_m} = (a^\prime)^{\eta_m}_{i_m} = g^{\eta_m}_{i_m}$ by assumption. \\[-3mm]

%ACHTUNG - $\Vdash$ oder $\Vdash_s$? UNTERSCHIED??

Now, we construct the conditions $\ol{p}$ and $\ol{p}^\prime$. \\[-3mm]

We start with the linking ordinals $\ol{a}^\sigma_i$ and $(\ol{a}^\prime)^\sigma_i$, with our aim that $\bigcup_{\sigma, i} \ol{a}^\sigma_i = \bigcup_{\sigma, i} (\ol{a}^\prime)^\sigma_i = : \bigcup \ol{a}$. We closely follows our construction from the \textit{Approximation Lemma} \ref{approx}; but now, some extra care is needed, since we additionally have to make sure that $\ol{a}^\sigma_i = (\ol{a}^\prime)^\sigma_i$ holds for all $\sigma \leq \eta$, $i < \beta$. \\[-3mm]

Similarly as in the \textit{Approximation Lemma} \ref{approx}, let 
\[s := \kappa_{\ol{\delta}} := \sup \{\kappa_\sigma\ | \ \sigma \in \Lim\; , \; \exists\, i < \alpha_\sigma\ (\sigma, i) \in \supp p_0\, \cup \supp p^\prime_0\}. \] 

Recall that we are assuming $\boldsymbol{\beta < \alpha_{\wt{\eta}}}$ \textbf{or} $\boldsymbol{Lim \, \cap \, (\wt{\eta}, \gamma) \neq \emptyset}$, where $\wt{\eta} := \max \{\sigma \leq \eta\ | \ \sigma \in Lim\}$. \\[-3mm]

%$\wt{\eta} := \sup \{\sigma < \eta\ | \ \sigma \in Lim\}$
In the case that $\kappa_{\ol{\delta}} = \kappa_\gamma$, we set $\ol{\gamma} := \ol{\delta}$ and take $((\sigma_k, l_k)\ | \ k < \omega)$ such that $\sup \{\kappa_{\sigma_k}\ | \ k < \omega\} = \kappa_{\ol{\gamma}} = \kappa_\gamma$, and $(\sigma_k, l_k) \notin \supp p_0\, \cup \, \supp p^\prime_0$ for all $k < \omega$, with the additional property
% By our assumption that $\beta < \alpha_{\wt{\eta}}$ or $Lim \, \cap \, (\wt{\eta}, \gamma) \neq \emptyset$ (where $\wt{\eta} := \sup \{\sigma \leq \eta\ | \ \sigma \in Lim\}$), we can arrange that for all $k < \omega$, 
that $\sigma_k > \wt{\eta}$ or $l_k \geq \beta$ for all $k < \omega$. 

If $\kappa_{\ol{\delta}} < \kappa_\gamma$ and $Lim\, \cap\, (\wt{\eta}, \gamma) \neq \emptyset$, let $\ol{\gamma} \in \Lim\, \cap\, (\wt{\eta}, \gamma)$ with $\ol{\gamma} \geq \ol{\delta}$, and take $((\sigma_k, l_k)\ | \ k < \omega)$ such that $(\sigma_k, l_k) = (\ol{\gamma}, l_k) \notin \supp p_0\, \cup \, \supp p^\prime_0$ for all $k < \omega$.

Finally, if $\kappa_{\ol{\delta}} < \kappa_\gamma$ and $Lim\, \cap\, (\wt{\eta}, \gamma) = \emptyset$, then $\beta < \alpha_{\wt{\eta}}$ follows. In this case, let $\ol{\gamma} := \wt{\eta} \geq \ol{\delta}$, and take $((\sigma_k, l_k)\ | \ k < \omega)$ with $(\sigma_k, l_k) = (\ol{\gamma}, l_k) \notin \supp p_0\, \cup \, \supp p^\prime_0$ for all $k < \omega$; with the additional property that $l_k \geq \beta$ for all $k < \omega$. \\[-3mm]

Let \[\ol{\supp}_0 := \supp \ol{p}_0 := \supp \ol{p}^\prime_0 := \supp p_0\, \cup \, \supp p^\prime_0\, \cup \, \{ (\sigma_k, l_k)\ | \ k < \omega\}.\]

We now construct the linking ordinals $\ol{a}^\sigma_i$. 
For any $(\sigma, i) \in \supp p_0$, we set $\ol{a}^\sigma_i := a^\sigma_i$; and whenever $(\sigma, i) \in \supp p^\prime_0 \setminus \supp p_0$ with $\sigma \leq \eta$, $i < \beta$, then $\ol{a}^\sigma_i := (a^\prime)^\sigma_i$. \\[-3mm]

Now, take a set $Z \subseteq \kappa_{\ol{\gamma}}$ such that for all intervals $[\kappa_{\nu, j}, \kappa_{\nu, j + 1}) \subseteq \kappa_{\ol{\gamma}}$, we have $|Z\, \cap \, [\kappa_{\nu, j}, \kappa_{\nu, j + 1})| = \aleph_0$, and $Z\, \cap\, \Big( \bigcup_{(\sigma, i) \in \supp p_0}\,a^\sigma_i\, \cup \, \bigcup_{ (\sigma, i) \in \supp p^\prime_0}\,(a^\prime)^\sigma_i \Big) = \emptyset$. Let \[\ol{Z} := Z\, \cup \, \bigcup_{\sigma, i} a^\sigma_i\, \cup \, \bigcup_{\sigma, i} (a^\prime)^\sigma_i.\] Our aim is to construct $\ol{p}$ and $\ol{p}^\prime$ with $\bigcup_{\sigma, i} \ol{a}^\sigma_i = \bigcup_{\sigma, i} (\ol{a}^\prime)^\sigma_i = \bigcup \ol{a} := \ol{Z}$. \\[-3mm]

%\colorbox{yellow}{ACHTUNG - müsste man die Kardinaltitäten nicht $\aleph_0$ schreiben?}

%Every interval $[\kappa_{\nu, j}, \kappa_{\nu, j + 1}) \subseteq \kappa_{\ol{\gamma}}$ is treated separately. \\[-3mm] 

Fix an interval $[\kappa_{\nu, j}, \kappa_{\nu, j + 1}) \subseteq \kappa_{\ol{\gamma}}$. Let \[ \ol{Z}_{\nu, j} := \Bigl( \: \bigcup \{ a^\sigma_i \ | \ (\sigma, i) \in \supp p_0\}\, \cup \, \bigcup \{ (a^\prime)^\sigma_i\ | \ (\sigma, i) \in \supp p^\prime_0\, , \, \sigma \leq \eta, i < \beta\}\ \Bigr)\ \cap\] \[\cap \ [\kappa_{\nu, j}, \kappa_{\nu, j + 1})\] and 
\[\{\xi_k (\nu, j)\ | \ k < \omega\} := \big( \ol{Z} \, \cap \, [\kappa_{\nu, j}, \kappa_{\nu, j + 1}) \big) \setminus \ol{Z}_{\nu, j} .\] This set has cardinality $\aleph_0$ by construction of $Z$. \\[-3mm]

Now, let \[\{(\ol{\sigma}_k, \ol{l}_k)\ | \ k < \omega\} =: \{ (\sigma, i) \in \supp \ol{p}_0 \setminus \supp p_0 \ | \ \kappa_{\nu, j} < \kappa_\sigma \mbox{ and } (\sigma > \eta \mbox{ or } i \geq \beta)\}.\]

This set also has cardinality $\aleph_0$ by construction of $\supp \ol{p}_0$. Now, for any $k < \omega$, we let \[\ol{a}^{\,\ol{\sigma}_k}_{\,\ol{l}_k}\, \cap\, [\kappa_{\nu, j}, \kappa_{\nu, j + 1}) := \{\xi_k (\nu, j) \}.\] We apply the same construction to the linking ordinals $(\ol{a}^\prime)^\sigma_i$ for $(\sigma, i) \in \supp \ol{p}^\prime_0 = \ol{\supp}_0$.
%, we obtain the linking ordinals $\ol{a}^\sigma_i$, $(\ol{a}^\prime)^\sigma_i$ for $(\sigma, i) \in \supp \ol{p}_0 = \supp (\ol{p}^\prime)_0$ such 
It is not difficult to see that $\bigcup_{\sigma, i} \ol{a}^\sigma_i = \bigcup_{\sigma, i} (\ol{a}^\prime)^\sigma_i = \bigcup \ol{a} = \ol{Z}$, the independence property holds, and $\ol{a}^\sigma_i = (\ol{a}^\prime)^\sigma_i$ whenever $\sigma \leq \eta$, $i < \beta$.  \\[-2mm]

Next, take $\ol{\dom}_0 := \dom \ol{p}_0 = \dom \ol{p}^\prime_0 = \bigcup_{\nu, j} [\kappa_{\nu, j}, \delta_{\nu, j})$ with the property that firstly, $\dom p_0\, \cup \, \dom p^\prime_0 \subseteq \ol{\dom}_0$, and secondly, for every interval $[\kappa_{\nu, j}, \kappa_{\nu, j + 1}) \subseteq \kappa_\gamma$ with $\ol{\dom}_0\, \cap\, [\kappa_{\nu, j}, \kappa_{\nu, j + 1}) \neq \emptyset$, it follows that $\ol{Z}\, \cap\, [\kappa_{\nu, j}, \kappa_{\nu, j + 1}) \subseteq \ol{\dom}_0$. \\[-2mm]

It remains to construct $\ol{p}_\ast$, $\ol{p}^\prime_\ast$, and $\ol{p}^\sigma_i$, $(\ol{p}^\prime)^\sigma_i$ for $(\sigma, i) \in \ol{\supp}_0$. \\[-3mm] 

First, we consider an interval $[\kappa_{\nu, j}, \kappa_{\nu, j + 1}) \subseteq \kappa_\eta$. \\[-3mm]
%First, we consider the case that ${\mathbf \kappa_{\nu, j} < \kappa_\eta}$. 

We start with the construction of $\ol{p}_\ast\, \uhr\, [\kappa_{\nu, j}, \kappa_{\nu, j + 1})^2 = \ol{p}^\prime_\ast\, \uhr \, [\kappa_{\nu, j}, \kappa_{\nu, j + 1})^2$. \\[-3mm]
% and we have to take care of the \textit{linking property}. \\[-3mm]

%We start with $\ol{p}_\ast\, \uhr\, [\kappa_{\nu, j}, \kappa_{\nu, j + 1})^2$ and $p^\prime_\ast \, \uhr\, [\kappa_{\nu, j}, \kappa_{\nu, j + 1})^2$. 
Let $\xi$, $\zeta \in [\kappa_{\nu, j}, \kappa_{\nu, j + 1}) \, \cap \, \ol{\dom}_0$.

\begin{itemize} \item In the case that $(\xi, \zeta) \in \dom p_0\, \times\, \dom p_0$, we set $\ol{p}^\prime_\ast (\xi, \zeta) := \ol{p}_\ast (\xi, \zeta) := p_\ast (\xi, \zeta)$. \item If $(\xi, \zeta) \in \dom p^\prime_0\, \times\, \dom p^\prime_0$, then $\ol{p}^\prime_\ast (\xi, \zeta) := \ol{p}_\ast (\xi, \zeta) := p^\prime_\ast (\xi, \zeta)$. \\[-3mm]

For $(\xi, \zeta) \in (\dom p_0\, \times\, \dom p_0) \, \cap\, (\dom p_0^\prime\, \times\, \dom p_0^\prime)$, this is not a contradiction, since $p_\ast\, \uhr\, \kappa_\eta^2$ and $p^\prime_\ast \, \uhr \, \kappa_\eta^2$ are compatible. \item If $\zeta \in \dom p_0 \setminus \dom p_0^\prime$ and $\xi \notin \dom p_0$, we proceed as follows: In the case that $\{\xi\} = a^\sigma_i\, \cap\, [\kappa_{\nu, j}, \kappa_{\nu, j + 1})$ for some $(\sigma, i) \in \supp p_0$ with $\sigma \leq \eta$, $i < \beta$ or $(\sigma, i) \in \{(\eta_m, i_m)\ | \ m < \omega\}$, we set $\ol{p}^\prime_\ast (\xi, \zeta) := \ol{p}_\ast (\xi, \zeta) := p^\sigma_i (\zeta)$. Otherwise, we set $\ol{p}^\prime_\ast (\xi, \zeta) = \ol{p}_\ast (\xi, \zeta)$ arbitrarily. 

\item In the case that $\zeta \in \dom p_0^\prime \setminus \dom p_0$ and $\xi \notin \dom p_0^\prime$, we proceed as before: If $\{\xi\} = (a^\prime)^\sigma_i \, \cap\, [\kappa_{\nu, j}, \kappa_{\nu, j + 1})$ for some $(\sigma, i) \in \supp p^\prime_0$ with $\sigma \leq \eta$, $i < \beta$ or $(\sigma, i) \in \{ (\eta_m, i_m)\ | \ m < \omega\}$, then $\ol{p}^\prime_\ast (\xi, \zeta) := \ol{p}_\ast (\xi, \zeta) := (p^\prime)^\sigma_i (\zeta)$. Otherwise, we set $\ol{p}^\prime_\ast (\xi, \zeta) = \ol{p}_\ast (\xi, \zeta)$ arbitrarily.
\item In all other cases, $\ol{p}^\prime_\ast (\xi, \zeta) = \ol{p}_\ast (\xi, \zeta)$ can be set arbitrarily.
\end{itemize}

%\colorbox{yellow}{FRAGE: \tbl can be arbitrary\tbr\,oder \tbl can be set arbitrarily\tbr\,besser schreiben?}

This defines $\ol{p}_\ast\, \uhr\, [\kappa_{\nu, j}, \kappa_{\nu, j + 1})^2 = \ol{p}^\prime_\ast\, \uhr\, [\kappa_{\nu, j}, \kappa_{\nu, j + 1})^2$. \\[-3mm]
%on all intervals $[\kappa_{\nu, j}, \kappa_{\nu, j + 1}) \subseteq \kappa_\eta$, and we have

%with the property that $\ol{p}_\ast\, \uhr\, \kappa_\eta^2 = \ol{p}_\ast^\prime\,\uhr\,\kappa_\eta^2$. \\[-2mm]

%$\ol{p}_\ast \, \uhr\, [\kappa_{\nu, j}, \kappa_{\nu, j + 1})^2 = \ol{p}^\prime_\ast\, \uhr\, [\kappa_{\nu, j}, \kappa_{\nu, j + 1})^2$.\\[-2mm]

Now, consider $(\sigma, i) \in \ol{\supp}_0$. We define $\ol{p}^\sigma_i$ and $(\ol{p}^\prime)^\sigma_i$ on the interval $[\kappa_{\nu, j}, \kappa_{\nu, j + 1}) \subseteq \kappa_\eta$ as follows:

\begin{itemize} \item For $(\sigma, i) \in \supp p_0$, we define $\ol{p}^\sigma_i \, \uhr\, [\kappa_{\nu, j}, \kappa_{\nu, j + 1}) \supseteq p^\sigma_i \, \uhr\, [\kappa_{\nu, j}, \kappa_{\nu, j + 1})$ according to the \textit{linking property}:
Let $\{\xi\} := a^\sigma_i\, \cap\, [\kappa_{\nu, j}, \kappa_{\nu, j + 1})$ and consider $\zeta \in [\kappa_{\nu, j}, \kappa_{\nu, j + 1})\, \cap\, \ol{\dom}_0$. If $\zeta \in \dom p_0$, we set $\ol{p}^\sigma_i (\zeta) := p^\sigma_i (\zeta)$; and $\ol{p}^\sigma_i (\zeta) := \ol{p}_\ast (\xi, \zeta)$ in the case that $\zeta \in \ol{\dom}_0 \setminus \dom p_0$. (Note that $\xi \in \ol{\dom}_0$ follows by construction.)

\item In the case that $(\sigma, i) \in \supp p_0^\prime$, we define $(\ol{p}^\prime)^\sigma_i \, \uhr\, [\kappa_{\nu, j}, \kappa_{\nu, j + 1}) \supseteq (p^\prime)^\sigma_i \, \uhr\, [\kappa_{\nu, j}, \kappa_{\nu, j + 1})$ according to the \textit{linking property} as before: Let $\{\xi\} := (a^\prime)^\sigma_i\, \cap\, [\kappa_{\nu, j}, \kappa_{\nu, j + 1})$, and consider $\zeta \in [\kappa_{\nu, j}, \kappa_{\nu, j + 1})\, \cap\, \ol{\dom}_0$. If $\zeta \in \dom p^\prime_0$, we set $(\ol{p}^\prime)^\sigma_i (\zeta) := (p^\prime)^\sigma_i (\zeta)$; and $(\ol{p}^\prime)^\sigma_i (\zeta) := (\ol{p}^\prime_\ast )(\xi, \zeta)$ in the case that $\zeta \in \ol{\dom}_0 \setminus \dom p^\prime_0$. (Again, $\xi \in \ol{\dom}_0$ by construction.) 

\item For $(\sigma, i) \in \supp p_0 \setminus \supp p^\prime_0$, let $(\ol{p}^\prime)^\sigma_i \, \uhr\, [\kappa_{\nu, j}, \kappa_{\nu, j + 1}) := \ol{p}^\sigma_i \, \uhr\, [\kappa_{\nu, j}, \kappa_{\nu, j + 1})$. 

\item For $(\sigma, i) \in \supp p^\prime_0 \setminus \supp p_0$, let $\ol{p}^\sigma_i \, \uhr\, [\kappa_{\nu, j}, \kappa_{\nu, j + 1}) := (\ol{p}^\prime)^\sigma_i \, \uhr\, [\kappa_{\nu, j}, \kappa_{\nu, j + 1})$. 

\item If $(\sigma, i) \in \ol{\supp}_0 \setminus (\supp p_0\, \cup \, \supp p_0^\prime)$, then $\ol{p}^\sigma_i \, \uhr\, [\kappa_{\nu, j}, \kappa_{\nu, j + 1}) = (\ol{p}^\prime)^\sigma _i \, \uhr\, [\kappa_{\nu, j}, \kappa_{\nu, j + 1})$ can be set arbitrarily on the given domain.
% $\ol{\dom}_0\, \cap\, \kappa_\sigma$. 

%\item In the case that $(\sigma, i) \in \ol{\supp}_0 \setminus (\supp p_0\, \cup \, \supp p^\prime_0)$, we set $(\ol{p}^\prime)^\sigma_i := \ol{p}^\sigma_i$. 

\end{itemize}

This defines all $\ol{p}^\sigma_i$ and $(\ol{p}^\prime)^\sigma_i$ for $(\sigma, i) \in \ol{\supp}_0$ on intervals $[\kappa_{\nu, j}, \kappa_{\nu, j + 1}) \subseteq \kappa_\eta$. \\[-2mm]

%Thus, we have defined $\ol{p}_\ast\, \uhr\, \kappa_\eta$, $\ol{p}^\prime_\ast\, \uhr\, \kappa_\eta$, and $\ol{p}^\sigma_i \, \uhr\, \kappa_\eta$, $(\ol{p}^\prime)^\sigma_i\, \uhr\, \kappa_\eta$ for all $(\sigma, i) \in \supp_0$. \\[-2mm]

We now have to verify that $\ol{p}^\sigma_i = (\ol{p}^\prime)^\sigma_i$ for any $(\sigma, i) \in \ol{\supp}_0$ with $\sigma \leq \eta$, $i < \beta$. We only have to treat the case that $(\sigma, i) \in \supp p_0\, \cap\, \supp p_0^\prime$. 

Consider an interval $[\kappa_{\nu, j}, \kappa_{\nu, j + 1}) \subseteq \kappa_\sigma \subseteq \kappa_\eta$. Then $p^\prime \in G$ and \[ (p^\beta\, \uhr \, (\eta + 1))^{( \eta_m, i_m)_{m < \omega}} \in (G^\beta\, \uhr \, (\eta + 1))^{( \eta_m, i_m)_{m < \omega}}\] implies that $p^\sigma_i$ and $(p^\prime)^\sigma_i$ are compatible, and $a^{\sigma}_i\, \cap\, [\kappa_{\nu, j}, \kappa_{\nu, j + 1}) = (a^\prime)^\sigma_i\, \cap\, [\kappa_{\nu, j}, \kappa_{\nu, j + 1}) =: \{ \xi\}$. \\[-3mm]

Let $\zeta \in [\kappa_{\nu, j}, \kappa_{\nu, j + 1}) \, \cap\, \ol{\dom}_0$. 

\begin{itemize} \item If $\zeta \in \dom p_0\, \cap\, \dom p^\prime_0$, then $\ol{p}^\sigma_i (\zeta) = p^\sigma_i (\zeta) = (p^\prime)^\sigma_i (\zeta) = (\ol{p}^\prime)^\sigma_i (\zeta)$. \item For $\zeta \in \ol{\dom}_0 \setminus (\dom p_0\, \cup \, \dom p^\prime_0)$, it follows that $\ol{p}^\sigma_i (\zeta) = \ol{p}_\ast (\xi, \zeta) = \ol{p}^\prime_\ast (\xi, \zeta) = (\ol{p}^\prime)^\sigma_i (\zeta)$ by construction, since we have arranged $\ol{p}_\ast \, \uhr\, \kappa_\eta^2 = \ol{p}^\prime_\ast \, \uhr\, \kappa_\eta^2$. \item Let now $\zeta \in \dom p_0 \setminus \dom p^\prime_0$, $\xi \notin \dom p_0$. Then $\ol{p}^\sigma_i (\zeta) = p^\sigma_i (\zeta)$, and $(\ol{p}^\prime)^\sigma_i (\zeta) = \ol{p}^\prime_\ast (\xi, \zeta)$. Since $\ol{p}^\prime_\ast (\xi, \zeta) = p^\sigma_i (\zeta)$ by construction of $\ol{p}^\prime_\ast$, this gives $\ol{p}^\sigma_i (\zeta) = (\ol{p}^\prime)^\sigma_i (\zeta)$ as desired.

The case that $\zeta \in \dom p^\prime_0 \setminus \dom p_0$, $\xi \notin \dom p^\prime_0$, can be treated similarly.

\item If $\zeta \in \dom p_0 \setminus \dom p^\prime_0$ and $\xi \in \dom p_0$, it follows that $\ol{p}^\sigma_i (\zeta) = p^\sigma_i (\zeta)$ and $(\ol{p}^\prime)^\sigma_i (\zeta) = \ol{p}^\prime_\ast (\xi, \zeta)$ as before; but in this case, we have set $\ol{p}^\prime_\ast (\xi, \zeta) := p_\ast (\xi, \zeta)$, so it remains to verify that $p^\sigma_i (\zeta) = p_\ast (\xi, \zeta)$. \\[-5mm]

Since $p^\prime \in G$, $(p^\beta\, \uhr\, (\eta + 1))^{( \eta_m, i_m)_{m < \omega}} \in (G^\beta\, \uhr\, (\eta + 1))^{( \eta_m, i_m)_{m < \omega}}$, we can take $q \in G$ with $(q^\beta\, \uhr\, (\eta + 1))^{(\eta_m , i_m)_{m < \omega}} \leq (p^\beta\, \uhr\, (\eta + 1))^{(\eta_m, i_m)_{m < \omega}}$, and assume w.l.o.g. that $q \leq p^\prime$. Then $q^\sigma_i (\zeta) = q_\ast (\xi, \zeta)$ by the linking property for $q \leq p^\prime$, since $(a^\prime)^\sigma_i\, \cap\, [\kappa_{\nu, j}, \kappa_{\nu, j + 1}) = \{\xi \}$. Moreover, $p^\sigma_i (\zeta) = q^\sigma_i (\zeta)$ and $p_\ast (\xi, \zeta) = q_\ast (\xi, \zeta)$, and we are done.

\item The remaining case is that $\zeta \in \dom p^\prime_0 \setminus \dom p_0$ and $\xi \in \dom p^\prime_0$. Then $\ol{p}^\sigma_i (\zeta) = \ol{p}_\ast (\xi, \zeta) = p^\prime_\ast (\xi, \zeta)$ and $(\ol{p}^\prime)^\sigma_i (\zeta)= (p^\prime)^\sigma_i (\zeta)$, and it remains to verify that $(p^\prime)^\sigma_i (\zeta) = p^\prime_\ast (\xi, \zeta)$. As before, take $q \in G$ with $q \leq p^\prime$ and $(q^\beta\, \uhr\, (\eta + 1))^{(\eta_m , i_m)_{m < \omega}} \leq (p^\beta\, \uhr\, (\eta + 1))^{( \eta_m, i_m)_{m < \omega}}$. The latter gives $q^\sigma_i (\zeta) = q_\ast (\xi, \zeta)$ by the linking property, since $\sigma \leq \eta$, $i < \beta$, $a^\sigma_i\, \cap\, [\kappa_{\nu, j}, \kappa_{\nu, j + 1}) = \{\xi\}$ and $\zeta \in \dom q_0 \setminus \dom p_0$. Moreover, from $q \leq p^\prime$ it follows that $(p^\prime)^\sigma_i (\zeta) = q^\sigma_i (\zeta)$ and $p^\prime_\ast (\xi, \zeta) = q_\ast (\xi, \zeta)$; hence, $(p^\prime)^\sigma_i (\zeta) = p^\prime_\ast (\xi, \zeta)$ as desired.
%The case $\zeta \in \dom p_0 \setminus \dom p^\prime_0$ (with $\xi \notin \dom p^\prime_0$ or $\xi \in \dom p^\prime_0$) ca be treated similarly.

%\colorbox{yellow}{FRAGE: Könnte man das weglassen?}
\end{itemize}

Thus, it follows that $\ol{p}^\sigma_i = (\ol{p}^\prime)^\sigma_i$ holds for all $(\sigma, i) \in \ol{\supp}_0$ with $\sigma \leq \eta$, $i < \beta$. \\[-2mm]

If $m < \omega$ with $\eta_m \leq \eta$, then $i_m < \beta$ follows by construction of $\beta$. Hence, $\ol{p}^{\eta_m}_{i_m} = (\ol{p}^\prime)^{\eta_m}_{i_m}$. It remains to make sure that whenever $m < \omega$ with $\eta_m > \eta$, then $\ol{p}^{\eta_m}_{i_m} \, \uhr\, \kappa_\eta = (\ol{p}^\prime)^{\eta_m}_{i_m}\, \uhr\, \kappa_\eta$ holds; which can be shown similarly as $\ol{p}^\sigma_i = (\ol{p}^\prime)^\sigma_i$ in the case that $\sigma \leq \eta$, $i < \beta$: We use that $a^{\eta_m}_{i_m} = (a^\prime)^{\eta_m}_{i_m}$ and $p^{\eta_m}_{i_m} (\zeta) = (p^\prime)^{\eta_m}_{i_m} (\zeta)$ for all $m < \omega$ and $\zeta \in \dom p_0\, \cap\, \dom p_0^\prime$; and now, it is important that we are using the forcing notion $(\m{P}^\beta\, \uhr\, (\eta + 1))^{( \eta_m, i_m)_{m < \omega}}$ instead of $\m{P}^\beta\, \uhr\, (\eta + 1)$; since we need the linking property below $\kappa_\eta$ for the $(\eta_m, i_m)$ with $\eta_m > \eta$. \\[-2mm]

It remains to construct $\ol{p}_\ast \, \uhr\, [\kappa_\eta, \kappa_\gamma)^2$, $\ol{p}^\prime_\ast\, \uhr\, [\kappa_\eta, \kappa_\gamma)^2$, and $\ol{p}^\sigma_i\, \uhr\, [\kappa_\eta, \kappa_\gamma)$, $(\ol{p}^\prime)^\sigma_i \, \uhr \, [\kappa_\eta, \kappa_\gamma)$ for all $(\sigma, i) \in \ol{\supp}_0$ with $\sigma > \eta$.

\begin{itemize} \item For $(\eta_m, i_m)$ with $\eta_m > \eta$, we take $\ol{p}^{\eta_m}_{i_m}\, \uhr\, [\kappa_\eta, \kappa_{\eta_m}) \supseteq p^{\eta_m}_{i_m}\, \uhr\, [\kappa_\eta, \kappa_{\eta_m})$, $(\ol{p}^\prime)^{\eta_m}_{i_m}\, \uhr\, [\kappa_\eta, \kappa_{\eta_m}) \supseteq (p^\prime)^{\eta_m}_{i_m} \, \uhr \, [\kappa_\eta, \kappa_{\eta_m})$ on the given domain, such that $\ol{p}^{\eta_m}_{i_m} \, \uhr \, [\kappa_\eta, \kappa_{\eta_m}) = (\ol{p}^\prime)^{\eta_m}_{i_m}\, \uhr \, [\kappa_\eta, \kappa_{\eta_m})$. This is possible, since $p^\prime \in G$ and $(p^\beta\, \uhr\, (\eta + 1))^{( \eta_m, i_m)_{m < \omega}} \in (G^\beta\, \uhr\, (\eta + 1))^{( \eta_m, i_m)_{m < \omega}}$; so $p^{\eta_m}_{i_m}$ and $(p^\prime)^{\eta_m}_{i_m}$ are compatible for all $m < \omega$.

%$p^{\eta_j}_{i_j} = (p^\prime)^{\eta_j}_{i_j}$. 
\item For the $(\sigma, i) \in \ol{\supp}_0$ remaining, we set $\ol{p}^\sigma_i\, \uhr\, [\kappa_\eta, \kappa_\gamma) \supseteq p^\sigma_i \, \uhr\, [\kappa_\eta, \kappa_\gamma)$ and $(\ol{p}^\prime)^\sigma_i \, \uhr\, [\kappa_\eta, \kappa_\gamma) \supseteq (p^\prime)^\sigma_i \, \uhr\, [\kappa_\eta, \kappa_\gamma)$ arbitrarily on the given domain. 
\item Consider an interval $[\kappa_{\nu, j}, \kappa_{\nu, j + 1}) \subseteq [\kappa_\eta, \kappa_\gamma)$. We define $\ol{p}_\ast\, \uhr\, [\kappa_{\nu, j}, \kappa_{\nu, j + 1})^2 \supseteq p_\ast\, \uhr\, [\kappa_{\nu, j}, \kappa_{\nu, j + 1})^2$ according to the \textit{linking property}: Whenever $\zeta \in \ol{\dom}_0 \setminus \dom p_0$ and $\{\xi\} = a^{\sigma}_i\, \cap\, [\kappa_{\nu, j}, \kappa_{\nu, j + 1})$ for some $(\sigma, i) \in \supp p_0$, then $\ol{p}_\ast (\xi, \zeta) := \ol{p}^{\sigma}_i (\zeta)$.\\ The construction of $\ol{p}^\prime_\ast\, \uhr\, [\kappa_{\nu, j}, \kappa_{\nu, j + 1})^2 \supseteq p^\prime_\ast \, \uhr\, [\kappa_{\nu, j}, \kappa_{\nu, j + 1})^2$ is similar. \end{itemize}

%ACHTUNG - man müsste erwähnen, dass man nur $\eta_j \in Lim$ betrachtet!

This completes our construction of $\ol{p}_0 \leq p_0$ and $\ol{p}_0^\prime \leq p_0^\prime$ with all the desired properties. \\[-2mm]

Similarly, one can construct $\ol{p}_1 \leq p_1$, $\ol{p}^\prime_1 \leq p^\prime_1$ such that $\ol{\supp}_1 := \supp \ol{p}_1 = \supp \ol{p}^\prime_1$, $\ol{\dom}_1 (\sigma) := \dom \ol{p}_1 (\sigma) = \dom \ol{p}_1^\prime (\sigma)$ for all $\sigma \in \ol{\supp}_1$; and $\ol{p}^\sigma_i = (\ol{p}^\prime)^\sigma_i$ for all $\sigma \leq \eta$, $i < \beta$ with $\sigma \in Succ$, and $\ol{p}^{\eta_m}_{i_m} = (\ol{p}^\prime)^{\eta_m}_{i_m}$ for all $m < \omega$ with $\eta_m \in Succ$. \\[-2mm]

We now proceed similarly as in the \textit{Approximation Lemma} \ref{approx} and construct an isomorphism $\pi$ such that $\pi$ a \textit{standard isomorphism for} $\pi \ol{p} = \ol{p}^\prime$. This determines all parameters of $\pi$ except the maps $G_0 (\nu, j): \supp \pi_0 (\nu, j) \rightarrow \supp \pi_0 (\nu, j)$, which will be defined as follows: Consider an interval $[\kappa_{\nu, j}, \kappa_{\nu, j + 1})$. Recall that we have the map $F_{\pi_0} (\nu, j): \supp \pi_0 (\nu, j) \rightarrow \supp \pi_0 (\nu, j)$, which is in charge of permuting the linking ordinals: We set $F_{\pi_0} (\nu, j) (\sigma, i) := (\lambda, k)$ for $(\ol{a}^\prime)^\sigma_i\, \cap\, [\kappa_{\nu, j}, \kappa_{\nu, j + 1}) = \ol{a}^\lambda_k \, \cap \, [\kappa_{\nu, j}, \kappa_{\nu, j + 1})$. We define $G_{\pi_0} (\nu, j) := F_{\pi_0} (\nu, j)$ for all $\kappa_{\nu, j} < \kappa_\eta$, and $G_{\pi_0} (\nu, j) := id$ whenever $\kappa_{\nu, j} \geq \kappa_\eta$. \\[-2mm]

By construction, it follows that $\pi \ol{p} = \ol{p}^\prime$. We will now check that $[\pi]$ is contained in the intersection $\bigcap_m Fix (\eta_m, i_m)\, \cap\, \bigcap_m H^{\lambda_m}_{k_m}$. 

\begin{itemize} \item Consider $m < \omega$ with $\eta_m \in Lim$ and $r \in D_{\pi}$, $r^\prime := \pi r$, with $(\eta_m, i_m) \in \supp r_0$. 

For an interval $[\kappa_{\nu, j}, \kappa_{\nu, j + 1}) \subseteq \kappa_{\eta_m}$ and $\zeta \in \dom \pi_0\, \cap\, [\kappa_{\nu, j}, \kappa_{\nu, j + 1})$, it follows by construction of the map $\pi_0 (\zeta)$
%as in the \textit{Approximation Lemma}
that $(r^\prime)^{\eta_m}_{i_m} (\zeta) = r^{\eta_m}_{i_m} (\zeta)$ holds; since $\ol{p}^{\eta_m}_{i_m}  = (\ol{p}^\prime)^{\eta_m}_{i_m}$. 

In the case that $\zeta \in [\kappa_{\nu, j}, \kappa_{\nu, j + 1}) \, \cap\, (\dom r_0 \, \setminus \, \dom \pi_0)$, it follows that $(r^\prime)^{\eta_m}_{i_m} (\zeta) = r^{\lambda}_k (\zeta)$ with $(\lambda, k) = G_{\pi_0} (\nu, j) (\eta_m, i_m)$. If $\kappa_{\nu, j} < \kappa_\eta$,  then $(\lambda, k) $ $=$  $G_{\pi_0} (\nu, j) (\eta_m, i_m)$ $=$ $F_{\pi_0} (\nu, j) (\eta_m, i_m)$ $=$ $(\eta_m, i_m)$, since $\ol{a}^{\eta_m}_{i_m} = (\ol{a}^\prime)^{\eta_m}_{i_m}$. In the case that $\kappa_{\nu, j} \geq \kappa_\eta$, we have $G_{\pi_0} (\nu, j) = id$; so again, $(\lambda, k) = (\eta_m, i_m)$. 

Hence, $r^{\eta_m}_{i_m} (\zeta) = (r^\prime)^{\eta_m}_{i_m} (\zeta)$ holds for all $\zeta \in \dom r_0\, \cap\, \kappa_{\eta_m}$. 

This proves $[\pi] \in Fix (\eta_m, i_m)$ in the case that $\eta_m \in Lim$. For $\eta_m \in Succ$, we obtain $[\pi] \in Fix (\eta_m, i_m)$ as in the \textit{Approximation Lemma}.

\item Consider $m < \omega$ with $\lambda_m \in Lim$. In the case that $\lambda_m > \eta$, we have $G_{\pi_0} (\nu, j) (\lambda_m, i)$ $=$ $(\lambda_m, i)$ for all $\kappa_{\nu, j} \in [\kappa_\eta, \kappa_{\lambda_m})$, and $[\pi] \in H^{\lambda_m}_{k_m}$ follows. If $\lambda_m \leq \eta$, it follows that $k_m < \beta$ by construction of $\beta$.
% since $\beta$ is \textit{large enough for} $(A_{\dot{f}})$. 
Hence, whenever $\kappa_{\nu, j} < \kappa_{\lambda_m}$ and $i \leq k_m$, we have $G_{\pi_0} (\nu, j) (\lambda_m, i) = F_{\pi_0} (\nu, j)(\lambda_m, i) = (\lambda_m, i)$; since $\ol{a}^{\lambda_m}_i = (\ol{a}^\prime)^{\lambda_m}_i$ follows from $\lambda_m \leq \eta$, $i < \beta$. 

In the case that $\lambda_m \in Succ$, we obtain $[\pi] \in H^{\lambda_m}_{k_m}$ as in the \textit{Approximation Lemma}.
\end{itemize}

Thus, we have shown that $[\pi] \in \bigcap_m Fix (\eta_m, i_m)\, \cap\, \bigcap_m H^{\lambda_m}_{k_m}$; which implies $\pi \ol{f}^{D_\pi} = \ol{f}^{D_\pi}$. It remains to make sure that $\pi \ol{\tau_{\varrho}(\dot{X})}^{D_\pi} = \ol{\tau_{\varrho}(\dot{X})}^{D_\pi}$. \\[-2mm]

Recall that we have an $\eta$-good pair $\varrho = ( (a_m)_{m < \omega}, (\ol{\sigma}_m, \ol{i}_m)_{m < \omega})$ with $\ol{i}_m < \beta$ for all $m < \omega$, and $\dot{X} \in \Name \big( (\ol{P}^\eta)^\omega\, \times \prod_{m < \omega} P^{\ol{\sigma}_m} \big)$ with \[\tau_{\varrho} (\dot{X}) = \Big\{\; \big(\tau_{\varrho}(\dot{Y}), q\big)\ \; \big| \; \ q \in \m{P}\; , \; \exists\, \Big( \dot{Y}\; , \; \big( ( p_\ast (a_m))_{m < \omega}\, , \, (\, p^{\ol{\sigma}_m}_{\ol{i}_m})_{m < \omega} \,\big)\, \Big) \in \dot{X}\; : \]\[ \ \forall\,m\ \ \Big(\,  q_\ast (a_m) \supseteq p_\ast (a_m)\, , \, q^{\ol{\sigma}_m}_{\ol{i}_m} \supseteq p^{\ol{\sigma}_m}_{\ol{i}_m}\, \Big) \; \Big\}. \] Then \[\ol{\tau_{\varrho}(\dot{X})}^{D_\pi} = \Big\{ \, \big(\,\ol{\tau_{\varrho}(\dot{Y})}^{D_\pi}, q\,\big)\ \, \big| \, \ q \in D_\pi\; , \; \dot{Y} \in \dom \dot{X}\, , \, q \Vdash_s \tau_{\varrho} (\dot{Y}) \in \tau_{\varrho} (\dot{X})\, \Big\}, \] and \[\pi \ol{\tau_{\varrho}(\dot{X})}^{D_\pi} = \Big\{ \, \big( \,\pi \ol{\tau_{\varrho}(\dot{Y})}^{D_\pi}, \pi q \,\big)\ \, | \ \, \pi q \in D_\pi\; , \; \dot{Y} \in \dom \dot{X}\, , \, q \Vdash_s \tau_{\varrho} (\dot{Y}) \in \tau_{\varrho} (\dot{X})\, \Big\} .\]

%\ol{\tau(\dot{X})}^{D_\pi} = \big\{ (\ol{\tau(\dot{Y})}^{D_\pi}, q)\ | \ q \in D_\pi\; , \; \exists\, \big( \dot{Y}\; , \; \prod_m p_\ast (a_m)\, \times\, \prod_m p^{\ol{\sigma}_m}_{\ol{i}_m} \big) \in \dot{X}\; : \ \forall j\  \ \big( q_\ast (a_m) \supseteq p_\ast (a_m), \]\[ q^{\ol{\sigma}_m}_{\ol{i}_m} \supseteq p^{\ol{\sigma}_m}_{\ol{i}_m} \big) \, \big\}. \]

%\colorbox{yellow}{FRAGE: Sollte man ${\wt{\dot{X}}}^{D_\pi}$ überhaupt einführen? Oder eher $\wt{X}^{D_\pi}$ schreiben?}

We will now check that $\pi$ is the identity on $\m{P}^\beta\, \uhr\, (\eta + 1)$. More precisely: Let $q \in D_\pi$, $q = (q_\ast, (q^\sigma_i, b^\sigma_i)_{\sigma, i}, (q^\sigma)_\sigma)$ with $\pi q = q^\prime = \big(q^\prime_\ast, ( (q^\prime)^\sigma_i, (b^\prime)^\sigma_i)_{\sigma, i}, ( (q^\prime)^\sigma)_{\sigma}\big)$. We prove  that $q^\prime_\ast\, \uhr\, \kappa_\eta^2 = q_\ast\, \uhr\, \kappa_\eta^2$; moreover, $(q^\prime)^\sigma_i = q^\sigma_i$, $(b^\prime)^\sigma_i = b^\sigma_i$ for all $\sigma \leq \eta$, $i < \beta$ with $\sigma \in Lim$, and $(q^\prime)^\sigma_i = q^\sigma_i$ for all $\sigma \leq \eta$, $i < \beta$ with $\sigma \in Succ$.

\begin{itemize} \item Since $\pi$ is a \textit{standard isomorphism for }$\pi \ol{p} = \ol{p}^\prime$, it follows that $q^\prime_\ast \, \uhr\, \kappa_\eta^2 = q_\ast\, \uhr\, \kappa_\eta^2$ for all $q \in D_\pi$; since firstly, $\ol{p}_\ast\, \uhr\, \kappa_\eta^2 = \ol{p}^\prime_\ast\, \uhr\, \kappa_\eta^2$, and secondly, $G_{\pi_0} (\nu, j) = F_{\pi_0} (\nu, j)$ for all $\kappa_{\nu, j} < \kappa_\eta$. The latter makes sure that $q^\prime_\ast (\xi^\sigma_i (\nu, j), \zeta) = q_\ast (\xi^\sigma_i (\nu, j), \zeta)$ whenever $\zeta \in \dom q_0 \setminus \dom \pi_0$, and $\{\xi^\sigma_i (\nu, j)\} := b^\sigma_i\, \cap\, [\kappa_{\nu, j}, \kappa_{\nu, j + 1})$ for some $(\sigma, i) \in \supp \pi_0 (\nu, j)$: We have $q^\prime_\ast (\xi^{\sigma}_i (\nu, j), \zeta) = q_\ast (\xi^\lambda_k (\nu, j), \zeta)$ with $(\lambda, k) = G_{\pi_0} (\nu, j)\, \circ\, (F_{\pi_0} (\nu, j))^{-1} (\sigma, i)$; so from 
%for all $(\sigma, i) \in \supp \pi_0 (\nu, j)$ and $\zeta \in \dom q_0 \setminus \dom \pi_0$. Hence, from 
$G_{\pi_0} (\nu, j) = F_{\pi_0} (\nu, j)$ it follows that $q^\prime_\ast (\xi^\sigma_i (\nu, j), \zeta) = q_\ast (\xi^\sigma_i (\nu, j), \zeta)$ as desired.

\item Let now $(\sigma, i) \in \supp \pi_0 = \supp \ol{p}_0$ with $\sigma \leq \eta$, $i < \beta$ and $\sigma \in Lim$. Then $\ol{a}^\sigma_i = (\ol{a}^\prime)^\sigma_i$; hence, $F_{\pi_0} (\nu, j) (\sigma, i) = (\sigma, i)$ for all $\kappa_{\nu, j} < \kappa_\sigma$. This gives $(b^\prime)^\sigma_i = b^\sigma_i$ as desired.
For $\zeta \in \dom \pi_0 = \dom \ol{p}_0$, it follows from $\ol{p}^\sigma_i = (\ol{p}^\prime)^\sigma_i$ by construction of $\pi_0$ that $(q^\prime)^\sigma_i (\zeta) = q^\sigma_i (\zeta)$ holds. Finally, if $\zeta  \in (\dom q_0 \setminus \dom \pi_0)$, and $\zeta$ is contained in an interval $[\kappa_{\nu, j}, \kappa_{\nu, j + 1}) \subseteq \kappa_\sigma$, then $(q^\prime)^\sigma_i (\zeta) = q^\lambda_k (\zeta)$ with $(\lambda, k) = G_{\pi_0} (\nu, j) (\sigma, i) = F_{\pi_0} (\nu, j) (\sigma, i) = (\sigma, i)$ as desired. 
%(We use that $\sigma \leq \eta$, $i < \beta$ implies $\ol{a}^\sigma_i = (\ol{a}^\prime)^\sigma_i$.)
Hence, it follows that $(q^\prime)^\sigma_i = q^\sigma_i$ for all $\sigma \leq \eta$, $i < \beta$. 

\item In the case that $\sigma \leq \eta$, $i < \beta$ with $\sigma \in Succ$, we obtain $(q^\prime)^\sigma_i = q^\sigma_i$ from $\ol{p}^\sigma_i = (\ol{p}^\prime)^\sigma_i$ as in the \textit{Approximation Lemma} \ref{approx}.
\end{itemize}

Hence, $\pi$ is the identity on $\m{P}^\beta\, \uhr \, (\eta + 1)$. \\[-3mm]

Now, it is not difficult to prove recursively that for every $\dot{Z}\in \Name ( (\ol{P}^\eta)^\omega\, \times\, \prod_{m < \omega} P^{\ol{\sigma}_m})$ the following holds: If $H$ is a $V$-generic filter on $\m{P}$,
%$(\ol{P}^\eta)^\omega\, \times\, \prod_{m < \omega} P^{\ol{\sigma}_m}_{\ol{i}_m}$, 
then $(\tau_{\varrho} (\dot{Z}))^{\pi H} = (\tau_{\varrho} (\dot{Z}))^H = (\tau_{\varrho} (\dot{Z}))^{\pi^{-1} H}$. 

%\colorbox{red}{ACHTUNG - für Filter wurde $\pi H$ nocht nicht definiert, aber schon vorher verwendet!}

%\colorbox{yellow}{FRAGE: Sollte man das genauer erklären?}

This implies $\ol{\tau_{\varrho} (\dot{X})}^{D_\pi} = \pi \ol{\tau_{\varrho} (\dot{X})}^{D_\pi}$, since for every $q \in D_\pi$ and $\dot{Y} \in \dom \dot{X}$, we have $q \Vdash_s \tau_{\varrho}(\dot{Y}) \in \tau_{\varrho}(\dot{X})$ if and only if $\pi q \Vdash_s \tau_{\varrho}(\dot{Y}) \in \tau_{\varrho}(\dot{X})$ holds. \\[-3mm]

%\colorbox{red}{wegen $\Vdash$ und $\Vdash_s$ nachfragen!}

%$\pi \ol{\tau (\dot{Y})}^{D_\pi} = \ol{\tau (\dot{Y})}^{D_\pi}$ for all $\dot{Y}\in \Name ( (\ol{P}^\eta)^\omega\, \times\, \prod_{j < \omega} P^{\ol{\eta}_j}_{\ol{i}_j})$: Firstly, we have just shown that $(\pi q)_\ast\, \uhr\, \kappa_\eta^2 = q_\ast\, \uhr\, \kappa_\eta^2$ for all $q \in D_\pi$, which gives hence, $(\pi q)_\ast (a_j) = q_\ast (a_j)$ for all $j < \omega$. Secondly, $(\pi q)^{\ol{\eta}_j}_{\ol{i}_j} = q^{\ol{\eta}_j}_{\ol{i}_j}$, for all $j < \omega$, since $\ol{\eta}_j \leq \eta$, $\ol{i}_j < \beta$. This implies $\pi \wt{\dot{X}}^{D_\pi} = \wt{\dot{X}}^{D_\pi}$ as desired.\\[-2mm]

Summing up, this gives our desired contradiction: Since $\ol{p} \Vdash_s \big(\tau_{\varrho} (\dot{X}), \alpha\big) \in \dot{f}$, it follows that $\pi \ol{p} \Vdash_s \big(\pi \ol{\tau_{\varrho}(\dot{X})}^{D_\pi}, \alpha\big) \in \pi \ol{f}^{D_\pi}$; hence, $\ol{p}^\prime \Vdash_s \big(\ol{\tau_{\varrho} (\dot{X})}^{D_\pi}, \alpha\big) \in \ol{f}^{D_\pi}$. But this contradicts $\ol{p}^\prime \Vdash_s (\tau_{\varrho}(\dot{X}), \alpha) \notin \dot{f}$. \\[-3mm]

Thus, our assumption that $(X, \alpha) \in (f^\beta)^\prime \setminus f^\beta$ was wrong, and it follows that $(f^\beta)^\prime = f^\beta$ as desired.

%ACHTUNG - $\Vdash$ oder $\Vdash_s$?

%FRAGE: Bezeichnungen beim Approximation Lemma??

\end{proof}

%From $f^\beta = (f^\beta)^\prime$, it follows that 

Hence, $f^\beta \in V[\, (G^\beta\, \uhr\, (\eta + 1))^{( \eta_m, i_m)_{m < \omega}}\, \times\, \prod_{m < \omega} G^{\eta_m}_{i_m}\, \uhr\, [\kappa_\eta, \kappa_{\eta_m})\, ]$.

\subsubsection*{C) ${\mathbf (\boldsymbol{\m{P}^\beta\, \uhr \, (\eta + 1))^{( \eta_m, i_m)_{m < \omega}}\, \times\, \prod_{m < \omega} P^{\eta_m}\, \uhr\, [\kappa_\eta, \kappa_{\eta_m})}}$ preserves cardinals $\boldsymbol{\geq \alpha_\eta}$.}

The next step is to show that cardinals $\geq \alpha_\eta$ are absolute between $V$ and $V[G^\beta\, \uhr \, (\eta + 1))^{( \eta_m, i_m)_{m < \omega}}\, \times\, \prod_{m < \omega} G^{\eta_m}_{i_m}\, \uhr\, [\kappa_\eta, \kappa_{\eta_m})]$. \\[-3mm]

Recall that we are assuming $GCH$ in our ground model $V$, which will be used implicitly throughout this Chapter \ref{6.2} C): When we claim that a particular forcing notion preserves cardinals, then we mean it preserves cardinals under the assumption that $GCH$ holds, if not stated differently. \\[-2mm]

%\colorbox{red}{WEGEN DIESER FORMULIERUNG NACHFRAGEN!!}

First, we have a look at the cardinality of $(\m{P}^\beta\, \uhr\, (\eta + 1))^{( \eta_m, i_m)_{m < \omega}}$. Recall that $\beta$ was an ordinal \textit{large enough for the intersection $(A_{\dot{f}})$} with $\kappa_\eta^\plus < \beta < \alpha_\eta$.

\begin{lem} \label{cardforc} $|(\m{P}^\beta\, \uhr\, (\eta + 1))^{( \eta_m, i_m)_{m < \omega}}| \leq |\beta|^\plus$. \end{lem}

\begin{proof} The forcing notion $(\m{P}^\beta\, \uhr\, (\eta + 1))^{( \eta_m, i_m)_{m < \omega}}$ is the set of all \[\big(\, p_\ast\, \uhr\, \kappa_\eta\; ,\; (p^\sigma_i, a^\sigma_i)_{\sigma \leq \eta, i < \beta}, (p^\sigma\, \uhr\, (\beta\, \times\ \dom_y p^\sigma))_{\sigma \leq \eta}\, , \, ( p^{\eta_m}_{i_m}\, \uhr\, \kappa_\eta, a^{\eta_m}_{i_m}\, \cap\, \kappa_\eta)_{m < \omega\, , \, \eta_m > \eta}\; , \; \wt{X}_p\, \big)\] for $p \in \m{P}$ with $|\{ (\sigma, i) \in \supp p_0\ | \ \sigma > \eta\, \vee\, i \geq \beta\}| = \aleph_0$, together with the maximal element $(\m{1}^\beta_{\eta + 1})^{(\eta_m, i_m)_{m < \omega}}$. Since $\wt{X}_p \subseteq \kappa_\eta$, there are only $\kappa_\eta^\plus \leq |\beta|$-many possibilities for $\wt{X}_p$; and there are only $\leq \kappa_\eta^\plus \leq |\beta|$-many possibilities for $p_\ast\, \uhr\, \kappa_\eta^2$ and $(p^{\eta_m}_{i_m}\, \uhr\, \kappa_\eta, a^{\eta_m}_{i_m}\, \cap\, \kappa_\eta)_{m < \omega}$. Concerning $(p^\sigma_i, a^\sigma_i)_{\sigma \leq \eta, i < \beta}$, there are $|\beta|^{\aleph_0} \leq |\beta|^\plus$-many possibilities for the countable support; and with the support fixed, we have $(2^{\kappa_\eta})^{\aleph_0} \leq \kappa_\eta^\plus \leq |\beta|$-many possibilities for countably many $(p^\sigma_i, a^\sigma_i)$ with $\sigma \leq \eta$, $i < \beta$. Finally, for $(p^\sigma\, \uhr\, (\beta\, \times\, \dom_y p^\sigma))_{\sigma \leq \eta}$, there are only $|\eta|^\omega \leq \kappa_\eta^\plus \leq |\beta|$-many possibilities for the countable support; and with the countable support fixed, there are $\leq (2^{|\beta|\, \cdot\, \kappa_\eta} )^{\aleph_0}  = |\beta|^\plus$-many possibilities for countably many $p^\sigma$ with $\dom p^\sigma \subseteq \beta\, \times \kappa_\sigma \subseteq \beta\, \times\, \kappa_\eta$.
Hence, it follows that the forcing notion $(\m{P}^\beta\, \uhr\, (\eta + 1))^{( \eta_m, i_m)_{m < \omega}}$ has cardinality $\leq |\beta|^\plus$. 
%By our assumption that $|\beta|^\plus < \alpha_\eta$, it follows that $|(\m{P}^\beta\, \uhr\, (\eta + 1))^{( \eta_m, i_m)_{m < \omega}}| < \alpha_\eta$ as desired.
\end{proof}

\begin{cor} \label{prescard2mitte} If $|\beta|^\plus < \alpha_\eta$, then $(\m{P}^\beta\, \uhr\, (\eta + 1))^{( \eta_m, i_m)_{m < \omega}}\, \times\, \prod_{m < \omega} P^{\eta_m}\, \uhr\, [\kappa_\eta, \kappa_{\eta_m})$ preserves cardinals $\geq \alpha_\eta$. \end{cor}

%\colorbox{red}{ACHTUNG - die Referenz auf \label{prescard2} müsste EVTL überprüft werden!}

\begin{proof} With the same arguments as in Lemma \ref{prescard1}, one can show that the forcing $\prod_{m < \omega} P^{\eta_m}\, \uhr \, [\kappa_\eta, \kappa_{\eta_m})$ preserves all cardinals. By Lemma \ref{cardforc} above, the forcing $(\m{P}^\beta\, \uhr\, (\eta + 1))^{( \eta_m, i_m)_{m < \omega}}$ has cardinality $\leq |\beta|^\plus$ (in $V$; and hence, also in any $\prod_{m < \omega} P^{\eta_m}\, \uhr\, [\kappa_\eta, \kappa_{\eta_m})$-generic extension). It follows that the product $(\m{P}^\beta\, \uhr\, (\eta + 1))^{( \eta_m, i_m)_{m < \omega}}\, \times\, \prod_{m < \omega} P^{\eta_m}\, \uhr\, [\kappa_\eta, \kappa_{\eta_m})$ preserves all cardinals $\geq |\beta|^{\plus \plus}$. \end{proof} 

It remains to consider the case that $|\beta|^\plus = \alpha_\eta$. Then by our assumptions on the sequence $(\alpha_\eta\ | \ 0 < \eta < \gamma)$ (cf. Chapter 2), it follows that $cf \,|\beta| > \omega$. Hence, $GCH$ gives $|\beta|^{\aleph_0} = |\beta| < \alpha_\eta$; and by our proof of Lemma \ref{cardforc}, it follows that all components of $(\m{P}^\beta\, \uhr\, (\eta + 1))^{( \eta_m, i_m)_{m < \omega}}$ have cardinality $\leq |\beta| < \alpha_\eta$; with the exception of $(p^\sigma\, \uhr\, (\beta\, \times\, \dom_y p^\sigma))_{\sigma \leq \eta}$, where there might be $(2^{|\beta|\, \cdot\, \kappa_\eta})^{\aleph_0} = |\beta|^\plus = \alpha_\eta$-many possibilities. \\[-2mm] 

We now have to distinguish several cases depending on whether $\eta$ is a limit ordinal or not, and depending on whether $\kappa_\eta$ is a limit cardinal or a successor cardinal (i.e. $\eta \in Lim$ or $\eta \in Succ$).

%\in Lim$ or $\kappa_\eta \in Succ$ and depending on whether 
We will have to separate one or two components $P^{\sigma}\, \uhr \, (\beta\, \times\, [\ol{\kappa_{\sigma}}, \kappa_{\sigma}))$, where $\sigma \in Succ$, $\sigma \leq \eta$, $\kappa_{\sigma} = \ol{\kappa_{\sigma}}^{\, \plus}$,
%?? Die Bedingungen $p^{\ol{\sigma}}\, \uhr\, (\beta\, \times\, \dom_y p^{\ol{\sigma}})$ erwähnen?? 
from the forcing notion $(\m{P}^\beta\, \uhr\, (\eta + 1))^{( \eta_m, i_m)_{m < \omega}}$; and obtain a forcing $\big((\m{P}^\beta\, \uhr\, (\eta + 1))^{( \eta_m, i_m)_{m < \omega}}\big)^\prime$ which has cardinality $< \alpha_\eta$, while the product of the remaining $P^{\sigma}\, \uhr \, (\beta\, \times\, [\ol{\kappa_{\sigma}}, \kappa_{\sigma}))$ and $\prod_{m < \omega} P^{\eta_m}\, \uhr\, [\kappa_\eta, \kappa_{\eta_m})$ preserves cardinals.

\begin{prop} \label{prescardalphaeta} The forcing notion $(\m{P}^\beta\, \uhr \, (\eta + 1))^{( \eta_m, i_m)_{m < \omega}}\, \times\, \prod_{m < \omega} P^{\eta_m}\, \uhr\, [\kappa_\eta, \kappa_{\eta_m})$ preserves all cardinals $\geq \alpha_\eta$.\end{prop}

\begin{proof} By Corollary \ref{prescard2mitte}, we only have to treat the case that $\alpha_\eta = |\beta|^\plus$. Then $\cf |\beta| > \omega$ and $|\beta|^{\aleph_0} = |\beta|$.

First, we assume that \textbf{ ${\boldsymbol \eta}$ is a limit ordinal}. Then by closure of the sequence $(\kappa_\sigma\ | \ 0 < \sigma < \gamma)$, it follows that $\boldsymbol{\eta \in \Lim}$, i.e. $\kappa_\eta = \sup \{\kappa_\sigma\ | \ 0 < \sigma < \eta\}$ is a limit cardinal. \\[-3mm]

%\colorbox{red}{ACHTUNG - STIMMT DAS ÜBERHAUPT?}

Since the sequence $(\alpha_\sigma \ | \ 0 < \sigma < \gamma)$ is strictly increasing (cf. Chapter \ref{the theorem}), it follows that $\alpha_\sigma < |\beta|$ for all $\sigma < \eta$. Hence, for any $\sigma \in Succ$ with $\sigma < \eta$, the forcing notion $P^\sigma\, \uhr \, (\beta\, \times\, [\ol{\kappa_\sigma}, \kappa_\sigma)) = P^\sigma\, \uhr \, (\alpha_\sigma \, \times\, [\ol{\kappa_\sigma}, \kappa_\sigma))$ has cardinality $\leq \alpha_\sigma^\plus \leq |\beta|$; and we conclude that there are only $\leq|\eta|^{\aleph_0}\, \cdot |\beta|^{\aleph_0} = |\beta|$-many possibilities for $(p^\sigma\, \uhr\, (\beta\, \times\, \dom_y p^\sigma))_{\sigma \leq \eta}$.

%\begin{itemize} \item 
%If $\boldsymbol{\eta \in Lim}$, i.e. $\kappa_\eta$ is a limit cardinal, it follows that there are only $\leq|\eta|^{\aleph_0}\, \cdot |\beta|^{\aleph_0} = |\beta|$-many possibilities for $(p^\sigma\, \uhr\, (\beta\, \times\, \dom_y p^\sigma))_{\sigma \leq \eta}$.

%the countable-support-product $ \prod_{\sigma \leq \eta} P^\sigma\, \uhr\, (\beta\, \times\, [\ol{\kappa_\sigma}, \kappa_\sigma)) = \prod_{\sigma < \eta} P^\sigma\, \uhr\, (\beta\, \times\, [\ol{\kappa_\sigma}, \kappa_\sigma))$ has cardinality $ \leq|\eta|^{\aleph_0}\, \cdot |\beta|^{\aleph_0} = |\beta|$. 

Hence, by the proof of Lemma \ref{cardforc}, it follows that $(\m{P}^\beta\, \uhr\, (\eta + 1))^{( \eta_m, i_m)_{m < \omega}}$ has cardinality $\leq |\beta| < \alpha_\eta$.
% since there are only $ \leq|\beta|$-many possibilities for $(p^\sigma\, \uhr\, (\beta\, \times\, \dom_y p^\sigma))_{\sigma \leq \eta}$. 
Like in Corollary \ref{prescard2mitte}, this implies that the product $(\m{P}^\beta\, \uhr\, (\eta + 1))^{( \eta_m, i_m)_{m < \omega}}\, \times\, \prod_{m < \omega} P^{\eta_m}\, \uhr\, [\kappa_\eta, \kappa_{\eta_m})$  preserves all cardinals $\geq |\beta|^\plus = \alpha_\eta$ as desired. \\[-2mm]

The remaining case is that $\boldsymbol{\eta}$ \textbf{ is a successor ordinal}. Let $\boldsymbol{\eta = \ol{\eta} + 1}$. We now have to distinguish four cases, depending on whether $\kappa_\eta$ and $\kappa_{\ol{\eta}}$ are successor cardinals or limit cardinals. \\[-2mm]

If $\boldsymbol{\eta \in Lim} \textbf{ and } \boldsymbol{\ol{\eta} \in Lim}$, it follows for any $P^\sigma\, \uhr\, (\beta\, \times\, [\ol{\kappa_\sigma}, \kappa_\sigma))$ with $\sigma \leq \eta$, $\sigma \in Succ$ that $\sigma < \ol{\eta}$ must hold; hence, $\alpha_\sigma < \alpha_{\ol{\eta}} < \alpha_\eta = |\beta|^\plus$, which implies $\alpha_\sigma < |\beta|$. Thus, the corresponding forcing notion $P^\sigma\, \uhr\, (\beta\, \times\, [\ol{\kappa_\sigma}, \kappa_\sigma)) = P^\sigma\, \uhr\, (\alpha_\sigma\, \times\, [\ol{\kappa_\sigma}, \kappa_\sigma))$ has cardinality $\leq \alpha_\sigma^\plus \leq |\beta|$; and as before, it follows that the  forcing $(\m{P}^\beta\, \uhr\, (\eta + 1))^{( \eta_m, i_m)_{m < \omega}}$ has cardinality $\leq |\beta|^{\aleph_0} = |\beta|$. Like in Corollary \ref{prescard2mitte}, this implies that the product $(\m{P}^\beta\, \uhr\, (\eta + 1))^{( \eta_m, i_m)_{m < \omega}}\, \times\, \prod_{m < \omega} P^{\eta_m}\, \uhr\, [\kappa_\eta, \kappa_{\eta_m})$  preserves all cardinals $\geq |\beta|^\plus = \alpha_\eta$ as desired. \\[-2mm]

If $\boldsymbol{\ol{\eta} \in Lim} \textbf{ and } \boldsymbol{\eta \in Succ}$, we consider the forcing notion $\big((\m{P}^\beta\, \uhr\, (\eta + 1))^{( \eta_m, i_m)_{m < \omega}}\big)^\prime$, which is obtained from $(\m{P}^\beta\, \uhr\, (\eta + 1))^{( \eta_m, i_m)_{m < \omega}}$ by excluding $P^\eta\, \uhr \, (\beta\, \times\, [\ol{\kappa_\eta}, \kappa_\eta))$; i.e. we consider \[(p^\sigma\, \uhr\, (\beta\, \times\, \dom_y p^\sigma))_{\sigma < \eta}= (p^\sigma\, \uhr\, (\beta\, \times\, \dom_y p^\sigma))_{\sigma < \ol{\eta}}\]

%the product \[\,\prod_{\substack{\sigma < \eta \\ \sigma \in \Succ}} P^\sigma\, \uhr\, (\beta\, \times\, [\ol{\kappa_\sigma}, \kappa_\sigma)) = \prod_{\substack{\sigma < \ol{\eta} \\ \sigma \in \Succ}} P^\sigma\, \uhr\, (\beta\, \times\, [\ol{\kappa_\sigma}, \kappa_\sigma))\] 

instead of $(p^\sigma\, \uhr\, (\beta\, \times\, \dom_y p^\sigma))_{\sigma \leq \eta}$.
%$\,\prod_{\sigma \leq \eta} P^\sigma\, \uhr\, (\beta\, \times\, [\ol{\kappa_\sigma}, \kappa_\sigma))$. 
Then $\big((\m{P}^\beta\, \uhr\, (\eta + 1))^{( \eta_m, i_m)_{m < \omega}}\big)^\prime$ has cardinality $\leq |\beta|$ as before; and it suffices to check that the remaining product \[P^\eta\, \uhr\, (\beta\, \times\, [\ol{\kappa_\eta}, \kappa_\eta))\; \times\; \prod_{m < \omega} P^{\eta_m}\, \uhr\, [\kappa_\eta, \kappa_{\eta_m})\] preserves all cardinals. 

%\colorbox{red}{ACHTUNG - man betrachtet wohl Intervalle $[\ol{\kappa}_\sigma, \kappa_\sigma)$ mit $\ol{\kappa}_\sigma$ Kardinalzahl-Vorgänger von $\kappa_\sigma$?}

%\colorbox{red}{Das heißt, es gibt \tbl Lücken\tbr ?} 

%\colorbox{red}{ACHTUNG - das stimmt so nicht!}

The forcing notion $\prod_{m < \omega} P^{\eta_m}\, \uhr\, [\kappa_\eta, \kappa_{\eta_m})$ preserves cardinals. Moreover, $\prod_{m < \omega} P^{\eta_m}\, \uhr\, [\kappa_\eta, \kappa_{\eta_m})$ is $\leq \kappa_\eta$-closed. Hence, in any $V$-generic extension by $\prod_{m < \omega} P^{\eta_m}\, \uhr\, [\kappa_\eta, \kappa_{\eta_m})$ the following holds: Firstly, $P^\eta\, \uhr\, (\beta\, \times\, [\ol{\kappa_\eta}, \kappa_\eta))$ is the same forcing notion as in $V$; and secondly, $P^\eta\, \uhr\, (\beta\, \times\, [\ol{\kappa_\eta}, \kappa_\eta))$ preserves cardinals, since $2^{< \kappa_\eta} = \kappa_\eta$. Thus, it follows that the product $P^\eta\, \uhr\, (\beta\, \times\, [\ol{\kappa_\eta}, \kappa_\eta))\; \times\; \prod_{m < \omega} P^{\eta_m}\, \uhr\, [\kappa_\eta, \kappa_{\eta_m})$ preserves all cardinals as desired. \\[-2mm]

If $\boldsymbol{\ol{\eta} \in \Succ} \textbf{ and } \boldsymbol{\eta \in \Lim}$, we proceed similarly, but exclude $P^{\ol{\eta}}\, \uhr \, (\beta\, \times\, [\ol{\kappa_{\ol{\eta}}}, \kappa_{\ol{\eta}}))$ instead of $P^\eta\, \uhr \, (\beta\, \times\, [\ol{\kappa_\eta}, \kappa_\eta))$. \\[-2mm]

If $\boldsymbol{\eta \in Succ}$ \textbf{ and } $\boldsymbol{\ol{\eta} \in Succ}$, then both $P^\eta\, \uhr\, (\beta\, \times\, [\ol{\kappa_\eta}, \kappa_\eta))$ and $P^{\ol{\eta}}\, \uhr\, (\beta\, \times\, [\ol{\kappa_{\ol{\eta}}}, \kappa_{\ol{\eta}}))$ have to be parted from $(\m{P}^\beta\, \uhr\, (\eta + 1))^{( \eta_m, i_m)_{m < \omega}}$. As before, it follows that firstly, the remaining forcing notion, denoted by $\big((\m{P}^\beta\, \uhr\, (\eta + 1))^{( \eta_m, i_m)_{m < \omega}}\big)^{\prime \prime}$, has cardinality $\leq |\beta|$; and secondly, the remaining product \[P^{\ol{\eta}}\, \uhr\, (\beta\, \times\, [\ol{\kappa_{\ol{\eta}}}, \kappa_{\ol{\eta}}))\, \times \, P^\eta\, \uhr\, (\beta\, \times\, [\ol{\kappa_\eta}, \kappa_\eta))\, \times\, \prod_{m < \omega} P^{\eta_m} \, \uhr\, [\kappa_\eta, \kappa_{\eta_m})\]

%of $\prod_{m < \omega} P^{\eta_m} \, \uhr\, [\kappa_\eta, \kappa_{\eta_m})$ with $P^\eta\, \uhr\, (\beta\, \times\, [\ol{\kappa_\eta}, \kappa_\eta))$ or $P^{\ol{\eta}}\, \uhr\, (\beta\, \times\, [\ol{\kappa_{\ol{\eta}}}, \kappa_{\ol{\eta}}))$  

preserves all cardinals. \\[-3mm]

%If $\eta \in Succ$ and $\ol{\eta} \in Succ$, then either $P^\eta\, \uhr\, (\beta\, \times\, [\ol{\kappa_\eta}, \kappa_\eta))$ or $P^{\ol{\eta}}\, \uhr\, (\beta\, \times\, [\ol{\kappa_{\ol{\eta}}}, \kappa_{\ol{\eta}}))$ or both have to be parted from $(\m{P}^\beta\, \uhr\, (\eta + 1))^{( \eta_m, i_m)_{m < \omega}}$. Then as before, it follows that the remaining forcing notion, denoted by $((\m{P}^\beta\, \uhr\, (\eta + 1))^\prime)^{( \eta_m, i_m)_{m < \omega}}$ again, has cardinality $\leq |\beta|$; and the remaining product of $\prod_{m < \omega} P^{\eta_m} \, \uhr\, [\kappa_\eta, \kappa_{\eta_m})$ with $P^\eta\, \uhr\, (\beta\, \times\, [\ol{\kappa_\eta}, \kappa_\eta))$ or $P^{\ol{\eta}}\, \uhr\, (\beta\, \times\, [\ol{\kappa_{\ol{\eta}}}, \kappa_{\ol{\eta}}))$ or both, preserves cardinals. \\[-3mm]

It follows that $(\m{P}^\beta\, \uhr\, (\eta + 1))^{( \eta_m, i_m)_{m < \omega}}\, \times\, \prod_{m < \omega} P^{\eta_m}\, \uhr\, [\kappa_\eta, \kappa_{\eta_m})$ preserves all cardinals $\geq \alpha_\eta$. \\[-2mm]

This concludes our proof by cases.

%FRAGE: Abkürzung für $P^\eta\, \uhr\, (\beta\, \times\, [\ol{\kappa}_\eta, \kappa_\eta))$ einführen? vll. $\m{P}^\beta (\eta)$?

\end{proof}

%\colorbox{red}{ACHTUNG - wo wurde vergessen, dass die Kardinalzahlen überabzählbar sein müssen???}

\subsubsection*{\textbf D) A set $\boldsymbol{\wt{\powerset}(\kappa_\eta) \supseteq \dom f^\beta}$ with an injection $\boldsymbol{\iota: \wt{\powerset} (\kappa_\eta) \hookrightarrow |\beta|^{\aleph_0}}$.}
%$\wt{\powerset}(\kappa_\eta) \supseteq \dom f^\beta$ with a
%\colorbox{red}{ACHTUNG - es ist NICHT $\wt{\powerset}(\kappa_\eta) \supseteq \powerset^N (\kappa_\eta)$!}

In this section, we construct in $V[(G^\beta\, \uhr\, (\eta + 1))^{( \eta_m, i_m)_{m < \omega}}\, \times \, \prod_{m < \omega} G^{\eta_m}_{i_m}\, \uhr\, [\kappa_\eta, \kappa_{\eta_m})]$ a set $\wt{\powerset}(\kappa_\eta)$ with $\wt{\powerset}(\kappa_\eta) \supseteq \dom f^\beta$, together with an injective function $\iota: \wt{\powerset} (\kappa_\eta) \hookrightarrow (|\beta|^{\aleph_0})^V < \alpha_\eta$. Since $f^\beta$ is contained in $V[(G^\beta\, \uhr\, (\eta + 1))^{( \eta_m, i_m)_{m < \omega}}\, \times \, \prod_{m < \omega} G^{\eta_m}_{i_m}\, \uhr\, [\kappa_\eta, \kappa_{\eta_m})]$ by Definition \ref{deffbetaprime} and Proposition \ref{fbetafbetaprime}, and $(\m{P}^\beta\, \uhr\, (\eta + 1))^{( \eta_m, i_m)_{m < \omega}}\, \times\, \prod_{m < \omega} P^{\eta_m}\, \uhr\, [\kappa_\eta, \kappa_{\eta_m})$ preserves cardinals $\geq \alpha_\eta$ by Proposition \ref{prescardalphaeta}, this will contradict our initial assumption that $f^\beta: \dom f^\beta \rightarrow \alpha_\eta$ was surjective. \\[-1mm]

Fix an $\eta$-good pair $ \varrho = \big( (a_m)_{m < \omega}, (\ol{\sigma}_m, \ol{i}_m)_{m < \omega}\big)$. Then $\prod_m G_\ast (a_m)\, \times\, \prod_m G^{\ol{\sigma}_m}_{\ol{i}_m}$ is a $V$-generic filter on $\prod_m \ol{P}^\eta\, \times\, \prod_m P^{\ol{\sigma}_m}$; and as in Lemma \ref{prescard1}, it follows that this forcing preserves cardinals and the $GCH$. 
%Frage: Kam das schon irgendwo vor? 
Hence, there is an injection $\chi: \powerset (\kappa_\eta) \hookrightarrow (\kappa_\eta^\plus)^V$ in $V[ \prod_m G_\ast (a_m)\, \times\, \prod_m G^{\ol{\sigma}_m}_{\ol{i}_m}]$. \\[-2mm]

Let $M_\beta$ be the set of all $\eta$-good pairs $ \big( (a_m)_{m < \omega}, (\ol{\sigma}_m, \ol{i}_m)_{m < \omega}\big)$ in $V$ with the property that $\ol{i}_m < \beta$ for all $m < \omega$. Then $M_\beta$ has cardinality $\leq (2^{\kappa_\eta})^{\aleph_0}\, \cdot\, |\eta|^{\aleph_0}\, \cdot \, |\beta|^{\aleph_0} \leq |\beta|^{\aleph_0}$. \\[-3mm]
%Oder $M$ nennen? \\[-3mm]

First, we consider the case that $\boldsymbol{|\beta|^\plus = \alpha_\eta}$. Then $cf\,|\beta| > \omega$; hence, $GCH$ gives $|\beta|^{\aleph_0} = |\beta|$ and there is an injection $\psi: M _\beta\hookrightarrow |\beta|$ in $V$.\\[-2mm] 

%\colorbox{red}{ACHTUNG - $m$ ist doch sonst der Laufindex!}

%From Proposition \ref{} Nachsehen! Korollar von Approximation Lemma, it follows that any $X \subseteq \kappa_\eta$, $X \in N$ 

By construction of $f^\beta$ (cf. Definition \ref{deffbeta}), it follows that any $X \subseteq \kappa_\eta$ with $X \in dom \,f^\beta$ is contained in a model $V[\prod_m G_\ast (a_m)\, \times\, \prod_m G^{\ol{\sigma}_m}_{\ol{i}_m}]$ for some $\eta$-good pair $\big( (a_m)_{m < \omega}, (\ol{\sigma}_m, \ol{i}_m)_{m < \omega}\big) \in M_\beta$. Hence, $dom \,f^\beta$ is a subset of \[\wt{\powerset} (\kappa_\eta) \ := \ \bigcup \Big \{\, \powerset(\kappa_\eta)\, \cap\, V[\prod_m G_\ast (a_m)\, \times\, \prod_m G^{\ol{\sigma}_m}_{\ol{i}_m}]\ \,\Big| \, \ \big((a_m)_{m < \omega}, (\ol{\sigma}_m, \ol{i}_m)_{m < \omega}\big) \in M_\beta\, \Big\}.\] 

The set $\wt{\powerset}(\kappa_\eta)$ can be defined in $V[(G^\beta\, \uhr\, (\eta + 1))^{( \eta_m, i_m)_{m < \omega}}\, \times\, G^{\eta_m}_{i_m}\, \uhr\, [\kappa_\eta, \kappa_{\eta_m})]$, since for any $\big( (a_m)_{m < \omega}, (\ol{\sigma}_m, \ol{i}_m)_{m < \omega}\big) \in M_\beta$, we have $a_m \subseteq \kappa_\eta$, and $\ol{\sigma}_m \leq \eta$, $\ol{i}_m < \beta$ for all $m < \omega$. \\[-3mm]

For the rest of this section, we work in $V[(G^\beta\, \uhr\, (\eta + 1))^{( \eta_m, i_m)_{m < \omega}}\, \times \, \prod_{m < \omega} G^{\eta_m}_{i_m}\, \uhr\, [\kappa_\eta, \kappa_{\eta_m})]$, and construct there an injective function $\iota: \wt{\powerset} (\kappa_\eta) \hookrightarrow |\beta|^V$. \\[-2mm]

For a set $X \in \wt{\powerset} (\kappa_\eta)$, let \[\wt{\iota} (X) := ( (a_m)_{m < \omega}, (\ol{\sigma}_m, \ol{i}_m)_{m < \omega})\]
if $( (a_m)_{m < \omega}, (\ol{\sigma}_m, \ol{i}_m)_{m < \omega}) \in M_\beta$ with $X \in \powerset(\kappa_\eta)\, \cap\, V[\prod_m G_\ast (a_m)\, \times\, \prod_m G^{\ol{\sigma}_m}_{\ol{i}_m}]$, and $\psi((a_m)_{m < \omega}, (\ol{\sigma}_m, \ol{i}_m)_{m < \omega})$ is least with this property. 

Now, we use the Axiom of Choice in $V[(G^\beta\, \uhr\, (\eta + 1))^{( \eta_m, i_m)_{m < \omega}}\, \times \, \prod_{m < \omega} G^{\eta_m}_{i_m}\, \uhr\, [\kappa_\eta, \kappa_{\eta_m})]$, and choose for all $( (a_m)_{m < \omega}, (\ol{\sigma}_m, \ol{i}_m)_{m < \omega}) \in M_\beta$ an injection \[\chi_{( (a_m)_{m < \omega}, (\ol{\sigma}_m, \ol{i}_m)_{m < \omega})}\ :  \ \Big(\, \powerset(\kappa_\eta)\, \cap\, V[\prod_m G_\ast (a_m)\, \times\, \prod_m G^{\ol{\sigma}_m}_{\ol{i}_m}]\, \Big) \hookrightarrow (\kappa_\eta^\plus)^V. \]

Now, we can define $\iota: \wt{\powerset} (\kappa_\eta) \hookrightarrow (\kappa_\eta^\plus)^V\, \cdot |\beta|^V$ as follows: For $X \in \wt{\powerset} (\kappa_\eta)$, let \[\iota (X) := \big(\,\chi_{\wt{\iota} (X)} (X), \psi (\wt{\iota} (X))\,\big).\] Since $\psi$ and the maps $\chi_{( (a_m)_{m < \omega}, (\ol{\sigma}_m, \ol{i}_m)_{m < \omega})}$ for $( (a_m)_{m < \omega}, (\ol{\sigma}_m, \ol{i}_m)_{m < \omega}) \in M_\beta$ are injective, it follows that also $\iota$ is injective; which finishes our construction in the case that $(|\beta|^\plus)^V = \alpha_\eta$. \\[-2mm]

If $\boldsymbol{|\beta|^\plus < \alpha_\eta}$ in $V$, we can take an injection $\psi: M_\beta \hookrightarrow (|\beta|^\plus)^V$, and construct an injective function $\iota: \wt{\powerset} (\kappa_\eta) \hookrightarrow (\kappa_\eta^\plus)^V\, \cdot\, (|\beta|^\plus)^V$ in $V[(G^\beta\, \uhr\, (\eta + 1))^{( \eta_m, i_m)_{m < \omega}}\, \times \, \prod_{m < \omega} G^{\eta_m}_{i_m}\, \uhr\, [\kappa_\eta, \kappa_{\eta_m})]$ similarly as before. \\[-3mm]

%\colorbox{red}{ACHTUNG - $\gamma$ für die Injektion ist SCHLECHT!} \\[-2mm]

%\colorbox{red}{ACHTUNG - die Klammerung bei den \tbl good pairs\tbr\, ist NICHT RICHTIG!}

%\colorbox{red}{TO DO: die Notation bei \tbl good pairs\tbr\, müsste man DURCHGEHEN! Wäre NICHT EINHEITLICH!!}

This gives the following proposition:

\begin{prop} \label{iota} If $(|\beta|^\plus)^V = \alpha_\eta$, then there is in $V[(G^\beta\, \uhr\, (\eta + 1))^{( \eta_m, i_m)_{m < \omega}}\, \times \, \prod_{m < \omega} G^{\eta_m}_{i_m}\, \uhr\, [\kappa_\eta, \kappa_{\eta_m})]$ an injection $\iota: \wt{\powerset} (\kappa_\eta) \hookrightarrow |\beta|^V$, where \[ \wt{\powerset} (\kappa_\eta) := \bigcup \Big \{\, \powerset(\kappa_\eta)\, \cap\, V[\prod_m G_\ast (a_m)\, \times\, \prod_m G^{\ol{\sigma}_m}_{\ol{i}_m}]\ \; \Big| \; \ ( (a_m)_{m < \omega}, (\ol{\sigma}_m, \ol{i}_m)_{m < \omega}) \in M_\beta\, \Big\}.\] 
%is a superset of $\dom f^\beta$.
If $(|\beta|^\plus)^V < \alpha_\eta$, there is in $V[(G^\beta\, \uhr\, (\eta + 1))^{( \eta_m, i_m)_{m < \omega}}\, \times \, \prod_{m < \omega} G^{\eta_m}_{i_m}\, \uhr\, [\kappa_\eta, \kappa_{\eta_m})]$ an injection $\iota: \wt{\powerset} (\kappa_\eta) \hookrightarrow (|\beta|^\plus)^V$. \end{prop}

%\colorbox{red}{TO DO: Nochmal durchsehen!! Kardinalzahlen $(|\beta|^\plus)^V$ etc. ? }

%\colorbox{yellow}{FRAGE: die Bezeichnung $M_\beta$ schon bei der Definition von $f_\beta$ einführen? Eher nicht?}

This leads to our desired contradiction: We assumed that $f^\beta: \dom f^\beta \rightarrow \alpha_\eta$ was surjective.
% where $\dom f^\beta \subseteq \wt{\powerset} (\kappa_\eta)$.  
By Chapter \ref{6.2} B), Definition \ref{deffbetaprime} and Proposition \ref{fbetafbetaprime}, it follows that $f^\beta \in V[(G^\beta\, \uhr\, (\eta + 1))^{( \eta_m, i_m)_{m < \omega}}\, \times\, \prod_m G^{\eta_m}_{i_m}\, \uhr\, [\kappa_\eta, \kappa_{\eta_m})]$; where $(G^\beta\, \uhr\, (\eta + 1))^{( \eta_m, i_m)_{m < \omega}}\, \times\, \prod_m G^{\eta_m}_{i_m}\, \uhr\, [\kappa_\eta, \kappa_{\eta_m})$ is a $V$-generic filter on the forcing notion $((\m{P}^\beta\, \uhr\, (\eta + 1))^{( \eta_m, i_m)_{m < \omega}}\, \times\, \prod_m P^{\eta_m}\, \uhr\, [\kappa_\eta, \kappa_{\eta_m})$, which preserves cardinals $\geq \alpha_\eta$ by Chapter \ref{6.2} C), Proposition \ref{prescardalphaeta}.

However, since $dom \,f^\beta \subseteq \wt{\powerset}(\kappa_\eta)$ and $|\beta|^V < \alpha_\eta$, it follows that $f^\beta$ together with the map $\iota$ from Proposition \ref{iota} above, collapses the cardinal $\alpha_\eta$ in $V[(G^\beta\, \uhr\, (\eta + 1))^{( \eta_m, i_m)_{m < \omega}}\, \times\, \prod_m G^{\eta_m}_{i_m}\, \uhr\, [\kappa_\eta, \kappa_{\eta_m})]$. Contradiction. \\[-2mm]

Thus, we have shown that our initial assumption that $f^\beta: \dom f^\beta \rightarrow \alpha_\eta$ was surjective, was wrong. \\[-3mm]

Hence, there must be $\alpha < \alpha_\eta$ with $\alpha \notin \rg f^\beta$. 

\subsubsection*{E) We use an isomorphism argument and obtain a contradiction.}

%We will need the following canonical name for $f^\beta$:

%\begin{definition} \label{namefbeta} Let $\dot{f}^\beta$ denote the collection of all $(\OR_{\m{P}} (\dot{Y}, \alpha), q)$ 
%\colorbox{yellow}{TO DO: Die Bezeichnung $\OR(...)$ müsste eingeführt werden!}
%such that $q \in \m{P}$ with $q \Vdash (\dot{Y}, \alpha) \in \dot{f}$, and there exists an $\eta$-good pair $\varrho = ( (a_m)_{m < \omega}, (\ol{\sigma}_m, \ol{i}_m)_{m < \omega}) \in M_\beta$ 
%with $\ol{i}_j < \beta$ for all $j < \omega$ 
%and a name $\dot{Z} \in \Name ( (\ol{P}^\eta)^\omega\, \times\, \prod_m P^{\ol{\sigma}_m})$ with $\dot{Y} = \tau_\varrho (\dot{Z})$.
%\end{definition}

%It is not difficult to check that indeed, $(\dot{f}^\beta)^G = f^\beta$. \\[-3mm]

We fix an ordinal $\alpha < \alpha_\eta$ with $\alpha \notin \rg f^\beta$. By surjectivity of $f$, there must be $X \subseteq \kappa_\eta$, $X \in N$, with $f(X) = \alpha$. Hence, there is an $\eta$-good pair $\varrho = ( (a_m)_{m < \omega}, (\ol{\sigma}_m, \ol{i}_m)_{m < \omega})$ with $X \in V[\prod_m G_\ast (a_m)\, \times\, \prod_m G^{\ol{\sigma}_m}_{\ol{i}_m}]$; but since $X \notin dom \,f^\beta$, there must be at least one index $m < \omega$ with $\ol{i}_m \geq \beta$. Let $S_0$ denote the set of all $(\ol{\sigma}_m, \ol{i}_m)$ with $\ol{i}_m < \beta$, and let $S_1$ be the set of all $(\ol{\sigma}_m, \ol{i}_m)$ with $\ol{i}_m \geq \beta$. Then $|S_1| \geq 1$.

For better clarity, we now switch to a slightly different notation, and write $(\ol{\lambda}_m, \ol{k}_m) := (\ol{\sigma}_m, \ol{i}_m)$ in the case that $m \in S_1$. We denote our $\eta$-good pair $\varrho$ by \[\varrho = \big( (a_m)_{m < \omega}, \big((\ol{\sigma}_m, \ol{i}_m)_{m \in S_0}, (\ol{\lambda}_m, \ol{k}_m)_{m  \in S_1}\big)\big). \] Then \[X = \dot{X}^{\prod_m G_\ast (a_m)\, \times\, \prod_{m \in S_0} G^{\ol{\sigma}_m}_{\ol{i}_m}\, \times\, \prod_{m \in S_1} G^{\ol{\lambda}_m}_{\ol{k}_m}} \]
for some $\dot{X} \in Name ( (\ol{P}^\eta)^\omega\, \times\, \prod_{m \in S_0} P^{\ol{\sigma}_m}\, \times\, \prod_{m \in S_1} P^{\ol{\lambda}_m})$, such that the following holds: \begin{itemize} \item $ (a_m\ | \ m < \omega)$ is a sequence of pairwise disjoint $\kappa_\eta$-subsets, such that for all $m < \omega$ and $\kappa_{\ol{\nu}, \ol{\j}} < \kappa_\eta$, it follows that $|a_m\, \cap \, [\kappa_{\ol{\nu}, \ol{\j}}, \kappa_{\ol{\nu}, \ol{\j} + 1})| = 1$, \item $S_0 \subseteq \omega$, and for all $m \in S_0$, we have $\ol{\sigma}_m \in Succ$ with $\ol{\sigma}_m \leq \eta$, $\ol{i}_m < \min\{\alpha_{\ol{\sigma}_m}, \beta\}$, \item if $m$, $m^\prime \in S_0$ with $m \neq m^\prime$, then $(\ol{\sigma}_m, \ol{i}_m) \neq (\ol{\sigma}_{m^\prime}, \ol{i}_{m^\prime})$, \item $\emptyset \neq S_1 \subseteq \omega$, and for all $m \in S_1$, we have $\ol{\lambda}_m \in Succ$ with $\ol{\lambda}_m \leq \eta$, $\ol{k}_m \in [\beta, \alpha_{\ol{\lambda}_m})$, \item if $m$, $m^\prime \in S_1$ with $m \neq m^\prime$, then  $(\ol{\lambda}_m, \ol{k}_m) \neq (\ol{\lambda}_{m^\prime}, \ol{k}_{m^\prime})$.\end{itemize}

%Let $\rho$ denote the good pair $( (a_m)_{j < \omega}, ( (\ol{\eta}_m, \ol{i}_m)_{m \in S_0}, (\ol{\lambda}_m, \ol{k}_m)_{m \in S_1}))$, and let 
%Setting $\wt{\dot{X}} := \tau_\varrho (\dot{X})$, it follows that \[{\wt{\dot{X}}}^G =\dot{X}^{\prod_m G_\ast (a_m)\, \times\, \prod_{m \in S_0} G^{\ol{\sigma}_m}_{\ol{i}_m}\, \times\, \prod_{m \in S_1} G^{\ol{\lambda}_m}_{\ol{k}_m}} = X.\] 

Since $(X, \alpha) \in f$, take $p \in G$ with \[p \Vdash_s (\tau_{\varrho} (\dot{X}), \alpha) \in \dot{f}.\]
%\mbox {\ \, and \, \ } p \Vdash_s \alpha \notin \rg \dot{f}^\beta,\] where $\dot{f}^\beta$ denotes the canonical name for $f^\beta$ from Definition \ref{namefbeta} above. 

Since we are using countable support, we can asssume w.l.o.g. that $\ol{\sigma}_m \in \supp p_1$, $\ol{i}_m \in \dom_x p_1 (\ol{\sigma}_m)$ for all $m \in S_0$; and $\ol{\lambda}_m \in \supp p_1$, $\ol{k}_m \in \dom_x p_1 (\ol{\lambda}_m)$ for all $m \in S_1$. \\[-2mm]

The idea can roughly be explained as follows: Recall that we have $\beta = \wt{\beta} \plus \kappa_\eta^\plus$ (addition of ordinals), where the ordinal $\wt{\beta}$ is \textit{large enough for $(A_{\dot{f}})$}. In particular, $\kappa_\eta^\plus < \wt{\beta} < \beta < \alpha_\eta$.
We will now extend $p$ and obtain a condition $q \in G$, $q \leq p$, such that there is a sequence $(\ol{l}_m\ | \ m \in S_1)$ with $\ol{l}_m \in (\wt{\beta}, \beta)$ for all $m \in S_1$, such that $q^{\ol{\lambda}_m}_{\ol{k}_m} = q^{\ol{\lambda}_m}_{\ol{l}_m}$ for all $m \in S_1$. Then we construct an isomorphism $\pi \in A$ that swaps any $(\ol{\lambda}_m, \ol{k}_m)$-coordinate with the according $(\ol{\lambda}_m, \ol{l}_m)$-coordinate.

%\colorbox{cyan}{Peter wegen Formulierung fragen!}

%\colorbox{red}{ACHTUNG - wieso \tbl Addition von Ordinalzahlen\tbr\, bei $\wt{\beta}$ und $\beta$? GEHT DAS ÜBERHAUPT?} 

%\colorbox{red}{NOCHMAL DURCHSEHEN?}

% and does nothing else. 
Then $\pi q = q$; and we will see that $\pi \in \bigcap_m Fix (\eta_m, i_m)\, \cap\, \bigcap_m H^{\lambda_m}_{k_m}$, since $\wt{\beta}$ is \textit{large enough for} $(A_{\dot{f}})$. Hence, $\pi \ol{f}^{D_\pi} = \ol{f}^{D_\pi}$; so from $q \Vdash_s (\tau_{\varrho}(\dot{X}), \alpha) \in \dot{f}$, we obtain that $q \Vdash_s (\pi \ol{\tau_{\varrho}(\dot{X})}^{D_\pi}, \alpha) \in \ol{f}^{D_\pi}$. Setting \[Y := \Big(\pi \ol{\tau_{\varrho}(\dot{X})}^{D_\pi}\Big)^G,\] it follows that \[(Y, \alpha) \in f. \]

However, we will see that $Y = \dot{X}^{\prod_m G_\ast (a_m)\, \times\, \prod_{m \in S_0} G^{\ol{\sigma}_m}_{\ol{i}_m}\, \times\, \prod_{m \in S_1} G^{\ol{\lambda}_m}_{\ol{l}_m}}$; where $\ol{i}_m < \beta$ for all $m \in S_0$, but also $\ol{l}_m < \beta$ for all $m \in S_1$. But then, the $\eta$-good pair $\varrho^\prime = ( (a_m)_{m < \omega}, ((\ol{\sigma}_m, \ol{i}_m)_{m \in S_0}, (\ol{\lambda}_m, \ol{l}_m)_{m \in S_1}))$ is an element of $M_\beta$, and it follows that \[(Y, \alpha) = \big(\dot{X}^{\prod_m G_\ast (a_m)\, \times\, \prod_{m \in S_0} G^{\ol{\sigma}_m}_{\ol{i}_m}\, \times\, \prod_{m \in S_1} G^{\ol{\lambda}_m}_{\ol{l}_m}}, \alpha\big) \in f^\beta.\] But this would be a contradiction towards $\alpha \notin \rg f^\beta$. \\

%\colorbox{red}{ACHTUNG - wurde die Schreibweise ${\wt{\dot{X}}}^{D_\pi}$ bzw. ${\wt{X}}^{D_\pi}$ irgendwo eingeführt? Wahrscheinlich NICHT!}

%\colorbox{red}{FRAGE: könnte mal EVTL überall $\tau_{\varrho} (\dot{X})$ schreiben anstatt $\wt{\dot{X}}$?}

%and $\pi \in \bigcap_m Fix (\eta_m, i_m)$: We have $\ol{\lambda}_m \leq \eta$, $\ol{k}_m \geq \beta > \wt{\beta}$, $\ol{l}_m > \wt{\beta}$ for all $m \in S_1$; but since $\wt{\beta}$ is \textit{large enough for} $(A_{\dot{f}})$, it follows that whenever $\eta_{j^\prime} \leq \eta$, then $i_{j^\prime} < \wt{\beta}$. Hence, it follows that neither $(\ol{\lambda}_m, \ol{k}_m)$ nor $(\ol{\lambda}_m, \ol{l}_m)$ can be equal to any $(\eta_m, i_m)$, which gives $\pi \in \bigcap Fix (\eta_m, i_m)$ as desired.. Similary

We start our proof with the following lemma:

\begin{lem} Let $D$ be the set of all $q \in \m{P}$ for which there exists a sequence of pairwise distinct ordinals $(\ol{l}_m\ | \ m \in S_1)$ with $\ol{l}_m \in (\wt{\beta}, \beta) \setminus \{ \ol{i}_m\ | \ m \in S_0\}$ for all $m \in S_1$, such that $q^{\ol{\lambda}_m}_{\ol{k}_m} = q^{\ol{\lambda}_m}_{\ol{l}_m}$ holds for all $m \in S_1$. Then $D$ is dense below $p$.
\end{lem}

\begin{proof} Consider $q \in \m{P}$ with $q \leq p$. We have to construct $\ol{q} \leq q$ with $q \in D$. The idea is that for every $m \in S_1$, we enlarge $\dom_x q(\ol{\lambda}_m)$ by some suitable $\ol{k}_m$, and set $\ol{q}(\ol{\lambda}_m) (\ol{k}_m, \zeta) := \ol{q} (\ol{\lambda}_m) (\ol{l}_m, \zeta) = q(\ol{\lambda}_m) (\ol{l}_m, \zeta)$ for all $\zeta \in \dom_y \ol{q} (\ol{\lambda}_m) = \dom_y q (\ol{\lambda}_m)$.\\[-3mm]

Note that for every $m \in S_1$, we have $\ol{\lambda}_m \in \supp q_1$ with $|\dom_x q_1 (\ol{\lambda}_m)| < \kappa_{\ol{\lambda}_m} \leq \kappa_\eta$, since $\ol{\lambda}_m \leq \eta$. Hence, it follows that $|\bigcup_{m \in S_1} \dom_x q (\ol{\lambda}_m)| \leq \kappa_\eta < \kappa_\eta^\plus$; and similarly, $|\bigcup_{m \in S_0} \dom_x q (\ol{\sigma}_m)| \leq \kappa_\eta < \kappa_\eta^\plus$. Thus, the set \[\Delta := (\wt{\beta}, \beta)\; \setminus\; \big(\, \bigcup_{m \in S_1} \dom_x q(\ol{\lambda}_m)\; \cup\; \bigcup_{m \in S_0} \dom_x q(\ol{\sigma}_m)\, \big)\] has cardinality $\kappa_\eta^\plus$. \\[-3mm]

Recall that for every $m \in S_0$, we have assumed that $\ol{i}_m \in \dom_x p (\ol{\sigma}_m) \subseteq \dom_x q (\ol{\sigma}_m)$; hence $\ol{i}_m \notin \Delta$. 

For $m \in S_1$, we have $\ol{k}_m \in [\beta, \alpha_{\ol{\lambda}_m})$; hence, $\beta < \alpha_{\ol{\lambda}_m}$ and $\Delta \subseteq (\wt{\beta}, \beta) \subseteq \alpha_{\ol{\lambda}_m}$ follows. \\[-3mm]

We take a sequence of pairwise distinct ordinals $(\ol{l}_m \ | \ m \in S_1)$ in $\Delta$ (then $\{\ol{l}_m\ | \ m \in S_1\} \subseteq (\wt{\beta}, \beta) \setminus \{\ol{i}_m\ | \ m \in S_0\}$), and define the extension $\ol{q} \leq q$ as follows: \\[-3mm]

Set $\ol{q}_0 := q_0$, and $\supp \ol{q}_1 = \supp q_1$. (From $q \leq p$ it follows that $\ol{\lambda}_m \in \supp q_1$ for all $m \in S_1$.) For $\sigma \in \supp \ol{q}_1$ with $\sigma \notin \{\ol{\lambda}_m\ | \ m \in S_1\}$, we set $\ol{q} (\sigma) := q (\sigma)$. For $\sigma \in \{\ol{\lambda}_m\ | \ m \in S_1\}$, we proceed as follows: Let $S_1 (\sigma) := \{m \in S_1\ | \ \sigma = \ol{\lambda}_m\}$. We set $\dom_y \ol{q} (\sigma) := \dom_y q (\sigma)$, and $\dom_x \ol{q} (\sigma) := \dom_x q(\sigma)\, \cup\, \{\ol{l}_m\ | \ m \in S_1 (\sigma)\}$. Note that by construction of $\Delta$ this union is disjoint, since $\ol{l}_m \notin \dom_x q(\sigma) = \dom_x q (\ol{\lambda}_m)$ for all $m \in S_1 (\sigma)$. \\[-3mm]
%for all $m \in S_1 (\sigma)$.\\[-3mm]

%\colorbox{yellow}{FRAGE: Sollte man eine Bezeichnung für die Menge $\{m \in S_1\ | \ \sigma = \ol{\lambda}_m\}$ einführen?}
Note that for every $m \in S_1 (\sigma)$, we have $\ol{k}_m \in \dom_x p(\sigma) \subseteq \dom_x q(\sigma) \subseteq \dom_x \ol{q} (\sigma)$. \\[-3mm] 

We let $\ol{q} (\sigma) (i, \zeta) := q (\sigma) (i, \zeta)$ whenever $(i, \zeta) \in \dom_x q(\sigma)\, \times\, \dom_y q(\sigma)$. If $(i, \zeta) \in \dom \ol{q} (\sigma) \setminus \dom q(\sigma)$, then $\zeta \in \dom_y q(\sigma)$ and $i = \ol{l}_{n}$ for some $n \in S_1 (\sigma)$, i.e. $n \in S_1$ with $\sigma = \ol{\lambda}_n$. In this case, we set $\ol{q} (\sigma) (i, \zeta) = \ol{q} (\ol{\lambda}_{n}) (\ol{l}_{n}, \zeta) := q (\ol{\lambda}_{n}) (\ol{k}_{n}, \zeta) = q (\sigma) (\ol{k}_{n}, \zeta)$. \\[-3mm]

%\colorbox{red}{ACHTUNG - $N$ IST DOCH DAS MODELL! DAS MÜSSTE MAN UMDEFINIEREN!}

%(Note that $\ol{l}_{\ol{m}} \notin \dom_x q (\ol{\lambda}_m)$ for all $m \in S_1$.)

This defines $\ol{q} \leq q$ with the property that $\ol{q}^{\ol{\lambda}_m}_{\ol{k}_m} = \ol{q}^{\ol{\lambda}_m}_{\ol{l}_m}$ holds for all $m \in S_1$. \\[-3mm]
% We have assumed that for every $m \in S_1$, $\ol{\lambda}_m \in \supp q_1$ and $\ol{k}_m \in \dom_x q (\ol{\lambda}_m)$. Hence, whenever $\zeta \in \dom_x \ol{q} (\ol{\lambda}_m) = \dom_x q (\ol{\lambda}_m$, then $\ol{q} (\ol{\lambda}_m) (\ol{k}_m, \zeta) = q (\ol{\lambda}_m) (\ol{k}_m, \zeta) = \ol{q} (\ol{\lambda}_m) (\ol{l}_m)$ by construction. 

Thus, it follows that $D$ is dense below $p$.

\end{proof}

Since $p \in G$, we can now take $q \in G$, $q \leq p$ with $q \in D$. Take $(\ol{l}_m\ | \ m \in S_1)$ as in the definition of $D$, with $\ol{l}_m \in (\wt{\beta}, \beta) \setminus \{\ol{i}_m\ | \ m \in S_0\}$ and $q^{\ol{\lambda}_m}_{\ol{k}_m} = q^{\ol{\lambda}_m}_{\ol{l}_m}$ for all $m \in S_1$. Then the sets $\{(\ol{\lambda}_m, \ol{l}_m) \ | \ m \in S_1\}$ and $\{(\ol{\sigma}_m, \ol{i}_m)\ | \ m \in S_0\}$ are disjoint. \\[-3mm]

Since $q \leq p$, we have \[q \Vdash_s (\tau_{\varrho} (\dot{X}), \alpha) \in \dot{f}. \]

%\ \ , \ \ q \Vdash_s \alpha \notin \rg \dot{f}^\beta.\]

The next step is to construct an isomorphism $\pi$ that swaps every $(\ol{\lambda}_m, \ol{k}_m)$-coordinate with the according $(\ol{\lambda}_m, \ol{l}_m)$-coordinate for $m \in S_1$, and does nothing else.

\begin{definition}
We define an isomorphism $\pi \in A$ as follows:

\begin{itemize} \item The map $\pi_0$ is the identity on $D_{\pi_0} = \m{P}_0$. \item We set $\supp \pi_1 := \supp q_1$, and for every $\sigma \in \supp q_1$, we let $\dom \pi_1 (\sigma) := \dom q_1 (\sigma)$. 

{\textrm{Then for all $m \in S_1$, it follows that $\ol{\lambda}_m \in \supp p_1 \subseteq \supp q_1 = \supp \pi_1$; and $\ol{k}_m \in \dom_x p_1 (\ol{\lambda}_m) \subseteq \dom_x q_1 (\ol{\lambda}_m) = \dom_x \pi_1 (\ol{\lambda}_m)$, $\ol{l}_m \in \dom_x q_1 (\ol{\lambda}_m) = \dom_x \pi_1 (\ol{\lambda}_m)$.}}

 \item Consider $\sigma \in \supp \pi_1$ with $\kappa_\sigma = \ol{\kappa_\sigma}^{\,\plus}$. In the case that $\sigma \notin \{\ol{\lambda}_m\ | \ m \in S_1\}$, we set $\supp \pi_1 (\sigma) := \emptyset$, and let $\pi_1 (\sigma) (i, \zeta): 2 \rightarrow 2$ be the identity map for all $(i, \zeta) \in \alpha_\sigma\, \times\, [\ol{\kappa_\sigma}, \kappa_\sigma)$. \item For $\sigma \in \{ \ol{\lambda}_m \ | \ m \in S_1\}$, consider the set $S_1 (\sigma) := \{ m \in S_1\ | \ \sigma = \ol{\lambda}_m\}$, and let $\supp \pi_1 (\sigma) := \{ \ol{k}_m\ | \ m \in S_1 (\sigma)\}\, \cup\, \{ \ol{l}_m\ | \ m \in S_1 (\sigma)\}$. Then $\supp \pi_1 (\sigma)$ is a subset of $\dom_x \pi_1 (\sigma)$. \\[-3mm]
% = \dom q_1 (\sigma)$, since $\ol{k}_m \in \dom_x p_1 (\ol{\lambda}_m) = \dom_x p_1 (\sigma) \subseteq \dom_x q_1 (\sigma)$ for all $m \in S_1 (\sigma)$, and $\ol{l}_m \in \dom_x q_1 (\ol{\lambda}_m) = \dom_x q_1 (\sigma)$ for all $m \in S_1 (\sigma)$ by construction.

The map $f_{\pi_1} (\sigma): \supp \pi_1 (\sigma) \rightarrow \supp \pi_1 (\sigma)$ is defined as follows: Let $f_{\pi_1} (\sigma) (\ol{k}_m) = \ol{l}_m$, and $f_{\pi_1} (\sigma) (\ol{l}_m) = \ol{k}_m$ for all $m \in S_1 (\sigma)$. 

{\textrm{Then $f_{\pi_1} (\sigma)$ is well-defined and bijective, since $\ol{k}_m \geq \beta$ for all $m \in S_1$, and $\ol{l}_m < \beta$ for all $m \in S_1$.}} \\[-3mm]

It remains to define the maps $\pi_1 (\zeta): 2^{\supp \pi_1 (\sigma)} \rightarrow 2^{\supp \pi_1 (\sigma)}$ for $\zeta \in \dom_y \pi_1 (\sigma)$: Let $\pi_1 (\zeta) (\epsilon_i\ | \ i \in \supp \pi_1 (\sigma)) := (\wt{\epsilon}_i\ | \ i \in \supp \pi_1 (\sigma))$, where $\wt{\epsilon}_{\ol{k}_m} := \epsilon_{\ol{l}_m}$, $\wt{\epsilon}_{\ol{l}_m} := \epsilon_{\ol{k}_m}$ for all $m \in S_1 (\sigma)$. \\[-3mm]

Finally, for every $(i, \zeta) \in \alpha_\sigma\, \times\, [\ol{\kappa_\sigma}, \kappa_\sigma)$, we let $\pi_1 (\sigma) (i, \zeta): 2 \rightarrow 2$ be the identity.

%ACHTUNG - wie sieht $\m{P}_1$ aus? Werden Teilmengen von $[\kappa_{\ol{\eta}}, \kappa_\eta)$ oder von $[\ol{\kappa_\eta}, \kappa_\eta)$% zugefügt?
%Das müsste man ÜBERLEGEN!

\end{itemize}

\end{definition}

This defines our automorphism $\pi \in A$.

%Insgesamt besser strukturieren! In Proposition etc. einteilen?

\begin{lem} \label{yalphainf} For $Y := \Big(\pi \ol{\tau_{\varrho}(\dot{X})}^{D_\pi}\Big)^G$, it follows that $(Y, \alpha) \in f$. \end{lem}

\begin{proof} By construction of $\pi$ it follows that whenever $r$ is a condition in $D_\pi$ with $r^\prime := \pi r$, then the following holds: Firstly, for all $m \in S_1$, we have $(r^\prime)^{\ol{\lambda}_m}_{\ol{k}_m} = r^{\ol{\lambda}_m}_{\ol{l}_m}$ and $(r^\prime)^{\ol{\lambda}_m}_{\ol{l}_m} = r^{\ol{\lambda}_m}_{\ol{k}_m}$.
Secondly, whenever $\sigma \in \supp r_1$, $i \in \dom_x r (\sigma)$ with $(\sigma, i) \notin \{(\ol{\lambda}_m, \ol{k}_m)\ | \ m \in S_1\} \, \cup\, \{ (\ol{\lambda}_m, \ol{l}_m)\ | \ m \in S_1\}$, then $(r^\prime)^\sigma_i = r^\sigma_i$. 

In particular, $(r^\prime)^{\ol{\sigma}_{m^\prime}}_{\ol{i}_{m^\prime}} = r^{\ol{\sigma}_{m^\prime}}_{\ol{i}_{m^\prime}}$ holds for all $m^\prime \in S_0$:

%\colorbox{yellow}{FRAGE: anderen Namen als $m^\prime$ für Elemente aus $S_0$ wählen?}

On the one hand, we have $(\ol{\sigma}_{m^\prime}, \ol{i}_{m^\prime}) \notin \{(\ol{\lambda}_m, \ol{k}_m)\ | \ m \in S_1\}$ for all $m^\prime \in S_0$, since $\ol{i}_{m^\prime} < \beta$; but $\ol{k}_m \geq \beta$ for all $m \in S_1$. On the other hand, $(\ol{\sigma}_{m^\prime}, \ol{i}_{m^\prime}) \notin \{(\ol{\lambda}_m, \ol{l}_m)\ | \ m \in S_1\}$ for all $m^\prime \in S_0$ follows by construction of the set $D$. \\[-3mm]

In other words: The map $\pi$ swaps for all $m \in S_1$ the $(\ol{\lambda}_m, \ol{k}_m)$-coordinate with the according $(\ol{\lambda}_m, \ol{l}_m)$-coordinate, and does nothing else.

Hence, it follows that $\pi q = q$; since $q^{\ol{\lambda}_m}_{\ol{k}_m} = q^{\ol{\lambda}_m}_{\ol{l}_m}$ for all $m \in S_1$. \\[-2mm]

Next, we want to show that $\pi \in \bigcap_m Fix (\eta_m, i_m)\, \cap \, \bigcap_m H^{\lambda_m}_{k_m}$. Then $\pi \ol{f}^{D_\pi} = \ol{f}^{D_\pi}$ follows. Regarding $\pi \in \bigcap_m Fix (\eta_m, i_m)$, it suffices to make sure that for all $m < \omega$, we have $(\eta_{m}, i_{m}) \notin \{ (\ol{\lambda}_{m^\prime}, \ol{k}_{m^\prime})\ | \ m^\prime \in S_1\}\, \cup\, \{ (\ol{\lambda}_{m^\prime}, \ol{l}_{m^\prime})\ | \ m^\prime \in S_1\}$. But this follows from the fact that $\ol{\lambda}_{m^\prime} \leq \eta$ and $\ol{k}_{m^\prime} \geq \beta > \wt{\beta}$, $\ol{l}_{m^\prime} > \wt{\beta}$ for all $m^\prime \in S_1$; but $\wt{\beta}$ is \textit{large enough for }$(A_{\dot{f}})$, so for any $\eta_{m}$ with $\eta_{m} \leq \eta$, it follows that $i_{m} < \wt{\beta}$. This implies $(\eta_{m}, i_{m}) \notin \{ (\ol{\lambda}_{m^\prime}, \ol{k}_{m^\prime})\ | \ m^\prime \in S_1\}\, \cup\, \{ (\ol{\lambda}_{m^\prime}, \ol{l}_{m^\prime})\ | \ m^\prime \in S_1\}$ for all $m < \omega$ as desired. \\ Hence, $\pi \in \bigcap_m Fix (\eta_m, i_m)$. \\[-3mm] 

%$(\eta_{m_0}, i_{m_0}) \neq (\ol{\lambda}_{m_1}, \ol{k}_{m_1})$ and $(\eta_{m_0}, i_{m_0}) \neq (\ol{\lambda}_{m_1}, \ol{l}_{m_1})$ for any $m_0$, $m_1$. But this follows from the fact that $\ol{\lambda}_{m_1} \leq \eta$ and $\ol{k}_{m_1} \geq \beta > \wt{\beta}$, $\ol{l}_{m_1} > \wt{\beta}$ for all $m_1 \in S_1$; but $\wt{\beta}$ is \textit{large enough for }$(A_{\dot{f}})$, so for any $\eta_{m_0}$ with $\eta_{m_0} \leq \eta$, it follows that $i_{m_0} < \wt{\beta}$. This implies $(\eta_{m_0}, i_{m_0}) \neq (\ol{\lambda}_{m_1}, \ol{k}_{m_1})$, $(\eta_{m_0}, i_{m_0}) \neq (\ol{\lambda}_{m_1}, \ol{l}_{m_1})$ for all $m_0$, $m_1$ as desired. \\ Hence, $\pi \in \bigcap_m Fix (\eta_m, i_m)$. \\[-3mm]

Regarding $\pi \in \bigcap_m H^{\lambda_m}_{k_m}$, we have to make sure that 
%for all $m_1$ it follows that $\supp \pi_1 (\lambda_{m_1}) \subseteq (k_{m_1}, \alpha_{\lambda_{m_1}})$. In other words: 
whenever $\lambda_m = \ol{\lambda}_{m^\prime}$ for some $m < \omega$ and $m^\prime \in S_1$, then $\supp \pi_1 (\lambda_{m}) = \supp \pi_1 (\ol{\lambda}_{m^\prime}) \subseteq (k_{m}, \alpha_{\lambda_{m}})$ holds; i.e. $\ol{k}_{m^\prime} > k_m$ and $\ol{l}_{m^\prime} > k_m$. Again, this follows from the fact that $\ol{\lambda}_{m^\prime} \leq \eta$ and $\ol{k}_{m^\prime} \geq \beta > \wt{\beta}$, $\ol{l}_{m^\prime} > \wt{\beta}$ for all $m^\prime \in S_1$; and $\wt{\beta}$ is \textit{large enough for} $(A_{\dot{f}})$, so whenever $\lambda_{m} \leq \eta$, then $k_{m} < \wt{\beta}$ follows. Hence, $\pi \in \bigcap_m H^{\lambda_m}_{k_m}$. \\[-3mm]

%\colorbox{yellow}{ACHTUNG - das ist SCHWIERIG mit $m_0$ und $m_1$!!}

%\colorbox{yellow}{EVTL bessere Bezeichnungen versuchen?}

%\colorbox{red}{HIER $m$ und $m^\prime$ VERTAUSCHEN?}

Thus, it follows that $\pi \ol{f}^{D_\pi} = \ol{f}^{D_\pi}$. \\[-2mm]

Now, from $q \Vdash_s (\tau_{\varrho} (\dot{X}), \alpha) \in \dot{f}$, we obtain $\pi q \Vdash_s (\pi \ol{\tau_{\varrho}(\dot{X})}^{D_\pi}, \alpha) \in \pi \ol{f}^{D_\pi}$; hence, $q \Vdash_s (\pi \ol{\tau_{\varrho} (\dot{X})}^{D_\pi}, \alpha) \in \ol{f}^{D_\pi}$. With \[Y := \big(\pi \ol{\tau_{\varrho} (\dot{X})}^{D_\pi}\big)^G,\] it follows from $q \in G$ that $(Y, \alpha) \in f$ as desired.

\end{proof}

%Moreover, $q \Vdash \alpha \notin \rg \dot{f}^\beta$
%Erwähnen, dass $\alpha \notin \rg f^\beta$! 

We will now show that $(Y, \alpha) \in f$ implies that also $(Y, \alpha) \in f^\beta$ must hold. 
%Since $q \in G$ with $q \Vdash_s \alpha \notin \rg \dot{f}^\beta$, 
This finally gives our desired contradiction, since $\alpha \notin \rg f^\beta$. \\[-3mm]

Indeed, we will prove that \[Y = \dot{X}^{\prod_m G_\ast (a_m)\, \times\, \prod_{m \in S_0} G^{\ol{\sigma}_m}_{\ol{i}_m}\, \times\, \prod_{m \in S_1} G^{\ol{\lambda}_m}_{\ol{l}_m}}.\] Since $\ol{i}_m < \beta$ for all $m \in S_0$ and $\ol{l}_m < \beta$ for all $m \in S_1$, it follows that the $\eta$-good pair \[\varrho^\prime := \big( (a_m)_{m < \omega}, ( (\ol{\sigma}_m, \ol{i}_m)_{m \in S_0}, (\ol{\lambda}_m, \ol{l}_m)_{m \in S_1})\big)\] is an element of $M_\beta$. Hence, $(Y, \alpha) \in f$
%$\big(\, \dot{X}^{\prod_m G_\ast (a_m)\, \times\, \prod_{m \in S_0} G^{\ol{\sigma}_m}_{\ol{i}_m}\, \times\, \prod_{m \in S_1} G^{\ol{\lambda}_m}_{\ol{l}_m}}, \alpha \, \big) \in f$ 
would then imply that also $(Y, \alpha) \in f^\beta$
%$\big(\, \dot{X}^{\prod_m G_\ast (a_m)\, \times\, \prod_{m \in S_0} G^{\ol{\sigma}_m}_{\ol{i}_m}\, \times\, \prod_{m \in S_1} G^{\ol{\lambda}_m}_{\ol{l}_m}}, \alpha \, \big) \in f^\beta$ 
must hold, and we are done. \\[-1mm]

Recall that 
%$\wt{\dot{X}}$ was our abbreviation for $\tau_{\varrho} (\dot{X})$, where 
%$\varrho$ is the $\eta$-good pair
\[\varrho := \big( (a_m)_{m < \omega}, ( (\ol{\sigma}_m, \ol{i}_m)_{m \in S_0}, (\ol{\lambda}_m, \ol{k}_m)_{m \in S_1})\big),\]
%$\tau_{\varrho} (\dot{X})$ is the canonical extension of 
$\dot{X} \in \Name ( (\ol{P}^\eta)^\omega\, \times\, \prod_{m \in S_0} P^{\ol{\sigma}_m}\, \times\, \prod_{m \in S_1} P^{\ol{\lambda}_m})$, and $\tau_{\varrho} (\dot{X})$ is the canonical extension of $\dot{X}$ to a name for $\m{P}$ (see Definition \ref{tauvarrho}). \\[-2mm] 
%to a name for $\m{P}$. \\[-2mm]

We will show recursively:

\begin{lem} For every $\dot{Y} \in \Name \big( (\ol{P}^\eta)^\omega\,\times\, \prod_{m \in S_0} P^{\ol{\sigma}_m}\, \times\, \prod_{m \in S_1} P^{\ol{\lambda}_m}\big)$, 
%and $\wt{\dot{Y}} := \tau_\rho (\dot{Y})$, 
it follows that \[\pi \, {\ol{\tau_\varrho (\dot{Y})}}^{D_\pi} = \ol{\tau_{\varrho^\prime} (\dot{Y})}^{D_\pi}.\] 

%\[\big(\, \pi \, \ol{\tau_\varrho (\dot{Y})}^{D_\pi}\, \big)^G = \big(\tau_{\varrho^\prime} (\dot{Y}) \big)^G = \dot{Y}^{\prod_m G_\ast (a_m)\, \times\, \prod_{m \in S_0} G^{\ol{\sigma}_m}_{\ol{i}_m}\, \times\, \prod_{m \in S_1} G^{\ol{\lambda}_m}_{\ol{l}_m}}.\]
\end{lem} 

\begin{proof} Consider $\dot{Y} \in \Name_{\alpha+1} ( (\ol{P}^\eta)^\omega\,\times\, \prod_{m \in S_0} P^{\ol{\sigma}_m}\, \times\, \prod_{m \in S_1} P^{\ol{\lambda}_m})$, and assume recursively that 
%ACHTUNG - nicht eher $\dot{Y} \in \Name ( (\ol{P}^\eta)^\omega\,\times\, \prod_{m \in S_0} P^{\ol{\sigma}_m}\, \times\, \prod_{m \in S_1} P^{\ol{\lambda}_m})$ ??
the claim was true for all $\dot{Z} \in \Name_\alpha( (\ol{P}^\eta)^\omega\,\times\, \prod_{m \in S_0} P^{\ol{\sigma}_m}\, \times\, \prod_{m \in S_1} P^{\ol{\lambda}_m})$. \\[-4mm]

%\colorbox{yellow}{Man müsste die $\Name_\alpha$-Hierarchie einführen!} \\[-3mm]

First, \[\ol{\tau_{\varrho} (\dot{Y})}^{D_\pi} = \big\{\, \big(\, \ol{\tau_{\varrho} (\dot{Z})}^{D_\pi}, r\, \big)\ \ \big| \ \ r \in D_\pi\;, \; \dot{Z} \in \dom \dot{Y}\; , \; r \Vdash_s \tau_{\varrho} (\dot{Z}) \in \tau_{\varrho} (\dot{Y})\, \big\},\]

and \[ \pi \ol{\tau_{\varrho} (\dot{Y})}^{D_\pi} = \big\{\, \big(\, \pi \ol{\tau_{\varrho} (\dot{Z})}^{D_\pi}, \pi r\, \big)\ \ \big| \ \ r \in D_\pi\; , \; \dot{Z} \in \dom \dot{Y}\, , \; r \Vdash_s \tau_{\varrho} (\dot{Z}) \in \tau_{\varrho} (\dot{Y})\, \big\}. \]

Now, for any $H$ a $V$-generic filter on $\m{P}$ and $\dot{Z}_0 \in \Name( (\ol{P}^\eta)^\omega\,\times\, \prod_{m \in S_0} P^{\ol{\sigma}_m}\, \times\, \prod_{m \in S_1} P^{\ol{\lambda}_m})$, it follows by construction of the map $\pi$ that \begin{eqnarray} \big(\tau_{\varrho} (\dot{Z}_0)\big)^H & = &\dot{Z}_0^{\prod_{m < \omega} H_\ast (a_m)\, \times\, \prod_{m \in S_0} H^{\ol{\sigma}_m}_{\ol{i}_m}\, \times\, \prod_{m \in S_1} H^{\ol{\lambda}_m}_{\ol{k}_m}} \nonumber \\  & = & \dot{Z}_0^{\prod_{m < \omega} (\pi H)_\ast (a_m)\, \times\, \prod_{m \in S_0} (\pi H)^{\ol{\sigma}_m}_{\ol{i}_m}\, \times\, \prod_{m \in S_1} (\pi H)^{\ol{\lambda}_m}_{\ol{l}_m}} \ \nonumber \\ & = & \big(\tau_{\varrho^\prime} (\dot{Z}_0)\big)^{\pi H}, \nonumber \end{eqnarray}

since $\pi$ swaps any $(\ol{\lambda}_m, \ol{k}_m)$-coordinate with the according $(\ol{\lambda}_m, \ol{l}_m$)-coordinate, and does nothing else. \\[-3mm]

Hence, whenever $r \in D_\pi$, then $r \Vdash_s \tau_{\varrho} (\dot{Z}) \in \tau_{\varrho} (\dot{Y})$ if and only if $\pi r \Vdash_s \tau_{\varrho^\prime} (\dot{Z}) \in \tau_{\varrho^\prime} (\dot{Y})$. Thus, by our recursive assumption, \begin{eqnarray} \pi \ol{\tau_{\varrho} (\dot{Y})}^{D_\pi} & = &\big\{\, \big(\, \pi \ol{\tau_{\varrho} (\dot{Z})}^{D_\pi}, \pi r\, \big)\ \ \big| \ \ \pi r \in D_\pi\; , \; \dot{Z} \in \dom \dot{Y}\, , \; \pi r \Vdash_s \tau_{\varrho^\prime} (\dot{Z}) \in \tau_{\varrho^\prime} (\dot{Y})\, \big\} \ \nonumber \\ & = & \big\{\, \big(\, \ol{\tau_{\varrho^\prime} (\dot{Z})}^{D_\pi}, r\, \big)\ \ \big| \ \  r \in D_\pi\; , \; \dot{Z} \in \dom \dot{Y}\, , \; r \Vdash_s \tau_{\varrho^\prime} (\dot{Z}) \in \tau_{\varrho^\prime} (\dot{Y})\, \big\} \ \nonumber \\  & = & \ol{\tau_{\varrho^\prime} (\dot{Y})}^{D_\pi}. \nonumber \end{eqnarray}

\end{proof}

Hence, \[ Y = \big( \pi \ol{\tau_{\varrho} (\dot{X})}^{D_\pi}\big)^G = \big( \ol{\tau_{\varrho^\prime} (\dot{X})}^{D_\pi}\big)^G = \big( \tau_{\varrho^\prime} (\dot{X})\big)^G = \]\[ = \dot{X}^{\prod_m G_\ast (a_m)\, \times\, \prod_{m \in S_0} G^{\ol{\sigma}_m}_{\ol{i}_m}\, \times\, \prod_{m \in S_1} G^{\ol{\lambda}_m}_{\ol{l}_m}}.\]

%= \big( \pi \ol{\tau_{\rho} (\dot{X})}^{D_\pi}\big)^G =  = Y.\] 
Hence, by Lemma \ref{yalphainf} above, it follows that \[(\dot{X}^{\prod_{m} G_\ast (a_m)\, \times\, \prod_{m \in S_0} G^{\ol{\sigma}_m}_{\ol{i}_m}\, \times\, \prod_{m \in S_1} G^{\ol{\lambda}_m}_{\ol{l}_m}}, \alpha) \in f.\] But $\ol{i}_m < \beta$ for all $m \in S_0$ and $\ol{l}_m < \beta$ for all $m \in S_1$; hence, \[(\dot{X}^{\prod_m G_\ast (a_m)\, \times\, \prod_{m \in S_0} G^{\ol{\sigma}_m}_{\ol{i}_m}\, \times\, \prod_{m \in S_1} G^{\ol{\lambda}_m}_{\ol{l}_m}}, \alpha) \in f^\beta.\] 

%\colorbox{red}{FRAGE: Braucht man überhaupt $q \Vdash_s \alpha \notin \rg \dot{f}^\beta$?}

%\colorbox{red}{Könnte man nicht auch den Namen $\dot{f}^\beta$ dann weglassen?} 

%\colorbox{red}{TO DO: ÄNDERN?} 

But this contradicts our choice of $\alpha \notin \rg f^\beta$. \\

Thus, in either case our assumption of a surjective function $f: \powerset^N (\kappa_\eta) \rightarrow \alpha_\eta$ in $N$ has lead to a contradiction, and it follows that indeed, $\theta^N (\kappa_\eta) \leq \alpha_\eta$. \\

Recall that we have assumed throughout our proof that $\kappa_{\eta + 1} > \kappa_\eta^\plus$. In the next Chapter 6.3, we will treat the case that $\kappa_{\eta + 1} = \kappa_\eta^\plus$, and discuss where the arguments from Chapter \ref{6.2} have to be modified. 

\subsection{$\boldsymbol{\forall\, \eta\ \big(\,  \kappa_{\eta + 1} = \kappa_\eta^\plus\, \longrightarrow\, \theta^N (\kappa_\eta) \leq \alpha_\eta\, \big)}$.} \label{6.3}

If $\kappa_{\eta + 1} = \kappa_\eta^\plus$, we need the notion of an \textit{$\eta$-almost good pair} (cf. Definition \ref{etaalmostgoodpair} and Proposition \ref{Xsubseteqkappaeta2}): For any $X \in  N$, $X \subseteq \kappa_\eta$, there exists an \textit{$\eta$-almost good pair} $( (a_m)_{m < \omega},  (\ol{\sigma}_m, \ol{i}_m)_{m < \omega})$ such that $X \in V[ \prod_m G_\ast (a_m)\, \times\, \prod_m G^{\ol{\sigma}_m}_{\ol{i}_m}\, \times\, G^{\eta + 1}]$. \\[-2mm]

Throughout this Chapter 6.3, we assume that \[\boldsymbol{\kappa_{\eta + 1} = \kappa_\eta^\plus}.\]

As before in Chapter \ref{6.2}, we assume towards a contradiction that there was a surjective function $f: \powerset^N (\kappa_\eta) \rightarrow \alpha_\eta$ in $N$ with $\pi \ol{f}^{D_\pi} = \ol{f}^{D_\pi}$ for all $\pi \in A$ with $[\pi]$ contained in the intersection \[\bigcap_{m < \omega} Fix (\eta_m, i_m)\, \cap\, \bigcap_{m < \omega}H^{\lambda_m}_{k_m} \hspace*{4cm} (A_{\dot{f}}). \] We take $\wt{\beta}$ \textit{large enough for $(A_{\dot{f}})$} as in Chapter \ref{6.2}, Definition \ref{largeenough}, and set $\beta := \wt{\beta} \plus \kappa_\eta^\plus$ (addition of ordinals). \\[-2mm]

Now, we can adapt our definition of $f^\beta$ to $\eta$-almost good pairs, and obtain: \[f^\beta := \Big\{ \; (X, \alpha) \in f\ \; \big|\; \ \exists\, ( (a_m)_{m < \omega}, (\ol{\sigma}_m, \ol{i}_m)_{m < \omega}) \ \ \eta\mbox{\textit{-almost good pair}} \, :\ (\forall m\ \,\ol{i}_m < \beta)\; \wedge\; \]\[ \exists\, \dot{X} \in \Name\big((\ol{P}^\eta)^\omega\, \times\, \prod_m P^{\ol{\sigma}_m}\, \times\, P^{\eta + 1}\big)\  \, X = \dot{X}^{\prod_m G_\ast (a_m)\, \times\, \prod_m G^{\ol{\sigma}_m}_{\ol{i}_m}\, \times\, G^{\eta + 1}}\; \Big\}.\]

First, we assume towards a contradiction that $\boldsymbol{f^\beta: \dom f^\beta \rightarrow \alpha_\eta}$ \textbf{is surjective}.

%\colorbox{yellow}{ACHTUNG - wie sollte man ein $\eta$-almost good pair notieren?}

\subsubsection*{A) Constructing $\boldsymbol{\wt{\m{P}}^\beta\, \uhr\, (\eta + 1)}$.}

As before, we only treat the case that \[ \boldsymbol{\beta < \alpha_{\wt{\eta}}} \mbox{\, or \, } \boldsymbol{Lim\, \cap\, (\eta, \gamma) \neq \emptyset}, \] where $\wt{\eta} := \sup \{\sigma < \eta\ | \ \sigma \in Lim\}$, i.e. we presume that there exist $(\sigma, i)$ with $\sigma \in Lim$ and $i \geq \beta$ or $\sigma > \eta$. \\[-3mm]

This times, we construct a forcing notion $\wt{\m{P}}^\beta\, \uhr \, (\eta + 1)$ instead of $\m{P}^\beta\, \uhr \, (\eta + 1)$; which
%\colorbox{yellow}{TO DO: eine bessere Bezeichnung für $\wt{\m{P}}^\beta\, \uhr \, (\eta + 1)$!}
should be like $\m{P}^\beta\, \uhr\, (\eta + 1)$, except that firstly, we use restrictions $p_\ast\, \uhr\, \kappa_{\eta + 1}^2$ instead of $p_\ast\, \uhr\, \kappa_\eta^2$, and secondly, we include $P^{\eta + 1}$.

\begin{definition} For $p \in \m{P}$, let \[\wt{p}^\beta\, \uhr\, (\eta  + 1) := \big(\, p_\ast\, \uhr\, \kappa_{\eta + 1}^2, (p^\sigma_i, a^\sigma_i)_{\sigma \leq \eta, i < \beta}, (p^\sigma\, \uhr\, (\beta\, \times\, \dom_y p^\sigma))_{\sigma \leq \eta}, p^{\eta + 1}, X_p\, \big),\]

and denote by $\wt{\m{P}}^\beta\, \uhr\, (\eta + 1)$ the collection of all $\wt{p}^\beta\, \uhr\, (\eta + 1)$ such that $p \in \ol{\m{P}}$ (i.e. $p \in \m{P}$ with $|\{ (\sigma, i) \in \supp p_0\ | \ \sigma > \eta\, \vee\, i \geq \beta\}| = \aleph_0$\,); together with the maximal element $\wt{\m{1}}^\beta_{\eta + 1}$. The order relation ${\wt{\leq}}^\beta_{\eta + 1}$ is defined as in \mbox{\upshape Definition \ref{defpbetauhretapluseins}}.
\end{definition} 

%\colorbox{red}{ACHTUNG - $\emptyset$ als Bezeichnung für das maximale Element weglassen?}

Like in Chapter 6.2 A), one can write down a projection of forcing posets ${\wt{\rho}}^\beta_{\eta + 1}: \ol{\m{P}} \rightarrow \wt{\m{P}}^\beta\, \uhr\, (\eta + 1)$ and conclude that 
\[\wt{G}^\beta\, \uhr\, (\eta + 1) := \big\{ \, p \in \wt{\m{P}}^\beta\, \uhr\, (\eta + 1) \ \ \big| \ \ \exists\, q \in G\, \cap\, \ol{\m{P}}\  \  \,\wt{q}^\beta\, \uhr \, (\eta + 1) \ \wt{\leq}^\beta_{\eta + 1} \ p\, \big\}\]

%\ : \ |\{ (\sigma, i) \in \supp q_0 \ | \ \sigma > \eta\, \vee \, i \geq \beta\}| = \omega \ \wedge\, \]\[ \wedge\ p \, \wt{\geq}^\beta_{\eta + 1} \, \wt{q}^\beta\, \uhr \, (\eta + 1)\, \big\}\] 

is a $V$-generic filter on $\wt{\m{P}}^\beta\, \uhr\, (\eta + 1)$.

%\colorbox{yellow}{FRAGE: Könnte man EVTL $\rho^\beta_{\eta + 1} (q) = \emptyset$ setzen, falls $|\{ (\sigma, i) \in \supp q_0 \ | \ \sigma > \eta\, \vee \, i \geq \beta\}| < \omega$? NEIN!!!}

\subsubsection*{B) Capturing $\boldsymbol{f^\beta}$.}

We define a forcing notion $(\wt{\m{P}}^\beta\, \uhr\, (\eta + 1))^{( \eta_m, i_m)_{m < \omega}}$, which will be obtained from $\wt{\m{P}}^\beta\, \uhr\, (\eta + 1)$ by using $\wt{X}_p$ instead of $X_p$ (cf. Chapter \ref{6.2} B)\,), and including for $\eta_m \in \Lim$, $\eta_m > \eta$ the verticals $p^{\eta_m}_{i_m}\, \uhr\, \kappa_{\eta + 1}$, and also $a^{\eta_m}_{i_m}\, \cap\, \kappa_{\eta + 1}$, the according linking ordinals up to $\kappa_{\eta + 1}$. \\[-3mm]

%\colorbox{yellow}{TO DO: nach einem besseren Begriff für \tbl vertical line\tbr\, suchen} \\[-3mm]

The restriction $(\wt{p}^\beta\, \uhr\, (\eta + 1))^{( \eta_m, i_m)_{m < \omega}}$ for $p \in \m{P}$ is defined as follows:
\[ (\wt{p}^\beta\, \uhr\, (\eta + 1))^{( \eta_m, i_m)_{m < \omega}} := \big(\,p_\ast\, \uhr\, \kappa_{\eta + 1}^2\; , \; (p^\sigma_i, a^\sigma_i)_{\sigma \leq \eta, i < \beta}\; , \; (p^{\eta_m}_{i_m}\, \uhr\, \kappa_{\eta + 1}, a^{\eta_m}_{i_m}\, \cap\, \kappa_{\eta + 1})_{m < \omega\, , \, \eta_m > \eta}, \]\[(p^\sigma\, \uhr\, ( \beta\, \times\, \dom_y p^\sigma))_{\sigma \leq \eta}\; , \; \wt{X}_p\; , \;p^{\eta + 1} \, \big).\]

Roughly speaking, the difference with the restrictions $(p^\beta\, \uhr\, (\eta + 1))^{( \eta_m, i_m)_{m < \omega}}$  introduced in Chapter 6.2 B) is, that we are now reaching up to $\kappa_{\eta + 1} = \kappa_\eta^\plus$ instead of $\kappa_\eta$. \\[-3mm]

We denote by $(\wt{\m{P}}^\beta\, \uhr\, (\eta + 1))^{( \eta_m, i_m)_{m < \omega}}$ the collection of all $(\wt{p}^\beta\, \uhr\, (\eta + 1))^{( \eta_m, i_m)_{m < \omega}}$ for $p \in \ol{\m{P}}$
%$p \in \m{P}$ with $| \{ (\sigma, i) \in \supp p_0\ | \ \sigma > \eta\, \vee\, i \geq \beta\}| = \omega$, 
together with the maximal element $(\wt{\m{1}}^\beta_{\eta + 1})^{( \eta_m, i_m)_{m < \omega}}$. The order relation \tbl $\leq$\tbr\,is defined like in Definition \ref{defpbetauhreta+1kompliziert}. \\[-3mm]

Finally, we include the verticals $P^{\eta_m}\, \uhr\, [\kappa_{\eta + 1}, \kappa_{\eta_m})$ for $\eta_m > \eta + 1$, which gives the product \[( \wt{\m{P}}^\beta\, \uhr\, (\eta + 1))^{(\eta_m, i_m)_{m < \omega}}\, \times\, \prod_{m < \omega} P^{\eta_m}\, \uhr\, [\kappa_{\eta + 1}, \kappa_{\eta_m}).\]

Let \[(\wt{G}^\beta\, \uhr\, (\eta + 1))^{( \eta_m, i_m)_{m < \omega}}\, \times\, \prod_{m < \omega} G^{\eta_m}_{i_m}\, \uhr\, [\kappa_{\eta + 1}, \kappa_{\eta_m})\] denote the collection of all \[\big((\wt{p}^\beta\, \uhr\, (\eta + 1))^{( \eta_m, i_m)_{m < \omega}}, (p^{\eta_m}_{i_m}\, \uhr\, [\kappa_{\eta + 1}, \kappa_{\eta_m})_{m < \omega})\big)\] such that there exists $q \in G\, \cap\, \ol{\m{P}}$ with $(\wt{q}^\beta\, \uhr \, (\eta + 1))_{(\eta_m ,i_m)_{m < \omega}} \leq (\wt{p}^\beta\, \uhr\, (\eta + 1))^{( \eta_m, i_m)_{m < \omega}}$ and $q^{\eta_m}_{i_m} \, \uhr\, [\kappa_{\eta + 1}, \kappa_{\eta_m}) \supseteq p^{\eta_m}_{i_m}\, \uhr\, [\kappa_{\eta + 1}, \kappa_{\eta_m})$ for all $m < \omega$. \\[-3mm]

As in Proposition \ref{projectionrho2}, one can construct a projection of forcing posets \[({\wt{\rho}}^\beta)^{( \eta_m, i_m)_{m < \omega}}: \ol{\m{P}} \rightarrow (\wt{\m{P}}^\beta\, \uhr\, (\eta + 1))^{( \eta_m, i_m)_{m < \omega}}\, \times\, \prod_{m < \omega} P^{\eta_m}\, \uhr\,  [\kappa_{\eta + 1}, \kappa_{\eta_m}),\]

%\colorbox{yellow}{ACHTUNG - irgendwo wurde der Begriff \tbl complete projection\tbr\,verwendet!}

and it follows that $(\wt{G}^\beta\, \uhr\, (\eta + 1))^{( \eta_m, i_m)_{m < \omega}}\, \times\, \prod_{m < \omega} G^{\eta_m}_{i_m}\, \uhr\, [\kappa_{\eta + 1}, \kappa_{\eta_m})$ is a $V$-generic filter on $(\wt{\m{P}}^\beta\, \uhr\, (\eta + 1))^{( \eta_m, i_m)_{m < \omega}}\, \times\, \prod_{m < \omega} P^{\eta_m}\, \uhr\,  [\kappa_{\eta + 1}, \kappa_{\eta_m})$. \\[-4mm]

Like in Chapter 6.2 B), we want to define a map $(f^\beta)^\prime$ contained in $V[\wt{G}^\beta\, \uhr\, (\eta + 1))^{( \eta_m, i_m)_{m < \omega}}\, \times\, \prod_{m < \omega} G^{\eta_m}_{i_m}\, \uhr\, [\kappa_{\eta + 1}, \kappa_{\eta_m})]$, and then use an isomorphism argument to show that $f^\beta = (f^\beta)^\prime$. \\[-2mm]

Before that, we have to modify our transformations of names $\tau_{\varrho}$ (where $\varrho$ is an $\eta$-good pair), and define transformations $\wt{\tau}_{\varrho}$(where $\varrho$ is an $\eta$-almost good pair) with \[\wt{\tau}_{\varrho}: \Name ( (\ol{P}^{\eta+ 1})^\omega\, \times \, \prod_{m < \omega} P^{\ol{\sigma}_m}\, \times \, P^{\eta + 1}) \rightarrow \Name (\m{P})\] as follows (cf. Definition \ref{tauvarrho}):

\begin{definition} \label{wttauvarrho} For an $\eta$-almost good pair $ \varrho = \big( (a_m)_{m < \omega}, (\ol{\sigma}_m, \ol{i}_m)_{m < \omega} \big)$, define recursively for $\dot{Y} \in \Name ( (\ol{P}^{\eta+ 1})^\omega\, \times \, \prod_{m < \omega} P^{\ol{\sigma}_m}\, \times \, P^{\eta + 1})$: \[\wt{\tau}_\varrho (\dot{Y}) := \big \{ \, (\wt{\tau}_{\varrho} (\dot{Z}), q)\ \, \big| \ \, q \in \m{P}\; , \;  \exists\, \big(\dot{Z}\, , \, \big( (p_\ast (a_m))_{m < \omega}\, , \,  (p^{\ol{\sigma}_m}_{\ol{i}_m})_{m < \omega}\, , \, p^{\eta + 1}\big)\big) \in \dot{Y}\; :\ \]\[\forall m\ \; \big(\, q_\ast (a_m) \supseteq p_\ast (a_m) \; , \; q^{\ol{\sigma}_m}_{\ol{i}_m} \supseteq p^{\ol{\sigma}_m}_{\ol{i}_m}\, \big) \; , \; q^{\eta + 1} \supseteq p^{\eta + 1}\,\big\}.\]

\end{definition}

%\colorbox{yellow}{FRAGE: Bustabe $\varrho$ für $\eta$-good pairs und $\eta$-almost good pairs nehmen? Ok? Oder eher $\wt{\varrho}$?}

Then $\dot{Y}^{\prod_{m < \omega} G_\ast (a_m)\, \times\, \prod_{m < \omega}G^{\ol{\sigma}_m}_{\ol{i}_m}\, \times\, G^{\eta + 1}} = (\wt{\tau}_{\varrho} (\dot{Y}))^G$ holds for all $\dot{Y} \in \Name ( (\ol{P}^{\eta+1})^\omega\, \times \, \prod_{m < \omega} P^{\ol{\sigma}_m}\, \times\, P^{\eta + 1})$. \\[-2mm]

%When the $\eta$-almost good pair $\varrho = \big( (a_m\ | \ m < \omega), ((\ol{\sigma}_m, \ol{i}_m) \ | \ m < \omega) \big)$ is clear from the context, we sometimes just write $\wt{\dot{Y}}$ instead of $\wt{\tau}_{\varrho} (\dot{Y})$. 

\begin{definition} 
Let $(f^\beta)^\prime$ denote the set of all $(X, \alpha)$ for which there exists an $\eta$-\textit{\upshape almost good pair} $\varrho = ( (a_m)_{m < \omega}, (\ol{\sigma}_m, \ol{i}_m)_{m < \omega})$ with $\ol{i}_m < \beta$ for all $m < \omega$, such that \[X = \dot{X}^{\prod_m G_\ast (a_m)\, \times \, \prod_m G^{\ol{\sigma}_m}_{\ol{i}_m}\, \times\, G^{\eta + 1}}, \] and there is a condition $p \in \m{P}$ with the following properties: 

\begin{itemize}  \item $|\{ (\sigma, i) \in \supp p_0\ | \ \sigma > \eta \mbox{ or } i \geq \beta\}| = \aleph_0$, \item $p \Vdash_s \big(\wt{\tau}_{\varrho} (\dot{X}), \alpha\big) \in \dot{f}$,\item $\big(\, (\wt{p}^\beta\, \uhr\, (\eta + 1))^{( \eta_m, i_m)_{m < \omega}}, ( p^{\eta_m}_{i_m}\, \uhr\, [\kappa_{\eta + 1}, \kappa_{\eta_m}))_{m < \omega}\, \big) \in (\wt{G}^\beta\, \uhr\, (\eta + 1))^{( \eta_m, i_m)_{m < \omega}}\, \times\, \prod_{m < \omega} G^{\eta_m}_{i_m} \, \uhr \, [\kappa_{\eta + 1}, \kappa_{\eta_m})$, \item $\forall\, \eta_m \in \Lim\ : \ (\eta_m, i_m) \in \supp p_0$ with $a^{\eta_m}_{i_m} = g^{\eta_m}_{i_m}$. \end{itemize} \end{definition}

Then $(f^\beta)^\prime \in V[(\wt{G}^\beta\, \uhr\, (\eta + 1))^{( \eta_m, i_m)_{m < \omega}}\, \times\, \prod_{m < \omega} G^{\eta_m}_{i_m}\, \uhr\, [\kappa_{\eta + 1}, \kappa_{\eta_m})]$.

%FRAGE: Sollte man erwähnen, dass die Definition von $\wt{\dot}$ hier modifiziert werden müsste?  

\begin{prop} $f^\beta = (f^\beta)^\prime$. \end{prop}
           
\begin{proof} We briefly outline, where the isomorphism argument form Proposition \ref{fbetafbetaprime} has to be modified.
We start with $(X, \alpha) \in (f^\beta)^\prime \setminus f^\beta$, $X = \dot{X}^{\prod_m G_\ast (a_m)\, \times\, \prod_m G^{\ol{\sigma}_m}_{\ol{i}_m}\, \times\, G^{\eta + 1}}$, for an $\eta$-almost good pair $\varrho = ( (a_m)_{m < \omega}, ( (\ol{\sigma}_m, \ol{i}_m)_{m < \omega}))$. Take $p$ as in the definition of $(f^\beta)^\prime$ with $p \Vdash_s (\wt{\tau}_{\varrho} (\dot{X}), \alpha) \in \dot{f}$, and $p^\prime \in G$ with $p^\prime \Vdash_s (\wt{\tau}_{\varrho} (\dot{X}), \alpha) \notin \dot{f}$. \\[-3mm]

The first step is the construction of extensions $\ol{p} \leq p$, $\ol{p}^\prime \leq p^\prime$, such that $\ol{p}$ and $\ol{p}^\prime$ have \tbl the same shape\tbr, agree on $\wt{\m{P}}^\beta\, \uhr\, (\eta + 1)$; and $\ol{p}^{\eta_m}_{i_m} = (\ol{p}^\prime)^{\eta_m}_{i_m}$ holds for all $m < \omega$, and $\ol{a}^{\eta_m}_{i_m} = (\ol{a}^\prime)^{\eta_m}_{i_m}$ holds for all $m < \omega$ with $\eta_m \in Lim$. \\[-3mm]

We proceed as in Proposition \ref{fbetafbetaprime}, with the following modifications: \begin{itemize} \item The construction of $\ol{p}_\ast$, $\ol{p}^\prime_\ast$ that we used in the Proposition \ref{fbetafbetaprime} for intervals $[\kappa_{\nu, j}, \kappa_{\nu, j + 1}) \subseteq \kappa_\eta$, has to be applied to all intervals $[\kappa_{\nu, j}, \kappa_{\nu, j + 1}) \subseteq \kappa_{\eta + 1}$ now, since we need $\ol{p}_\ast$ and $\ol{p}^\prime_\ast$ agree on $\kappa_{\eta + 1}^2$.
 \item Analogously, the construction of $\ol{p}^\sigma_i$, $(\ol{p}^\prime)^\sigma_i$ for $\sigma \in Lim$, $i < \alpha_\sigma$ for intervals $[\kappa_{\nu, j}, \kappa_{\nu, j + 1}) \subseteq \kappa_\eta$, has to be applied to all intervals $[\kappa_{\nu, j}, \kappa_{\nu, j + 1}) \subseteq \kappa_{\eta + 1}$ now, in the case that $\sigma > \eta + 1$. \item Additionally, we have to make sure that $\ol{p}^{\eta + 1} = (\ol{p}^\prime)^{\eta + 1}$. 
%FRAGE: $p^{\eta + 1}$ oder $p (\eta + 1)$? 
\end{itemize}
 
The next step is the construction of an isomorphism $\pi$ such that $\pi \ol{p} = \ol{p}^\prime$, $\pi \ol{f}^{D_\pi} = \ol{f}^{D_\pi}$, and $\pi \,\ol{\wt{\tau}_{\varrho} (\dot{X})}^{D_\pi} = \ol{\wt{\tau}_{\varrho}(\dot{X})}^{D_\pi}$. Again, we take for $\pi$ a \textit{standard isomorphism for }$\pi \ol{p} = \ol{p}^\prime$; but this time, we set $G_{\pi_0} (\nu, j) := F_{\pi_0} (\nu, j)$ for all intervals $[\kappa_{\nu, j}, \kappa_{\nu, j + 1}) \subseteq \kappa_{\eta + 1}$ (instead of only intervals $[\kappa_{\nu, j}, \kappa_{\nu, j +1})\subseteq \kappa_\eta$), and $G_{\pi_0} (\nu, j) = id$ for all $\kappa_{\nu, j} \geq \kappa_{\eta + 1}$ (instead of all $\kappa_{\nu, j} \geq \kappa_\eta$). Then as before, it follows that $\pi \in \bigcap_m Fix (\eta_m, i_m)\, \cap\, \bigcap_m H^{\lambda_m}_{k_m}$. 

For verifying $\pi \,\ol{\wt{\tau}_{\varrho} (\dot{X})}^{D_\pi} = \ol{\wt{\tau}_{\varrho} (\dot{X})}^{D_\pi}$, we now additionally have to make sure that $\pi$ is the identity on $P^{\eta + 1}$. But since we have arranged $\ol{p}^{\eta + 1} = (\ol{p}^\prime)^{\eta + 1}$, this is clear by construction of $\pi_1$.\\[-2mm]

%\colorbox{red}{ACHTUNG - wurde $\wt{X}^{D_\pi} := \ol{\wt{X}}^{D_\pi}$ überhaupt definiert?}

Now, it follows from $\ol{p} \Vdash_s \big(\wt{\tau}_{\varrho} (\dot{X}), \alpha\big) \in \dot{f}$ that $\pi \ol{p} \Vdash_s \big(\pi \,\ol{\wt{\tau}_{\varrho} (\dot{X})}^{D_\pi}, \alpha\big) \in \pi \ol{f}^{D_\pi}$. Hence, $\ol{p}^\prime \Vdash_s \big(\ol{\wt{\tau}_{\varrho} (\dot{X})}^{D_\pi}, \alpha\big) \in \ol{f}^{D_\pi}$, which is a contradiction towards $\ol{p}^\prime \Vdash_s \big(\wt{\tau}_{\varrho} (\dot{X}), \alpha\big) \notin \dot{f}$.
%FRAGE: GENÜGT DAS WIRKLICH? 
\end{proof}

Thus, $f^\beta = (f^\beta)^\prime \in V[(\wt{G}^\beta\, \uhr\, (\eta + 1))^{( \eta_m, i_m)_{m < \omega}}\, \times\, \prod_m G^{\eta_m}_{i_m}\, \uhr\, [\kappa_{\eta + 1}, \kappa_{\eta_m})]$ as desired. 

\subsubsection*{C) $\boldsymbol{(\wt{\m{P}}^\beta\, \uhr\, (\eta + 1))^{( \eta_m, i_m)_{m < \omega}}\, \times\, \prod_m P^{\eta_m}\, \uhr\, [\kappa_{\eta + 1}, \kappa_{\eta_m})}$ preserves cardinals $\boldsymbol{\geq \alpha_\eta}$.}

Now, we will show that cardinals $\geq \alpha_\eta$ are absolute between $V$ and $V[(\wt{G}^\beta\, \uhr \, (\eta + 1))^{( \eta_m, i_m)_{m < \omega}}\, \times\, \prod_{m < \omega} G^{\eta_m}_{i_m}\, \uhr\, [\kappa_{\eta + 1}, \kappa_{\eta_m})]$. \\[-3mm]

As in Chapter \ref{6.2} C), we are using that $GCH$ holds in our ground model $V$, and when we argue that a particular forcing notion preserves cardinals, we mean that it preserves cardinals under $GCH$, if not stated differently. \\[-2mm]

% that we are assuming $GCH$ in our ground model $V$, which will be used implicitly throughout this Chapter \ref{6.2} C): When we argue that a particular forcing notion preserves cardinals, then we always mean it preserves cardinals under the assumption that $GCH$ holds. \\[-2mm]

%\colorbox{red}{WEGEN DIESER FORMULIERUNG NACHFRAGEN!!}

%First, we have a look at the cardinality of $(\m{P}^\beta\, \uhr\, (\eta + 1))^{( \eta_m, i_m)_{m < \omega}}$, and use that $\beta$ was an ordinal \textit{large enough for the intersection $(A_{\dot{f}})$} with $\kappa_\eta^\plus < \beta < \alpha_\eta$.

\begin{lem} If $|\beta|^\plus < \alpha_\eta$, then $(\wt{\m{P}}^\beta\, \uhr\, (\eta + 1))^{( \eta_m, i_m)_{m < \omega}}\, \times\, \prod_m P^{\eta_m}\, \uhr\, [\kappa_{\eta + 1}, \kappa_{\eta_m})$ preserves cardinals $\geq \alpha_\eta$. \end{lem}

\begin{proof} We closely follow the proof of Lemma \ref{cardforc} and Corollary \ref{prescard2mitte}. \\[-3mm]

The forcing notion $(\wt{\m{P}}^\beta\, \uhr\, (\eta + 1))^{(\eta_m ,i_m)_{m < \omega}}$ is the set of all \[ \big(\,p_\ast\, \uhr\, \kappa_{\eta + 1}^2\; , \; (p^\sigma_i, a^\sigma_i)_{\sigma \leq \eta, i < \beta}\; , \; (p^{\eta_m}_{i_m}\, \uhr\, \kappa_{\eta + 1}, a^{\eta_m}_{i_m}\, \cap\, \kappa_{\eta + 1})_{m < \omega, \eta_m > \eta}, \]\[(p^\sigma\, \uhr\, ( \beta\, \times\, \dom_y p^\sigma))_{\sigma \leq \eta}\; , \; \wt{X}_p\; , \; p^{\eta + 1} \, \big),\] where $p \in \m{P}$ with $|\{ (\sigma, i) \in \supp p_0\ | \ \sigma > \eta\, \vee\, i \geq \beta\}| = \aleph_0$. \\[-3mm]

Since $\kappa_{\eta + 1} = \kappa_\eta^\plus$, it follows that the $p_\ast \, \uhr\, \kappa_{\eta+1}^2$, as well as $(p^{\eta_m}_{i_m}\, \uhr\, \kappa_{\eta + 1}, a^{\eta_m}_{i_m}\, \cap\, \kappa_{\eta + 1})$ for $m < \omega$ are bounded below $\kappa_{\eta + 1}$; which gives only $\leq (\kappa_{\eta + 1} \, \cdot\, 2^{\kappa_\eta})^\omega = \kappa_{\eta + 1} = \kappa_\eta^\plus \leq |\beta|$-many possibilities.
% for $p_\ast\, \uhr\, \kappa_{\eta + 1}^2$ and $(p^{\eta_m}_{i_m}\, \uhr\, \kappa_{\eta + 1}, a^{\eta_m}_{i_m}\, \cap\, \kappa_\eta)_{m < \omega}$.
 
Since $\wt{X}_p \subseteq \kappa_\eta$, there are only $\leq \kappa_\eta^\plus \leq |\beta|$-many possibilities for $\wt{X}_p$, as well. Regarding $(p^\sigma_i, a^\sigma_i)_{\sigma \leq \eta\, , \, i < \beta}$ and $(p^\sigma\, \uhr\, (\dom_x p^\sigma\, \times\, \beta))_{\sigma \leq \eta}$, it follows as in Lemma \ref{cardforc} that there are only $\leq |\beta|^\plus \, \cdot\, \kappa_\eta^\plus = |\beta|^\plus$-many possibilities. \\[-3mm]  

We denote by $\big((\wt{\m{P}}^\beta\, \uhr\, (\eta + 1))^{( \eta_m, i_m)_{m < \omega}}\big)^\prime$ the forcing notion that is obtained from $(\wt{\m{P}}^\beta\, \uhr\, (\eta + 1))^{(\eta_m, i_m)_{m < \omega}}$ by excluding $P^{\eta + 1}$. Then $(\wt{\m{P}}^\beta\, \uhr\, (\eta + 1))^{( \eta_m, i_m)_{m < \omega}}$ is isomorphic to the product $\big((\wt{\m{P}}^\beta\, \uhr\, (\eta + 1))^{(\eta_m, i_m)_{m < \omega}}\big)^\prime \, \times\, P^{\eta + 1}$. \\[-3mm]

By what we have just argued, it follows that the forcing notion $\big((\wt{\m{P}}^\beta\, \uhr\, (\eta + 1))^{( \eta_m, i_m)_{m < \omega}}\big)^\prime$ has cardinality $\leq |\beta|^\plus$; and the remaining product $P^{\eta + 1}\, \times\, \prod_{m < \omega} P^{\eta_m} \, \uhr\, [\kappa_{\eta + 1}, \kappa_{\eta_m})$ preserves all cardinals by similar arguments as in Proposition \ref{prescardalphaeta}. Hence, it follows that $\big((\wt{\m{P}}^\beta\, \uhr\, (\eta + 1))^{( \eta_m, i_m)_{m < \omega}}\big)^\prime\, \times\, P^{\eta + 1}\, \times\, \prod_{m < \omega} P^{\eta_m}\, \uhr\, [\kappa_{\eta + 1}, \kappa_{\eta_m})$ preserves all cardinals $\geq |\beta|^{\plus\plus}$.

\end{proof}
%ACHTUNG - die Bezeichnung $\wt{\m{P}}^\beta\, \uhr\, (\eta + 1))$ wurde in Paragraph C) schon benutzt! Das muesste man AENDERN!!

%TO DO: bessere Bezeichung fuer $\widehat{\wt{\m{P}}}^\beta\, \uhr \, (\eta + 1)$ ueberlegen!

\begin{prop} The forcing $(\wt{\m{P}}^\beta\, \uhr\, (\eta + 1))^{( \eta_m, i_m)_{m < \omega}}\, \times\, \prod_m P^{\eta_m}\, \uhr\, [\kappa_{\eta + 1}, \kappa_{\eta_m})$ preserves cardinals $\geq \alpha_\eta$.\end{prop} 

%It remains to make sure that also in the case $|\beta|^\plus = \alpha_\eta$, the product $(\wt{\m{P}}^\beta\, \uhr\, (\eta + 1))_{(\eta_j, i_j)_{j < \omega}}\, \times\, \prod_{j < \omega} P^{\eta_j}_{i_j}\, \uhr \, [\kappa_{\eta + 1}, \kappa_{\eta_j})$ preserves cardinals $\geq \alpha_\eta$. \\ 
\begin{proof} We only have to treat the case that $\alpha_\eta = |\beta|^\plus$. Then $cf \,\beta > \omega$, and $GCH$ gives $|\beta|^{\aleph_0} = |\beta|$. The proof is similar as for Proposition \ref{prescardalphaeta}: We distinguish several cases, and construct $\big((\wt{\m{P}}^\beta\, \uhr \, (\eta + 1))^{( \eta_m, i_m)_{m < \omega}}\big)^{\prime\prime}$ from $(\wt{\m{P}}^\beta\, \uhr\, (\eta + 1))^{( \eta_m, i_m)_{m < \omega}}$ by splitting up $P^{\eta + 1}$, and also one or two factors $P^{\sigma}\, \uhr\, (\beta\, \times\, [\ol{\kappa_{\sigma}}, \kappa_{\sigma}))$ for $\sigma \in Succ$ with $\sigma = \eta$, or $\sigma = \ol{\eta}$ in the case that $\eta$ is a successor ordinal with $\eta = \ol{\eta} + 1$.
% with $\kappa_{\ol{\sigma}} = : \ol{\kappa}_{\ol{\sigma}}^\plus$. ACHTUNG - man muesste KLAEREN ob die Rechtecke bei $\ol{\kappa}_{\ol{\sigma}}$ wie hier anfangen sollen, oder lieber bei $\kappa_{\ol{\sigma}^\prime}$ mit $\ol{\sigma} = \ol{\sigma}^\prime + 1$? 
Then as in the proof of Proposition \ref{prescardalphaeta}, it follows that $\big((\wt{\m{P}}^\beta\, \uhr \, (\eta + 1))^{( \eta_m, i_m)_{m < \omega}}\big)^{\prime \prime}$ has cardinality $\leq |\beta| < \alpha_\eta$, and the product of the remaining $P^{\sigma}\, \uhr\, (\beta\, \times\, [\ol{\kappa_{\sigma}}, \kappa_{\sigma}))$, $P^{\eta + 1}$ and $\prod_m P^{\eta_m}\, \uhr\, [\kappa_{\eta + 1}, \kappa_{\eta_m})$ preserves all cardinals. \end{proof}
%(The only difference is that we now have one more factor $P^{\eta + 1}$ to care about.)
%\\ Hence, it follows that $(\wt{\m{P}}^\beta\, \uhr\, (\eta + 1))^{( \eta_m, i_m)_{m < \omega}}$ preserves all cardinals $\geq \alpha_\eta$.

\subsubsection*{D) A set $\boldsymbol{\wt{\powerset} (\kappa_\eta) \supseteq \dom f^\beta}$ with an injection $\boldsymbol{\iota: \wt{\powerset} (\kappa_\eta) \hookrightarrow |\beta|^{\aleph_0}}$.}

%An injection $\iota: \wt{\powerset} (\kappa_\eta) \hookrightarrow |\beta|^\omega$ with $\wt{\powerset} (\kappa_\eta) \supseteq \powerset^N (\kappa_\eta)$ NEIN! $\wt{\powerset} (\kappa_\eta) \supset \dom f^\beta$!}
          
 %oder: \tbl A set $\wt{\powerset} (\kappa_\eta) \supseteq \dom f^\beta$ with an injection $\iota: \wt{\powerset} (\kappa_\eta) \hookrightarrow |\beta|^\omega$? \tbr. \\
          
For an $\eta$-almost good pair $\varrho = ( (a_m)_{m < \omega}, (\ol{\sigma}_m, \ol{i}_m)_{m < \omega})$, it follows that $\prod_m G_\ast (a_m)\, \times\, \prod_m G^{\eta_m}_{i_m}\, \times\, G^{\eta + 1}$ is a $V$-generic filter on $\prod_m (\ol{P}^\eta)^\omega\, \times\, \prod_m P^{\ol{\sigma}_m}\, \times\, P^{\eta + 1}$, and \[\big(2^\alpha\big)^{\prod_m G_\ast (a_m)\, \times\, \prod_m G^{\ol{\sigma}_m}_{\ol{i}_m}\, \times\, G^{\eta + 1}} = (\alpha^\plus)^V\] holds for all $\alpha \leq \kappa_\eta$ by the same proof as for Lemma \ref{prescard1}, since $P^{\eta + 1}$ ist $\leq \kappa_\eta$-closed. \\ Thus, there is an injection $\chi: \powerset(\kappa_\eta) \hookrightarrow (\kappa_\eta^\plus)^V$ in $V[\prod_m G_\ast (a_m)\, \times\, \prod_m G^{\ol{\sigma}_m}_{\ol{i}_m}\, \times\, G^{\eta + 1}]$.  \\[-2mm]

Let $\wt{M}_\beta$ denote the set of all $\eta$-almost good pairs $\varrho = ( (a_m)_{m < \omega}, (\ol{\sigma}_m, \ol{i}_m)_{m < \omega})$ in $V$ with the property that $\ol{i}_m < \beta$ for all $m < \omega$. Then $\wt{M}_\beta$ has cardinality $\leq \kappa_{\eta + 1}\, \cdot\, |\eta|^{\aleph_0}\, \cdot\, |\beta|^{\aleph_0} = |\beta|^{\aleph_0}$. \\[-3mm]
%FRAGE: eher Bezeichnung $\wt{M}_\beta$ anstatt $M_\beta$ nehmen? 

Moreover, $dom \,f^\beta$ is a subset of $\wt{\powerset} (\kappa_\eta) :=$ \[ \bigcup \Big\{\; \powerset(\kappa_\eta)\, \cap\, V[\prod_m G_\ast (a_m)\, \times\, \prod_m G^{\ol{\sigma}_m}_{\ol{i}_m}\, \times\, G^{\eta + 1}]\ \ \Big| \ \ \big( (a_m)_{m < \omega}, (\ol{\sigma}_m, \ol{i}_m)_{m < \omega}\big) \in \wt{M}_\beta\; \Big\}.\]   

Now, we can proceed as in Chapter 6.2 D) and construct in $V[(\wt{G}^\beta\, \uhr\, (\eta + 1))^{( \eta_m, i_m)_{m < \omega}}\, \times\, \prod_m G^{\eta_m}_{i_m}\, \uhr\, [\kappa_{\eta + 1}, \kappa_{\eta_m})]$ an injection $\iota: \wt{\powerset} (\kappa_\eta) \hookrightarrow |\beta|^V$ in the case that $\alpha_\eta = (|\beta|^\plus)^V$, and an injection $\iota: \wt{\powerset} (\kappa_\eta) \hookrightarrow (|\beta|^\plus)^V$ in the case that $\alpha_\eta > (|\beta|^\plus)^V$. Together with Chapter 6.3 B) and 6.3 C), this gives the desired contradiction. \\[-2mm]

%\colorbox{yellow}{ACHTUNG - sollte man im Titel lieber 6.2 B) etc. schreiben?}

Thus, we have shown that the map $f^\beta: \dom f^\beta \rightarrow \alpha_\eta$ must not be surjective.

% and there must be $\alpha < \alpha_\eta$ with $\alpha \notin \rg f^\beta$.

\subsubsection*{E) We use an isomorphism argument and obtain a contradiction.}

The arguments for this part are the very same as in the case that $\kappa_{\eta + 1} > \kappa_\eta^\plus$, except that we are now working with \textit{$\eta$-almost good pairs} $\varrho = ((a_m)_{m < \omega},  (\ol{\sigma}_m, \ol{i}_m)_{m < \omega})$ instead of $\eta$-good pairs. \\[-3mm]

%; so for names $\dot{X} \in \Name ( (\ol{P}^\eta)^\omega\, \times\, \prod_m P^{\ol{\sigma}_m}_{\ol{i}_m}\, \times\, P^{\eta + 1})$ we have to modify the extension of names $\wt{\dot{X}} := \tau_{\rho} (\dot{X})$ with $\wt{\dot{X}}^G = \dot{X}^{\prod G_\ast (a_m)\, \times\, \prod_m G^{\ol{\sigma}_m}_{\ol{i}_m}\, \times\, G^{\eta + 1}}$ accordingly. \\[-2mm]

Thus, also in the case that $\kappa_{\eta + 1} = \kappa_\eta^\plus$, it follows that $\theta^N (\kappa_\eta) = \alpha_\eta$.

%Thus, together with Chapter 6.2  and Chapter 6.3, it follows that indeed, $\theta^N (\kappa_\eta) = \alpha_\eta$ holds for all $0 < \eta < \gamma$. \\

\subsection{The remaining cardinals in the \tbl gaps\tbr\;$\boldsymbol{\lambda \in (\kappa_\eta, \kappa_{\eta + 1})}$.} \label{6.4}

So far, we have shown that $\theta^N (\kappa_\eta) = \alpha_\eta$ holds for all $0 < \eta < \gamma$. Recall that in the very beginning (see Chapter \ref{the theorem}), we started by \tbl thinning out\tbr\,our sequence $(\kappa_\eta\ | \ 0 < \eta < \gamma)$ and assuming w.l.o.g. that $(\alpha_\eta\ | \ 0 < \eta < \gamma)$ is strictly increasing. Thus, it remains make sure that for all cardinals $\lambda \in (\kappa_\eta, \kappa_{\eta + 1})$ in the \tbl gaps\tbr\,, $\theta^N (\lambda)$ gets the smallest possible value, i.e. $\theta^N (\lambda) = \max\{\alpha_\eta, \lambda^{\plus \plus}\}$. This will be our aim for this Chapter 6.4. \\ After that, in Chapter 6.5, we make sure that also for all cardinals $\lambda \geq \kappa_\gamma$, the value $\theta^N (\lambda)$ will be the smallest possible. \\
%dealt with in the next section.

We consider a cardinal $\lambda$ in a \tbl gap\tbr\, $\lambda \in (\kappa_\eta, \kappa_{\eta + 1})$ (then $\kappa_{\eta + 1} > \kappa_\eta^\plus$), and set $\alpha (\lambda) := \max\{\lambda^{\plus \plus}, \alpha_\eta\}$. Then $\theta^N (\lambda) \geq \alpha (\lambda)$ is clear, and it remains to make sure that there is no surjective function $f: \powerset^N (\lambda) \rightarrow \alpha (\lambda)$ in $N$. \\[-2mm]

First, we want to describe the intermediate generic extensions where the $\lambda$-subsets $X \in \powerset^N (\lambda)$ are located.\\[-3mm]

Let $\lambda \in [\kappa_{\eta, \ol{\j}}, \kappa_{\eta, \ol{\j} + 1})$, where $\ol{\j} < \,cf\, \kappa_{\eta + 1}$ in the case that $\eta + 1 \in Lim$, and $\ol{\j} = 0$ with $\lambda \in (\kappa_{\eta, 0}, \kappa_{\eta, 1}) = (\kappa_\eta, \kappa_{\eta + 1})$ in the case that $\eta + 1 \in Succ$. \\[-3mm]

We will modify our definition of an $\eta$-good pair and obtain the notion of an \textit{$\eta$-good pair for $\lambda$}, which will be used to describe the intermediate generic extensions where the sets $X \in \powerset^N (\lambda)$ are located:

\begin{definition} We say that $\big( (a_m)_{m < \omega}, (\ol{\sigma}_m, \ol{i}_m)_{m < \omega} \big)$ is an {\upshape $\eta$-good pair for $\lambda$}, if the following hold: \begin{itemize} \item $ (a_m\ | \ m < \omega)$ is a sequence of pairwise disjoint subsets of $\kappa_{\eta, \ol{\j}}$, such that for all $\kappa_{\ol{\nu}, \ol{l}} < \kappa_{\eta, \ol{\j}}$, it follows that $|a_m\, \cap \, [\kappa_{\ol{\nu}, \ol{l}}, \kappa_{\ol{\nu}, \ol{l} + 1})| = 1$, \item for all $m < \omega$, we have $\ol{\sigma}_m \in Succ$ with $\ol{\sigma}_m \leq \eta$, and $\ol{i}_m < \alpha_{\ol{\sigma}_m}$, \item if $m \neq m^\prime$, then  $(\ol{\sigma}_m, \ol{i}_m) \neq (\ol{\sigma}_{m^\prime}, \ol{i}_{m^\prime})$.\end{itemize} 

\end{definition}

Similarly as in Proposition \ref{Xsubseteqkappaeta}, we obtain:

\begin{prop} For every $X \in N$, $X \subseteq \lambda$, there is an {\upshape $\eta$-good pair for $\lambda$}, denoted by $\varrho = \big( (a_m)_{m < \omega}, (\ol{\sigma}_m, \ol{i}_m)_{m < \omega} \big)$, such that $X \in V[\prod_{m < \omega} G_\ast (a_m)\, \times\, \prod_{m < \omega} G^{\ol{\sigma}_m}_{\ol{i}_m}]$. 

\end{prop}

\begin{proof} As in Proposition \ref{Xsubseteqkappaeta}, it follows by the \textit{Approximation Lemma} \ref{approx} that any $X \in N$, $X \subseteq \lambda$ is contained in a generic extension \[X \in V[\prod_{m < \omega} G_\ast (g^{\sigma_m}_{i_m})\, \times\, \prod_{m < \omega} G^{\ol{\sigma}_m}_{\ol{i}_m}],\] where $( (\sigma_m, i_m)\ | \ m < \omega)$ and $( (\ol{\sigma}_m, \ol{i}_m)\ | \ m < \omega)$ are sequences of pairwise distinct pairs with $\sigma_m \in Lim$, $i_m < \alpha_{\sigma_m}$, and $\ol{\sigma}_m \in Succ$, $\ol{i}_m < \alpha_{\ol{\sigma}_m}$ for all $m < \omega$. \\[-2mm]

The forcing $\prod_{m < \omega} P^{\sigma_m}\, \times\, \prod_{m < \omega} P^{\ol{\sigma}_m}$ can be factored as \[\big(\, \prod_{m < \omega} P^{\sigma_m}\, \uhr\, \kappa_{\eta, \ol{\j}}\, \times\, \prod_{\ol{\sigma}_m \leq \eta} P^{\ol{\sigma}_m}\, \big)\, \times\, \big(\, \prod_{m < \omega} P^{\sigma_m}\, \uhr\, [\kappa_{\eta, \ol{\j}}, \kappa_{\eta_m})\, \times\, \prod_{\ol{\sigma}_m > \eta} P^{\ol{\sigma}_m}\, \big).\]
In the case that $\lambda \in (\kappa_{\eta, \ol{\j}}, \kappa_{\eta, \ol{\j} + 1})$, it follows that the \tbl lower part\tbr\ has cardinality $\leq \kappa_{\eta, \ol{\j}}^\plus \leq \lambda$, and the \tbl upper part\tbr\,is $\leq \lambda$-closed. \\ If $\lambda = \kappa_{\eta, \ol{\j}}$, then firstly, the \tbl lower part\tbr\,has cardinality $\leq \kappa_{\eta, \ol{\j}}^\plus = \lambda^\plus$, and secondly, it follows that $\ol{\j} > 0$ and $\kappa_{\eta + 1} \in Lim$, so $\kappa_{\eta, \ol{\j} + 1} \geq \kappa_{\eta, \ol{\j}}^{\plus \plus}$ by construction. Hence, the \tbl upper part\tbr\,is $\leq \lambda^\plus$-closed.

In either case, we obtain \[X \in V[\prod_{m < \omega} G_\ast (g^{\sigma_m}_{i_m}\, \cap\, \kappa_{\eta, \ol{\j}})\, \times \, \prod_{\ol{\sigma}_m \leq \eta} G^{\ol{\sigma}_m}_{\ol{i}_m}]. \]

With $a_m := g^{\sigma_m}_{i_m}\, \cap\, \kappa_{\eta, \ol{\j}}$ for all $m < \omega$, it follows by the \textit{independence property} that $\big((a_m)_{m < \omega}, (\ol{\sigma}_m, \ol{i}_m)_{m < \omega}\big)$ is an \textit{$\eta$-good pair for $\lambda$} with \[X \in V[\prod_{m < \omega} G_\ast (a_m)\, \times\, \prod_{\ol{\sigma}_m \leq \eta} G^{\ol{\sigma}_m}_{\ol{i}_m}].\] 

%\colorbox{yellow}{ACHTUNG - die Konstruktion bei Lemma 8 ist NICHT RICHTIG! die $\zeta \in \dom \ol{p}$ mit $\zeta \notin \dom q_m$ wurden nicht berücksichtigt!}
%\colorbox{red}{ACHTUNG - Lemma 8 müsste man nochmal DURCHSEHEN!}

\end{proof}

(Note that $\prod_m G_\ast (a_m)\, \times\, \prod_{\ol{\sigma}_m \leq \eta} G^{\ol{\sigma}_m}_{\ol{i}_m}$ is a $V$-generic filter on the forcing $(\ol{P}^{\,\eta + 1}\, \uhr\, \kappa_{\eta, \ol{\j}})^\omega\, \times \, \prod_{\ol{\sigma}_m \leq \eta} P^{\ol{\sigma}_m}$). \\

As before, we assume towards a contradiction that there was a surjective function $f: \powerset^N (\lambda) \rightarrow \alpha (\lambda)$ in $N$, where $\pi \ol{f}^{D_\pi} = \ol{f}^{D_\pi}$ holds for all $\pi \in A$ with $[\pi]$ contained in the intersection \[\bigcap_{m < \omega} Fix (\eta_m, i_m)\, \cap\, \bigcap_{m < \omega}H^{\lambda_m}_{k_m} \hspace*{4cm} (A_{\dot{f}}). \]

%As before, we assume there was a surjective function $f: \powerset(\lambda) \rightarrow \alpha (\lambda)$ in $N$ with $(A_{\dot{f}})$ and lead this assumption into a contradiction. \\[-2mm] 

We take $\wt{\beta}$ \textit{large enough for the intersection} $(A_{\dot{f}})$ as in Chapter \ref{6.2}, Definition \ref{largeenough}, and set $\beta := \wt{\beta} \plus \kappa_\eta^\plus$ (addition of ordinals). \\[-2mm]

Let \[f^\beta := \Big\{ \, (X, \alpha) \in f\ \: \big| \ \:\exists\, \big( (a_m)_{m < \omega}, (\ol{\sigma}_m, \ol{i}_m)_{m < \omega}\big)\ \:\eta\mbox{\textit{-good pair for }} \lambda \  : \  (\forall m\ \, \ol{i}_m < \beta)\ \wedge\] \[\exists\, \dot{X} \in \Name\big((\ol{P}^{\,\eta + 1}\, \uhr\, \kappa_{\eta, \ol{\j}})^\omega\, \times \, \prod_{\ol{\sigma}_m \leq \eta} P^{\ol{\sigma}_m}\big)\ \ X = \dot{X}^{\prod_m G_\ast (a_m)\, \times\, \prod_m G^{\ol{\sigma}_m}_{\ol{i}_m}}\, \Big\}.\] 

First, we assume towards a contradiction that $\boldsymbol{f^\beta: \dom f^\beta \rightarrow \alpha (\lambda)}$ \textbf{is surjective}.

\subsubsection*{A) Constructing $\boldsymbol{\wt{\wt{\m{P}}}^\beta\, \uhr\, (\eta + 1)}$} 

We proceed as in Chapter 6.2 A) and 6.3 A), except that we have to use $\m{P}_\ast\, \uhr\, \kappa_{\eta, \ol{\j}}^2$ instead of $\m{P}_\ast\, \uhr\, \kappa_\eta^2$, and do \textit{not} include $P^{\eta + 1}$:

For $p \in \m{P}$, we set \[ \wt{\wt{p}}^\beta\, \uhr\, (\eta  + 1) := \big(\, p_\ast\, \uhr\, \kappa_{\eta, \ol{\j}}^2\, , \, (p^\sigma_i, a^\sigma_i)_{\sigma \leq \eta, i < \beta}\, , \, (p^\sigma\, \uhr\, (\beta\, \times\, \dom_y p^\sigma))_{\sigma \leq \eta}\, , \, X_p\, \big), \] and denote by $\wt{\wt{\m{P}}}^\beta \, \uhr\, (\eta + 1)$ the collection of all $\wt{\wt{p}}^\beta\, \uhr\, (\eta + 1)$, where $p \in \ol{\m{P}}$ (i.e. $p \in \m{P}$ with $|\{ (\sigma, i) \in \supp p_0\ | \ \sigma > \eta\, \vee\, i \geq \beta\}| = \aleph_0$\,); together with the maximal element $\wt{\wt{\m{1}}}^\beta_{\eta + 1}$, and the order relation ${\wt{\wt{\leq}}\,}^\beta_{\eta+1}$ defined similarly as in Definition \ref{defpbetauhretapluseins}.\\[-2mm] 

We denote by $\wt{\wt{G}}^\beta\, \uhr\, (\eta + 1)$ the set of all $p \in \wt{\wt{\m{P}}}^\beta\, \uhr\, (\eta + 1)$ such that there exists $q \in G\, \cap\, \ol{\m{P}}$ with $\wt{\wt{q}}^{\,\beta}\, \uhr\, (\eta + 1) \: {\wt{\wt{\leq}}\,}^\beta_{\eta + 1} \: p$. Then as in Chapter 6.2 A), Proposition \ref{projectionrho}, it follows that $\wt{\wt{G}}^\beta\, \uhr \, (\eta + 1)$ is a $V$-generic filter on $\wt{\wt{\m{P}}}^\beta\, \uhr\, (\eta + 1)$. 

%ACHTUNG - $\wt{\wt{\m{P}}}^\beta\, \uhr\, (\eta + 1)$ ist KEINE gute Bezeichnung!

%FRAGE: $\ol{\j}$ versuchen?

%\colorbox{yellow}{FRAGE: die Definitionen EVTL nur zitieren?}

\subsubsection*{B) Capturing $\boldsymbol{f^\beta}$}

For $p \in \m{P}$, the restriction $(\wt{\wt{p}}^\beta\, \uhr\, (\eta + 1))^{( \eta_m, i_m)_{m < \omega}}$ is defined as follows: \[(\wt{\wt{p}}^\beta\, \uhr\, (\eta + 1))^{( \eta_m, i_m)_{m < \omega}} := \big(\,p_\ast\, \uhr\, \kappa_{\eta,\ol{\j}}^2\; , \; (p^\sigma_i, a^\sigma_i)_{\sigma \leq \eta, i < \beta}\; , \; (p^{\eta_m}_{i_m}\, \uhr\, \kappa_{\eta, \ol{\j}}, a^{\eta_m}_{i_m}\, \cap\, \kappa_{\eta , \ol{\j}})_{m < \omega\, , \, \eta_m > \eta}\, , \]\[(p^\sigma\, \uhr\, ( \beta\, \times\, \dom_y p^\sigma))_{\sigma \leq \eta}\, , \, \wt{X}_p \, \big).  \]

We define $(\wt{\wt{\m{P}}}^\beta\, \uhr\, (\eta + 1))^{( \eta_m, i_m)_{m < \omega}}$ and $(\wt{\wt{G}}^\beta\, \uhr\, (\eta + 1))^{( \eta_m, i_m)_{m < \omega}}$ as in Chapter 6.2 B) and 6.3 B). Then \[(\wt{\wt{G}}^\beta\, \uhr \, (\eta + 1))^{( \eta_m, i_m)_{m < \omega}}\, \times\, \prod_{m < \omega} G^{\eta_m}_{i_m}\, \uhr\, [\kappa_{\eta, \ol{\j}}, \kappa_{\eta_m})\] is a $V$-generic filter on \[(\wt{\wt{\m{P}}}^\beta\, \uhr\, (\eta + 1))^{( \eta_m, i_m)_{m < \omega}}\, \times\, \prod_{m < \omega} P^{\eta_m}\, \uhr\, [\kappa_{\eta, \ol{\j}}, \kappa_{\eta_m}).\]

The construction of $(f^\beta)^\prime$ as well as the proof of $f^\beta  = (f^\beta)^\prime$ are as in Chapter 6.2 B) and 6.3 B); except that this time, the isomorphism $\pi$ from the proof of Proposition \ref{fbetafbetaprime} has to be the identity on $\m{P}_\ast \, \uhr \, \kappa_{\eta, \ol{\j}}^2$ (not only on $\m{P}_\ast\, \uhr\, \kappa_\eta^2$). This can be achieved by the following modifications: Firstly, we demand that $p_\ast$ and $p^\prime_\ast$ cohere on $\m{P}_\ast\, \uhr\, \kappa_{\eta, \ol{\j}}^2$ (not only $\m{P}_\ast \, \uhr \, \kappa_\eta^2$); secondly, we arrange $\ol{p}_\ast\, \uhr\, \kappa_{\eta, \ol{\j}}^2 = \ol{p}^\prime_\ast \, \uhr\, \kappa_{\eta, \ol{\j}}^2$ (instead of just $\ol{p}_\ast\, \uhr\, \kappa_\eta^2 = \ol{p}_\ast^\prime\, \uhr \, \kappa_\eta^2$); and thirdly, when constructing the isomorphism $\pi$, we set $G_{\pi_0} (\nu, j) := F_{\pi_0} (\nu, j)$ for all $\kappa_{\nu, j} < \kappa_{\eta, \ol{\j}}$ now, and $G_{\pi_0} (\nu, j) = \,id$ whenever $\kappa_{\nu, j} \geq \kappa_{\eta, \ol{\j}}$. \\[-2mm]

It follows that $f^\beta = (f^\beta)^\prime \in V[( \wt{\wt{G}}^\beta\, \uhr\, (\eta + 1))^{( \eta_m, i_m)_{m < \omega}} \, \times\, \prod_m G^{\eta_m}_{i_m}\, \uhr\, [\kappa_{\eta, \ol{\j}}, \kappa_{\eta_m})]$.

\subsubsection*{C) $\boldsymbol{(\wt{\wt{\m{P}}}^\beta\, \uhr \, (\eta + 1))^{( \eta_m, i_m)_{m < \omega}}\, \times\, \prod_{m < \omega} P^{\eta_m}\, \uhr\, [\kappa_{\eta, \ol{\j}}, \kappa_{\eta_m})}$ preserves cardinals \\ $\boldsymbol{\geq \alpha (\lambda) = \max \{ \lambda^{\plus \plus}, \alpha_\eta\}}$.}

The arguments here are similar as in Chapter \ref{6.2} C) and \ref{6.3} C), since there are only $\leq (2^{\kappa_{\eta, \ol{\j}}})^{\aleph_0} = \kappa_{\eta, \ol{\j}}^\plus \leq \lambda^\plus < \alpha (\lambda)$-many possibilities for $p_\ast\, \uhr\, \kappa_{\eta, \ol{\j}}^2$ and $(p^{\eta_m}_{i_m}\, \uhr\, \kappa_{\eta, \ol{\j}}, a^{\eta_m}_{i_m}\, \cap\, \kappa_{\eta, \ol{\j}})_{m < \omega}$. 

\subsubsection*{D) A set $\boldsymbol{\wt{\powerset}(\lambda) \supseteq \dom f^\beta}$ with an injection $\boldsymbol{\iota: \wt{\powerset} (\lambda) \hookrightarrow \lambda^\plus\, \cdot\, |\beta|^{\aleph_0}}$.}

We proceed as in Chapter \ref{6.2} D) and \ref{6.3} D). Whenever $( (a_m)_{m < \omega}, (\ol{\sigma}_m, \ol{i}_m)_{m < \omega})$ is an $\eta$-good pair for $\lambda$, it follows that $\prod_m G_\ast (a_m)\, \times\, \prod_m G^{\ol{\sigma}_m}_{\ol{i}_m}$ is a $V$-generic filter on $(\ol{P}^{\, \eta + 1}\, \uhr\, \kappa_{\eta, \ol{\j}})^\omega\, \times\, \prod_m P^{\ol{\sigma}_m}$; and \[\big(2^\alpha \big)^{V[\prod_m G_\ast (a_m)\, \times\, \prod_m G^{\ol{\sigma}_m}_{\ol{i}_m}}] = (\alpha^\plus)^V\] holds for all cardinals $\alpha$ by the same proof as in Lemma \ref{prescard1}. Hence, it follows that in $V[\prod_m G_\ast (a_m)\, \times\, \prod_m G^{\ol{\sigma}_m}_{\ol{i}_m}]$, there is an injection $\chi: \powerset(\lambda) \hookrightarrow (\lambda^\plus)^V$. \\[-2mm]

%\colorbox{yellow}{FRAGE: EVTL $\wt{\powerset} (\lambda)$ ersetzen durch $\ol{\powerset} (\lambda)$?} \\[-2mm]

Let $\wt{\wt{M}}_\beta$ be the set of all $\varrho = ( (a_m)_{m < \omega}, (\ol{\sigma}_m, \ol{i}_m)_{m <\omega})$ in $V$ such that $\varrho$ is an $\eta$-good pair for $\lambda$ with the property that $\ol{i}_m < \beta$ for all $m < \omega$. Then $\wt{\wt{M}}_\beta$ has cardinality $\leq (\kappa_{\eta, \ol{\j}}^\plus)^{\aleph_0}\, \cdot \, |\beta|^{\aleph_0} \leq \lambda^\plus\, \cdot\, |\beta|^{\aleph_0}$. \\[-2mm] 

By construction, it follows that $dom\,f^\beta$ is a subset of \[\wt{\powerset} (\lambda) := \bigcup \big\{\, \powerset(\lambda)\, \cap\, V[\prod_m G_\ast (a_m)\, \times\, \prod_m G^{\ol{\sigma}_m}_{\ol{i}_m}]\ \big| \ ( (a_m)_{m < \omega}, (\ol{\sigma}_m, \ol{i}_m)_{m < \omega}) \in \wt{\wt{M}}_\beta\, \big\}.\]

As in Chapter 6.2 D), we can now work in $V[(\wt{\wt{G}}^\beta\, \uhr\, (\eta + 1))^{( \eta_m, i_m)_{m < \omega}}\, \times\, \prod_m G^{\eta_m}_{i_m}\, \uhr\, [\kappa_{\eta, \ol{\j}}, \kappa_{\eta_m})]$ and construct there an injection $\iota: \wt{\powerset} (\lambda) \hookrightarrow (\lambda^\plus)^V\, \cdot\, |\beta|^V$ in the case that $\alpha_\eta = (|\beta|^\plus)^V$, and an injection $\iota: \wt{\powerset} (\lambda) \hookrightarrow (\lambda^\plus)^V\, \cdot (|\beta|^\plus)^V$ in the case that $\alpha_\eta > (|\beta|^\plus)^V$. 

%Since $\lambda^\plus\, \cdot |\beta|^V < \alpha (\lambda)$, and $(\lambda^\plus)^V \, \cdot\, (|\beta|^\plus)^V < \alpha (\lambda)$ in the case that $\alpha_\eta > |(\beta|^\plus)^V$, this gives 
Together with Chapter \ref{6.4} B) and \ref{6.4} C), this gives the desired contradiction. \\[-3mm]

Hence, it follows that there must be $\alpha < \alpha (\lambda)$ with $\alpha \notin \rg f^\beta$.

\subsubsection*{E) We use an isomorphism argument and obtain a contradiction.}

%\colorbox{yellow}{EVTL anderer Titel?}

The arguments for this part are the same as in Chapter \ref{6.2} E); except that we are working with \textit{$\eta$-good pairs for} $\lambda$ now (instead of $\eta$-good pairs). \\[-2mm]
%$( (a_m)_{m < \omega}, ( (\ol{\sigma}_m, \ol{i}_m)_{m < \omega}))$. \\[-2mm]

%, so for names $\dot{X} \in \Name ( (\ol{P}^{\eta + 1}\, \uhr\, \kappa_{\eta, \ol{j}})^\omega\, \times\, \prod_j P^{\ol{\eta}_j}_{\ol{i}_j})$, we have to modify the extension $\wt{\dot{X}} = \tau_\rho (\dot{X})$ with $\wt{\dot{X}}^G = \dot{X}^{\prod_j G_\ast (a_j)\, \times\, \prod_j }$ accordingly. 

Thus, we have shown that for all cardinals $\lambda \in (\kappa_\eta, \kappa_{\eta + 1})$ in a \tbl gap\tbr\,, the value $\theta^N (\lambda)$ is the smallest possible: $\theta^N (\lambda) = \alpha(\lambda) = \max\{\lambda^{\plus \plus}, \alpha_\eta\}$. \\

It remains to consider the cardinals $\lambda \geq \kappa_\gamma := \sup \{\kappa_\eta\ | \ 0 < \eta < \gamma\}$. We prove that for all $\lambda \geq \kappa_\gamma$, again, $\theta^N (\lambda)$ takes the smallest possible value. \\[-3mm]

This will be the aim of the next Chapter \ref{6.5}.

%\colorbox{yellow}{\tbl Takes the least possible value\tbr ? \tbl the value is the least possible\tbr ? }

\subsection{The cardinals $\boldsymbol{\lambda \geq \kappa_\gamma := \sup \{\kappa_\eta\ | \ 0 < \eta < \gamma\}}$.} \label{6.5}

%We will show that for all cardinals $\lambda \geq \kappa_\gamma$, the values $\theta^N (\lambda)$ are the least possible: 

Let $\alpha_\gamma := \sup \{\alpha_\eta\ | \ 0 < \eta < \gamma\}$, and consider a cardinal $\lambda \geq \kappa_\gamma$. We want to show that $\theta^N (\lambda)$ takes the smallest possible value $\alpha (\lambda)$, defined as follows: 

\begin{itemize} \item In the case that $cf \,\alpha_\gamma = \omega$, we set $\alpha (\lambda) = \max\{ \alpha_\gamma^{\plus \plus}, \lambda^{\plus \plus}\}$. \item In the case that $\alpha_\gamma = \alpha^\plus$ for some $\alpha$ with $cf \,\alpha = \omega$, we set $\alpha (\lambda) = \max\{ \alpha_\gamma^{\plus}, \lambda^{\plus \plus}\}$. \item In other cases, we set $\alpha (\lambda) := \max\{ \alpha_\gamma, \lambda^{\plus \plus}\}$. \end{itemize} 

Then by our remarks from Chapter 2, it follows that indeed, $\theta^N (\lambda) \geq \alpha (\lambda)$ holds for all $\lambda \geq \kappa_\gamma$. \\[-2mm]

%\colorbox{red}{Folgt das wirklich aus dem Bemerkungen? NACHSEHEN!} \\[-2mm]

First, we 
%consider a cardinal $\lambda \geq \kappa_\gamma$ with 
assume that \[\boldsymbol{\alpha (\lambda) > \alpha_\gamma}.\] It remains to prove that there is no surjective function $f: \powerset^N (\lambda) \rightarrow \alpha (\lambda)$ in $N$. \\[-3mm]

We start with the following observation (again, we use that $V \vDash GCH$):

\begin{lem} \label{prescardlambda} Let $\lambda \geq \kappa_\gamma$ with $\alpha(\lambda) > \alpha_\gamma$. Then $\m{P}$ preserves cardinals $\geq \alpha (\lambda)$. \end{lem}

\begin{proof} For every $p \in \m{P}$, $p = (p_\ast, (p^\sigma_i, a^\sigma_i)_{\sigma < \gamma, i< \alpha_\sigma}, (p^\sigma)_{\sigma < \gamma})$, there are

\begin{itemize} \item $\leq \kappa_\gamma^\plus$-many possibilities for $p_\ast$, \item $\leq \alpha_\gamma^{\aleph_0}$-many possibilities for the countable support of $(p^\sigma_i, a^\sigma_i)_{\sigma < \gamma, i < \alpha_\sigma}$, \item $\leq \kappa_\gamma^\plus$-many possibilities for $(p^\sigma_i, a^\sigma_i)_{\sigma < \gamma, i < \alpha_\sigma}$ when the support is given. 

%\item $\leq |\gamma|^\omega \leq \kappa_\gamma^\plus$-many possibilities for the (countable) support of $(p^\sigma)_{\sigma < \gamma}$ \item $\leq (\alpha_\gamma^\plus)^\omega = \alpha_\gamma^\plus$-many possibilities for $(p^\sigma)_{\sigma}$ when the support is given. 

\end{itemize} 

In the case that \textbf{ $\boldsymbol{\gamma}$ is a limit ordinal}, it follows by the strict monotonicity of the sequence $(\alpha_\sigma \ | \ 0 < \sigma < \gamma)$ that $\alpha_\sigma < \alpha_\gamma$ holds for all $0 < \sigma < \gamma$. Hence, for any $\sigma \in Succ$, the forcing notion $P^\sigma$ has cardinality $\leq \alpha_\sigma^\plus \leq \alpha_\gamma$; and it follows by countable support that we have $\leq |\gamma|^{\aleph_0}\, \cdot\, \alpha_\gamma^{\aleph_0} = \alpha_\gamma^{\aleph_0}$-many possibilities for $(p^\sigma)_{\sigma < \gamma}$. Hence, the forcing $\m{P}$ has cardinality $\leq \kappa_\gamma^\plus\, \cdot \, \alpha_\gamma^{\aleph_0} \leq \lambda^\plus\, \cdot\, \alpha_\gamma^{\aleph_0}$. If $cf \,\alpha_\gamma > \omega$, $GCH$ gives $|\m{P}| \leq \lambda^\plus\, \cdot\, \alpha_\gamma$, and $\alpha (\lambda) = \max \{ \lambda^{\plus \plus}, \alpha_\gamma^\plus\}$. Hence, $\m{P}$ preserves cardinals $\geq \alpha(\lambda)$ as desired. If $cf \,\alpha_\gamma = \omega$, then $\alpha (\lambda) =  \max \{ \lambda^{\plus \plus}, \alpha_\gamma^{\plus \plus}\} \geq |\m{P}|^\plus$; and again, it follows that $\m{P}$ preserves cardinals $\geq \alpha (\lambda)$.

%, $|\m{P}| < \alpha (\lambda)$ follows. \\[-3mm]

It remains to consider the case that \textbf{$\boldsymbol{\gamma = \ol{\gamma} + 1}$ is a successor ordinal}. Then our sequences $(\kappa_\sigma \ | \ 0 < \sigma < \gamma)$ $=$ $(\kappa_\sigma\ | \ 0 < \sigma \leq \ol{\gamma})$ and $(\alpha_\sigma \ | \ 0 < \sigma < \gamma)$ $=$ $(\alpha_\sigma \ | \ 0 < \sigma \leq \ol{\gamma})$ have a maximal element, and $\kappa_\gamma = \kappa_{\ol{\gamma}}$, $\alpha_\gamma = \alpha_{\ol{\gamma}}$.

If $\boldsymbol{\ol{\gamma} \in Lim}$, i.e. $\kappa_{\ol{\gamma}}$ is a limit cardinal, it follows that for any $\sigma \in Succ$, we have $\sigma < \ol{\gamma}$; hence, $\alpha_\sigma^\plus \leq \alpha_{\ol{\gamma}} = \alpha_\gamma$. This gives $|\m{P}| \leq \kappa_\gamma^\plus\, \cdot\, \alpha_\gamma^{\aleph_0} \leq \lambda^\plus\, \cdot\, \alpha_\gamma^{\aleph_0}$ as before, and $\alpha (\lambda) \geq |\m{P}|^\plus$.

 If $\boldsymbol{\ol{\gamma} \in Succ}$, i.e. $\kappa_{\ol{\gamma}}$ is a successor cardinal, then $P^{\ol{\gamma}}$ has to be treated separately. We factor $\m{P} \cong \m{P}^\prime\, \times\, P^{\ol{\gamma}}$ with $\m{P}^\prime := \{ \, (p_\ast, (p^\sigma_i, a^\sigma_i)_{\sigma \leq \ol{\gamma} \, , \, i < \alpha_\sigma}, (p^\sigma)_{\sigma < \ol{\gamma}})\ | \ p \in \m{P}\, \}$. Then $P^{\ol{\gamma}}$ preserves cardinals, and $\m{P}^\prime$ has cardinality $\leq (\lambda^\plus)^V\, \cdot\, (\alpha_\gamma^{\aleph_0})^V$ as before
%$ < \alpha (\lambda)$
(in $V$, and hence, also in any $P^{\ol{\gamma}}$-generic extension); where $\alpha (\lambda) \geq |\m{P}^\prime|^\plus$. Hence, the forcing $\m{P} \cong \m{P}^\prime\, \times\, P^{\ol{\gamma}}$ preserves cardinals $\geq \alpha (\lambda)$ as desired.

%\colorbox{yellow}{ACHTUNG - die Bezeichnung $\wt{\m{P}}$ is bestimmt schon irgendwo vergeben!}

\end{proof}

Now, we assume towards a contradiction that there was a surjective function $f: \powerset^N (\lambda) \rightarrow \alpha(\lambda)$ in $N$. \\ By the \textit{Approximation Lemma \ref{approx}}, it follows that any $X \in N$, $X \subseteq \lambda$, is contained in an intermediate generic extension $V[\prod_{m < \omega} G^{\sigma_m}_{i_m}]$, with a sequence $( (\sigma_m, i_m)\ | \ m < \omega)$ of pairwise distinct pairs in $V$ such that $0 < \sigma_m < \gamma$, $i_m < \alpha_{\sigma_m}$ for all $m < \omega$. Denote by $M$ the collection of these $( (\sigma_m, i_m)\ | \ m < \omega)$. Then $|M| \leq \alpha_\gamma^{\aleph_0}$ in $V$; and $\alpha_\gamma^{\aleph_0} < \alpha (\lambda)$ as argued before. \\ The product $\prod_m P^{\sigma_m}$ preserves cardinals and the $GCH$. Hence, it follows that in any generic extension $V[\prod_m G^{\sigma_m}_{i_m}]$, there is an injection $\chi: \powerset(\lambda) \hookrightarrow (\lambda^\plus)^V$. Now, one can argue as in Chapter 6.2 D), and define in $V[G]$ a set $\wt{\powerset} (\lambda) \supseteq \powerset^N (\lambda)$ with an injection $\iota: \wt{\powerset} (\lambda) \hookrightarrow (\lambda^\plus)^V\, \cdot\, \alpha_\gamma$, or $\iota: \wt{\powerset} (\lambda) \hookrightarrow (\lambda^\plus)^V\, \cdot\, (\alpha_\gamma^\plus)^V$ in the case that $\alpha (\lambda) \geq (\alpha_\gamma^{ \plus \plus})^V$. Together with Lemma \ref{prescardlambda}, this gives the desired contradiction. \\[-2mm]

Thus, we have shown that in the case that $\alpha (\lambda) > \alpha_\gamma$, there can not be a surjective function $f: \powerset (\lambda) \rightarrow \alpha (\lambda)$ in $N$. \\[-2mm]

It remains to consider the case that \[\boldsymbol{\alpha (\lambda) = \alpha_\gamma}.\] Then $\lambda^\plus < \alpha_\gamma$, $cf \,\alpha_\gamma > \omega$; and if $\alpha_\gamma = \alpha^\plus$ for some $\alpha$, then $cf \,\alpha > \omega$. \\[-3mm] 

Assume towards a contradiction that there was a surjective function $f: \powerset^N (\lambda) \rightarrow \alpha (\lambda)$ in $N$, $f = \dot{f}^G$ with $\pi \ol{f}^{D_\pi} = \ol{f}^{D_\pi}$ for all $\pi \in A$ with $[\pi]$ contained in the intersection \[\bigcap_{m < \omega} Fix (\eta_m, i_m)\, \cap\, \bigcap_{m < \omega} H^{\lambda_m}_{k_m} \hspace*{5,5cm} (A_{\dot{f}}).\] Similary as before, we take $\wt{\beta} < \alpha (\lambda)$ \textit{large enough for the intersection} $(A_{\dot{f}})$, i.e. $\wt{\beta} > \lambda^\plus$ with $\wt{\beta} > \sup \{i_m\ | \ m < \omega\}\, \cup \, \sup \{k_m\ | \ m < \omega\}$ (this is possible, since $cf\,\alpha (\lambda) > \omega)$. Let $\beta := \wt{\beta}\, \plus\, \kappa_\gamma^\plus$ (addition of ordinals). Then $\kappa_\gamma^\plus \leq \lambda^\plus < \alpha (\lambda)$ gives $\lambda^\plus < \beta < \alpha (\lambda)$. \\[-2mm]

By the \textit{Approximation Lemma} \ref{approx}, it follows as in Proposition \ref{Xsubseteqkappaeta} that any $X \in N$, $X \subseteq \lambda$, is contained in an intermediate generic extension $V[\prod_m G_\ast (a_m)\, \times\, \prod_m G^{\ol{\sigma}_m}_{\ol{i}_m}]$, where $( (a_m)_{m < \omega}, $ $(\ol{\sigma}_m, \ol{i}_m)_{m < \omega})$ is a \textit{good pair for $\kappa_\gamma$}, i.e. \begin{itemize} \item $ (a_m\ | \ m < \omega)$ is a sequence of pairwise disjoint subsets of $\kappa_\gamma$, such that for all $\kappa_{\ol{\nu}, \ol{\j}} < \kappa_\gamma$ and $m < \omega$, it follows that $|a_m\, \cap \, [\kappa_{\ol{\nu}, \ol{\j}}, \kappa_{\ol{\nu}, \ol{\j} + 1})| = 1$, \item for all $m < \omega$, we have $\ol{\sigma}_m \in Succ$, $0 < \sigma < \gamma$, and $\ol{i}_m < \alpha_{\ol{\sigma}_m}$, \item if $m \neq m^\prime$, then  $(\ol{\sigma}_m, \ol{i}_m) \neq (\ol{\sigma}_{m^\prime}, \ol{i}_{m^\prime})$.\end{itemize}

As before, let \[f^\beta := \Big\{ \, (X, \alpha) \in f\ \ \big| \ \ \exists\, ( (a_m)_{m < \omega}, (\ol{\sigma}_m, \ol{i}_m)_{m < \omega}) \mbox{ \textit{good pair for} } \kappa_\gamma \ : \ (\forall m\ \; \ol{i}_m < \beta)\: \wedge\] \[\wedge \: \exists\, \dot{X} \in \Name\big((\ol{P}^\gamma)^\omega\, \times\, \prod_m P^{\ol{\sigma}_m}\big)\ \: X = \dot{X}^{\prod_m G_\ast (a_m)\, \times\, \prod_m G^{\ol{\sigma}_m}_{\ol{i}_m}}\, \Big\}.\] 

First, we assume towards a contradiction that $\boldsymbol{f^\beta: \dom f^\beta \rightarrow \alpha (\lambda)}$ \textbf{is surjective}.

\subsubsection*{A) + B) Constructing $\boldsymbol{\m{P}^\beta}$ and capturing $\boldsymbol{f^\beta}$.}

%This time, we define a forcing notion $\wt{\m{P}}^\beta$: 
For a condition $p \in \m{P}$, let \[p^\beta := (\, p_\ast, (p^\sigma_i, a^\sigma_i)_{\sigma \in \Lim \, , \, i < \beta}, (p^\sigma\, \uhr\, (\beta\, \times\, \dom_x p^\sigma))_{\sigma \in \Succ}, X_p\, ), \] where \[X_p := \bigcup \, \{a^\sigma_i\ | \ \sigma \in \Lim\, , \, i \geq \beta\}. \] We define $\m{P}^\beta$ and $G^\beta$ as before. \\[-3mm]

The construction of $(f^\beta)^\prime \in V[G^\beta]$ and the isomorphism argument for $f^\beta = (f^\beta)^\prime$ are as in Chapter 6.2 and 6.3; except that when constructing the isomorphism $\pi$, we now have to set $G_{\pi_0} (\nu, j) := F_{\pi_0} (\nu, j)$ for all $\kappa_{\nu, j} < \kappa_\gamma$.

\subsubsection*{C) $\boldsymbol{\m{P}^\beta}$ preserves cardinals $\boldsymbol{\geq \alpha (\lambda) = \alpha_\gamma = \sup \{\alpha_\eta\ | \ 0 < \eta < \gamma\}}.$}

The arguments here are similar as in Chapter 6.2 C): If $\alpha_\gamma > |\beta|^\plus$, it follows as in Lemma \ref{cardforc} that $|\m{P}^\beta| \leq \kappa_\gamma^\plus\, \cdot\, |\beta|^{\aleph_0} \leq \lambda^\plus\, \cdot |\beta|^\plus < \alpha_\gamma$. In the case that $\alpha_\gamma = |\beta|^\plus$, it follows that $cf \,|\beta| > \omega$, and as before, we distinguish several cases, whether $\gamma$ is a limit ordinal or $\gamma = \ol{\gamma} + 1$, and in the latter case, whether $\ol{\gamma} \in Lim$, or $\ol{\gamma} \in Succ$ with $\ol{\gamma} = \ol{\ol{\gamma}} + 1$ etc. We separate $P^{\ol{\gamma}}$ (or $P^{\ol{\ol{\gamma}}}$, or both), 
%\textit{and} $P^{\ol{\gamma}}$), 
and obtain that $P^{\ol{\gamma}}$ (or $P^{\ol{\ol{\gamma}}}$, or the product $P^{\ol{\ol{\gamma}}}\, \times\, P^{\ol{\gamma}}$) preserves cardinals, while the remaining forcing is now sufficiently small.

%\colorbox{red}{TO DO: Injektionen $\gamma$ ersetzen durch $\psi$!} marking

%\colorbox{yellow}{Genügt das?}

%\colorbox{yellow}{ACHTUNG - $\alpha_\gamma$ oder $\alpha (\lambda)$ schreiben?}

\subsubsection*{D) A set $\boldsymbol{\wt{\powerset} (\lambda) \supseteq \dom f^\beta}$ with an injection $\boldsymbol{\iota: \wt{\powerset} (\lambda) \hookrightarrow \lambda^\plus\, \cdot |\beta|^{\aleph_0}}$.}

As in Chapter \ref{6.2} D) and \ref{6.4} D), we construct in $V[G^\beta]$ a set $\wt{\powerset} (\lambda) \supseteq \dom f^\beta$ with an injection $\iota: \wt{\powerset} (\lambda) \hookrightarrow (\lambda^\plus)^V\, \cdot (|\beta|^\plus)^V$ in the case that $(|\beta|^\plus)^V < \alpha (\lambda)$, and an injection $\iota: \wt{\powerset} (\lambda) \hookrightarrow (\lambda^\plus)^V\, \cdot\, |\beta|^V$ in the case that $(|\beta|^\plus)^V = \alpha (\lambda)$. \\[-2mm] 
%(then $(cf \,|\beta|)^V > \omega$ and $(|\beta|^{\aleph_0})^V = |\beta|^V$). \\[-2mm]

Together with Chapter 6.5 B) and 6.5 C), this gives the desired contradiction. \\[-2mm]

Hence, it follows that there must be some $\alpha < \alpha (\lambda)$ with $\alpha \notin \rg f^\beta$.

\subsubsection*{E) We use an isomorphism argument and obtain a contradiction.}
With the same isomorphism argument as in Chapter 6.2 E), it follows that $\theta^N (\lambda) = \alpha (\lambda)$ as desired. \\[-2mm]

Thus, we have shown that also for all cardinals $\lambda \geq \kappa_\gamma$, $\theta^N (\lambda)$ takes the smallest possible value. \\[-2mm]

This was the last step in the proof of our main theorem.

%TO DO: Nummer 1 für das Theorem weg?

\section {Discussion and Remarks} \label{discussion} 

%One could ask whether it is possible to turn $\m{P}$ into a class forcing.. WEGLASSEN?

Our result confirms Shelah's thesis from \cite[p.2]{pcfwc} that in $ZF\, \plus\, DC\, \plus\, AX_4$ it is \tbl better\tbr\, to look at $\big(\, [\kappa]^{\aleph_0}\ | \ \kappa \mbox{ a cardinal}\, \big)$ rather than $\big(\, \powerset (\kappa)\ | \ \kappa \mbox{ a cardinal}\, \big)$, in the sense that by what we have shown, the only restrictions that can be imposed on the $\theta$-function on a set of cardinals in $ZF\, \plus \, DC\, \plus \, AX_4$, are the obvious ones. \\[-2mm]

From Theorem 1 in \cite{AK}, it follows that increasing the surjective size of $[\aleph_\omega]^{\aleph_0}$ together with preserving $GCH$ below $\aleph_\omega$ requires a measurable cardinal, which again underlines how differently $\powerset(\aleph_\omega)$ and $[\aleph_\omega]^{\aleph_0}$ behave without the \textit{Axiom of Choice}.
%in the $\neg AC$-context. 
In further investigation, one might look at the cardinal arithmetic in our constructed model, such as possible (surjective) sizes of $(\kappa^\lambda\ | \ \kappa \mbox{  a cardinal })$ for $\lambda < < \kappa$. \\

Another question one might ask is, under what circumstances certain $\neg AC$-large cardinal properties are preserved in our symmetric extension $N$. As an example, we will now briefly look at the question whether an inaccessible cardinal $\kappa$ in the ground model could remain inaccessible in $N$. \\[-3mm]

The notion of inaccessibility in $ZFC$ reads as follows: \textit{A cardinal $\kappa$ is \textbf{inaccessible} (or \textbf{strongly inaccessible}) if $\kappa$ is regular and $2^\lambda < \kappa$ holds for all cardinals $\lambda < \kappa$.}

Hence, it can not be transferred directly to the $\neg AC$-context, since the power sets $\powerset(\lambda)$ for $\lambda < \kappa$ are ususally not well-ordered.
In \cite[Chapter 2]{ID2}, one can find several characterizations how inaccessibility can be defined in $ZF$:
\begin{Definitionwn}[\cite{ID2}] \begin{itemize}\item A regular uncountable cardinal $\kappa$ is \textbf{$\mathbf{i}$-inaccessible} if for all $\lambda < \kappa$, there is an ordinal $\alpha < \kappa$ with an injection $\iota: \powerset(\lambda) \hookrightarrow \alpha$. 
\item A regular uncountable cardinal $\kappa$ is \textbf{$\mathbf{v}$-inaccessible} if for all $\lambda < \kappa$, there is no surjection $s: V_\lambda \rightarrow \kappa$. \item A regular uncountable cardinal $\kappa$ is \textbf{$\mathbf{\ol{s}}$-inaccessible} if for all $\lambda < \kappa$, there is no surjection $s: \powerset (\lambda) \rightarrow \kappa$. \end{itemize}
%\item An uncountable cardinal $\kappa$ is \textit{$\ol{i}$-inaccessible} if $\kappa$ is regular and for any $\lambda < \kappa$, there is no injection $\iota: \kappa \hookrightarrow \powerset (\lambda)$. 
\end{Definitionwn}

%\colorbox{yellow}{FRAGE: Auch $\ol{i}$-inaccessible oder $s$-inaccessible ergänzen?}

Note that $i$-inaccessibility implies $v$-inaccessibility, and $v$-inaccessibility implies $\ol{s}$-inaccessibility. It is not difficult to see that a cardinal $\kappa$ is $v$-inaccessible if and only if $V_\kappa$ is a model of second-order $ZF$ (see \cite[Chapter 2]{ID2}). \\[-2mm]

Let now $\kappa$ be an inaccessible cardinal in the setting of our theorem: $V \vDash ZFC\, \plus\, GCH$ with sequences $(\kappa_\eta\ | \ 0 < \eta < \gamma)$, $(\alpha_\eta\ | \ 0 < \eta < \gamma)$ as before, with the additional property that for all $\kappa_\eta < \kappa$, it follows that also $\alpha_\eta < \kappa$. Then by construction, it follows that $\kappa$ is $\ol{s}$-inaccessible in $N$, while $i$-inaccessibility of $\kappa$ is out of reach, since we do not have our power set well-ordered. 

The question remains whether $\kappa$ is $v$-inaccessible in $N$. By induction over $\lambda$, we could prove (using several isomorphism and factoring arguments, similar to those in Chapter \ref{6.2}):

\begin{Proposition} Let $V$ be a ground model of $ZFC\, \plus\, GCH$ with $\gamma \in \Ord$, and sequences of uncountable cardinals $(\kappa_\eta\ | \ 0 < \eta < \gamma)$ and $(\alpha_\eta\ | \ 0 < \eta < \gamma)$ with the properties listed in Chapter $2$. Moreover, let $N \supseteq V$ denote the symmetric extension constructed  in Chapter $3$, $4$ and $5$.

If $\kappa$ is an inaccessible cardinal in $V$ with the property that for all $\kappa_\eta < \kappa$ it follows that $\alpha_\eta < \kappa$, then $\kappa$ is {\upshape $v$-inaccessible} in $N$: For any $\lambda < \kappa$, there is no surjective function $s: V_\lambda^N \rightarrow \kappa$ in $N$. \end{Proposition}

%\colorbox{red}{TO DO: Man müsste NACHPRÜFEN, ob das überhaupt stimmt!}

In our inductive proof, we show that for any cardinal $\lambda < \kappa$, there exists some $\kappa_{\ol{\nu}, \ol{\j} } (\lambda) < \kappa$ and a cardinal $\beta_\lambda < \kappa$ with an injection $\iota: V_\lambda^N \hookrightarrow \beta_\lambda$ in $V[G\, \uhr\, \kappa_{\ol{\nu}, \ol{\j}} (\lambda)]$. \\

Our next remark is about the following requirement that we put on the sequences $(\kappa_\eta\ | \ 0 < \eta < \gamma)$, $(\alpha_\eta\ | \ 0 <\eta < \gamma)$: \[\forall\, \eta\ (\alpha_\eta = \alpha^\plus\, \rightarrow \, \cf \alpha > \omega).\]

We mentioned in Chapter \ref{the theorem} that this condition is necessary if we want $\forall \eta\ \theta^N (\kappa_\eta) = \alpha_\eta$ under $AX_4$.

Moreover, we proved in Chapter 2 that whenever we start from a ground model $V \vDash ZFC\, \plus\, GCH$, and construct a symmetric extension $N \supseteq V$ with $N \vDash ZF\, \plus\, DC$ such that $V$ and $N$ have the same cardinals and cofinalities, the following holds: \begin{center} \textit{If $\kappa$, $\alpha \in \Card$ with $\theta^N (\kappa) = \alpha^\plus$, then $\cf^N (\alpha) > \omega$.} \end{center} 

One could ask what happens if we drop the requirement that $N$ should extend a ground model $V \vDash ZFC\, \plus\, GCH$ cardinal-preservingly. 

Can there be any inner model $N \vDash ZF\, \plus\, DC$ with cardinals $\kappa$, $\alpha$ such that $\theta^N (\kappa) = \alpha^\plus$ and $cf^N (\alpha) = \omega$? \\[-2mm]

%whether there could be an inner??? model $N \vDash ZF\, \plus\, DC$ with $\theta^N (\kappa) = \alpha^\plus$ and $\cf \alpha = \omega$, if we do not assume the existence of a ground model $V$. \\[-2mm]

Let $s: 2^\kappa \rightarrow \alpha$ denote a surjective function in $N$. Then with $DC$, it follows that there is also a surjection $s_1: (2^\kappa)^\omega \rightarrow \alpha^\omega$ in $N$; and we also have a surjective function $s_0: 2^\kappa \rightarrow (2^\kappa)^\omega$. 

Recall that in Chapter 2, we then took a surjection $\wt{s}_2: (\alpha^\omega)^V \rightarrow (\alpha^\plus)^V$ from our ground model $V$, which gave a surjection $s_2: (\alpha^\omega)^N \rightarrow (\alpha^\plus)^N$ in $N$. Then $s_2\, \circ\, s_1\, \circ\, s_0: 2^\kappa \rightarrow \alpha^\plus$ was a surjective function in $N$; hence, $\theta^N (\kappa) \geq \alpha^{\plus \plus}$. \\[-2mm]

In a more general setting, where we can not refer to a ground model $V$, we try to use the constructible universe $L = L^N$ instead. Under the assumption that $0^\sharp$ does not exist, it follows by \textit{Jensen's Covering Theorem} (\cite{marginaliatts}) that $L$ does not differ drastically from $N$: In particular, $L$ and $N$ have the same successors of singular cardinals; so if $cf^N (\alpha) = \omega$, then $(\alpha^\plus)^L = (\alpha^\plus)^N$. \\[-3mm]

This gives the following lemma: 

%\colorbox{red}{ACHTUNG - das müsste man ZITIEREN! Jensen's Covering Theorem?}

\begin{Lemma} \label{zerosharpprop} Let $N$ be an inner model of $ZF\, \plus\, DC$ with $N \vDash$ \tbl\,$0^\sharp$ does not exist\tbr, and $\alpha \in \Card^N$ with $\cf^N (\alpha) = \omega$. Then there exists a surjective function $s_2: (\alpha^\omega)^N \rightarrow (\alpha^\plus)^N$ in $N$. \end{Lemma}

\begin{proof} Let $(\alpha_i\ | \ i < \omega)$ denote a strictly increasing sequence in $N$ which is cofinal in $\alpha$.
%with $\sup \{ \alpha_i\ | \ i < \omega\} = \alpha$. 
First, we construct in $N$ an injection $\iota: (2^\alpha)^L \hookrightarrow (\alpha^\omega)^N$, $\iota = \iota_2\, \circ\, \iota_1\, \circ\, \iota_0$, as follows: \begin{itemize} \item Let $\iota_0: (2^\alpha)^L \rightarrow \prod_{i < \omega} (2^{\alpha_i})^L$ denote the injection that maps any $g : \alpha \rightarrow 2$, $g \in L$, to the sequence of its restrictions $\big((g\, \uhr\, \alpha_i)\ | \ i < \omega\big)$. \item For any $i < \omega$, there is in $L$ an injection $\gamma: (2^{\alpha_i})^L \hookrightarrow (\alpha_i^\plus)^L$; so with $DC$ in $N$, we can choose a sequence of injective maps $(\gamma_i\ | \ i < \omega)$ such that $\gamma_i: (2^{\alpha_i})^L \hookrightarrow (\alpha_i^\plus)^L$ for all $i < \omega$. Then we define in $N$ an injection $\iota_1: \prod_{i < \omega} (2^{\alpha_i})^L \rightarrow \prod_{i < \omega} (\alpha_i^\plus)^L$ by setting $\iota_1 (X_i\ | \ i < \omega) := ( \gamma_i (X_i)\ | \ i < \omega)$. \item Finally, since $(\alpha_i^\plus)^L \leq (\alpha_i^\plus)^N < \alpha$ for all $i < \omega$, it follows that there is in $N$ an injective map $\iota_2: \prod_{i < \omega} (\alpha_i^\plus)^L \hookrightarrow (\alpha^\omega)^N$.

%ACHTUNG FRAGE: Geht es um $AC_\omega$ oder um $DC$? Würde $DC$ überhaupt folgen aus dem $\omega$-Abschluss von Forcing und Filter??
\end{itemize}
Thus, $\iota:= \iota_2\, \circ\, \iota_1\, \circ\, \iota_0: (2^\alpha)^L \hookrightarrow (\alpha^\omega)^N$ is an injection in $N$; which yields a surjection $s: (\alpha^\omega)^N \rightarrow (2^\alpha)^L$, or $\ol{s}: (\alpha^\omega)^N \rightarrow (\alpha^\plus)^L$. \\[-2mm]

%Now, \textit{Jensen's Covering Lemma} can be applied in $N$, 
Since we have assumed that $N \vDash$ \tbl $0^\sharp$ does not exist\tbr\,and $cf^N (\alpha) = \omega$, it follows by \textit{Jensen's Covering Lemma} in $N$ that $(\alpha^\plus)^L = (\alpha^\plus)^N$. \\[-3mm]

This gives our surjecion $s_2: (\alpha^\omega)^N \rightarrow (\alpha^\plus)^N$ in $N$ as desired.

%\colorbox{red}{HIER FEHLEN ZITATE!}

%Now, Jensen's Covering Lemma applied in $N$
\end{proof} 

%\colorbox{red}{TO DO: Zusammen durchgehen!!}

\begin{Corollary} Let $N$ be an inner model of $ZF\, \plus DC$ with $N \vDash$ \tbl\,$0^\sharp$ does not exist\tbr, and cardinals $\kappa$, $\alpha$ such that $\theta^N (\kappa) = \alpha^\plus$. Then $\cf^N (\alpha) > \omega$. \end{Corollary}

%Beweis weglassen?
\begin{proof} Let $s: 2^\kappa \rightarrow \alpha$ denote a surjective function in $N$, and assume towards a contradiction that $cf^N (\alpha) = \omega$. As mentioned before, we have surjections $s_0: 2^\kappa \rightarrow (2^\kappa)^\omega$ and $s_1: (2^\kappa)^\omega \rightarrow \alpha^\omega$. By the previous lemma, it follows that there is also a surjection $s_2: \alpha^\omega \rightarrow \alpha^\plus$ in $N$. Setting $s := s_2\, \circ\, s_1\, \circ\, s_0$, we obtain in $N$ a surjective function  $s: 2^\kappa \rightarrow \alpha^\plus$. Contradiction.\end{proof}

Thus, without large cardinal assumptions, it is not possible to achieve $\theta^N (\kappa) = \alpha^\plus$ for cardinals $\kappa$, $\alpha$ with $cf^N (\alpha) = \omega$. \\

%in our setting, where we wish to avoid large cardinal assumptions, we can not achieve $\theta^N (\kappa) = \alpha^\plus$ for $\cf^N \alpha = \omega$. 

%\colorbox{red}{FRAGE: näher auf $0^\sharp$ eingehen? Oder erst EVTL in der Arbeit?} \\

%\colorbox{red}{Zitieren, dass $0^\sharp$ sich zwischen Jonsson cardinal und weakly compact cardinal befindet?}

%\colorbox{yellow}{ACHTUNG - man sollte die selben Bezeichnungen nehmen wie in Kapitel 2!! NACHSEHEN!} \\[-2mm]

%there exists a model $N \vDash ZF\, \plus\, 

%\colorbox{red}{TO DO: den \tbl alten Text\tbr\, durchsehen bzgl. Discussion!}\\[-2mm]

Finally, we remark that our theorem gives a result 
%with this construction, we 
%only obtain results 
about possible behaviors of the $\theta$-function on a \textit{set} of uncountable cardinals. Unfortunately, a straightforward generalization of our forcing notion to ordinal-length sequences $(\kappa_\eta\ | \ \eta \in \Ord)$, $(\alpha_\eta\ | \ \eta \in \Ord)$ does not result in a $ZF$-model: \\[-3mm]

Denote by $\ol{\m{P}}$ the class forcing which canonically generalizes our forcing notion $\m{P}$ to sequences $(\kappa_\eta\ | \ \eta \in \Ord)$, $(\alpha_\eta\ | \ \eta \in \Ord)$ of ordinal length; denote by $\ol{G}$ a $V$-generic filter on $\ol{\m{P}}$, and $\ol{N} := V(\ol{G})$. 
Then $\ol{N} \nvDash$ $Power$ $Set$: 
%$\ol{N}$ the $V$-generic symmetric extension by the class forcing $\ol{\m{P}}$ generalizing our forcing notion $\m{P}$ to sequences $(\kappa_\eta\ | \ \eta \in \Ord)$, $(\alpha_\eta\ | \ \eta \in \Ord)$ of class length, then $N \nvDash $ $Power$ $Set$. 
%the Axiom of Power Set does not hold true in $\ol{N}$:
%us briefly outline why, for instance, $\powerset (\aleph_1)$ would not be a set in $\ol{N}$, if now $ 
Assume towards a contradiction that 
%$\ol{G}$ was a $V$-generic filter on $\ol{\m{P}}$, $\ol{N} := V(\ol{G})$, and $
$Z := \powerset^{\ol{N}} (\aleph_1) \in \ol{N}$.
% and denote by $\ol{G}$ 
Then there would be an ordinal $\gamma$ and a symmetric name $\dot{Z} \in HS\, \cap \, \Name (\ol{\m{P}}\, \uhr\, \gamma)$ with $Z = \dot{Z}^{\ol{G}\, \uhr\, \gamma}$, where $\ol{\m{P}}\, \uhr\, \gamma$ denotes the initial part of $\ol{\m{P}}$ up to $\kappa_\gamma$. Now, by an isomorphism argument similar as in the \textit{Approximation Lemma} \ref{approx}, one can show that any set $X \in \powerset^{\ol{N}} (\aleph_1)$ is contained in an intermediate generic extension $V[\prod_{m < \omega} G^{\eta_m}_{i_m}]$ with $\eta_m < \gamma$, $i_m < \alpha_{\eta_m}$ for all $m < \omega$. Consider $X := G^{\gamma + 1}_i\, \uhr\, \aleph_1$ for some $i < \alpha_{\gamma + 1}$. Then $X \in \powerset^{\ol{N}} (\aleph_1)$; hence, $ X = G^{\gamma + 1}_i\, \uhr\, \aleph_1\in V[\prod_{m < \omega} G^{\eta_m}_{i_m}]$ for a sequence $((\eta_m, i_m)\ | \ m < \omega)$, such that
%$\big((\eta_m, i_m)\ | \ m < \omega\big)$ with 
$\eta_m < \gamma$, $i_m < \alpha_{\eta_m}$ holds for all $m < \omega$.
But this is not possible, since $G^{\gamma + 1}_i$ is $V[\prod_{m < \omega} G^{\eta_m}_{i_m}]$-generic on $P^{\gamma + 1}$. \\[-2mm]

Broadly speaking, the point is that a class-sized version of our forcing construction never stops adding new subsets of $\aleph_1$ (or any other uncountable cardinal). Although we can try and keep control over the surjective size of $\powerset^N (\aleph_1)$, it is not possible to capture $\powerset^N (\aleph_1)$ in an appropriate set-sized intermediate generic extension; and it remains a future project to settle this problem and find a countably closed forcing notion that is also suitable for sequences $(\kappa_\eta\ | \ \eta \in \Ord)$, $(\alpha_\eta\ | \ \eta \in \Ord)$ of ordinal length.

\bibliographystyle{alpha}
\nocite{*}
\bibliography{Literaturverzeichnis}

\vspace*{3mm}

{\scshape{ \footnotesize Anne Fernengel, Mathematisches Institut, Rheinische Friedrich-Wilhelms-Universit\"at, \\[-1mm] Bonn, Germany }}\\[-0,5mm]
{\rmfamily \itshape \footnotesize E-Mail address: }{ \ttfamily \footnotesize anne@math.uni-bonn.de} \\[1mm]

{\scshape{ \footnotesize Peter Koepke, Mathematisches Institut, Rheinische Friedrich-Wilhelms-Universit\"at, \\[-1mm] Bonn, Germany }} \\[-0,5mm]
{\rmfamily \itshape \footnotesize E-Mail address:} { \ttfamily \footnotesize koepke@math.uni-bonn.de}

\end{document}